\documentclass[letterpaper]{amsart} 
\usepackage{amsthm}
\usepackage{mathtools,mathrsfs,amssymb,latexsym,graphics,graphicx,enumerate,cite}
\usepackage[mathscr]{eucal}
\usepackage{amsmath,amsfonts,amsthm,amssymb,mathrsfs,bbm,cite}
\usepackage{color,xcolor}
\usepackage{hyperref}
\usepackage[left=1in,right=1in,top=1in,bottom=1in]{geometry}
\numberwithin{equation}{section}

\newcommand{\rr}{\mathbb{R}}

\newcommand{\lan}{\langle}
\newcommand{\ran}{\rangle}
\newcommand{\be}{\begin{eqnarray*}}
\newcommand{\bel}{\begin{eqnarray}}
\newcommand{\ee}{\end{eqnarray*}}
\newcommand{\eel}{\end{eqnarray}}
\newcommand{\ba}{\begin{aligned}}
\newcommand{\ea}{\end{aligned}}
\newcommand{\de}{\Delta}
\newcommand{\al}{\alpha}
\newcommand{\na}{\nabla}
\newcommand{\ep}{\epsilon}

\newcommand{\pa}{\partial}
\newcommand{\wh}{\widehat}
\newcommand{\wt}{\widetilde}
\newcommand{\uu}{\mathcal{U}}
\newcommand{\nq}{{\neq}}
\newcommand{\te}{{\zeta}}
\newcommand{\Z}{{\mathcal{Z}}}

\newcommand{\paz}{\partial_{z}}
\newcommand{\lf}{\left}
\newcommand{\rg}{\right}
\newcommand{\palt}{{\partial_x^i \Gamma_y^j\partial_z^k}}

\newcommand{\gt}{{\Gamma_{y;t}^{ijk}}}
\newcommand{\gts}{{\Gamma_{y;t_\star}^{ijk}}}
\newcommand{\CC}{C_{2,\infty}}
\newcommand{\BB}{\mathcal{B}}
\newcommand{\cz}{{\lan  \mathfrak {C}\ran}}

\newcommand{\nz}{{\lan n\ran}}

\newcommand{\cc}{{ \mathfrak {C}}}
\newcommand{\Mass}{{ \mathfrak {M}}}
\newcommand{\mf}{\mathfrak}
\newcommand{\myb}[1]{{#1}}
\newcommand{\myr}[1]{{#1}}
\newcommand{\myc}[1]{}

\newcommand{\FM}{\mathbb{F}^{t_r;M}_{G,Q}}
\newcommand{\FMZ}{\mathbb{H}^{M}_G}
\newcommand{\DD}{\mathbb{D}^{t_r;M}_{G,Q}}
\newcommand{\VV}{\mathbb{V}_{ijk}}
\newcommand{\mH}{\mathfrak{H}_{ijk}}
\newcommand{\mG}{\mathfrak{G}}
\newtheorem{theorem}{Theorem}

\newtheorem{lem}{Lemma}
\newtheorem{pro}{Proposition}
\newtheorem{rmk}{Remark}
\newtheorem{remark}{Remark}

\numberwithin{theorem}{section}
\numberwithin{cor}{section}
\numberwithin{pro}{section}
\numberwithin{rmk}{section}
\numberwithin{lem}{section}

\newcommand{\norm}[1]{\left\| #1\right\|}

\newcommand\Torus{{\mathbb T}}

\newcommand{\dss}{\displaystyle}

\allowdisplaybreaks

\mathtoolsset{showonlyrefs=true}

\title[Time-dependent Flows and Their Applications]{Time-dependent Flows and Their Applications in Parabolic-parabolic Patlak-Keller-Segel Systems\\ Part I: Alternating Flows}
\date{}

\author{Siming He}
\address{Department of Mathematics, 
University of South Carolina, Columbia, SC, 29208}
\email{
siming@mailbox.sc.edu
}

\begin{document}
\maketitle
\begin{center}\large\emph{Dedicated to Eitan Tadmor on occasion of his 70th birthday}\end{center}
\begin{abstract}
We consider the three-dimensional parabolic-parabolic Patlak-Keller-Segel equations (PKS) subject to ambient flows. Without the ambient fluid flow, the equation is super-critical in three-dimension and has finite-time blow-up solutions with arbitrarily small $L^1$-mass.  In this study, we show that a family of time-dependent alternating shear flows, inspired by the clever ideas of Tarek Elgindi \cite{Elgindi20}, can suppress the chemotactic blow-up in these systems.   
\end{abstract}

\setcounter{tocdepth}{1}{\small\tableofcontents}

\section{Introduction}

\myc{\color{blue}{\bf General Summary}:

1) Theorem 1.1 is the main theorem;

2) Based on Theorem 1.3, we wrote another paper on suppression of blow-up with shear flows;

3) Section 2 is the general ideas. 

4) For the linear theory, we first derive $\nu^{1/3}$-enhanced dissipation on the $L^2$ level. And then we upgrade it to the higher ``Gliding regularity'' norms (Right after (2.3)). The idea to prove the enhanced dissipation is standard. The main challenge in treating the higher regularity spaces is to control commutators of the form $\nu[\Gamma^n,\pa_{yy}]$. The key lemma is Lemma  \ref{lem:cm_L_trm}.

5) To upgrade the idea to treat nonlinear problem, we need the functional \eqref{F_M}. Same challenges appear when we handle commutator and nonlinear terms. The key lemmas are  Lemma  \ref{lem:cm_L_trm}, \ref{lem:est_nNL},   \ref{lem:est_ctrm}.

6) The alternating construction is like a rubric cube.

}
\myc{In the introduction, we can also include the recent progress on the alternating flow generated anomalous dissipation and mixing flow.}
In this paper, we consider the parabolic-parabolic Patlak-Keller-Segel systems (PKS) on the three-dimensional torus, which model the chemotaxis phenomena in fluid flows:
\begin{align}\label{ppPKS_basic}
\lf\{\begin{array}{cc}\ba
\pa_t n+&U_A(t)\cdot\na n+\na \cdot( n\na \cc)=\de n, \\
\pa_t  \mathfrak C+&U_A(t)\cdot\na  \mathfrak C=\de  \mathfrak C+n-\overline{n}, \\
n(t&=0)=n_{\mathrm{in}},\quad  \mathfrak C(t=0)=\cc_{\mathrm{in}}.\ea \end{array}\rg.
\end{align}
\ifx\textcolor{red}{Today Tarek tells me that the following shear flow (?)
$u(t,y_1,y_2)=(\cos(y_2+\log(1+t)),0)$ have enhanced dissipation estimate:
\begin{align}
\|\rho_{\neq}(t)\|_2\leq C_{ED}\|\rho_{\neq;\mathrm{in}}\|_2e^{-\delta \nu^{1/3}t}, \quad \forall {t\in[0,\infty)}
\end{align}
decay. Therefore, we might be able to use this time dependent shear flow on the $\Torus^3$ to derive the $8\pi$ critical mass. By alternating these flows, we might get a flow suppressing two dimension $(u(t,y),0,0)$ and $(0,u(t,x),0)$. But the gradient is also growing fast? Is it possible to redo Section \ref{Sec:Linfty_est} with alternating flow? }
\fi
Here $n,  \mathfrak C$ denote the cell density and the chemical density, respectively. The first equation describes the time evolution of the cell density, incorporating processes such as the chemical-triggered aggregation, diffusion, and fluid transportation. The second equation describes the dynamics of the chemical density. We subtract the spatial average of cell density ($\overline{n}$) to normalize the chemical density equation. This adjustment has no impact on the cell density dynamics, as only the chemical gradient influences the $n$-equation. Additionally, we assume that the initial chemical density possesses a zero spatial average, denoted as $\overline{\cc_{\mathrm{in}}} =0$. Both equations involve strong fluid advection, characterized by time-dependent fluid vector field $U_A(t)$ that is divergence-free, with an amplitude denoted as $A := \|U_A\|_{L^\infty}$. Throughout the paper, we employ the notation $(x,y,z)$ to represent points on the domain $\mathbb{T}^3=(-\pi,\pi]^3$.

It is worth mentioning that if the chemical reaches equilibrium in a much faster time scale than the fluid transportation and nonlinear aggregation, one can derive the following  important variants of the equations \eqref{ppPKS_basic}, which are called the advective parabolic-elliptic PKS equations:

\begin{align} \label{pe-PKS}
\pa_t n+ U_A (t )\cdot\na n+\na &\cdot(n\na \cc )=\de n,\quad -\de \cc=n, 
 \quad n(t=0 )=n_{\mathrm{in}}.
\end{align}  
It is also worth mentioning that another way to model the chemotaxis phenomena in the fluid flow is to couple the PKS equation with the Navier-Stokes equation, the Stokes equation, or other types of fluid equations. The literature on this topic is vast, we refer the interested readers to the following papers, \cite{Lorz10,Lorz12,LiuLorz11,DuanLorzMarkowich10,FrancescoLorzMarkowich10, Winkler12,TaoWinkler,ChaeKangLee13,
KozonoMiuraSugiyama,Tuval05} and the references therein.

If there is no ambient fluid flows, i.e., $U_A(t)\equiv 0$, the equation is the classical parabolic-parabolic PKS equation. The PKS equations were first derived by C. Patlak \cite{Patlak}, E. Keller, and L. Segel \cite{KS}. The literature on the analysis of the classical parabolic-parabolic PKS equations and their variants  is large, and we refer the interested readers to the papers \cite{CarrapatosoMischler17,EganaMischler16,CalvezCorrias,Schweyer14,Winkler13, BlanchetEJDE06,BlanchetCarrilloMasmoudi08, Biler95,CorriasEscobedoMatos14} and the references within. We summarize the results on the blow-up and global regularity result of the classical parabolic-parabolic PKS equations here. In two-dimension, the total mass of cells $\mathfrak{M}:=\|n\|_{L^1}$ characterizes the long time behavior of the solution. If the total mass is strictly less than $8\pi$,  V. Calvez and L. Corrias showed that the solutions are globally regular, \cite{CalvezCorrias}. On the other hand, if the total mass is large enough, singularities  form in a finite time, see, R. Schweyer \cite{Schweyer14}. In dimension three, the parabolic-parabolic PKS equations become supercritical. The total conserved mass $\mathfrak{M}$ becomes a supercritical quantity and is not enough to derive sufficient regularity control over the solutions. In the classical paper \cite{Winkler13}, M. Winkler showed that there exist solutions,  which have arbitrary small masses, blow up in a finite time. 

If the ambient fluid flow is present, the long time dynamics of the PKS systems change. In a series of work initiated by \cite{KiselevXu15}, it was shown that by introducing passive fluid  flow into the system, \emph{mixing} and \emph{fast-spreading} effects of the  fluid flow regularize the long time dynamics of the PKS equations. These works are mainly focusing on the parabolic-elliptic PKS equations \eqref{pe-PKS}. In the paper \cite{KiselevXu15}, A. Kiselev and X. Xu showed that if the ambient fluid flow is relaxation enhancing, which is introduced in the seminal paper \cite{CKRZ08}, and the magnitude $A$ of the flow is large enough, then potential chemotactic blow-up ceases to exist. Their argument was simplified and generalized in recent work by G. Iyer, X. Xu, and A. Zlato\v s \cite{IyerXuZlatos}. In the paper \cite{BedrossianHe16}, J. Bedrossian and the author proved that the suppression of the blow-up effect persists if the ambient fluid flow is the simple shear flow. The above works dwell on the \emph{mixing induced enhanced dissipation} properties of the passive scalar equations (advection-diffusion equations). On the other hand, in the paper \cite{HeTadmor172}, E. Tadmor and the author showed that \emph{fast-spreading effect} of the hyperbolic flow $(-x,y)$ also has the potential to suppress the blow-up of the parabolic-elliptic PKS systems \eqref{pe-PKS}. More interestingly, in a forthcoming paper, we also observe that shear flows in the infinite long channel $\mathbb{T}\times \rr$ have the fast-spreading effect. In the advective-reaction-diffusion equation literature, it is referred to as the quenching effect, which is closely related to L. H\"ormander's hypoellipticity \cite{Hormander67}. For further references, see, e.g., P. Constantin, A. Kiselev and L. Ryzhik \cite{CKR00} and A. Kiselev and Zlato\v s \cite{KiselevZlatos}. 

On the contrary, the study of the fluid flow-induced regularization effect in the parabolic-parabolic system \eqref{ppPKS_basic} is limited. In the work \cite{He}, the author showed that the strictly monotone shear flows could suppress the chemotactic blow-up in two-dimension. Later, L. Zeng, Z. Zhang, and R. Zi extended this result to coupled Patlak-Keller-Segel-Navier-Stokes systems (\cite{ZengZhangZi21}). In both of these works, an additional smallness assumption on the initial chemical gradient $\na\cc_{\text{in}}$ is employed. In a recent work \cite{DengShiWang24}, the authors are able to prove the suppression of blow-up through Couette flow on $\rr^3$. To understand the new challenges, we highlight the differences between the parabolic-elliptic regime \eqref{pe-PKS} and the parabolic-parabolic regime \eqref{ppPKS_basic}. It is enough to focus on the dynamics of the chemical gradient $\na \cc$, which determines the aggregation nonlinearity. In the parabolic-elliptic regime, since the chemical gradient is determined through an elliptic type relation $\na c=\na (-\de)^{-1}n$, the strong fluid advection has little impact on the aggregation nonlinearity. As a result, it is easy to invoke various regularization mechanisms from fluid mechanics to stabilize the system. On the other hand, the chemical density $\cc$ in the parabolic-parabolic regime is governed by an advection-diffusion type equation. A strong fluid advection can destabilize the dynamics by creating fast transient growth in the chemical gradient. This destabilizing effect rules out most of the regularization mechanisms applicable in the parabolic-elliptic regime. In the papers \cite{He, ZengZhangZi21}, the smallness in the initial chemical gradient is needed to compensate for this destabilizing effect.

In this work, we prove the suppression of chemotactic blow-up for the $3$-dimensional parabolic-parabolic PKS equations.
\begin{theorem}\label{thm_main}
Consider the solutions $(n,\cc)$ to the equation \eqref{ppPKS_basic} subject to smooth initial data $n_{\mathrm{in}}\in  C^\infty(\Torus^3)$, $\cc_{\mathrm{in}}\in  C^\infty(\Torus^3)$. There exists a family of time-dependent flows $U_A\in L_t^\infty C^\infty_{x,y,z}$ such that the solutions  are globally smooth on the time horizon $[0,\infty)$. 
\end{theorem}
\myc{Check the dependence of $A_0$, the amplitude depends on the profile. Is the phrasing here OK?}

Our basic building blocks for the flow $U_A$ are a family of time-dependent alternating shear flows. We extend this result to the time-dependent shear flow case in a companion paper. 
\ifx
\begin{theorem}\label{thm_1}
Consider the solutions to the equation \eqref{ppPKS_basic} subject to initial data $n_{\mathrm{in}}\in   H^{M}(\Torus^3)$, $\cc_{\mathrm{in}}\in  H^{M+1}(\Torus^3),\, M\geq 3$. Further assume that {\color{blue} (Does $\ep$ depend on the shear profile $u$? or $\|\na u\|_{W^{M+3,\infty}}$?)}
\begin{align}\label{smll_c_nq}
\textcolor{blue}{\| \cc_{\mathrm{in} }\|_{ H^{M+1} (\mathbb{T}^3)}\leq \epsilon( \|n_{\mathrm{in}}\|_{H^{M}(\Torus^3)},u).}
\end{align}
Then there exists a family of time-dependent flows $u\in L_t^\infty C^\infty_{x,y,z}$ such that if their magnitudes $A$ are greater than $A_0(\|n_{\mathrm{in}}\|_{  H^{M}},u)$, then the solutions $(n,\cc)$ are globally smooth. 
\end{theorem}
\textcolor{red}{It is enough to gain smallness of the data at time $CA^{1/3+\te}$, i.e., $\|n-\overline{n}\|_{H^s}(CA^{1/3+\te})+\|c\|_{H^{s+1}}(CA^{1/3+\te})\leq \frac{1}{A^{...}}+C\ep$. Then we choose the $A^{-1}$ and $\ep$ small enough such that the solution is tiny at time $CA^{1/3+\te}$, after that time, we can withdraw all the flows and use small data argument to get that the solution decay to zero as time goes to infinity. }
\fi
\myc{\begin{rmk}
\color{red}I think with slightly more effort, we can upgrade the regularity of $u$ to $C^\infty_{t,x,y,z}$. Basically, we need to introduce an order one transition time layer between different phases. Since the time is short, the regularity norms are controlled.  
\end{rmk}}
\subsection{Sketch of the Proof}
To motivate the idea and highlight the challenges, we first present the blow-up mechanism of the Patlak-Keller-Segel type equations (\eqref{ppPKS_basic}, \eqref{pe-PKS}). Then we introduce the regularization effects induced by fluid advection. Finally, we highlight the obstacles in applying these regularization mechanisms in the parabolic-parabolic case \eqref{ppPKS_basic} and our ideas to address them. 

In the PKS type systems \eqref{ppPKS_basic} -  \eqref{pe-PKS}, there are two competing forces - the nonlinear aggregation ($\na \cdot( n\na c)$) and the diffusion ($\de n$). On the one hand, the cells aggregate to form Dirac singularities, while on the other hand, cell diffusion regularizes the dynamics. The solutions remain smooth when diffusion prevails over nonlinear aggregation (\cite{BlanchetEJDE06, CalvezCorrias,CalvezCarrillo06,Wei181} ). However, if the aggregation dominates, singularities can develop  (\cite{JagerLuckhaus92,BlanchetCarrilloMasmoudi08,Schweyer14, Winkler13, GhoulMasmoudi18,CollotGhoulMasmoudiNguyen19,CollotGhoulMasmoudiNguyen192} ). One natural approach to suppress the blow-up is to enhance the diffusion to counteract the nonlinear aggregation. This can be achieved by replacing the diffusion operator with porous media type diffusion, e.g., \cite{Blanchet09,BRB10}. Alternatively, the presence of external fluid flow can also achieve the same goal. The primary mechanism here is that strong fluid transportation creates fast oscillations in the cell density, thereby improving diffusion. This regularization effect of fluid flows, commonly referred to as the ``enhanced dissipation phenomena'' in the literature, is applicable to various fluid-related problems. For instance, in the study of hydrodynamic stability, the enhanced dissipation effect is crucial in deriving the sharp stability threshold associated with various shear flows, see, e.g., \cite{BMV14, BVW16, BGM15I,BGM15II,BGM15III,ChenLiWeiZhang18, IbrahimMaekawaMasmoudi19, WeiZhangZhao20, LiWeiZhang20, MasmoudiZhao20, MasmoudiZhao22,LiMasmoudiZhao22, Jia22,BedrossianHe19,CotiZelatiDolce20}. 
Furthermore, the enhanced dissipation phenomena find applications in a wide range of areas, ranging from plasma physics to mathematical biology, see, e.g., \cite{Bedrossian17, BedrossianCotiZelatiDolce22,CotiZelatiDolceFengMazzucato, CotiZelatiDietertGerardVaret22,FengFengIyerThiffeault20, AlbrittonOhm22,Shvydkoy21, HeKiselev21,GongHeKiselev21}.

To relate the system \eqref{ppPKS_basic} to the existing theory of enhanced dissipation, we first divide the equation \eqref{ppPKS_basic} by the amplitude of the flow $A=\|U_A\|_{L_{t,x,y,z}^\infty}$, and rescale time properly to obtain
\begin{align}\label{PKS_rsc}
\lf\{\begin{array}{cc}\ba
\pa_t n+&u_A \cdot\na n=\frac{1}{A}\de n-\frac{1}{A}\na \cdot(n\na \cc),\\
\pa_t  \mathfrak C+&u_A \cdot\na  \mathfrak C=\frac{1}{A}\de  \mathfrak C+\frac{1}A\lf(n-\overline{n}\rg),\quad u_A:=\frac{U_A}{\|U_A\|_{L^\infty}}, \\
n(t&=0)=n_{\mathrm{in}},\quad  \mathfrak C(t=0)=\cc_{\mathrm{in}}.\ea \end{array}\rg.
\end{align} 
Here we still use $t$ to denote the new time variable. In the large amplitude $A$ regime, we can view the equation \eqref{PKS_rsc} as a perturbation to the passive scalar equation  
\begin{align*}
\pa_t f+u \cdot \na f=\frac{1}{A}\de f. 
\end{align*}
For our purpose, it is enough to focus on the passive scalar equations on $\mathbb{T}^3$ subject to shear flows:
\begin{align}
\pa_t f+u(t,y)\pa_x f  =\frac{1}{A}\de f,\quad  f(t=0)=f_{\text{in}},\quad   0<A^{-1}\ll 1.\label{ps_intro_1}
\end{align}
To motivate the key regularizing mechanism, we decompose the solutions to \eqref{ps_intro_1} into the average in the shearing direction and the remainder: 
\begin{align}
\lan f\ran(y,z)=\frac{1}{|\Torus|}\int f(x,y,z)dx,\quad f_\nq(x,y,z) =f(x,y,z) -\lan f\ran(y,z).\label{shr_avg_rmd}
\end{align}
It is easy to see that the $x$-average $\lan f\ran$ solves a heat equation; hence the $x$-average dissipates on a time scale $\mathcal O(A)$. The time scale $\mathcal{O}(A)$ is long if the $A$ is large. On the other hand, under suitable assumptions, the remainder $f_\nq$ dissipates on a time scale much faster than $\mathcal{O}(A). $ This deviation in dissipative time scales, caused by fluid advection, is called the \emph{enhanced dissipation}, and it has attracted much attention in recent years. Most analyses on the enhanced dissipation phenomenon are carried out in a 2-dimensional setting but can be easily extended to a 3-dimensional one. We first consider the \emph{stationary} shear flow, i.e., $u(t,y)=u(y)$. If the shear flow profile  $u(y)$ has only finitely many nondegenerate critical points, then the flow is called nondegenerate shear flow. In the paper \cite{BCZ15}, J. Bedrossian and M. Coti Zelati showed that if the stationary shear flows are nondegenerate, then there exist positive constants $\delta, C $ such that the following estimate holds:
\begin{align}\label{ED_intro_1}
\|f _{\neq}(t)\|_{L^2}\leq C \|f _{\mathrm{in};\neq}\|_{L^2} e^{-\delta  {A^{-1/2}|\log A|^{ -2}}  {t}},\quad \forall t\geq 0.
\end{align}
If the parameter $A^{-1}$ is small, the dissipation time scale $\mathcal{O}(A^{ 1/2}|\log A|^2)$ is much shorter than the heat dissipation time scale $\mathcal{O}(A)$. 
In the paper, the authors constructed explicit hypocoercivity functionals in the spirit of C. Villani \cite{villani2009}, and showed that these functionals  decay with enhanced rate $\mathcal{O}(A^{-1/2})$. Similar estimates are derived for other degenerate shear flows. The result was improved in the paper \cite{Wei18}. By applying resolvent estimates and a Gearhart-Pr\"{u}ss type theorem, D. Wei  proved that 
\begin{align}\label{ED_intro_2}
\|f _{\neq}(t)\|_{L^2}\leq C \|f _{\mathrm{in};\neq}\|_{L^2} e^{-\delta A^{-1/2}t},\quad \forall t\geq 0.
\end{align}
In the paper \cite{CotiZelatiDrivas19}, M. Coti-Zelati and D. Drivas  applied stochastic methods to show that the $A^{-1/2}$-enhanced dissipation rate are sharp for non-degenerate stationary shear flows. 
We also refer the interested readers to the works by Tarek Elgindi,  M. Coti-Zelati and M. G. Delgadino \cite{ElgindiCotiZelatiDelgadino18}, Y. Feng and G. Iyer \cite{FengIyer19}, which derive the explicit relationship between the mixing effect of fluid flow and their enhanced dissipation rate. Recently, D. Albritton, R. Beekie, and M. Novack proved the estimate \eqref{ED_intro_2} on bounded channel \cite{AlbrittonBeekieNovack21}. The enhanced dissipation phenomena also appear in fractional dissipative systems, see, e.g., \cite{ElgindiCotiZelatiDelgadino18}, \cite{He21} and \cite{LiZhao21}. 

Note that the above enhance-dissipation estimate is sharp for  \emph{stationary} shear flows, which leaves open the question that whether one can improve the dissipation rate by relaxing the \emph{stationary} constraint. The first step to prove  Theorem \ref{thm_main} is to show that by introducing time dependency into the shear flow, the enhanced dissipation rate can be improved from $\mathcal{O}(A^{-1/2})$ to $\mathcal{O}(A^{-1/3})$. 
\begin{theorem}\label{thm:L_ED}
Consider the solutions $f_{\neq}$ to the passive scalar equations \eqref{ps_intro_1} subject  to shear flow $(u(t,y),0,0).$ There exists a family of shear flows $u_A\in C^\infty_{t,y}$ such that the following enhanced dissipation estimate is satisfied
\begin{align}
\| f_{\neq}(s+t)\|_{L ^2}\leq &C_0   \| f_{\neq}(s)\|_{L^2 }e^{-\delta_0 A^{-1/3}t},\label{ED_tdsh_intr}
\end{align}
for all positive time $s,t\geq 0$. The constants $\delta_0$ and $C_0$ depend on the shear $u_A$, and they are independent of $A$ and the solutions. Moreover, the spatial Sobolev norms of the velocity fields are bounded independent of $A$, i.e., $\|\pa_y u_A\|_{L^\infty_t W_y^{M,\infty}}\leq C_M,\,\forall M\in \mathbb{N}. $
\end{theorem}
\ifx\begin{rmk}
Before renormalizing in time, the estimate \eqref{ED_tdsh_intr} corresponds to the following
\begin{align*}
\|f_{\neq}(s+t)\|_{L ^2}\leq C \|f_{\neq}(s)\|_{L^2 }e^{-\delta A^{2/3} t},\quad \forall s, t\geq 0.
\end{align*}
\end{rmk}
\fi
\begin{rmk}
This construction is based on Tarek Elgindi's  logarithmic-shifted shear flows (\cite{Elgindi20} ). We reproduce all the details of his construction in Section \ref{Sec:Lnr}. Thanks to a ``rewinding'' procedure, the resulting shear flows $u_A$ have a mild dependence on the amplitude $A$. However, this dependence will not alter the spacial Sobolev norm of the shear. The construction is general in the sense that for most shear profile functions, we can design time-dependent flows that achieve the enhanced dissipation \eqref{ED_tdsh_intr}. 
\end{rmk}

{

Compared to existing enhanced dissipation flows in the literature, Theorem \ref{thm:L_ED} provides time-dependent shear flows on $\Torus^3$ that balance the enhanced dissipation and the transient growth of the passive scalar solutions. We first recall that there are many freedoms in choosing fluid flows to suppress the blow-ups in the parabolic-elliptic PKS system \eqref{pe-PKS}. For example, one can choose shear flows with an enhanced dissipation estimate \eqref{ED_intro_1} (\cite{BedrossianHe16} ) or flows with sufficiently short dissipation time (\cite{KiselevXu15,IyerXuZlatos,
BedrossianBlumenthalPunshonSmith192} ). However, the story changes drastically for coupled systems. Here, we provide a heuristic argument to show that the flows constructed in Theorem \ref{thm:L_ED} are optimal in a certain sense. Motivated by the parabolic-parabolic PKS systems \eqref{ppPKS_basic}, we introduce a toy model,
\begin{align*}
\pa_t \rho +u\cdot\na \rho=\frac{1}{A}\de \rho-\frac{1}{A}\de g, \quad \pa_t g+u\cdot\na g=\frac{1}{A}\de g.
\end{align*}
Here $\rho$ plays the role of the cell density, and $g$ is the chemical density, respectively. The linear forcing term $\dss-\frac{1}{A}\de g$ in the $\rho$-equation mimics the aggregation nonlinearity. We assume suitable average-free conditions on the data and focus on the growth of the solution $\rho$. It is enough to estimate the net contribution from the forcing $\dss-\frac{1}{A}\de g.$ 
We expect the higher derivatives of the solution $g$ to undergo transient growth, and enhanced dissipation, which can be summarized as follows 
\begin{align}
\|\de g(t) \|_\infty\leq \mathfrak{G}(t) \exp\{-\mathfrak{R} t\}. 
\end{align} 
We start by considering the stationary nondegenerate shear flows, which have enhanced dissipation rate $ \mathfrak R=\mathcal{O}{(A^{-1/2})}$ \eqref{ED_intro_2} and transient growth $ \mathfrak G(t)=\mathcal{O}(t^2)$. If we employ this type of flow, the contribution from the chemical might not be negligible, i.e.,  
\begin{align}
\frac{1}{A}\int_0^\infty \|\de g(t)\|_\infty dt\leq \frac{C}{A}\int_0^\infty t^2\exp\lf\{-\frac{t}{CA^{1/2}}\rg\}dt= CA^{1/2}. 
\end{align} 
Similarly, we can consider the stochastic enhanced dissipation flows constructed in \cite{BedrossianBlumenthalPunshonSmith192, BlumenthalCotiZelatiGvalani22}. Here,  the growth rate of the gradients and the enhanced dissipation rate are $ \mathfrak G(t)=e^{C_1 t},\quad  \mathfrak R=|\log A|^{-1}/C_2$. Hence the net contribution of the chemicals to the system can be large for $A\gg1 $,
\begin{align}
\frac{1}{A}\int_0^\infty\|\de g(t)\|_\infty dt\leq \frac{1}{A}\int_0^\infty \exp\lf\{C_1 t-\frac{ t}{C_2|\log A|}\rg\}dt=\infty.
\end{align}
For the time-dependent shear flows $u$ constructed in Theorem \ref{thm:L_ED}, the growth factor is $ \mathfrak G(t)=\mathcal{O}(t^2)$ and the enhanced dissipation rate $ \mathfrak R =\mathcal{O}(A^{-1/3})$. In this case, the gradient of $g$ has bounded contribution to the $\rho$-dynamics, i.e.,
\begin{align}
\frac{1}{A}\int_0^\infty \|\de g(t)\|_\infty dt\leq \frac{C}{A}\int_0^\infty t^2\exp\lf\{-\frac{t}{CA^{1/3}}\rg\}dt\leq C. 
\end{align}
In conclusion, constructing a smooth flow that balances the transient growth of gradients and enhanced dissipation is crucial for our analysis. Theorem \ref{thm:L_ED} achieves this goal.  

}


As a result of Theorem \ref{thm:L_ED}, if we introduce the strong shear flow $A(u_A(t,y),0,0)$ to the system, the remainder $n_\nq,\, \na \cc_\nq $ decay fast. Hence it is reasonable to expect that after a short amount of time, the remainders become small and the solutions become quasi-two-dimensional. 

This is the content of the next main theorem. 
\myb{
\begin{theorem}\label{thm:short_t}
Consider the solutions $(n,\cc)$ to the equation \eqref{PKS_rsc} initiated from the data $n_{\mathrm{in}}\in  H^\mathbb{M}(\Torus^3)$, $\cc_{\mathrm{in}}\in  H^{\mathbb{M}+1}(\Torus^3),\ \mathbb{M}\geq 5$. Define a parameter 
\begin{align}\label{te_M_chc}
\te(\mathbb{M})= \frac{1}{108(2+\mathbb{M})}.
\end{align} Further assume that the shear flows $u_A$ in the equation \eqref{PKS_rsc} are the ones constructed in Theorem \ref{thm:L_ED}. There exists a threshold $A_0=A_0(\|n_{\mathrm{in}}\|_{H^\mathbb{M}}, \|\cc_{\mathrm{in}}\|_{H^{\mathbb{M}+1}}, \|u_A\|_{L_t^\infty W^{\mathbb{M}+3,\infty}}, \mathbb{M})$ such that if $A\geq A_0$, then there exists a universal constant $C$ such that the following estimate holds at time instance $A^{1/3+\te(\mathbb{M})}$,
\begin{align}\label{est_short_t}
\|n_\nq(A^{1/3+\te})\|_{H^{\mathbb{M}-1}}^2+\|\cc_\nq(A^{1/3+\te})\|_{H^{\mathbb{M}}}^2\leq C \exp\lf\{-\frac{A^{\te}}{C}\rg\}.
\end{align} Moreover, the $x$-average is bounded as follows
\begin{align}\label{est_shrt_t_0}
\|\lan n\ran (A^{1/3+\te}) \|_{\myr{H^{\mathbb{M}-1}}}^2+\|\lan \cc\ran(A^{1/3+\te})\|_{\myr{H^{\mathbb{M}}}}^2\leq \mathfrak{B}(\|n_{\mathrm{in}}\|_{H^\mathbb{M}}, \|\cc_{\mathrm{in}}\|_{H^{\mathbb{M}+1}}, \|u_A\|_{L_t^\infty W^{\mathbb{M}+3,\infty}}, \mathbb{M}). 
\end{align} 
\end{theorem}
\begin{rmk}We highlight that we obtain a fast decay of the remainder in a rougher Sobolev space. On the other hand, obtaining enhanced dissipation for the top-order Sobolev norm is challenging.   
\end{rmk}}

The final step to prove Theorem \ref{thm_main} is understanding the long-time dynamics. If we continue to use the shear flows constructed in Theorem \ref{thm:L_ED}, the nonlinear aggregation will eventually kick in through the $x$-averages $\lan n\ran, \, \lan \cc\ran$-dynamics after the time-scale $\mathcal O(A)$. To get control over the solutions, one has to assume the subcritical mass constraint $\mathfrak{M}=\| n\|_1< 8\pi|\Torus|$. We address this case in a companion paper \cite{ElgindiHe22II}. 
In Theorem \ref{thm_main}, the total mass $\Mass$ is arbitrary, and one cannot expect that the $x$-averages $\lan n\ran, \, \lan \cc\ran$ stay bounded for all time. To overcome the nonlinearity effect, we introduce the last ingredient of the proof. It is an alternating construction of time-dependent flows from the paper \cite{HeKiselev21}. If one alternates the shear direction of the flow in a particular time scale, the enhanced dissipation can dampen all the information fast. By carefully implementing this idea, we can complete the proof.  It is worth mentioning that alternating shear flows have found applications in various field of fluid mechanics, see, e.g., \cite{BrueColomboCrippaDeLellisSorella22,BrueDeLellis23,JeongYoneda22,JeongYoneda21,HeKiselev21}.
\ifx
With the enhanced dissipation effect of shear flow  introduced, we are ready to lay out the battle plan. First we observe that the validity of the enhanced dissipation estimates \eqref{ED_intro_1} and \eqref{ED_intro_2} dwell on the constraint that the $x$-averages of the passive scalar solutions $f^\nu$ vanish at almost every $y$. To better understand this constraint, we consider passive scalar solutions which are homogeneous in the $x$-direction, i.e., $f (x,y)=f (y)$. Since these types of solutions also solve the one-dimensional heat equation, the fluid transport effect is vacuous and no enhanced dissipation is expected. Therefore, the $x$-average constraint in \eqref{ps_intro_1} is necessary. To capture this heterogeneous nature of the shear flow enhanced dissipation effect, it is natural to decompose the solution $n$ into 
the $x$-average part $\lan n\ran^x$ and the remainder $n_{\neq}$ as follows:
\begin{align} 
\lan n\ran^x=\frac{1}{|\Torus|}\int_{\Torus} n(x,y,z)dx,\quad n_{\nq}^x= n(x,y,z)-\lan n\ran^x.
\end{align} The detailed equations for each component will be presented in the next section. The main idea of the proof is that we can treat the evolution of the remainder $n_{\neq}$ as a perturbation to the passive scalar equation \eqref{ps_intro_1}. Since the solution to the passive scalar dissipates in a fast time scale, we expect that the remainder $n_{\neq}$ decays fast and will not alter the dynamics of the whole system. On the other hand, as the $x$-average $\lan n\ran^x$ is invariant under the shear flow, it solves a modified parabolic-parabolic PKS equation in dimension two. Since the total mass of the solution is a critical quantity in dimension two, the $8\pi$-mass constraint \eqref{mass_constraint} guarantees higher regularity estimation on the solution.   

The above idea works well in the parabolic-elliptic PKS equations. However, new challenges arize when we come to the parabolic-parabolic PKS equations. In the `paper \cite{He}, the author made the observation that effect of the strong shear on the system \eqref{ppPKS_basic} is two-fold. On the one hand, it enhances the diffusion effect of the cell density. On the other hand, it also triggers a strong growth of the chemical gradient, which in turn destabilizes the dynamics of the cell density through the aggregation nonlinearity. To overcome this destabilization effect, the smallness condition \eqref{smll_c_nq} is introduced. 
\fi
\ifx
\textcolor{red}{If we allow alternating shear flows, we can derive the following result
\begin{theorem}\label{thm_2}
Consider the solutions to the equation \eqref{ppPKS_basic} subject to smooth initial data $n_{\mathrm{in}}\in   W^{s,\infty}(\Torus^3)$, $c_{\mathrm{in}}\in  W^{s+1,\infty}(\Torus^3),\, s\geq 3$.  Further assume that 
\begin{align}\label{smallness_of_c_neq_thm_2}
{\| c_{\mathrm{in};\neq}\|_{ L_y^\infty W_x^{s+1,\infty} (\mathbb{T}^3)}\leq \epsilon( n_{\mathrm{in}},c_{\mathrm{in};0}).}
\end{align}
Then there exist an alternating shear flow and a constant $A_0(M,\eta,\|n_{\mathrm{in}}\|_{  W^{s,\infty}}, \|c_{\mathrm{in}}\|_{  W^{s+1,\infty}})$ such that if the magnitude $A$ of the fluid vector field is greater than $A_0$, i.e., $A\geq A_0$, then the solution $(n,c)$ is globally smooth. 
\end{theorem}
}\fi

In Section \ref{Sec:Bttl_Pln}, we lay down the structure of the proof.  

\subsection{Notation} Throughout the paper, the constants $C$ can change from line to line. 
Moreover, to avoid cumbersome notation, we allow the implicit constant $C$ to depend on the $L_t^\infty W_{x,y,z}^{M,\infty}$-norm of the velocity and regularity level $M$. We recall that the velocity field $u_A$ we construct has the property that $\|u_A\|_{L^\infty_t W^{M,\infty}_{x,y,z}}\leq  C_M, \ \forall M\in\mathbb{N}$. Hence this notation convention will not cause confusion. 

To avoid complicated notation, we will reuse the notion $T_i$. For the proof of each individual lemma, we are going to define different ``local'' quantities $T_i$'s. Once the proof is finished, the current $T_i$'s will no longer be in use.  


For $\iota\in\{x,y,z\}$, we use $\lan f\ran^\iota, \   f_{\nq}^\iota$ to denote the $\iota$-average and $\iota$-remainder:
\begin{align}\label{f_iota}
\lan f\ran^\iota:=\frac{1}{|\Torus|}\int_{\Torus} f(x,y,z)d\iota,\quad f_{\nq}^\iota(x,y,z):= f(x,y,z)-\lan f\ran^\iota.
\end{align}
The following double average notation (and its natural analogues) is also applied in the text:
\begin{align}\label{f_iota_2}
\lan \lan f\ran\ran^{x,y}(z):=\lan \lan f\ran^{x}\ran^y(z)=\frac{1}{|\Torus|^2}\iint f(x,y,z)dxdy.
\end{align}
The following vector field (and its natural analogues)  and multi-index notation are used:
\begin{align}
\Gamma_{y;t}=\pa_y+\int_0^t \pa_y u(T+s,y)ds\pa_x, \quad \Gamma_{y;t}^{ijk}:=\pa_x^i\Gamma_{y;t}^j\paz^k,\qquad |i,j,k|:=i+j+k.
\end{align}
The choice of the reference time $T$ will be specified in a case-by-case scenario. If the $t$ is clear from the text, we will also drop the `$;t$' in the subscript, i.e., $\Gamma_y,\, \Gamma_{y}^{ijk}$.

In Section \ref{Sec:Lnr} of the paper, we apply Fourier transformation in the $x$-variable or in the  $(x,z)$ variables. The  frequency variables corresponding $x$ and $z$ are denoted by $k$ and $\ell$, respectively. 
The notation $\widehat{(\cdot)}$ is used to denote the Fourier transform and the $(\cdot)^\vee$ is used to denote the inverse transform. If we consider the transformation only in the $x$-variable, the Fourier transform and its inverse has the following form
\begin{equation*}
\wh{f}_\al(y,z) := \frac{1}{2\pi}\int_{-\pi}^{\pi} e^{-\sqrt{-1}\al x} f(x,y,z) dx , \quad \check{g}(x,y,z) = \sum_{\al=-\infty}^\infty g_\al(y,z) e^{\sqrt{-1}\al x}.
\end{equation*}
The Fourier variables corresponding to $x,y,z$ are $\al, \beta,\gamma$, respectively. 
\ifx
Similarly, we can define Fourier transform in the $x,y_2$-variables. In that case, we use $\wh f_{k,\ell}(y_1)$ to represent the transformed functions. The notation $|\pa_x|^{1/2}$ should be understood as Fourier multipliers with symbol $|k|^{1/2}$, i.e.,
\begin{align} 
 M(\pa_x,\pa_{y_2})f(x,y_1,y_2)=\left(M(ik,i\ell)\widehat{f}(k,y_1,\ell)\right)^\vee.
\end{align}\fi
Recall the classical $L^p$ norms and Sobolev $H^m$ norms:
\begin{align}
\|f\|_{L^p}=&\|f\|_p=\left(\int |f|^p dV\right)^{1/p};\quad \|f\|_{L_t^q([0,T]; L^p )}=\left(\int_0^T\|f(t )\|_{L ^p}^qdt\right)^{1/q};\\
\|f\|_{H ^m}=&\left(\sum_{|i,j,k|=0}^{M}\|\pa_x^i\pa_y^j\pa_z^k f\|_{L^2}^2\right)^{1/2};\quad
\|f\|_{\dot H^m}=\left(\sum_{|i,j,k|=m}\|\pa_x^i\pa_y^j\pa_z^k f\|_{L ^2}^2\right)^{1/2}.
\end{align}

Throughout the paper, we use the notation $S_{t_r}^{t_r+\tau}$ to denote the semigroup corresponding to the passive scalar equation initiating from time $t_r$
\begin{align}
\pa_\tau \rho+u(t_r+\tau,y)\pa_x \rho =\frac{1}{A} \de \rho.
\end{align}
Adopting this notation, $S_{t_r}^{t_r+\tau} \rho_{\neq}$ is the solution to the passive scalar subject to the initial data $\rho_{\neq}(t_r)$ at time $t_r$. We will also use the notation $S_{t_r,t_r+\tau}^\al$ to denote the solution semigroup corresponding to the $\al$-by-$\al$ equation
\begin{align}
\pa_\tau \wh\rho_\al+u(t_r+\tau,y)i\al \wh\rho_\al =\frac{1}{A} (-|\al|^2+\pa_{yy}+\pa_{zz}) \wh\rho_\al, \quad \,\forall (y,z)\in\Torus ^2,\label{k-by-k-PS}
\end{align}
which is initiating from $t=t_r$.

The notation $t_r$ is the reference time and is defined in \eqref{dcmp}. 

\myc{\color{red}
Notations: $t$: time, \quad $s,t$ is the classical time, $t_\star$ is a starting time, $\tau$ is the increment from time $t_\star$;   $\iota\in\{x,y, z\}$, denotes the coordinate axis variables; $m$: regularity level, $M$: maximal regularity level; $f$ viscous solution to the passive scalar; $\eta$: inviscid solution to the passive scalar; $\te$: extra time factor in $A^{1/3+\te}$. To avoid confusion, $\Phi(t_r+\tau)$ only means $\Phi$ evaluated at time $t_r+\tau$. It is not $\Phi$ times $(t_r+\tau). $ }
 \section{Road Map}\label{Sec:Bttl_Pln}
To present the main ideas involved in the proof of Theorem \ref{thm_main}, we need several preparations. First, we prove the linear enhanced dissipation estimate for a special family of time-dependent shear flows. These time-dependent shear flows have optimized the balance between enhanced dissipation and gradient creation. Moreover, we can upgrade the linear estimates to include higher-order gliding regularity norms based on this balance. These preparations will be accomplished in subsection \ref{Sec_s:L}. Secondly, we develop the nonlinear theory for the system \eqref{PKS_rsc} in subsection \ref{Sec_s:nl}. A functional is introduced to characterize various components of the system in higher order gliding regularity norms. We can show that if the initial chemical gradient is moderately small, then nonlinear enhanced dissipation holds.
On the other hand, if the initial chemical gradient is large, the function is still bounded for a sufficiently long time. After these preparations, we introduce an alternating construction in subsection \ref{Sec_s:Alt} to show that the solutions to equation \eqref{PKS_rsc} are globally regular if the flow is strong enough.

\subsection{Linear Theory of Time-dependent Shear Flows}\label{Sec_s:L}
The starting point of the analysis is the enhanced dissipation phenomena induced by a special family of time dependent shear flows. We consider the passive scalar equations subject to small viscosity:
\begin{align}\label{PS_intr} 
\pa_t f_{\neq}+u_A(t,y)\pa_x f_{\neq} =\frac{1}{A} \de f_{\neq},\quad f_{\neq}(t=0)=f_{\text{in};\nq}, \quad \frac{1}{|\Torus|}\int_{\Torus}f_{\text{in};\neq}(x,y ,z)dx \equiv 0. 
\end{align} 

If we do the Fourier transform in the $(x,z)$-variables, we obtain that
\begin{align}\label{PS_Fourier} 
\pa_t \wh f_{\al,\gamma}(t,y)+u_A(t, y)i\al\wh f_{\al,\gamma}(t,y) =\frac{1}{A} (-|\al|^2+\pa_{yy}-|\gamma|^2 )\widehat f_{\al,\gamma}(t,y),\quad \al\neq 0.
\end{align}
Since the passive scalar equations  \eqref{PS_intr} is linear, we can consider each Fourier mode independently. Hence the theorem below directly implies Theorem \ref{thm:L_ED}.
\begin{theorem}\label{thm:L_ED_al} Consider solution $f_{\al,\gamma}$ to the equation  \eqref{PS_Fourier}. \myr{Given any shear profile $\uu\in C^\infty(\Torus)$, such that $\uu'$ is not identically zero.} There exist time dependent shear flows $u_A(t,y)=\Psi_A(t) \uu (y+\Phi_A(t))\in C^\infty_{t,y}$ such that the following enhanced dissipation estimate holds 
\begin{align}\label{ED_tdp_sh_a}
\|\wh f_{\al,\gamma}(s+t)\|_2\leq &C \|\wh f_{\al,\gamma}(s)\|_2e^{-\delta_0  \frac{1}{A^{1/3}}t}.
\end{align}
Here $\delta_0\in(0,1),\ C\geq 1$ are constants that depend only on the shear profile $\uu$. Moreover, the Sobolev norms of the velocity fields is bounded independent of $A$, i.e., $\|\pa_y u_A\|_{L^\infty_t W_y^{M,\infty}}\leq C_M,\,\forall M\in \mathbb{N}. $
\end{theorem} 
Let us present the main idea of the construction. 

The main obstacle for a smooth stationary shear flow $(u(y),0,0)$ on the torus to achieve the enhanced dissipation rate $A^{-1/3}$ is that its profile function $u$ always contains critical points. Near the critical points, the enhanced dissipation effect will slow down (\cite{CotiZelatiDrivas19}). By introducing time dependence, we can ensure that the critical point of the profile moves around and will not occupy a specific region for a long time. As a result, the flow can induce a faster dissipation in the time-average sense. 

Let us start by considering time-dependent shears taking the form  $(u(y+\log(1+t)),0,0)$, which are logarithmic-in-time shifts of general smooth functions $u\in C^\infty_y$. We first show that the $L^2$-norm of the solution to \eqref{PS_Fourier} decays by a fixed amount on the time interval $[0, A^{1/3}]$, see, e.g., Lemma \ref{lem_ED_1}. Hence, it is tempting to believe that the solutions decay exponentially. However, the logarithmic shifts of the shear flow profiles slow down as time progresses. Hence we will smoothly truncate the flows and restart them after an appropriately chosen time interval. This smooth rewinding procedure will not alter the Sobolev norm $\|\pa_y u_A\|_{L_t^\infty W^{M+3,\infty}}$. After the rewinding process, our proof is complete.   

\ifx
, we consider the $L^\infty$-norm of the solutions \eqref{PS_intr}. In our application, it turns out that an $L^\infty$-enhanced dissipation with sharp rate $\mathcal{O}(A^{-1/3})$ is desired. Classically, one can use the duality argument and Nash inequality, together with the $L^2$-estimate developed thusfar \eqref{ED_tdsh_intr} to derive the following (see, e.g., \cite{HeKiselev21}),
\begin{align}
\|f_\nq(s+t)\|_\infty\leq C\|f_\nq(s)\|_\infty e^{-\delta\frac{t}{A^{1/3}|\log A|^2}},\quad \forall s,t\geq 0.
\end{align}
However, it turns out that the logarithmic loss in the decay rate is not sufficient for our analysis. To derive the $L^\infty$ enhanced dissipation with sharp rate, we choose to establish enhanced dissipation for 

and then invoke the Sobolev embedding
\fi
The next object of study the higher $L^2$-based Sobolev norms. The starting point is the following observation. There exists a vector field that commute with the transport part of \eqref{PS_intr}, i.e.
\begin{align}\label{Gamma}
\Gamma_{y;t}:=\pa_y+\int_0^t \pa_y u(s,y)ds\pa_x,\quad [\Gamma_{y;t}, \pa_t +u\pa_x]=0.
\end{align}
As a result, the advection term ($u\pa_x$) in the equation \eqref{PS_intr} does not drive a transient growth of the $H^1$-norm induced by the $\Gamma_{y;t}$-vector field, $\|\Gamma_{y;t} f_\nq\|_2$. On the contrary, direct energy estimates yields that the canonical $H^1$-norm $\|f_\nq\|_{H^1}$ undergoes a transient linear growth ($\sim \mathcal{O}(t)$) thanks to the advection in  \eqref{PS_intr}.
In conclusion, Sobolev norms induced by the vector fields $\pa_x,\, \Gamma _{y;t},\, \pa_z$ are well-adapted to the passive scalar equations \eqref{PS_intr}. From now on, we call these norms ``gliding regularity norms'' as in the celebrated work \cite{MouhotVillani11}. \myr{To simplify the notation, we also use the following variants throughout the text
\begin{align}
\Gamma_{y;t}^{ijk}=\Gamma_y^{ijk}=\pa_x^i\Gamma_{y;t}^j\pa_z^k.
\end{align}}

An obstacle in analyzing these gliding regularity norms comes from the diffusion operator $\frac{1}{A}\de$. When one applies $\Gamma_{y;t}$-derivative to the equation \eqref{PS_intr}, the commutator term $\frac{1}{A}[\Gamma _y,\de]$ arises. This term involves growth of the form $\mathcal{O}(\frac{t^2}{A})$. Nevertheless, we can use the following time weight to control the  commutator terms, 
\begin{align}{
\Phi(t)=\frac{1}{1+ \frac{t^{3}}{A}  }}.\label{varphi}
\end{align}
Combining all these considerations, we obtain the norm
\begin{align}
\|f_\nq(t,\cdot)\|_{{ \mathcal{Z}}_{G,\Phi}^M}^2:=&{\sum_{|i,j,k|=0}^{M}G^{2i}\Phi(t)^{2j}\|\Gamma_{y;t}^{ijk} f_\nq(t,\cdot)\|_{L^2}^2}:=\sum_{i+j+k=0}^M G^{2i}\Phi(t)^{2j}\lf\|\pa_x^i\Gamma_{y;t}^{j}\pa_z^k f_\nq(t,\cdot)\rg\|_{L^2}^2.\label{Gliding_norm}
\end{align}We apply the notations $\Gamma_{y;t}^{ijk}:=\pa_x^i\Gamma_{y;t}^j\pa_z^k$, and $|i,j,k|:=i+j+k.$ Moreover, the argument of the time weight $\Phi(t)$ is \emph{always} the time variable of the target function $f_\nq(t,\cdot)$ throughout the paper.    
\myr{Here $G$ is a constant depending only on the norm $\|u_A\|_{L_t^\infty W^{M+2,\infty}}$ and  the regularity level $M$, but independent of $A$.} We note that the higher gliding regularity norm will slowly deplete over time. However, an enhanced dissipation estimate is enough to compensate for the decay. Later, we simplify the notation to $\mathcal{Z}^M$. The enhanced dissipation estimates of the gliding regularity norms are summarized in the following theorem.  

\begin{theorem}\label{thm:ED_gldrg}
Consider solutions to the equation \eqref{PS_intr}. There exists a threshold $G_M:=G_M(\| u_A\|_{L_t^\infty W^{M+2,\infty}})$ such that for any parameter $G$ that is above the threshold $G_M$, the following enhanced dissipation estimate holds 
 \begin{align}
\|f_\nq(s+t)\|_{ \mathcal{Z} ^{M}_{G,\Phi}}^2\ \leq& \ 2e^2 \ \|f_\nq(s)\|_{ \mathcal{Z}^{ M}_{G,\Phi}}^2\ \exp\left\{-\myr{2}\delta_{\mathcal Z}\frac{t}{A^{1/3}}\right\},   \quad  \forall \, s, t\in[0,\infty);\label{Gld_Rg_ED}\\
\delta_{\mathcal Z}:=&\delta_0\lf(\log (8 C_0 )\rg)^{-1}.\label{Chc_del_M0}
\end{align}
Here the parameters $\delta_0,\, C_0$ are defined in the  estimate  \eqref{ED_tdsh_intr}. \myc{\footnote{Moreover, the $\delta_{\mathcal Z}$ is in fact independent of the $M$. Can we use $\delta_\mathcal{Z}$ instead?}} 
\myr{Moreover, there exists a constant $C$, which depends only on the $M,\,\delta_{\mathcal Z}^{-1}$ such that  the following estimate holds
\begin{align}\label{Z_implc}
\sum_{|i,j,k|=0}^{M}G^{2i}\|\gt f_\nq(t)\|_{L^2}^2\leq C\sum_{|i,j,k|=0}^{M}G^{2i} \|\pa_x^i\pa_y^j\pa_z^k f_\nq(0)\|_{L^2}^2\exp\lf\{-\delta_{\mathcal Z} \frac{ t}{A^{1/3}}\rg\},\quad \forall t\geq 0.
\end{align}}  
\end{theorem}
\begin{rmk}
We observe that even though the $\mathcal{Z}$-norm depletes over time, it still grants us enough enhanced dissipation as specified in \eqref{Z_implc}. Another comment is that the parameter $\delta_{\mathcal Z}$ can be chosen independent of the regularity level $M. $
\end{rmk}
\myc{If we add the combinatorial weight in front of the norms, then we might control the combinatorial factors in the commutator. As a result, we might get enhanced dissipation of the Gevrey norm.}
This concludes the linear theory.\ifx
\begin{lem}\label{lem:ED_gld}Assume all conditions in Theorem \ref{thm:L_ED}. The following estimate holds
\begin{align}\label{l_ED_gld}
\|f_{\neq}(s+t)\|_{ \mathcal{Z}^M_{G,\Phi}}^2\leq C\|f_{\nq}(s)\|_{ \mathcal{Z}^M_{G,\Phi}}^2 e^{-2\delta_{\mathcal Z}\frac{t}{A^{1/3}}},\quad \forall 0\leq s, t .
\end{align} 
Here the parameters $0<\delta_{\mathcal Z}\leq \delta_0, \, C$ are constants depending only on $u_A,\ M$. Moreover, they are independent of the parameter $A$. The parameter $G$ will be chosen depending on $\myr{\| u_A\|_{L_t^
\infty W^{M+2,\infty}}}, M$. 
\end{lem}\fi
\subsection{Nonlinear Theory for the Time-dependent Shear Flows}\label{Sec_s:nl}
In this subsection, we discuss the nonlinear theory associated with the parabolic-parabolic PKS system \eqref{PKS_rsc} subject to  time dependent shear flows. The goal is to prove Theorem \ref{thm:short_t}. 

First of all, we specify the main set of equations and their local existence theory. We rewrite the equation \eqref{PKS_rsc} as follows
\begin{subequations}\label{ppPKS}
\begin{align}
\label{ppPKS_n_1}\pa_t n+&u_A\pa_x n+\frac{1}{A}\na \cdot(n\na \cc )=\frac{1}{A}\de n,\\
\label{ppPKS_c_1}\pa_t \cc+&u_A\pa_x \cc=\frac{1}{A}\de \cc+\frac{1}{A}n-\frac{1}{A}\overline{n},\\
n(t=&0)=n_{\mathrm{in}} ,\quad \cc(t=0 )=\cc_{\mathrm{in}}, \quad (x,y,z)\in \mathbb{T}^3. 
\end{align}
\end{subequations}
The fluid velocity drift is normalized, i.e., $\dss \|u_A\|_{L^\infty_{t,x,y,z}}=1$. To prove Theorem \ref{thm:short_t}, it is enough to consider the time horizon $t\in [0, A^{1/3+\te}]$, where $\te=\te(M)$ is defined in \eqref{te_M_chc}. To rigorous capture the transient growth of the chemical gradient, we decompose the time horizon into two parts:
\begin{align}
\label{tm_hrzn_dcm}[0, A^{1/3+\te}]=[0, A^{1/3+\te/2})\cup[A^{1/3+\te/2}, A^{1/3+\te}]=:[0, T_h)\cup [T_h, A^{1/3+\te}]=:\mathcal{P}_{\text{gr}}\cup \mathcal{P}_{\text{dc}}. 
\end{align}
Here $\mathcal{P}_{\text{gr}}$ is the transient growth phase, and $\mathcal{P}_{\text{dc}}$ represents the decaying phase. 

Next, to characterize the enhanced dissipation, we decompose the solution $n, \, \cc$ into the $x$-average part $\lan n\ran,\,\lan \cc\ran$ and the remainder part $n_{\neq}, \, \cc_\nq$  \eqref{shr_avg_rmd}. Taking the $x$-average of  \eqref{ppPKS_n_1}, \eqref{ppPKS_c_1} yields the $(\lan n\ran,\, \lan \cc\ran)$-equations, i.e.,
\begin{subequations}\label{ppPKS_0_mode}
\begin{align}
\pa_t \lan n\ran+&\frac{1}{A}\na_{y,z} \cdot\left( \lan n\ran\na_{y,z} \lan \cc\ran\right)+\frac{1}{A}\left\lan\na_{y,z}\cdot(n_{\neq}\na_{y,z} \cc_{\neq})\right\ran=\frac{1}{A}\de_{y,z} \lan n\ran,\label{ppPKS_<n>}\\
\pa_t \lan \cc\ran=&\frac{1}{A}\de_{y,z} \lan \cc\ran+\frac{1}{A}\lan n\ran-\frac{1}{A}\overline{n},\label{ppPKS_<c>}\\
\lan n\ran(t=&0)=\lan n_{\mathrm{in}}\ran ,\quad \lan \cc\ran(t=0)=\lan \cc_{\mathrm{in}}\ran,\quad (y,z)\in\mathbb{T}^2. 
\end{align}
\end{subequations}
Then we observe that since the chemical equation is linear, it is possible to decompose the remainder $\cc_\nq$ as follows:
\begin{subequations}\label{ppPKS_neq}
\begin{align}\cc_{ \neq}(t_r+\tau)=c_{\neq}^{t_r}(\tau)+&d_{\neq}^{t_r}(\tau), \label{dcmp}\\
\pa_\tau c^{t_r}_{\neq}+u_A(t_r+\tau,y)\pa_x c_{\neq}^{t_r}=&\frac{1}{A}\de c_{\neq}^{t_r}+\frac{1}{A}n_{\neq},\quad c_{\neq}^{t_r}(\tau=0)=0 ,\label{ppPKS_c_neq}\\
\pa_\tau d_{\neq}^{t_r}+u_A(t_r+\tau,y)\pa_x d_{\neq}^{t_r}=&\frac{1}{A}\de d_{\neq}^{t_r},\quad d_{\neq}^{t_r}(\tau=0)=\cc_{\neq}(t=t_r).\label{ppPKS_d_neq}
\end{align}
Here $t_r$ is a reference time taking values in $\{0, T_h\}$ \eqref{tm_hrzn_dcm}. If $t_r=0$, then the initial data for $d_\nq^{t_r}$ is $\cc_{\mathrm{in};\nq}$. In the text, we also use the notation $S_{t_r}^{t_r+\tau}\cc_{\nq}$ to represent the passive scalar solution $d_\nq^{t_r}(\tau)$ that initiated from $(t_r,\cc_{\nq}(t_r))$ \eqref{ppPKS_d_neq}. We view $d_\nq^{t_r}(\tau)$ as the main part of the chemical remainder and $c_\nq^{t_r}(\tau)$ as the deviation. Both of these components undergo transient growth. However, thanks to the zero initial condition, the deviation $c_\nq^{t_r}$ is small. Here we emphasize that the definition of $c_\nq^{t_r}$ and $d_\nq^{t_r}$ is sensitive to the reference time $t_r$, and we only choose $t_r\in\{0, T_h\}$ \eqref{tm_hrzn_dcm}. 
\myc{\footnote{\bf Is it reasonable to consider the Prop. 2.2 and Prop. 2.3  only?}
}
Finally, the equation for the remainders $n_{\neq}$ reads as follows
\begin{align}
\pa_t n_{\neq}+u_A \pa_x n_{\neq}+&\frac{1}{A}\na_{y,z} \cdot( n_{\neq}\na_{y,z} \lan \cc\ran)+\frac{1}{A}\na \cdot( \lan n\ran\na \cc_{\neq} )+\frac{1}{A}\na\cdot(n_{\neq}\na \cc_{\neq})_{\neq}=\frac{1}{A}\de n_{\neq},\label{ppPKS_n_neq}\ \
n_{\neq}(t=0) =n_{\mathrm{in};\neq}. 
\end{align}
\end{subequations}
First of all, we present the following local well-posedness result, which can be proven through standard argument. \myc{(? Check!?)}
\begin{theorem}\label{thm:lcl_exst}
Consider solutions $(n,\cc)$ to the  equation \eqref{ppPKS} subject to initial data $n_{\mathrm{in}}\in H^{M}(\Torus^3),\, \cc_{\mathrm{in}}\in H^{M+1}(\Torus^3),\, M\geq 3$ and regular flow $u\in L_t^\infty W^{M+3,\infty}_{x,y,z}$. There exists a small constant $T_\varepsilon(\|n_{\mathrm{in}}\|_{H^M},\|\cc_{\mathrm{in}}\|_{H^{M+1}})$ such that the unique solution exists on the time interval  $[0,T_\varepsilon]$. 
\end{theorem}

Next, we specify the norms that we use to measure the solutions. Motivated by the functional introduced in the linear setting, we consider the following coupled functional for the nonlinear system \eqref{ppPKS_neq}:
\begin{align}\label{F_M} \qquad
&\mathbb{F}^{t_r;M}_{G,Q}[t_r+\tau,n_\nq,\cc_\nq]\\
&=\sum_{|i,j,k|=0}^{M}\Phi^{2j} G^{4+2i}\ \lf\|\ \pa_x^i\Gamma_{y;t_r+\tau}^j\pa_z^k \ n_\nq(t_r+\tau)\ \rg\|_{L^2}^2\\
&\quad+\sum_{|i,j,k|=0}^{M+1}\Phi^{2j\myr{+2}}G^{2i}\ (A^{2/3}\mathbbm{1}_{j<M+1}+\mathbbm{1}_{j=M+1})\ \lf\|\ \pa_x^i\Gamma_{y;t_r+\tau}^j\pa_z^k\  \lf[ \ \cc_\nq(t_r+\tau)-S_{t_r}^{t_r+\tau}(\cc_{\nq}(t_r))\ \rg]\rg \|_{L^2}^2\\
&\quad+ \myr{Q^2}\sum_{|i,j,k|=0}^{M+1}\lf \|\ \pa_x^i\Gamma_{y;t_r}^j\pa_z^k \ \cc_{\nq}(t_r)\ \rg\|^2_{L^{2} }\ \myr{\exp\lf\{- \frac{  \delta_{\mathcal Z}}{2A^{1/3}}\tau\rg\}},\quad \Phi=\Phi(t_r+\tau).
\end{align}\myc{
\footnote{Previous Functional:
\begin{align}
F_M[n_\nq,\cc_\nq]=&\sum_{|i,j,k|=0}^{M}\Phi^{2j} G^{4+2i}\|\pa_x^i\Gamma_{y;t}^j\pa_z^k n_\nq\|_{L^2}^2\\
&+\sum_{|i,j,k|=0}^{M+1}\Phi^{2j\myr{+2}}G^{2i}(A^{2/3}\mathbbm{1}_{j<M+1}+\mathbbm{1}_{j=M+1})\|\pa_x^i\Gamma_{y;t}^j\pa_z^k c_\nq\|_{L^2}^2+ \frac{1}{\ep ^2}\|\cc_{\nq}(0)\|^2_{H^{M+1} }e^{-\delta_{\mathcal Z} \frac{t}{2A^{1/3}}}.
\end{align}}} 
Here we make several comments concerning the functional. The reference time is $t_r$, which takes values at two time instances, i.e., $t_r=0$ and $t_r=T_h$ \eqref{tm_hrzn_dcm}. The parameter $\tau\geq 0$ is the time increment. If the reference time $t_r$ is zero, $\tau$ is nothing but the current time $t$.   \myc{Check the dependence of $G$. It depends on $M, \delta_{\mathcal Z}/\delta_{\mathcal Z}, \|u_A\|_{L^\infty_t W^{M+3/+4?,\infty}}?$}
Here the parameter $G\geq1$ will be chosen depending on $\{\| u_A\|_{L_t^\infty W^{M+3,\infty}_y}, M\}$ and the parameter $Q\geq1$ depends on the initial conditions $\|n_{\mathrm{in}}\|_{H^M},\ \|\lan \cc_{\mathrm{in}}\ran\|_{H^{M+1}}$, i.e., $Q=Q(\|n_{\mathrm{in}}\|_{H^M},\|\lan \cc_{\mathrm{in}}\ran\|_{H^{M+1}})$. The parameter $\delta_{\mathcal Z}$ is defined in Theorem  \ref{thm:ED_gldrg}. 
Similar to the $\mathcal{Z}$-norm, a fast decay of the functional $\FM$ corresponds to the enhanced dissipation \eqref{Z_implc}. Moreover, thanks to the extra $A^{2/3}$-weight in the chemical  component, boundedness of the functional can be translated to smallness of the chemical deviation $c_\nq^{tr}=\cc_\nq-S_{t_r}^{t_r+\tau}\cc_\nq$ in lower order gliding regularity spaces.

Finally, we specify the scheme to prove Theorem \ref{thm:short_t}. We consider two variants of the prototype functional $\mathbb{F}$, i.e., 
\begin{align}
\mathbb{H}_{G}^{\mathbb{M}+1}[t,n_\nq,\cc_\nq]\ :=\ &\mathbb{F}^{t_r=0;\mathbb{M}+1}_{G,1}[t,n_\nq,\cc_\nq],\quad t\in \mathcal{P}_{\mathrm{gr}}=[0, T_h=A^{1/3+\zeta/2}]\label{defn_H};\\
\mathbb{L}_{G}^{\mathbb{M}}[t,n_\nq,\cc_\nq]\ :=\ &\myr{\mathbb{F}^{t_r=T_h;\mathbb{M}}_{G,A^{1/4}}[t,n_\nq,\cc_\nq]},\quad t\in \mathcal{P}_{\mathrm{dc}}=[A^{1/3+\zeta/2}, A^{1/3+\zeta}]\label{defn_L}.
\end{align}
The transient growth phase $\mathcal{P}_{\text{gr}}$ and the decaying phase $\mathcal{P}_{\text{dc}}$ are defined in \eqref{tm_hrzn_dcm}. In the transient growth phase, we propagate the functional $\mathbb{H}$. The boundedness of $\mathbb{H}$ is translated to smallness of the chemical deviation $c_\nq^{t_r=0}=\cc_\nq-S_{0}^{t}\cc_{\mathrm{in};\nq}$ \eqref{dcmp} in a lower order regularity space. This smallness, when combined with the linear enhanced dissipation \eqref{Z_implc} of the passive scalar solution $S_0^t\cc_{\mathrm{in};\nq}$, yields that the chemical gradient $\na\cc_\nq$ is small at the transition time $T_h$. In the decaying phase $\mathcal{P}_{\text{dc}}$, we capitalize the smallness of $\na\cc_\nq(T_h)$ into the enhanced dissipation of the lower norm $\mathbb{L}$. We explicitly spell out these heurisitics in the following two propositions. 

\noindent
\textbf{Growth Phase $\mathcal{P}_{\mathrm{gr}}$: Propagation of the higher regularity norm $\mathbb{H}$. }
In the growth phase, we apply energy method to prove the following proposition. 
\begin{pro}\label{pro:prp_FM}
Consider the solutions to the equation \eqref{ppPKS}  initiated from initial data $n_{\mathrm{in}}\in H^{M}(\Torus^3),\, \cc_{\mathrm{in}}\in H^{M+1}(\Torus^3),\, M\geq 5$. Assume that the ambient shear flow $u_A$ is the one defined in Theorem \ref{thm:L_ED}. Recall the definition of $ \te\ \eqref{te_M_chc},\,T_h=A^{1/3+\zeta/2} \ \eqref{tm_hrzn_dcm}$, and consider the functional $\mathbb{H}$ \eqref{defn_H}. 
There exist thresholds $G_{\mathbb{H}},\, \, A_\mathbb{H} $ such that if the following constraints are satisfied, 
\begin{align}G\geq& G_\mathbb{H}(\|u_A\|_{L_t^\infty W_y^{M+3}}, M),
\quad A\geq A_\mathbb{H}(M, \myr{\|u_A\|_{L_t^\infty W_y^{M+3}}, G},n_{\mathrm{in}},\cc_{\mathrm{in}}),
\end{align}
the following estimate holds for all $t\in[0, T_h]$,
\myb{
\begin{align}\label{est_pro}
{\mathbb{H}^{M}_{G}}[t]+\|\lan n\ran(t)\|_{H^M}^2+\|\lan \cc\ran(t)\|_{H^{M+1}}^2
\leq&  \mathfrak{B} \left(M,\|u_A\|_{L_t^\infty W^{M+3,\infty}},\delta_{\mathcal Z}^{-1},\|\cc_{\mathrm{in}} \|_{H^{M+1}},\| n_\mathrm{in}\|_{H^M}\right).
\end{align}}\myc{\footnote{Previous:
\myr{\begin{align}\FM[t]+&\|\lan n\ran(t)\|_{H^M}^2+\|\lan \cc\ran(t)\|_{H^{M+1}}^2\\
\leq &C(M,\|u_A\|_{L_t^\infty W^{M+3,\infty}},\delta_{\mathcal Z}^{-1},\|\cc_{\nq}(0)\|_{H^{M+1}})\left(\FM[0]+\|\lan n\ran(0)\|_{H^M}^2+\|\lan \cc\ran(0)\|_{H^{M+1}}^2\right),\quad \forall t\in[0, A^{1/3+\te}].
\end{align}}
\myr{Question: Why the expression is linear in terms of the initial data? Is it true? }
Here $\te>0$ is a small parameter depending on $\{\|\pa_yu\|_{L^\infty_y W^{M+3,\infty}}, \, M\}$ and $C$ is a universal constant.
\begin{align}\label{te_choice}
A^{12\te(2+M_0)}\leq A^{\frac{1}{9}}.
\end{align} }}
\end{pro}
\begin{rmk}Throughout the paper, the regularity index $M$ will be changing. But they will be smaller than the $\mathbb{M}$ in Theorem \ref{thm:short_t} and Theorem \ref{thm:alt_1}.
\end{rmk}

\noindent
\textbf{Decaying Phase $\mathcal{P}_{\mathrm{dc}}$: Enhanced dissipation of   the lower regularity norm $\mathbb{L}$.} 

To prove the enhanced dissipation of $ \mathbb{L}$, we use a bootstrap argument, see, e.g., \cite{BedrossianHe16}, \cite{He}. We set the reference time $t_r= T_h=A^{1/3+\te/2}$ \eqref{tm_hrzn_dcm}. Assume that $[t_r,T_\star]$ is the largest interval on which the following hypotheses hold:

\noindent
\begin{subequations}\label{Hypotheses}

\noindent
1) Remainders' enhanced dissipation estimates: 
\begin{align}
\mathbb{L}_{G}^{M}[t_r+\tau, n_\nq,\cc_\nq]\leq& 2C_{ED}\ \mathbb{L}_{G}^{M} [t_r,n_\nq,\cc_\nq]\ \exp\lf\{-\frac{2\delta \tau}{A^{1/3}}\rg\},\quad\forall\ t_r+\tau\in [t_r, T_\star];\label{HypED}
\end{align} 

\noindent
2) Uniform-in-time estimates of the $x$-averages:
\begin{align}
\|\lan n\ran\|_{L_t^\infty([t_r,T_\star];H_{y,z}^M)}\leq& 2\BB_{\lan n\ran;H^M};\label{Hyp_<n>}\\
\| \lan \cc\ran\|_{L_t^\infty([t_r,T_\star];H_{y,z}^{M+1})}\leq& 2 \BB_{ \lan \cc\ran;H^{M+1}}.\label{Hyp_<c>}
\end{align} 
\ifx
3)  Uniform-in-time $L^\infty$ estimates: 
\begin{align}
\|n\|_{L_t^\infty([0,T_\star];L^\infty)}\leq& 2C_{n;L^\infty};\label{Hyp_n_Linf}\\
\|\na c_{\neq}\|_{L_t^\infty([0,T_\star];L^\infty)}\leq &2 C_{\na c_{\neq};L^\infty};\label{Hyp_na_c_neq_Linf}
\end{align}\fi
\end{subequations}
\myr{The parameter $\delta$ is chosen depending only on the constants $\delta_0,\, C_0$ \eqref{ED_tdsh_intr} and $\delta_{\mathcal Z}$ \eqref{Chc_del_M0}, 
\begin{align}\label{Chc_del}
\delta=\delta(\delta_0,\, C_0,\,\delta_{\mathcal Z})\myc{?\lf(=\frac{\delta_{\mathcal Z}}{2\log(32 C_0^2)}\rg)}.
\end{align} \myc{If the $\delta_{\mathcal Z}$ is actually independent of $M$, we can replace $\delta_{\mathcal Z}$ by $\delta_{\mathcal Z}$.}}
By the local well-posedness of the equation \eqref{ppPKS} in $H^M,\, M\geq 3$ (Theorem \ref{thm:lcl_exst}), we have that the interval $[t_r,T_\star]$ is non-empty.   

The nonlinear enhanced dissipation is the consequence of the following proposition. We recall that $\mathbb{M}$ is the regularity level specified in Theorem \ref{thm:short_t}.
{
\begin{pro}\label{Pro:2}
Consider the solutions to the equation \eqref{ppPKS}  initiated from the initial data $n_{\mathrm{in}}\in H^{M}(\Torus^3),\, \cc_{\mathrm{in}}\in H^{M+1}(\Torus^3),\, 3\leq M\leq\mathbb{M}$. Assume that the ambient shear flow $u_A$ is the one defined in Theorem \ref{thm:L_ED},  and \myr{the following ``gluing'' constraint} is satisfied at time $t_r=T_h=A^{1/3+\te/2}$ \eqref{te_M_chc}, \eqref{tm_hrzn_dcm},
\begin{align} 
\label{Glu_con}\myr{\sum_{|i,j,k|=0}^{M}\Phi(T_h)^{2j}\|\pa_x^i\Gamma_{y;T_h}^{j}\pa_z^k n_\nq(T_h)\|_{L^2}^2+\sum_{|i,j,k|=0}^{M+1}A^{2/3}\myr{\Phi(T_h)^{2j+2}}\| \pa_x^i\Gamma_{y;T_h}^{j} \pa_z^k\cc_{\neq}(T_h)\|_{ L^2 (\mathbb{T}^3)}^2\leq \BB_1^2}.\quad 
\end{align}Here the bound  $\BB_1$ is independent of $A$ and $ \Gamma_{y;T_h} = \pa_y+\int_0^{T_h} \pa_y u_A(s,y) ds\pa_x$. \myc{We only use the classical $\mathcal{Z}$ norm in the $d_\nq^{T_h}$ part. But the bound in \eqref{est_pro} only gives us the bound with ``$\Phi^{2j+2}$'' in $c_\nq^{t_r=0}$ and $e^{-A^{\cdots}}$ in $d_\nq^{t_r=0}$. So our actual $d_\nq^{t_r=T_h}$ bound is like 
\begin{align}
\|d_{\nq}^{T_h}(\tau)\|_{\mathcal{Z}^{M+1}}^2\lesssim& \sum_{|i,j,k|=0}^{M+1}G^{2i}\Phi(T_h)^{2j}\|\Gamma_{y;T_h}^{ijk}\cc_{\nq}(T_h)\|_{L^2}^2\exp\lf\{-\frac{\delta_\mathcal{Z}}{A^{1/3}}\tau \rg\}\\
\lesssim&\frac{1}{\Phi^2(T_h)}\sum_{|i,j,k|=0}^{M+1}G^{2i}\Phi(T_h)^{2j+2}\|\Gamma_{y;T_h}^{ijk}\cc_{\nq}(T_h)\|_{L^2}^2\exp\lf\{-\frac{\delta_\mathcal{Z}}{A^{1/3}}\tau \rg\}\\
\lesssim&\frac{A^{3\te}}{A^{2/3}}\mathcal{B}_1^2 \exp\lf\{-\frac{\delta_\mathcal{Z}}{A^{1/3}}\tau \rg\}.
\end{align}}

Let $[t_r,T_\star]$ be the maximal time interval  on which the hypotheses \eqref{Hypotheses} hold.  
{There exist thresholds $G_\mathbb{L},\, A_\mathbb{L}, \, Q_\mathbb{L}$ such that if the following constraints are satisfied, 
\begin{align}G\geq& G_\mathbb{L}(C_0, \delta_0, \delta_{\mathcal Z}^{-1}, \|u_A\|_{L_t^\infty W_y^{M+3}}, M),\\
\quad A\geq& A_\mathbb{L}(C_0, \delta_0,M,\delta_{\mathcal Z}^{-1}, \|u_A\|_{L_t^\infty W_y^{M+3}}, \myr{G},\BB_1, \|\lan n\ran (T_h)\|_{H^{M}},\|\lan\cc\ran (T_h)\|_{H^{M+1}}),
\end{align}}then the following stronger estimates can be developed:

\begin{subequations} 
1) Remainders' enhanced dissipation estimates: 
\begin{align}
\mathbb{L}_{G}^{M}[T_h+\tau]\leq C_{ED}\ \mathbb{L}_{G}^{M}[T_h]\ \exp\left\{-2\frac{\delta \tau}{A^{1/3}}\right\},\quad \forall \ T_h+\tau \in[T_h, T_\star].\label{ConED_glue}
\end{align}

2) Uniform-in-time estimates of the $x$-averages:
\begin{align}
\|\lan n\ran\|_{L_t^\infty([T_h,T_\star];H_{y,z}^M)}\leq&
\BB_{\lan n\ran;H^M}:= 2\|\lan n\ran(T_h)\|_{H^{M}}+2; \label{Con_<n>_glue}\\
\| \lan \cc\ran\|_{L_t^\infty([T_h,T_\star];H_{y,z}^{M+1})}\leq& \BB_{ \lan \cc\ran;H^{M+1}}:= \myr{\frac{1}{A^{1/5}}+\|\lan \cc(T_h)\ran \|_{H^{M+1}}} .\label{Con_<cc>_glu}
\end{align}
\end{subequations}
\ifx 
Here the constants are chosen as follows\myr{
\begin{align}
\BB_{\lan n\ran;H^M};\ 
\BB_{ \lan \cc\ran;H^{M+1}}=\BB_{ \lan \cc\ran;H^{M+1}}\myr{\left(\sum_{|i,j,k|=0}^{M}\|\pa_x^i\Gamma_{y;T_h}^j\pa_z^k n(T_h)\|_{L^2}^2, \sum_{|i,j,k|=0}^{M+1}\|\pa_x^i\pa_{y;T_h}^j\pa_z^k\cc(T_h)\|_{L^{2}}\right)}.
\end{align}}\fi
As a consequence of the bootstrap argument, the estimates \eqref{ConED_glue}, \eqref{Con_<n>_glue}, \eqref{Con_<cc>_glu} hold on the time horizon $[T_h, A^{1/3+\te}]$. 
\end{pro}}
\ifx
If the parameter $G$ is large compared to $\{\|\pa_y u_A\|_{L^\infty_t W^{M+3,\infty}_y}, M\}$, \, and the parameters $ A$ are chosen large depending on $\{G,\|\pa_y u_A\|_{L^\infty_t W^{M+3,\infty}_y}, M, \|n_{\mathrm{in}}\|_{H^M}, \|\cc_{\mathrm{in}}\|_{H^{M+1}}\}$,
Further define a variant of the functional $F_M$ \eqref{F_M} on the time interval $[T_h, A^{1/3+\te}]$:    
\begin{align}
\mathbb{G}_M [n_\nq,\cc_\nq](T_h+\tau)=&\sum_{|i,j,k|=0}^{M}\Phi^{2j} G^{4+2i}\|\pa_x^i\Gamma_{y;T_h+\tau}^j\pa_z^k n_\nq\|_{L^2}^2\\
&+\sum_{|i,j,k|=0}^{M+1}\Phi^{2j\myr{+2}}(T_h+\tau)G^{2i}(A^{2/3}\mathbbm{1}_{j<M+1}+\mathbbm{1}_{j=M+1})\|\pa_x^i\Gamma_{y;T_h+\tau}^j\pa_z^k (\cc_\nq-S_{T_h}^{T_h+\tau}\cc_\nq(T_h))\|_{L^2}^2\\
&+ \sum_{|i,j,k|=0}^{M+1}\|\pa_x^i\Gamma_{y;T_h}^j\pa_z^k \cc_{\nq}(T_h)\|^2_{L^{2} }e^{-\delta_{\mathcal Z} \frac{t}{2A^{1/3}}}.\label{G_M}
\end{align}
{\bf In the alternating process, the gluing at $A^{1/3+\te/2}$ is a problem, we need the $1/A^{2/3-}$ to help? We have $\int_0^{A^{1/3+\te}}\Phi_A^{-2M}(t)\frac{t^2}{A^{5/3}}e^{-(t-A^{1/3+\te/2})_+}dt<A^{-\varepsilon}?$}
\myr{There is an ambiguity in the definition of $F_M$ here! The $F_M$ implicitly means that we decompose the $\cc$ in a special way. But when we derive this gluing estimate, we are decomposing $\cc_\nq$ in a different fashion! So I suggest we define the following new functional on $[T_h:=A^{1/3+\te/2}, A^{1/3+\te}]$:
\begin{align}c_\nq (T_h+\tau)&:=\cc_\nq(T_h+\tau)-S_{T_h}^{T_h+\tau}\cc_\nq(T_h);\\
\mathbb{G}[n_\nq,\cc_\nq](T_h+\tau)=&\sum_{|i,j,k|=0}^{M}\Phi^{2j} G^{4+2i}\|\pa_x^i\Gamma_{y;T_h+\tau}^j\pa_z^k n_\nq\|_{L^2}^2\\
&+\sum_{|i,j,k|=0}^{M+1}\Phi^{2j\myr{+2}}G^{2i}(A^{2/3}\mathbbm{1}_{j<M+1}+\mathbbm{1}_{j=M+1})\|\pa_x^i\Gamma_{y; T_h+\tau}^j\pa_z^k c_\nq\|_{L^2}^2\\
&+ \sum_{|i,j,k|=0}^{M+1}\|\pa_x^i\Gamma_{y;T_h+\tau}^j\pa_z^k\cc_{\nq}(T_h)\|^2_{H^{M+1} }e^{-\delta_{\mathcal Z} \frac{\tau}{2A^{1/3}}}.
\end{align} }\fi
\myb{
The proof of Theorem \ref{thm:short_t} is completed once we combine the two propositions above:
\begin{proof}[Proof of Theorem \ref{thm:short_t}]\myc{I am confused by the $M, \mathbb{M}$ counts here. } We decompose the time horizon as in \eqref{tm_hrzn_dcm}. On the time interval $[0,T_h]$, we apply Proposition \ref{pro:prp_FM} and obtain that 
\begin{align}
\sum_{|i,j,k|\leq \mathbb{ M}}\Phi(T_h)^{2j+2}\|\pa_x^i \Gamma_{y;T_h}^j \pa_z^k \cc_{\neq}(T_h)\|_{ L^2 (\mathbb{T}^3)}^2\leq\frac{ C}{A^{2/3}}\mathfrak{B} \left( \mathbb M,\|u_A\|_{L_t^\infty W^{ \mathbb M+3,\infty}},\delta_{ \mathcal Z}^{-1},\|\cc_{\mathrm{in}} \|_{H^{ \mathbb M+1}},\| n_\mathrm{in}\|_{H^ \mathbb M}\right).
\end{align} 
This is the gluing condition \eqref{Glu_con} in Proposition \ref{Pro:2} with regularity level $M=\mathbb{M}-1$. Now an application of Proposition \ref{Pro:2} yields \eqref {est_short_t} and \eqref{est_shrt_t_0}. This concludes the proof.
\end{proof}}
\myc{\footnote{\textcolor{red}{ The Order of Choosing the Constants. Add some details here. Note that $A$ will be chosen large compared to all the constants appeared in the proof, including the constants appeared in the initial $W^{s,\infty}$ estimates, which might not appear in the hypotheses above.}}}
To prepare ourselves for the alternating shear construction, we present the following proposition, which is in the same vein as Proposition \ref{Pro:2}.
\begin{pro}\label{Pro:main}
Consider the solutions to the equation \eqref{ppPKS} initiated from initial data $n_{\mathrm{in}}\in H^{M}(\Torus^3),\, \cc_{\mathrm{in}}\in H^{M+1}(\Torus^3),\, M\geq 3$. Assume that the ambient shear flow $u_A$ is the one defined in Theorem \ref{thm:L_ED}, and the initial chemical remainder is small in the sense that 
\begin{align}\label{smll_c_nq}
\myb{\| \cc_{\mathrm{in};\neq}\|_{ H^{M+1} (\mathbb{T}^3)}\leq \epsilon.  }  
\end{align} 
\myr{There exist thresholds $G_0,\, \ep_0, \, A_0, \, Q_0$ such that if the following constraints are satisfied, 
\begin{align}\label{chc_GAQ}
G\geq& G_0(\|u_A\|_{L_t^\infty W_y^{M+3}}, M), \quad \ep^{-1}\geq \ep_0^{-1}(n_{\mathrm{in}},\cc_{\mathrm{in}}),\\
\quad A\geq& A_0(M, \|u_A\|_{L_t^\infty W_y^{M+3}}, \myb{G,\ep^{-1},}n_{\mathrm{in}},\cc_{\mathrm{in}}),\quad Q\geq Q_0(M,\|u_A\|_{L_t^\infty W_y^{M+3}}, \myb{G,\ep^{-1}},n_{\mathrm{in}},\cc_{\mathrm{in}}),
\end{align}}
the following set of conclusions hold

 \begin{subequations}\label{Conclusions}
1) Remainders' enhanced dissipation estimates:  
\begin{align}\label{ConED} 
\mathbb{F}_{G,Q}^{t_r=0;M}[t,n_\nq,\cc_\nq]\ \leq\ & C_{ED}
\mathbb{F}_{G,Q}^{t_r=0;M}[0,n_\nq,\cc_\nq]\exp\lf\{-\frac{2\delta t}{A^{1/3}}\rg\}\\ \leq\  &C(\|n_{\mathrm{in};\neq}\|_{H^M}^2+\| \cc_{\mathrm{in};\nq}\|_{H^{M+1}}^2+ 1)\exp\lf\{-\frac{2\delta t}{A^{1/3}}\rg\};
\end{align}

2) Uniform-in-time estimates of the $x$-averages:
\begin{align}
\|\lan n\ran\|_{L_t^\infty([0,T_\star];H_{y,z}^M)}\leq& 2\|\lan n_{\mathrm{in}}\ran\|_{H^M_{y,z}}+2;\label{Con_<n>}\\
\| \na\lan \cc\ran\|_{L_t^\infty([0,T_\star];H_{y,z}^{M})}\leq&{A^{-1/5}}+2\|\na \lan \cc_{\mathrm{in}}\ran\|_{H^{M}_{y,z}}.\label{Con_<grd_c>}
\end{align} 
\end{subequations}
\end{pro}
\begin{rmk}\label{rmk:sml_c}
We highlight that the derivation of the enhanced dissipation estimate \eqref{ConED} requires smallness of the chemical gradient $\na \cc$ \eqref{smll_c_nq}. This is the main obstacle to overcome in the alternating construction.
\end{rmk}

This concludes Section \ref{Sec_s:nl}. 
\ifx
If we consider general initial condition, namely, if we drop the smallness constraint \eqref{smll_c_nq}, then we can do the following. We can allow the $\|n\|_{H^M}$ to grow to an extremely large level but independent of $A$. And we can use the following observation: if $\lan n\ran^\iota$ or $\lan \lan n\ran\ran^{s,s'}$ is bounded in terms of the bootstrapping constants, we can take $A$ large such that the growth of $\na\lan c\ran ^\iota$ or $\na\lan\lan c\ran\ran^{s,s'}$ is of order $CA^{-1/2}$. For example, after the first phase, we have that $\lan \lan c\ran\ran ^{y,z}$ is small, and we can make sure that it will be small for all later time in phase $\mathfrak{b}$ and phase $\mathfrak{c}$. So after phase $\mathfrak{a}$, $\na \lan \lan c\ran \ran ^{y,z}$ is small and after phase $\mathfrak{b}$, $\na \lan c\ran^{y}$ is small. Now after the first $\mathfrak{a}, \mathfrak{b},\mathfrak{c}$ alternate, we have that even though the $n$ is amplified, the $\na \cc$ is small ($O(A^{-1/2})$). Afterwards, we use the main theorem here.  My hope is that even though the size of the initial chemical gradient is large, it can only force the $\lan n\ran^\iota$	 LINEARLY. Hence the $n$ will not have nonlinear blow up tendency. So we can pump the energy from chemical to the density in the first round of alternating shear process.
\fi
\subsection{\myb{Suppression of Blow-up with Alternating Shear Flows}}\label{Sec_s:Alt}
In this section, we sketch the idea of suppression of chemotactic blow-up through alternating shear flows. Several lemmas will be presented along with the argument. The proof of these technical lemmas will be postponed to Section \ref{Sec:Alt}. We consider the following time-dependent alternating shear flows acting on the system:
\begin{align}
\text{Phase }\mathfrak{a}: \widetilde{u}_A:=&(u_A(t,y),0,0)^\top  ,\quad \forall t\in[0, A^{1/3+\te}];\\
\text{Phase }\mathfrak{b}: \wt u_A:=&(0,0, u_A(t,x))^\top ,\quad \forall t\in(A^{1/3+\te}, 2A^{1/3+\te}];\\
\text{Phase }\mathfrak{c}: \wt u_A:=&(0,u_A(t,z),0)^\top,\quad \forall t\in(2A^{1/3+\te}, 3A^{1/3+\te}].
\end{align}Here $u_A$ is the flow defined in Theorem \ref{thm:L_ED}.  The explicit form of $\wt u_A$ can be found in Figure \ref{Fig:alternating_shear}. 
\begin{figure}[h]
\centering
\includegraphics[scale=0.8]{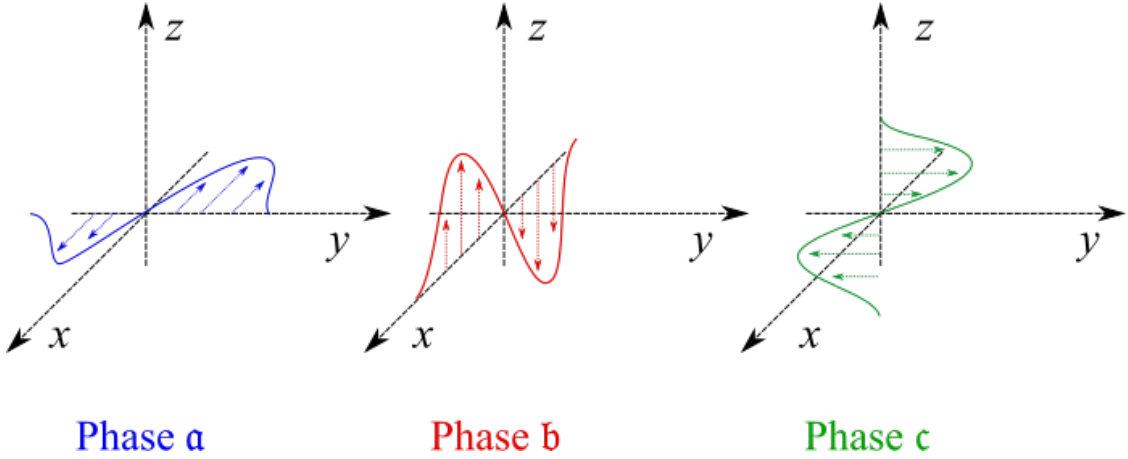}
\caption{The alternating shear flows} \label{Fig:alternating_shear}
\end{figure} 
{We define the following switching time instances:
\begin{subequations}
\begin{align}
T_{I\mathfrak{a}}=& 3IA^{1/3+\te} ,\ \
\quad T_{I\mathfrak{b}}=(1+3I)A^{1/3+\te} ,\ \
\quad T_{I\mathfrak{c}}= (2+3I)A^{1/3+\te} , \quad I\in \{0, 1,2,....\}=\mathbb{N}.
\end{align}\end{subequations} 
The subscripts $\mathfrak{a},\mathfrak{b},\mathfrak{c}$ indicate the phase of the shear flows and the subscript $I$ indicates which period the system is in. Since the shearing direction changes over time, we consider three distinct gliding regularity norms induced by the following vector fields:
\begin{align}
\text{Phase }\mathfrak{a}:\qquad&\pa_x, \, \Gamma_{y;t},\, \pa_z,\quad \Gamma_{y;t}:=\pa_y+\int_0^t \pa_yu(T_{I\mathfrak{a}}+s,y)ds\pa_x,\quad \Gamma_{y;t}^{ijk}:=\pa_x^i\Gamma_{y;t}^j\pa_z^k,\quad t\in[0, A^{1/3+\te}];\\
\text{Phase }\mathfrak{b}:\qquad&\Gamma_{x;t}, \, \pa_y,\, \pa_z,\quad \Gamma_{x;t}:=\pa_x+\int_0^t \pa_xu(T_{I\mathfrak{b}}+s,x)ds\pa_z,\quad \Gamma_{x;t}^{ijk}:=\Gamma_{x;t}^i\pa_y^j\pa_z^k,\quad t\in[0, A^{1/3+\te}];\\
\text{Phase }\mathfrak{c}:\qquad&\pa_x, \, \pa_y,\, \Gamma_{z;t} ,\quad \Gamma_{z;t}:=\pa_z+\int_0^t \pa_zu(T_{I\mathfrak{c}}+s,z)ds\pa_y,\quad \Gamma_{z;t}^{ijk}:=\pa_x^i\pa_y^j\Gamma_{z;t}^k,\quad t\in[0, A^{1/3+\te}].
\end{align}

With all the notation introduced, we are ready to present the main theorems in the alternating construction. As it turns out, there are two distinct time phases in the system, i.e., 
\begin{align}\label{ph_1}
\text{Phase }\#\ 1:\ & I=0, \quad t\in[T_{0\mathfrak{a}},T_{1\mathfrak{a}}]=[0, 3A^{1/3+\te});\\ \label{ph_2}
\text{Phase  }\#\ 2:\ &I\geq 1, \quad t\in \cup_{I=1}^\infty [T_{I\mathfrak{a}},T_{(I+1)\mathfrak{a}}]=[3A^{1/3+\te},\infty ).
\end{align}
In Phase \# 1, the solutions undergo transient growth. However, by paying three derivatives, one can show that the chemical gradient $\na \cc$ decays to a low level. The following theorem summarized the state of the system at the end of Phase \# 1:
\begin{theorem}\label{thm:alt_1}
Consider solutions to \eqref{PKS_rsc} subject to the initial data $(n_{\mathrm{in}},\cc_{\mathrm{in}})\in H^{\mathbb{M}+3}\times H^{\mathbb{M}+4},\, \mathbb{M}\geq 5$. There exists a threshold $A_0=A_0(\mathbb{M},\|\na \wt u_A\|_{L_t^\infty W^{\mathbb{M}+4,\infty}}, \|n_{\mathrm{in}}\|_{H^{\mathbb{M}+3}},\| \cc_{\mathrm{in}}\|_{H^{\mathbb{M}+4}})$ such that if the parameter $A\geq A_0,$ then the following bounds hold at the time instance $t=T_{1\mathfrak a}$, 
\begin{align}
\|n(T_{1\mathfrak{a}})\|_{H^{\mathbb{M}}}^2\leq  \mathfrak B(\|n_{\mathrm{in}}\|_{H^{\mathbb{M}+3}},\|\cc_{\mathrm{in}}\|_{H^{\mathbb{M}+4}}) ,\quad \|\cc(T _{1\mathfrak{a}})\|_{H^{\mathbb{M}+1}}^2\leq \frac{\mathfrak B(\|n_{\mathrm{in}}\|_{H^{\mathbb{M}+3}},\|\cc_{\mathrm{in}}\|_{H^{\mathbb{M}+4}})}{A^{1/3}}.\label{Goal_Phase_1}
\end{align}
Moreover, $\|n(t)\|_{H^{\mathbb{M}}} +\|\cc(t)\|_{H^{\mathbb{M}+1}} <\infty$ for all $t\in [0, 3A^{1/3+\te}]$. 
\end{theorem}
As a result, we observe that at the end of Phase \# 1, the chemical gradient becomes small. Hence, Proposition \ref{Pro:main} becomes applicable. By repeatedly applying Proposition \ref{Pro:main} in Phase \# 2, we expect the following conclusion: 
\begin{theorem}\label{thm:alt_2}
Consider solution to the equation \eqref{PKS_rsc} in Phase \# 2, i.e., $t\in\cup_{I=1}^\infty [T_{I\mathfrak{a}},T_{(I+1)\mathfrak{a}}]=[3A^{1/3+\te},\infty)$. Further assume that the condition \eqref{Goal_Phase_1} holds at time $T_{1\mathfrak a}$. Then there exists a threshold \[A_0=A_0( \mathbb{M},\|\na \wt u_A\|_{L_t^\infty W^{\mathbb{M}+1,\infty}}, \|n(T_{1\mathfrak a})\|_{H^{\mathbb{M}}},\| \cc(T_{1\mathfrak a})\|_{H^{\mathbb{M}+1}}) \] such that if $A\geq A_0,$ then the  following enhanced dissipation holds
\begin{align}
\|(n-\overline {n})(T_{1\mathfrak a}+t)\|_{H^\mathbb{M}}^2+\|\cc(T_{1\mathfrak a}+t)\|_{H^{\mathbb{M}+1}}^2\leq& C(\|n(T_{1\mathfrak a})\|_{H^{\mathbb{M} }}^2+\|\cc(T_{1\mathfrak a})\|_{H^{\mathbb{M}+1}}^2)e^{-\delta t/A^{1/3+\te}},\quad\forall t\geq 3A^{1/3+\te}.\label{Goal_Phase_2} 
\end{align}Here $C,\, \delta$ are constants that depends on  $\mathbb{M},\|\na \wt u_A\|_{L_t^\infty W^{\mathbb{M}+1,\infty}}$.   
\end{theorem}
With these two theorems, we are ready to prove Theorem \ref{thm_main}:
\begin{proof}[Proof of Theorem  \ref{thm_main}] 
Combining Theorem \ref{thm:alt_1} and Theorem \ref{thm:alt_2}, we observe that  the Sobolev norms of the solution is bounded globally in time and the well-posedness follows from Theorem \ref{thm:lcl_exst}.  \myc{local existence should be enough.} 
\end{proof}
\ifx we prove the following estimate:
\begin{align}
\|(n-\overline{n})(T_{(I+1)\mathfrak{a}})\|_{H^{\mathbb{M} }}^2+\|\cc(T_{(I+1)\mathfrak{a}})\|_{H^{\mathbb{M}+1}}^2\leq \frac{1}{2}(\|(n-\overline{n})(T_{I\mathfrak{a}})\|_{H^{\mathbb{M} }}^2+\|\cc(T_{I\mathfrak{a}})\|_{H^{\mathbb{M}+1}}^2).
\end{align}\fi 
 \myc{(Need a lemma relating $H^\mathbb{M}$ and $H_\Gamma^\mathbb{M}$. They differ by polynomial powers of $A$. However, after $A^{1/3+\zeta/2}$ time, the ED has $e^{-A^\ep}$ decay and overpowers the $A^\mathbb{M}$. Be careful, you really need to run $A^{1/3+\zeta/2}$ time to guarantee that the change of norm make sense!)}

}

\ifx\subsubsection{Analysis in Phase $\mathfrak{a}$}\label{Sec:Bootstrap}
\textcolor{red}{To avoid complicated notations, we focus on the first time interval $[0, A^{1/3+\te}]$, then apply the result to other directions without loss of generality. Need a special section explaining how to glue all the argument together. }

\subsubsection{Analysis in Phase $\mathfrak{b}$ \& $\mathfrak{c}$}
\myb{It seems that the estimate in Phase $\mathfrak{b}$ \& $\mathfrak{c}$ are similar.?}
\fi

The remaining part of the paper is organized as follows: in Section \ref{Sec:Lnr}, we prove the Theorem \ref{thm:L_ED}; in Section \ref{Sec:ED}, we prove the enhanced dissipation of the remainder $n_{\neq}^\iota,\, \na c_{\neq}^\iota$; in Section \ref{Sec:Alt}, we prove the theorems and lemmas in subsection \ref{Sec_s:Alt}. 
\section{Linear Theory}\label{Sec:Lnr}
In this section, we present the proof of Theorem \ref{thm:L_ED} and Theorem  \ref{thm:ED_gldrg}. %
\subsection{Time-dependent Shear Flows with Logarithmic Shift}
In this section, we recall the equations \eqref{PS_intr}, \eqref{PS_Fourier}  subject to diffusion coefficient $\frac{1}{A}$.
Before proving Theorem \ref{thm:L_ED}, we present a lemma which captures the decay of the $L^2$ norm.
\begin{lem}\label{lem_ED_1}
Consider the equation \eqref{PS_Fourier} subject  to the shear flow  $u (t,y)=\uu(y+\log(t+1)).$  Assume that the profile $\uu\in C^\infty$ is not constant, i.e., there exists $y_\ast\in \Torus$ such that  $\uu'(y_\ast)\neq 0$. Then the solutions to \eqref{PS_Fourier} satisfy the estimate
\begin{align}
\|\wh f_{\al ,\gamma  } (|\al |^{-2/3}A^{1/3})\|_{L_y^2}\leq (1-\kappa) \|\wh f_{\al ,\gamma  } (0)\|_{L_y^2},\quad \forall \al \neq 0,\label{ED_tdp_shr}
\end{align}
for a constant $0<\kappa<1$ that depends only on $\uu$.
\end{lem}\ifx
\begin{remark}
Note that \emph{all} space-time derivatives of $u$ are uniformly bounded in space and time.
As far as I know, this is the best rate available for any smooth incompressible velocity field on $\mathbb{T}^2$ (shear flow or not). Observe also that the choice of $\log(t+1)$ in the velocity field is somewhat crucial. Indeed, the idea is to vary the velocity field in such a way that no value of $y$ is a critical point for all time. If we vary the velocity field too quickly, there will be cancelations and in fact no enhanced dissipation (this is the case for example when $\log(1+t)$ is replaced by $t$).  
\end{remark}
\fi
\begin{proof}
We organize the proof in several steps.

\noindent
\textbf{Step \#1: General setup.} First of all, we identify the range of the wave number $|\al |$ on which we focus. 
If $|\al |\geq A^{1/2}/K$ for some universal constant $K$, direct application of the non-expansive property of the $L^2$ norm of solutions to \eqref{PS_Fourier} yields \begin{align}
\|\wh f_{\al ,\gamma  } (t)\|_{2}\leq \|\wh f_{\al ,\gamma  } (0)\|_{2}e^{-\frac{|\al |^{2/3}}{K^{4/3}A^{ 1/3}}t },\quad |\al |\geq A^{1/2}/K.\label{ED_tdp_sh_hm}
\end{align}
Hence at time instance $|\al|^{-2/3}A^{1/3}$, the $L^2$ norm decays as in \eqref{ED_tdp_shr} with $\kappa=1-e^{-1/K^{4/3}}$. Hence in the remaining proof, we always assume that $|\al |\leq A^{1/2}/K$. Without loss of generality, we assume $\al \geq 1$. We choose the constant $K$ such that 
\begin{align}\label{K_choice}
A^{1/3}|\al|^{-2/3}\geq K^{2/3}\geq e^{100\pi}.
\end{align}
We are going to refine the choice of $K$ in \eqref{Choice_H,K_1} and \eqref{Choice_H,K_2}. 

Rearranging the terms in \eqref{PS_Fourier} yields that 
\begin{align}
\pa_t \left(\wh f_{\al ,\gamma  } (t,y)e^{\frac{1}{A}(|\al |^2+|\gamma  |^2)t}\right)+iu(t,y)\al \left(\wh f_{\al ,\gamma  }(t,y)e^{\frac{1}{A}(|\al |^2+|\gamma  |^2)t}\right) =&\frac{1}{A}\pa_{yy}\left(\wh f_{\al ,\gamma  }(t,y)e^{\frac{1}{A}(|\al |^2+|\gamma  |^2)t}\right).\label{PS_hypo_1}
\end{align}
As a result, we define $F =\wh f_{\al ,\gamma  } (t,y)e^{\frac{1}{A}(|\al |^2+|\gamma  |^2)t}$ and consider the following equation
\begin{align}
\pa_t F(t,y)+iu\al   F(t,y) =&\frac{1}{A}\pa_{yy}F(t,y).\label{PS_hypo}
\end{align}
We show that for $t=|\al |^{-2/3}A^{1/3}$, the solutions to \eqref{PS_hypo} must decrease in $L^2$ by a fixed amount. Assume without loss of generality that $\|F(0)\|_{L_y^2}=1$. Observe that we have the following energy identities:
\begin{align}
\|F(t)\|_{L_y^2}^2=&1-\frac{2}{A}\int_{0}^{t}\|\partial_y F(s)\|_{L_y^2}^2ds,\label{L_2_f_nu}\\
 \qquad \|\partial_y F(t)\|_{L_y^2}^2+ &\frac{2}{A} \int_{0}^{t}\|\partial_{yy}F(s)\|_{L_y^2}^2ds\leq 2\|\pa_y u\|_{L^\infty_{t,y}}|\al| \int_{0}^t\|\partial_y F(s)\|_{L_y^2}ds+\|\partial_{y}F (0)\|_{L_y^2}^2.\label{H_1_f_nu}
 \end{align}
To derive the second inequality, we also use the fact that $\|F(s)\|_{L^2}\leq \|F(0)\|_{L^2}=1. $
Assume then, that
\begin{align}\frac{ 2}{A} \int_{0}^{|\al |^{-2/3}A^{1/3}} \|\partial_y F(s)\|_{L_y^2}^2ds<\kappa.\label{Cntra_Assmp}
\end{align}
Here $\kappa$ is a small constant to be chosen later in \eqref{choice_ka}. Thanks to \eqref{L_2_f_nu}, we have 
\begin{align}
\|F(t)\|_{L^2_y}^2\geq 1-\kappa,& \quad \forall t\in[0,A^{1/3}|\al |^{-2/3}]. 
\end{align} 

Next, we identify a time instance $\tau_0$ with the property that:
\begin{align}\label{tau_0_cnfg}
\|\partial_y F(\tau_0 )&\|_{L_y^2}^2< {(2\|\pa_y u\|_{L^\infty_{t,y}} +1)H\sqrt\kappa|\al |^{2/3}}{A^{2/3}}, \quad \|\partial_{yy} F(\tau_0 )\|_{L_y^2}^2< {(2\|\pa_y u\|_{L^\infty_{t,y}} +1)H^2\sqrt\kappa|\al |^{4/3}}{A^{4/3}},\\
\frac{2}{A} &\int_{\tau_0}^{|\al |^{-2/3}A^{1/3}}\|\partial_y F(s )\|_{L_y^2}^2ds<\kappa, \quad\tau_0\in\lf[0,\frac{4}{H}|\al |^{-2/3}A^{1/3}\rg].
\end{align}
Here \myr{$H\geq 4e^{16\pi}$} is a large constant to be chosen later in \eqref{Choice_H,K_1}, \eqref{Choice_H,K_2}.  The last inequality is a  direct consequence of \eqref{Cntra_Assmp}. 
The explicit argument involves an application of the Chebyshev inequality. First we note that by the assumption \eqref{Cntra_Assmp}, 
\begin{align}
\bigg|\left\{t\in\left[0, |\al|^{-2/3}A^{1/3}\right]\bigg|\ \|\pa_y F(t)\|_{L_y^2}^2> {H\kappa|\al |^{2/3}}{A^{2/3}}\right\}\bigg|\leq \frac{\kappa {A} }{{H\kappa|\al |^{2/3}}{A^{2/3}}}=\frac{A^{1/3}}{H|\al |^{2/3}},
\end{align}
so on the time interval $[0,|\al |^{-2/3}A^{1/3}]$, a fraction of $1-\frac{1}{H}$ points have the $\dot H^1$ bound $\|\pa_y F(t)\|_{L_y^2}^2\leq {H\kappa|\al |^{2/3}}{A^{2/3}}$. Now we choose $\tau_0'$ as the minimum of all these points, i.e.
\begin{align}
\tau_0'=\min\left\{t\big| \|\pa_y F(t)\|_{L_y^2}^2\leq {H\kappa|\al |^{2/3}}{A^{2/3}}\right\}.
\end{align}
We observe that $\tau_0'\leq \frac{2}{H}|\al |^{-2/3}A^{1/3}$. Now on the interval $[\tau_0',|\al |^{-2/3}A^{1/3})$, we apply the estimate \eqref{H_1_f_nu} subject to initial time $\tau_0'$ and the assumption \eqref{Cntra_Assmp}  to obtain that
\begin{align}\label{defn_C_1}
 \|&\partial_y F(t)\|_{L_y^2}^2+\frac{2}{A}\int_{\tau_0'}^t\|\pa_{yy}F(s)\|_{L^2_y}^2ds\\
 \leq& 2\|\pa_y u\|_{L^\infty_{t,y}} |\al |A^{1/2} \sqrt{t}\Big(\frac{1}{A}\int_{\tau_0'}^t \|\pa_y F(s)\|_{L_y^2}^2ds\Big)^{1/2}+{H\kappa |\al |^{2/3}}{A^{2/3}}\leq{ (2\|\pa_y u\|_{L^\infty_{t,y}} +1) H\sqrt{\kappa}|\al |^{2/3}}{A^{2/3}}\\
 =:&C_1(\|\pa_y u\|_{L_{t,y}^\infty})H\sqrt{\kappa}|\al |^{2/3}A^{2/3},\quad \forall t\in \left[\tau_0', |\al|^{-2/3}A^{1/3}\right].
\end{align}
Now the Chebyshev inequality yields that
\begin{align}
\left|\left\{t\in \left[\tau_0', |\al|^{-2/3}A^{1/3}\right]\bigg| \|\pa_{yy}F(t)\|_{L^2_y}^2>{H^2C_1 \sqrt{\kappa}|\al |^{4/3}}{A^{4/3}}\right\}\right|\leq \frac{C_1H\sqrt{\kappa}|\al |^{2/3}A^{5/3}}{H^2C_1 \sqrt{\kappa}|\al |^{4/3}A^{4/3}}\leq \frac{1}{H}|\al |^{-2/3}A^{1/3}.
\end{align}
Therefore, by \eqref{defn_C_1}, we can find $\tau_0\in[\tau_0',\tau_0'+\frac{2}{H} A^{1/3}|\al |^{-2/3}]\subset [0,\frac{4}{H}A^{1/3}|\al |^{-2/3}]$ such that \begin{align}\|\partial_y F(\tau_0)\|_{L_y^2}^2<{C_1(u)H\sqrt\kappa|\al |^{2/3}}{A^{2/3}}, \qquad \|\partial_{yy} F(\tau_0)\|_{L_y^2}^2<{C_1(u)H^2\sqrt\kappa|\al |^{4/3}}{A^{4/3}}.
\end{align}This is \eqref{tau_0_cnfg}.

We denote the solution to the corresponding inviscid problem by $\eta$, i.e., $\eta$ solves
\begin{align}
\partial_t \eta(t,y)+iu\al   \eta (t,y)=0,\quad \eta(\tau_0,y)=F(\tau_0,y).\label{Invscd_pb}
\end{align}

\noindent
\textbf{Step \# 2: Quantitative estimates. }
In the second step, we provide some necessary estimates. 

\myr{ First of all, we present some upper bounds for the viscous/inviscid solutions. The starting point is the following claim
\begin{align}\label{upperbound}
\|\pa_{yy}&\eta(\tau_0+t,y)\|_{L_y^2}+\|\partial_{yy}F(\tau_0+t,y)\|_{L_y^2}\\
\leq & C|\al |^2t^2\|F (\tau_0,y)\|_{L_y^2} +C|\al |t\left(\|\partial_y F (\tau_0,y)\|_{L_y^2}+\|F(\tau_0,y)\|_{L_y^2}\right)+\|\partial_{yy} F(\tau_0,y)\|_{L_y^2},\quad\forall t\geq 0. 
\end{align} Here the constant $C$ depends on the $L_t^\infty W_y^{2,\infty}$-norm of the shear $u$. 
To prove this bound, we observe that a direct $L^2$-based energy estimate on the solutions to \eqref{PS_hypo} and \eqref{Invscd_pb} yields that
\begin{align}
\label{dt_Feta}\quad\quad\left\{\begin{array}{cc}\displaystyle\|F\|_{L^2}\leq\|F(\tau_0)\|_{L^2},& \qquad \frac{d}{dt}\|\partial_y F\|_{L^2}\leq C|\al | \|F\|_{L^2},\qquad \frac{d}{dt} \|\partial_{yy}F\|_{L^2}\leq C|\al |(\|\partial_y F\|_{L^2}+\|F\|_{L^2});\\
\|\eta\|_{L^2}= \|F(\tau_0)\|_{L^2},&\qquad \frac{d}{dt}\|\partial_y \eta\|_{L^2}\leq C|\al | \|\eta\|_{L^2},\qquad \frac{d}{dt} \|\partial_{yy}\eta\|_{L^2}\leq C|\al |(\|\partial_y \eta\|_{L^2}+\|\eta\|_{L^2}).\end{array}\right. 
\end{align}  Now a direct integration in time yields the upper bound \eqref{upperbound}. 

If we focus on the time interval $\left[\tau_0,|\al |^{-2/3}A^{1/3}\right]$, finer estimates can be obtained. The energy estimate \eqref{dt_Feta}, together with the initial configuration \eqref{tau_0_cnfg}, yields that 
\begin{align}\|\partial_y \eta (t)\|_{L_y^2}\leq {C(\|\pa_yu\|_{L^\infty_{t,y}})\sqrt{H}|\al |^{1/3}}{A^{1/3}},\quad \forall t\in\Big[\tau_0,|\al |^{-2/3}A^{1/3}\Big].\label{pay_eta_L2}
\end{align}
Combining the upper bounds of the $\dot H^2$-norms  \eqref{upperbound} and the fact that the solutions $F$ and $\eta$ are initiated from identical data at the initial time $\tau_0$,  we obtain
\begin{align}\label{upbnd_2}
\|\partial_{yy}&\eta (t)\|_{L_y^2}+\|\partial_{yy}F(t)\|_{L_y^2}\\
\leq  & C|\al |^2t^2\|F (\tau_0,y)\|_{L_y^2} +C|\al |t\left(\|\partial_y F (\tau_0,y)\|_{L_y^2}+\|F(\tau_0,y)\|_{L_y^2}\right)+\|\partial_{yy} F(\tau _0,y)\|_{L_y^2}\\
 \leq &{C H|\al |^{2/3}}{A^{2/3}}+{CH|\al |^{1/3}}{A^{1/3}}\leq{CH|\al |^{2/3}}{A^{2/3}},\quad \forall t\in\Big[\tau_0,|\al |^{-2/3}A^{1/3}\Big].&
\end{align}
}

The remaining part of step \# 2 is devoted to the proof of the following lower bound for all smooth solutions to the inviscid equation \eqref{Invscd_pb}:
\begin{equation} \label{lowerbound}\int_{\tau_0}^{A^{1/3}|\al|^{-2/3}}\|\partial_y \eta(s) \|_{L_y^2}^2ds \geq  {\delta} A\|\eta (\tau_0) \|_{L_y^2}^2-Ct\|\partial_y \eta (\tau_0)\|_{L_y^2}^2, \quad 0<|\al |\leq A^{1/2}/K.
\end{equation}
To prove the lower bound \eqref{lowerbound}, we first rewrite the solution to \eqref{Invscd_pb} as
\[\eta(t,y)=\exp\left\{-i \al \int_{\tau_0}^{t}u(s,y)ds\right\}\eta (\tau_0,y),\quad \forall t\geq \tau_0.\]
The $\pa_y$-derivative reads,
\[\partial_y \eta (t,y) = \exp\left\{-i\al \int_{\tau_0}^t u(s,y)ds\right\}\pa_y \eta (\tau_0,y) -i\al \left(\int_{\tau_0}^t \partial_y u(s,y)ds\right) \exp\left\{-i\al \int_{\tau_0}^t u(s,y)ds\right\}\eta (\tau_0,y).\]
Thus, there exists a universal constant $C$ so that the following estimate holds
\[\|\partial_y \eta(t) \|_{L_y^2}^2\geq \frac{1}{C}|\al |^2\norm{\int_{\tau_0}^t \partial_y u(s,\cdot)\eta(\tau_0,\cdot)ds}_{L_y^2}^2-C\|\partial_y \eta (\tau_0,\cdot)\|_{L_y^2}^2.\] Integration in time yields that for all $t\geq \tau_0$, 
\[\int_{\tau_0}^t\|\partial_y\eta (s)\|_{L_y^2}^2ds\geq \frac{1}{C}|\al |^2\int_{\tau_0}^t\norm{\int_{\tau_0}^s \pa_y u(\tau,\cdot)\eta  (\tau_0,\cdot)d\tau}_{L_y^2}^2ds-Ct\|\partial_y \eta (\tau_0,\cdot)\|_{L_y^2}^2.\]
Recalling the explicit form of the flow $u(t,y)=\uu(y+\log(1+t))$, we will establish \eqref{lowerbound} once we show that
\begin{align}\label{lowr_bnd_1}\int_{\tau_0}^{A^{1/3}|\al |^{-2/3}}\bigg|\int_{\tau_0}^s\uu'( y+\log(1+\tau))d\tau\bigg|^2ds\geq \frac{A}{C_\ast(\uu)|\al |^2},\quad \forall y\in\Torus.
\end{align}
Now we use the change of variable $ \mathfrak{h}:=\log(1+\tau)$ to rewrite the above inequality
\begin{align}
\int_{\tau_0}^{A^{1/3}|\al |^{-2/3} }\bigg|\int_{\log (1+\tau_0)}^{\log(1+s)}\uu'(y+ \mathfrak{h})e^ \mathfrak{h} d \mathfrak{h}\bigg|^2ds\geq \frac{A }{C_\ast(\uu)|\al |^2},\quad \forall y\in\Torus.
\end{align}
Here we assume that $\uu'(y)$ is not identically zero on the interval $[0,2\pi]$. Then we prove the following claims:

\noindent
a) There exists a point $y_0$ on $\Torus $ such that 
\begin{align}\label{Claim}
\mathcal{R}_0:=\bigg|\int_0^{2\pi}\uu'(y_0+ \mathfrak{h})e^ \mathfrak{h}d \mathfrak{h}\bigg|>0.
\end{align}

\noindent 
b)  there exist a point $a_0\in[0, 2\pi)$ such that  for $\forall \mathfrak{h}_0\in[a_0,2\pi]$, the integral has the following lower bound
\begin{align}\label{Claim_2}
\bigg|\int_0^{ \mathfrak{h}_0}\uu'(y_0+ \mathfrak{h}) e^{ \mathfrak{h}}d \mathfrak{h}\bigg|\geq \mathcal{R}_0/2>0, \quad \forall \mathfrak{h}_0\in[a_0,2\pi]. 
\end{align}
Here $a_0$ depends only on the profile $\uu.$ 
We further define \begin{align}
\label{tau_a} \tau_{a;0} :=e^{a_0}-1<e^{2\pi}-1. 
\end{align}

The proof of \eqref{Claim} is through contradiction argument. Assume that the claim \eqref{Claim} is false, then we have that 
\begin{align}
\pa_y\int_0^{2\pi} \uu'(y+ \mathfrak{h})e^ \mathfrak{h} d \mathfrak{h}=0\Rightarrow \int_0^{2\pi}-\uu''(y+ \mathfrak{h})e^ \mathfrak{h} d \mathfrak{h}=0,\quad \forall y\in \Torus. 
\end{align}
Now we apply the integration by parts and the vanishing conditions above to derive that
\begin{align}
\int_0^{2\pi} \uu'(y+\mathfrak{h})&e^ \mathfrak{h} d \mathfrak{h}=\int_0^{2\pi} \uu'(y+ \mathfrak{h})\frac{d}{d \mathfrak{h}}e^{ \mathfrak{h}}d \mathfrak{h}=\uu'(y+ \mathfrak{h})e^{ \mathfrak{h}}\bigg|_{ \mathfrak{h}=0}^{2\pi}-\int_0^{2\pi}\uu''(y+ \mathfrak{h})e^ \mathfrak{h} d \mathfrak{h},\quad \forall y\in\Torus.
\end{align}
Since $\uu$ is smooth and periodic, $\uu'$ is periodic. Hence,
\begin{align}
\uu'(y)(e^{2\pi}&-1)=0,\quad \forall y\in\Torus. 
\end{align}
This is a contradiction to the assumption that $\uu'(y)$ is not identically zero and hence yields \eqref{Claim}. 

The claim \eqref{Claim_2} is a natural corollary of \eqref{Claim} through a continuity argument. Without loss of generality, we assume that $\dss\int_0^{2\pi}\uu'(y_0+\mathfrak{h})e^{\mathfrak{h}}d\mathfrak{h}=\mathcal{R}_0>0.$ Since the function $\displaystyle\int_0^{z}\uu'(y_0+\mathfrak{h})e^{\mathfrak{h}}d\mathfrak{h}$ is continuous with respect to $z$, there exists a small neighborhood of $z=2\pi$ such that the function is above $\mathcal{R}_0/2.$ This concludes the proof of the claim. 

\ifx
To prove \eqref{lowr_bnd_1}, we first observe the average-free condition, i.e.,
\begin{align}\label{assumption_0}
\int_{0}^{2\pi}\pa_y u(y+y')dy'=0,\quad\forall y\in\Torus.
\end{align} 

Next we observe a geometric series property of the integral
\begin{align}
\int_0^{s}\pa_y u(y_0+\log(\tau+1))d\tau=\int_{1}^{s+1}\pa_y u(y_0+\log(\tau+1))d(\tau+1)=\int_{0}^{\log(s+1)}\pa_y u(y_0+ \mathfrak{h}) e^{ \mathfrak{h}}d \mathfrak{h}.
\end{align}\fi
Now we consider the following integral ($m\in\mathbb{N}\backslash\{0\}$) and apply the periodicity of $\uu'(\cdot)$ and the claim \eqref{Claim}  to get
\begin{align}\label{int_period}
\bigg|\int_{2m\pi}^{2(m+1)\pi} \uu'(y_0+ \mathfrak{h})e^{ \mathfrak{h}}d \mathfrak{h}\bigg|=&\bigg|\int_{2(m-1)\pi}^{2m\pi} \uu'(y_0+ \mathfrak{h})e^{ \mathfrak{h}+2\pi}d \mathfrak{h}\bigg|\\
=&e^{2m\pi}\bigg|\int_0^{2\pi}\uu'(y_0+ \mathfrak{h})e^{ \mathfrak{h}}d \mathfrak{h}\bigg|=e^{2m\pi} \mathcal R_0>0.
\end{align}
Now we find the smallest $L\in \mathbb{N}\backslash\{0\}$ and largest $U\in\mathbb{N}$ such that
\begin{align}\myb{
 \bigcup_{m=L}^U [e^{2m\pi}, e^{2(m+1)\pi }]\subset [\tau_0+1,A^{1/3}|\al |^{-2/3}+1]. }
\end{align}
The definition of $L, U$ yields that
\begin{align}\label{exp_mpi_t}
 {L=\left\lceil \frac{\log (\tau_0+1)}{2\pi}\right\rceil \quad\Rightarrow\quad }& {\tau_0+1}\leq e^{2L\pi}\leq {e^{2\pi}}(\tau_0+1); \\
\label{exp_mpi_U}{U=\left\lfloor\frac{\log (A^{1/3}|\al|^{-2/3}+1)}{2\pi}\right\rfloor-1 \quad \Rightarrow \quad}  &e^{-4\pi}(A^{1/3}|\al |^{-2/3}+1) \leq e^{2U\pi}\leq A^{1/3}|\al |^{-2/3}+1.
\end{align} 
Since the time interval is long, i.e.,  \eqref{K_choice}, and the $H$ is large, i.e., $H\geq 4e^{16\pi}$ \eqref{tau_0_cnfg},  we have that \myc{(\bf new! check!)}
\myb{\begin{align}L\leq&\frac{ \log (\tau_0+1)}{2\pi}+1\leq \frac{\log(e^{-16\pi}A^{1/3}|\al|^{-2/3}  +1)}{2\pi} +1 \\
\leq&-7 +\frac{\log(A^{1/3}|\al|^{-2/3} +e^{16\pi})}{2\pi} \leq \frac{\log (A^{1/3}|\al|^{-2/3}+1)}{2\pi}-6\leq  U-4.
\end{align}}
Further recall the quantity $H$ in \eqref{tau_0_cnfg} and the constraint on $K$ \eqref{K_choice}.  We have that 
\begin{align}
 \frac{e^{2L\pi}}{e^{2U\pi}}\leq \frac{e^{2\pi}(\tau_0+1)}{e^{-4\pi}(A^{1/3}|\al |^{-2/3}+1)}\leq e^{6\pi}\left(\frac{4}{H}+\frac{1}{K^{2/3}+1}\right).\label{Quot_bnd}
\end{align} 

We focus on the point $y_0\in \Torus$ and prove estimate \eqref{lowr_bnd_1} for $y_0$. On the interval $[a_0,2\pi]$,
\begin{align}\label{sign}
\left(\int_0^{ \mathfrak{h}_0 }  \uu'(y_0+ \mathfrak{h})e^{ \mathfrak{h}}d \mathfrak{h}\right)\times\left(\int_0^{2\pi  } \uu'(y_0+ \mathfrak{h}) e^{ \mathfrak{h}}d \mathfrak{h}\right)\geq 0,\quad \forall \mathfrak{h}_0\in[ a_{0},2\pi].
\end{align} 
Now we use the sign property \eqref{sign} to derive the following estimate with $q\in [\log(\tau_{a;0}+1)+2U\pi, (1+U)2\pi]$, 
\begin{align}
\bigg|\int_{\log(1+\tau_0)}^{q}&\uu'(y_0+ \mathfrak{h})e^{ \mathfrak{h}}d \mathfrak{h}\bigg|\\
\geq& \sum_{m=L}^{U-1} e^{2m\pi}\bigg|\int_0^{2\pi} \uu'(y_0+ \mathfrak{h})e^{ \mathfrak{h}}d \mathfrak{h}\bigg|-\left|\int_{\log(1+\tau_0)}^{2L\pi}\uu'(y_0+ \mathfrak{h})e^ \mathfrak{h} d \mathfrak{h}\right|+e^{2U\pi}\bigg|\int_{0}^{q-2U\pi} \uu'(y_0+ \mathfrak{h})e^{ \mathfrak{h}}d \mathfrak{h}\bigg|. 
 \end{align} Recall the  quotient bound \eqref{Quot_bnd}. 
Now we implement a similar argument as in \eqref{int_period} and choose the $H$ in \eqref{tau_0_cnfg} and $K$ in \eqref{K_choice} to be large compared to $\mathcal{R}_0$ and $\dss\max_{z\in [0,2\pi]}\lf|\int_z^{2\pi}\uu'(y_0+ \mathfrak{h})e^{ \mathfrak{h}}d \mathfrak{h}\rg|$ to derive that 
\begin{align}\label{Choice_H,K_1}
\bigg|\int_{\log(1+\tau_0)}^{q}&\uu'(y_0+ \mathfrak{h})e^{ \mathfrak{h}}d \mathfrak{h}\bigg|\\
\geq&   \frac{\mathcal{R}_0 e^{2U\pi}}{2(e^{2\pi}-1)}-e^{2(L-1)\pi}\max_{z\in[0,2\pi]}\bigg|\int_z^{2\pi}\uu'(y_0+ \mathfrak{h})e^{ \mathfrak{h}}d \mathfrak{h}\bigg|\\
\geq &e^{2U\pi}\left(\frac{\mathcal{R}_0}{2(e^{2\pi}-1)} -e^{6\pi}\left(\frac{4}{H}+\frac{1}{K^{2/3}}\right)\max_{z\in[0,2\pi]}\bigg|\int_z^{2\pi}\uu'(y_0+ \mathfrak{h})e^{ \mathfrak{h}}d \mathfrak{h}\bigg|\right)\\
 \geq &\frac{ e^{2U\pi}\mathcal{R}_0}{4(e^{2\pi}-1)} ,\quad \forall q:=\log(1+s)\in [\log(\tau_{a;0}+1)+2U\pi, (1+U)2\pi].
\end{align}We note that $q\in [\log(\tau_{a;0}+1)+2U\pi, (1+U)2\pi]$ corresponds to $s\in [(\tau_{a;0}+1)e^{2U\pi}-1, e^{2(U+1)\pi}-1]$.  Hence, we have that by \eqref{Choice_H,K_1},
\begin{align}
\int_{\tau_0}^{A^{1/3}|\al |^{-2/3}}&\bigg|\int_{\tau_0}^s \uu' (y_0+\log(1+\tau))d\tau\bigg|^2ds\\
\geq&\int_{(\tau_{a;0}+1)e^{2U\pi}-1}^{ e^{2(U+1)\pi}-1}\bigg|\int_{\tau_0}^s \uu' (y_0+\log(1+\tau))d\tau\bigg|^2ds 
\geq (e^{2\pi}-\tau_{a;0}-1)\frac{e^{6U\pi}}{16(e^{2\pi}-1)^2}\mathcal{R}_0^2.  
\end{align} 
Now we recall the relations \eqref{tau_a}, \eqref{exp_mpi_U}, and obtain that
\begin{align}
\int_{\tau_0}^{A^{1/3}|\al |^{-2/3}}&\bigg|\int_{\tau_0}^s \uu' (y_0+\log(1+\tau))d\tau\bigg|^2ds\geq \frac{A}{ C_\ast(\uu)|\al |^2}.
\end{align} This is $\eqref{lowr_bnd_1}_{y=y_0}$.

To generalize the result to $\forall y\in \Torus$, we note that for any point $y\in \Torus$, there exists a $\delta_y\in[0,2\pi]$ so that $y=y_0+\delta_y (\text{mod } 2\pi) $. Hence
\begin{align}\label{gen_y}
\int_{\tau_0 }^{ A^{1/3}|\al |^{-2/3} }&\left|\int_{\log (1+\tau_0)}^{\log (1+s)} \uu'(y + \mathfrak{h}) e^{ \mathfrak{h}}d \mathfrak{h}\right|^2 ds=\int_{\tau_0 }^{ A^{1/3}|\al |^{-2/3} }\left|\int_{\log (1+\tau_0)}^{\log (1+s)} \uu'(y_0+\delta_y+ \mathfrak{h}) e^{ \mathfrak{h}}d \mathfrak{h}\right|^2 ds\\
=&\int_{ \tau_0 }^{ A^{1/3}|\al |^{-2/3} }\left|\int_{\log(1+\tau_0)+\delta_y}^{\log(1+s)+\delta_y}  \uu'(y_0+ \mathfrak{h})e^{ \mathfrak{h}-\delta_y}d \mathfrak{h}\right|^2ds.
\end{align}
Now we identify the smallest integer $L_y\in \mathbb{N}$ and largest integer $U_y\in \mathbb{N}$ such that 
\begin{align}\myb{
[2L_y\pi,2(U_y+1)\pi]\subset [\log(1+\tau_0)+\delta_y, \log(1+A^{1/3}|\al |^{-2/3})+\delta_y].}
\end{align}
Thanks to the bound $0\leq\delta_y\leq 2\pi$, we have that,
\begin{align}\label{Ly_Uy}
L_y=\left\lceil \frac{\log(1+\tau_0)+\delta_y}{2\pi}\right\rceil\in [L,L+1],\quad U_y=\left\lfloor\frac{ \log(1+A^{1/3}|\al |^{-2/3})+\delta_y}{2\pi}\right\rfloor-1\in[U,U+1].
\end{align}
Now we focus on specific points $q\in [\log(\tau_{a;0}+1)+2U_y\pi, (1+U_y)2\pi]$, and carry out the estimate with \eqref{sign},
\begin{align}\bigg|&\int_{\log(1+\tau_0)+\delta_y}^{q}\uu'(y_0+ \mathfrak{h})e^{ \mathfrak{h}}d \mathfrak{h}\bigg|e^{-\delta_y}\\
\geq& e^{-\delta_y}\left(\sum_{m=L_y}^{U_y-1} e^{2m\pi}\bigg|\int_0^{2\pi} \uu'(y_0+ \mathfrak{h})e^{ \mathfrak{h}}d \mathfrak{h}\bigg|-\bigg|\int_{\log(1+\tau_0)+\delta_y}^{2L_y\pi}\uu'(y_0+ \mathfrak{h})e^ \mathfrak{h} d \mathfrak{h}\bigg|+e^{2U_y\pi}\bigg|\int_{0}^{q-2U_y\pi} \uu'(y_0+ \mathfrak{h})e^{ \mathfrak{h}}d \mathfrak{h}\bigg|\right).  
\end{align}
Now we recall the definition \eqref{Claim}, the estimate \eqref{Quot_bnd}, the fact that $U\geq L+4$, and the relation \eqref{Ly_Uy} to obtain the estimate
 \begin{align}\label{Choice_H,K_2}\bigg|\int_{\log(1+\tau_0)+\delta_y}^{q}&\uu'(y_0+ \mathfrak{h})e^{ \mathfrak{h}}d \mathfrak{h}\bigg|e^{-\delta_y}\\
\geq&   \frac{\mathcal{R}_0 e^{2U_y\pi}e^{-\delta_y}}{2(e^{2\pi}-1)}-e^{2(L_y-1)\pi}\max_{z\in[0,2\pi]}\bigg|\int_z^{2\pi}\uu'(y_0+ \mathfrak{h})e^{ \mathfrak{h}}d \mathfrak{h}\bigg|e^{-\delta_y}\\
\geq&e^{2U\pi}\left(\frac{\mathcal{R}_0 }{2e^{2\pi}(e^{2\pi}-1)} -e^{6\pi}\left(\frac{4}{H}+\frac{1}{K^{2/3}}\right)\max_{z\in[0,2\pi]}\bigg|\int_z^{2\pi}\uu'(y_0+ \mathfrak{h})e^{ \mathfrak{h}}d \mathfrak{h}\bigg|\right)\\
 \geq &\frac{\mathcal{R}_0 }{4e^{2\pi}(e^{2\pi}-1)}e^{2U\pi} ,\quad \forall q:=\log(1+s)+\delta_y\in [\log(\tau_{a;0}+1)+2U_y\pi, (1+U_y)2\pi].
\end{align}
The remaining part of the proof is similar to the $y=y_0$ case. We have that by \eqref{gen_y} and \eqref{Ly_Uy}, 
\begin{align} 
\int_{\tau_0 }^{ A^{1/3}|\al |^{-2/3} }&\left|\int_{\log (1+\tau_0)}^{\log (1+s)} \uu'(y + \mathfrak{h}) e^{ \mathfrak{h}}d \mathfrak{h}\right|^2 ds\\
\geq&\int_{(\tau_{a;0}+1)e^{2U_y\pi-\delta_y}-1}^{ e^{2(U_y+1)\pi-\delta_y}-1}\left|\int_{\log(1+\tau_0)+\delta_y}^{\log(1+s)+\delta_y}  \uu'(y_0+ \mathfrak{h})e^{ \mathfrak{h}-\delta_y}d \mathfrak{h}\right|^2ds\\
\geq& (e^{2\pi}-\tau_{a;0}-1)\frac{e^{6U\pi}}{16e^{6\pi}(e^{2\pi}-1)^2}\mathcal{R}_0^2\geq \frac{A}{ C_\ast(\uu)|\al |^2}.  
\end{align} This is $\eqref{lowr_bnd_1}$.

\noindent
\textbf{Step \# 3: Decay estimates. }
To prove the decay estimate, we consider the difference between the viscous solution and the inviscid solution, which solves the equation 
\[\partial_t(F-\eta)+i u\al  (F-\eta)=\frac{1}{A}\pa_{yy} F.\] 
Recalling the estimate of $\pa_y\eta$ \eqref{pay_eta_L2}, the hypothesis \eqref{Cntra_Assmp}, and the time constraint $t\leq A^{1/3}|\al |^{-2/3}$, we have that the $L^2$-difference is bounded as follows
\begin{align}
\|\eta-F\|_{L_y^2}^2(t)\leq &\frac{1}{A}\int_{\tau_0}^{t} \|\partial_y \eta(s)\|_{L_y^2}\|\partial_y F(s)\|_{L_y^2}ds\\
\leq& C(\|\pa_yu\|_{L^\infty_{t,y}})A^{-2/3}|\al |^{1/3}\sqrt{H}\sqrt{t}\sqrt{\int_{\tau_0}^t\|\partial_yF(s)\|_{L_y^2}^2ds}<C\sqrt{H\kappa} .
\end{align}
Now, by interpolation and the $H^2$-estimate \eqref{upbnd_2}, we obtain the bound
\begin{align}
\|\eta-F\|_{\dot H_y^1} \leq C \|\eta-F\|_{L_y^2}^{1/2}\|\eta-F\|_{\dot H_y^2}^{1/2}<C\sqrt{H^{5/4}\kappa^{1/4} {|\al |^{2/3}}{A^{2/3}}}={CH^{5/8}\kappa^{1/8}|\al |^{1/3}}{A^{1/3}}.
\end{align}
Thus, the triangular inequality yields that
\[\|F\|_{\dot H_y^1}\geq  \|\eta\|_{\dot H_y^1}-\|\eta-F\|_{\dot H_y^1} \geq \|\eta\|_{\dot H_y^1} - {CH^{5/8}\kappa^{1/8}|\al |^{1/3}}{ A^{1/3}}.\]
Hence, 
\[\|F(t)\|_{\dot H_y^1}^2 \geq \frac{1}{2}\|\eta\|_{\dot H_y^1}^2-CH^{5/4}{\kappa^{1/4}|\al |^{2/3}}{A^{2/3}},\quad\forall t\in[\tau_0,A^{1/3}|\al|^{-2/3}].\]
Now we integrate from $\tau_0$ to $|\al |^{-2/3}A^{1/3}$ and apply the lower bound \eqref{lowerbound} to obtain that
\[\frac{1} {A}\int_{\tau_0}^{|\al |^{-2/3}A^{1/3}}\|F(s)\|_{\dot H_y^1}^2 ds\geq \frac{\delta}{16}-C {\sqrt{\kappa} H} -CH^{5/4}{\kappa^{1/4}} ,\]
Combining it with \eqref{Cntra_Assmp}, we then see that:
\[CH^{5/4}(\kappa^{1/2}+\kappa^{1/4})\geq \delta.\]
It follows that $1\geq\kappa\geq \delta^4/(CH^5)$. By choosing the $\kappa$ to be
\begin{align}
\kappa<\frac{\delta^4}{CH^5},  \label{choice_ka}
\end{align}
we obtain a contradiction. As a result,  there exists a constant $\kappa>0$ such that
\begin{align}&
\frac{2}{A}\int_0^{|\al |^{-2/3}A^{1/3}} \|\pa_y F(s)\|_{L^2_y}^2ds \geq \kappa,\quad |\al |\leq \frac{A^{1/2}}{K}. 
\end{align}
By \eqref{L_2_f_nu}, we have the decay \eqref{ED_tdp_shr} for $|\al |\leq \frac{1}{K}A^{1/2}$. Combining  it with the estimate in the high modes \eqref{ED_tdp_sh_hm}, we have completed the proof of the lemma.  
\end{proof}
\begin{proof}[Proof of Theorem \ref{thm:L_ED}]
We first make the observation that the shear flow $u(y+\log (t+1))$ is not uniformly enhanced dissipation in time. As time becomes larger and larger, the shear flow changes slower and slower. Hence the decay rate might deplete over time. To obtain the uniform-in-time version, we introduce the  rewinding of the flow. Recall that $A^{1/3}$ is the time when significant decay ($1-\kappa$) happens, then we define the following flow 
\begin{align}
U(t,y)=u(y+\phi(t)),\quad \phi(t)=\sum_{J=0}^\infty \chi_J(t)\log(1+(t-J(A^{1/3}+1))), \quad t\geq 0,
\end{align}
where the $\chi_J(t)$ is smooth cut-off function such that
\begin{align}
\chi_J(t)=\left\{\begin{array}{cc}1 ,&\quad t\in \left [2J A^{1/3},{(2J+1)}A^{1/3}\right];\\
 \text{monotone},&\quad t\in\left[(2J-1)A^{1/3},{2J}A^{1/3}\right)\cup\left( {(2J+1)}A ^{1/3},2(J+1) A^{1/3} \right],\\
0,&\quad \mathrm{others}.\end{array}\right.
\end{align} 
For this time periodic flow with period $2A^{1/3}$, we show the estimate \eqref{ED_tdsh_intr}. 
If $s, t\in 2A^{1/3}\mathbb{N}$, then by Lemma \ref{lem_ED_1}, we have the estimate \eqref{ED_tdsh_intr}. In the general case, we find  the largest integer $N_1\in\mathbb{N}$ so that $2A^{1/3} N_1\leq s+t$ and the smallest integer $N_2\in\mathbb{N}$ such that $2A^{1/3} N_2\geq s$. Further note that if $t<4A^{1/3}$, then the result \eqref{ED_tdsh_intr} is direct, i.e.,
\begin{align}
\|f(s+t)\|_2\leq \|f(s)\|_2 \leq (1-\kappa)^{-2}\|f(s)\|_2 e^{- \frac{1}{2}|\log(1-\kappa)|\frac{t}{A^{1/3}}},\quad \forall t\in[0, 4A ^{1/3}).
\end{align} Hence we assume that $t\geq 4A^{1/3}$. As a consequence,  $t-(N_1-N_2)2A^{1/3}\leq 4A^{1/3}$. Thanks to the dissipative nature of the $L^2$-norm and the enhanced dissipation estimate for $s,t\in 2A^{1/3}\mathbb{N}$ obtained before, we have
\begin{align}
\|f(s+t)\|_2\leq& \|f (2N_1A^{1/3})\|_2\leq \|f (2N_2A^{1/3})\|_2(1-\kappa)^{(N_2-N_1)}\leq \|f(2N_2A^{1/3})\|_2e^{ -\frac{1}{2}A^{-1/3}(t-4A^{1/3})\log(1-\kappa)^{-1}}\\
\leq &(1-\kappa)^{-2}\|f(s)\|_2e^{-\frac{1}{2}|\log(1-\kappa)| \frac{t}{A^{1/3}}},\quad \forall s\in[0,\infty),\,t\in[4A^{1/3},\infty).
\end{align} 
This concludes the proof of \eqref{ED_tdsh_intr}. 
\end{proof}

\subsection{Enhanced Dissipation in the Gliding Regularity Spaces}
In this subsection, we prove Theorem \ref{thm:ED_gldrg}. The proof develops a general framework to upgrade the $L^2$ enhanced dissipation to the higher gliding regularity spaces. The main obstacle to derive the enhanced dissipation estimates stem from the  commutator terms involving the $\Gamma$ vector field and the diffusion operator $\frac{1}{A}\de$. These challenges are addressed by the \emph{key Lemma \ref{lem:cm_L_trm}}. 

\myc{\footnote{\textcolor{red}{\textbf{Some new question: Can we upgrade this to $y^2$ flow? Namely, given a $\nu^{1/2}$-ED flow, can we derive a functional of the higher gliding regularity such that the ED of the $L^2$-level persists? The commutator terms are particularly weak near the critical points! And using the $t^2$ to model the commutator near the critical point is wasteful! So maybe combining the localization and the commutator estimate yields the Gliding norm ED estimate. \myb{Answer: Yes, there is indeed a commuting vector field for the Poiseuille flow. With this at hand, we can prove higher gliding regularity estimates. A nice question is to identify the general strategy of deriving commuting vector fields for nice enough non-degenerate shear flows.} \myr{Another key point: We can just do finitely many movement of the construction of Michele. If the process converges at the end, then we should be able to extract the first few key terms and show that the remainder term is small on time scale $[0,\nu^{-1/2-\ep}]$.}}}}
}


The remaining part of the subsection is devoted to the proof of Theorem \ref{thm:ED_gldrg}. We organize the proof in four main steps. 
  
\noindent
{\bf Step \# 1: General Setup. } We fix an arbitrary starting time $t_\star\in[0,\infty)$, and observe that Theorem \ref{thm:L_ED}, together with the choice of $\delta_{\mathcal Z}$ \eqref{Chc_del_M0},  yields that
\begin{align}
\left\|S_{t_\star}^{t_\star+\delta_{\mathcal Z}^{-1}A^{1/3}}\mathbb{P}_\nq\right\|_{L^2\rightarrow L^2}\leq \frac{1}{8}.\label{Chc_del_M}
\end{align} 
Here $\mathbb{P}_{\nq}$ is the projection operator such that $\mathbb{P}_{\nq}f=f_\nq.$ 
Next we observe that the following regularity and decay estimates guarantee the enhanced dissipation \eqref{Gld_Rg_ED},
\begin{align}
\|f_\nq(t_\star+\tau)\|_{\mathcal{Z}^M_{G,\Phi}}^2 \leq\  &2\ \|f_\nq(t_\star)\|_{\mathcal{Z}^M_{G,\Phi}}^2,\quad \forall \tau\in[0, \delta_{\mathcal Z}^{-1}A^{1/3}];\label{Reglrty_est}\\
\|f_{\nq}(t_\star+\delta_{\mathcal Z}^{-1}A^{1/3})\|_{\mathcal{Z}_{G,\Phi}^M}^2 \leq\ &\frac{1}{e^2}\ \|f_\nq(t_\star)\|_{\Z_{G,\Phi}^M}^2.\label{Decay_est}
\end{align}
The explicit argument is a variant to the one applied in the proof of Theorem \ref{thm:L_ED}.  Consider general $s,\,t\in[0,\infty)$.  
If the time instance $t$  in \eqref{Gld_Rg_ED} take values in the set $\delta_{\mathcal Z}^{-1}A^{1/3}\mathbb{N}$, iterative application of the decay estimate \eqref{Decay_est} yields the result  \eqref{Gld_Rg_ED}. In the general case, we find  the largest integer $N_1\in\mathbb{N}$ so that $\delta_{\mathcal Z}^{-1}A^{1/3} N_1\leq t$. As a consequence, we obtain the relation $t-N_1\delta_{\mathcal Z}^{-1}A^{1/3}< \delta_{\mathcal Z}^{-1}A^{1/3}\,\Rightarrow\, \delta_{\mathcal Z}A^{-1/3}(t-\delta_{\mathcal Z}^{-1}A^{1/3})<N_1$. Combining the regularity estimate \eqref{Reglrty_est} and the decay estimate \eqref{Decay_est} yields that, 
\begin{align}\label{ED_argument}\\
\|f_\nq(s+t)\|_{\Z^M_{G,\Phi(s+t)}}^2\leq\ 2\ \|f_\nq& (s+N_1\delta_{\mathcal Z}^{-1}A^{1/3})\|_{\Z^M_{G,\Phi(s+N_1\delta_{\mathcal Z}^{-1}A^{1/3})}}^2\leq\ 2\ \|f_\nq (s)\|_{\Z^M_{G,\Phi(s)}}^2\ e^{-2N_1}\\
\leq \ 2\ \|f_\nq(s)&\|_{\Z^M_{G,\Phi(s)}}^2\ \exp\left\{ -2\delta_{\mathcal Z} \frac{t-\delta_{\mathcal Z}^{-1}A^{1/3}}{A^{1/3}}\right\}
\leq  2e^2\ \|f_\nq(s)\|_{\Z^M_{G,\Phi(s)}}^2\ \exp\left\{-2\delta_{\mathcal Z}\frac{t}{A^{1/3}}\right\}.
\end{align} 
This is the result \eqref{Gld_Rg_ED}. Hence, it is enough to prove the estimates \eqref{Reglrty_est} and \eqref{Decay_est}. 

\noindent
{\bf Step \# 2: Proof of the Regularity Estimate \eqref{Reglrty_est}. } 
The {\bf Step \# 2} and {\bf Step \# 3} are mainly devoted to estimating  commutator terms. The  following \emph{key lemma} plays a major role. 
\begin{lem}\label{lem:cm_L_trm}
Consider function $ \mathfrak H_{ijk}\in H^1$ and the family $\{\mathfrak G_{ijk} \}_{|i,j,k|=m}\subset H^1$, where $0\leq m\myr{\leq M}$. Further assume that the family of functions $\{\mf G_{ijk}\}$ satisfies the following relations 
\begin{align}\label{G_cmp_rl}\pa_x\mG_{ijk}=\mG_{(i+1)jk},\quad \Gamma_y\mG_{ijk}=\mG_{i(j+1)k},\quad \pa_z \mG_{ijk}=\mG_{ij(k+1)}.
\end{align} Then there exists a constant $C$, which depends only on $M$ and \myr{$\|u\|_{L_t^\infty W^{M+2,\infty}}$\myc{(Checked once: matching the range above)}}, such that the following estimate holds 
\begin{align} \label{T_cm_est}
\bigg|\frac{G^{ 2i} \Phi^{2j}}{A}&\int \mathfrak{H}_{ijk} [\pa_{yy},\Gamma_y^j]\mathfrak G_{i0k}dV\bigg|\\
\leq&\frac{G^{2i}\Phi^{2j}}{8A}\|\na \mathfrak H_{ijk} \|_2^2+ \left(\frac{C \Phi}{A^{1/3}G^2}+\frac{C}{G(A^{2/3}+t^2)}\right)\lf(G^{2i}\Phi^{2j}\| \mathfrak H _{ijk} \|_2^2+\sum_{\substack{i'+j'+k'\leq m\\ j'\leq j-1}}G^{2i'}\Phi^{2j'} \| \mathfrak G_{i'j'k'}\|_{2}^2 \rg).
\end{align}\ifx
\begin{align}\label{T_com_L_12}
|T_1^{L}+T_2^L|&:=\lf|\sum_{|i,j,k|=0}^{M}\frac{G^{ 2i} \Phi^{2j}}{A}\int \mathfrak{H}_{ijk}  \sum_{\ell=0}^{j-1}\left(\begin{array}{rr}j\\ \ell\end{array}\right)\lf (-2B^{(j-\ell+1)}\pa_x\Gamma_y+\pa_y^{(j-\ell)}(B^{(1)})^2\pa_{xx}  \rg)  \mathfrak{G}_{i\ell k}  dV\rg|\\
&\leq \sum_{|i,j,k|=0}^{M}\frac{G^{\myr{?4+}2i}\Phi^{2j}}{8A}\|\na \mathfrak H_{ijk}\|_2^2+\left(\frac{C \Phi}{A^{1/3}G^2}+\frac{C}{G(A^{2/3}+t^2)}\right)\sum_{|i,j,k|=0}^{M}G^{2i}\Phi^{2j}\left(\| \mathfrak H_{ijk} \|_2^2+\| \mathfrak G_{ijk}\|_{2}^2\right) \\ \label{T_com_L_3}
|T_3^{L}|&:=\lf|\sum_{|i,j,k|=0}^{M}\frac{G^{2i}\Phi^{2j}}{A}\int \mathfrak H_{ijk}  \  \sum_{\ell=0}^{j-1}\left(\begin{array}{rr} j\\ \ell\end{array}\right)B^{(j-\ell+2)}\pa_x\mathfrak G_{i\ell k}   dV\rg|\\
&\leq\frac{C }{(A^{2/3}+t^2)G}\sum_{|i,j,k|=0}^{M}G^{2i}\Phi^{2j}\left(\| \mathfrak H_{ijk} \|_2^2+\| \mathfrak G_{ijk}\|_{2}^2\right).
\end{align}\fi
\end{lem}\begin{proof}The proof of the lemma is postponed to the end of the subsection. 
\end{proof}

To prove the regularity estimate \eqref{Reglrty_est}, we express the time evolution of $\|f_\nq\|_{\Z_{G,\Phi}^ M}^2$ as follows 
\begin{align}
\label{l_TLR}\frac{d}{d\tau}&\frac{1}{2}\sum_{|i,j,k|=0}^{M}G^{2i}\Phi^{2j} \|\palt f_\nq\|_2^2\\
\leq &-\sum_{|i,j,k|=0}^{M}\frac{G^{2i} \Phi^{2j}}{A}\|\na (\palt f_\nq)\|_2^2+\sum_{|i,j,k|=0}^{M}\frac{G^{2i} }{A}\Phi^{2j}\int \palt f_\nq \ [\Gamma_y^j,\pa_{yy}]\pa_x^i\pa_z^k f_\nq dV\\
=:&-\sum_{|i,j,k|=0}^{M}\frac{G^{2i} \Phi^{2j}}{A}\|\na (\palt f_\nq)\|_2^2+T_{cm}^{L;R}.
\end{align}
Here ``$L$'' stands for ``linear'' and ``$R$'' stands for ``regularity''. 
 
By setting $\mH=\palt f_\nq$ and $\mG_{ijk}=\palt f_\nq$ in Lemma \ref{lem:cm_L_trm} and checking that the condition \eqref{G_cmp_rl} holds, we reach the conclusion that the commutator term $T_{cm}^{L;R}$ in \eqref{l_TLR} is bounded as follows
\begin{align}
|T_{cm}^{L;R} |\leq& \frac{1}{8A}\sum_{|i,j,k|=0}^{M}G^{2i} \Phi^{2j}\|\na (\palt f_\nq)\|_2^2+\lf(\frac{C\Phi}{G^2A^{1/3}}+\frac{C}{G(A^{2/3}+ {t^2})}\rg)\|f_\nq\|_{\Z^M_{G,\Phi}}^2. 
\end{align} 
Here $C$'s are constants depending only on $M, \ \| u_A\|_{L^\infty_t W^{M+2,\infty}}$.  Hence,  for all $
 \tau\in[0,\delta_ M^{-1}A^{1/3}]$, 
\begin{align}
\frac{d}{d\tau}\|f_\nq(t_\star+\tau)\|_{\Z^ M_{G,\Phi}}^2\leq C\left(\frac{1}{G(A^{2/3}+(t_\star+\tau)^2)}+\frac{1}{A^{1/3}G^2(1+(t_\star+\tau)^3/A)}\right)&\|f_\nq(t_\star+\tau)\|_{\Z^M_{G,\Phi}}^2.
\end{align} 
Now by solving the differential relation and choosing the threshold $G_ M>0  $ large enough compared to $M$, $\delta_{\mathcal Z}^{-1}$ (as defined in \eqref{Chc_del_M}), and $\| u\|_{L_t^\infty W^{M+2,\infty}}$, we reach the result  \eqref{Reglrty_est}. This concludes the {\bf Step \# 2}.

\noindent
{\bf Step \# 3: Proof of the Decay Estimate \eqref{Decay_est}. }
Next we prove the decay estimate \eqref{Decay_est}. The handle we apply is the passive scalar solution $S_{t_\star}^{t_\star+\tau}( \pa_x^i\Gamma^{j}_{y;t_\star}\pa_z^k  f_\nq(t_\star))$, whose $\Z^ M_{G,\Phi}$-norm naturally decays to $1/8$ of its value at $t_\star$ after time $ \delta_{ M}^{-1}A^{1/3}$ \eqref{Chc_del_M}. To avoid lengthy notation, we use the simplified notions \[\Gamma_{y;t}^{ijk}=\pa_x^i\Gamma_{y;t_\star+\tau}^j\pa_z^k,\qquad S_{t_\star}^{t_\star+\tau} \Gamma^{ijk}_{y;t_\star}f_\nq=S_{t_\star}^{t_\star+\tau}\lf( \pa_x^i\Gamma^{j}_{y;t_\star}\pa_z^k f_\nq(t_\star)\rg).\] The remaining task is to show that the deviation between the solution $ \Gamma^{ijk}_{y;t} f_\nq$ and the passive solution $S_{t_\star}^{t_\star+\tau} \Gamma^{ijk}_{y;t_\star} f_\nq$ is small. To this end, we express the time evolution of the $\Z^M_{G,\Phi}$-norm of the difference as follows 
\begin{align}
\frac{d}{d\tau}&\frac{1}{2}\sum_{|i,j,k|=0}^{M}G^{2i} \Phi^{2j}\| \gt f_\nq-S_{t_\star}^{t_\star+\tau} \Gamma^{ijk}_{y;t_\star}  f_\nq\|_2^2\\
=&-\sum_{|i,j,k|=0}^{M}\frac{G^{2i} }{A}\Phi^{2j}\|\na (\gt  f_\nq-S_{t_\star}^{t_\star+\tau}\gts  f_\nq)\|_2^2 \\
&+\sum_{|i,j,k|=0}^{M} {G^{2i} j\Phi^{2j-1}\Phi'}  \|\gt f_\nq-S_{t_\star}^{t_\star+\tau}\gts  f_\nq\|_2^2 \\
  &+\sum_{|i,j,k|=0}^{M}\frac{G^{2i} \Phi^{2j} }{A}\int (\gt f_\nq-S_{t_\star}^{t_\star+\tau}\gts  f_\nq) [\Gamma_y^j,\pa_{yy}]\pa_x^i\pa_z^k f_\nq dV\\
=:&-\mathfrak D_1-\mathfrak D_2+T_{cm}^{L;D}.
\end{align}
Here ``$L$'' stands for ``linear'' and ``$D$'' stands for ``decay''. To estimate the commutator term,  we set $\mH=\gt  f_\nq-S_{t_\star}^{t_\star+\tau}\gts  f_\nq$ and $\mG_{ijk}=\gt f_\nq$ in Lemma \ref{lem:cm_L_trm}. The condition \eqref{G_cmp_rl} can be checked directly. 
Hence an application of Lemma \ref{lem:cm_L_trm} yields that
\begin{align}
|T_{cm}^{L;D}  | 
\leq& \frac{1}{8}\mathfrak D_1+\left(\frac{C \Phi(t_\star+\tau)}{A^{1/3}G^2}+\frac{C}{G(A^{2/3}+(t_\star+\tau)^2)}\right)\sum_{|i,j,k|=0}^{M}G^{2i}\Phi^{2j}\|\gt f_{\neq}-S_{t_\star}^{t_\star+\tau}\gts  f_\nq\|_2^2\\
&+\left(\frac{C\Phi(t_\star+\tau)}{A^{1/3}G^2}+\frac{C}{G(A^{2/3}+(t_\star+\tau)^2)}\right)\sup_{\tau\in[0,\delta_{\mathcal Z}^{-1}A^{1/3}]}\|f_{\nq}(t_\star+\tau)\|_{\Z^M_{G,\Phi}}^2.
\end{align}  
Here the constant $C$ depends only on the regularity level $M$ and the norm $\| u_A\|_{L_t^\infty W ^{M+2,\infty}}$.  Combining this bound with the regularity estimate \eqref{Reglrty_est}, and the decomposition \eqref{T_cm_est} yields that \begin{align}
 \frac{d}{d\tau}\sum_{|i,j,k|=0}^{M}&G^{2i}\Phi^{2j}\|\gt f_\nq-S_{t_\star}^{t_\star+\tau}\gts  f_\nq\|_{L^2}^2\\
 \leq & {C} \left(\frac{\Phi(t_\star+\tau)}{A^{1/3}G^2}+\frac{1}{G((t_\star+\tau)^2+A^{2/3})}\right)\\
 &\qquad \times \left(\sum_{|i,j,k|=0}^{M}G^{2i}\Phi^{2j}\|\gt f_\nq-S_{t_\star}^{t_\star+\tau}\gts f_\nq\|_{L^2}^2+\|f_{\neq}(t_\star)\|_{\Z^ M_{G, \Phi}}^2\right).
 \end{align}  
By applying a Gr\"onwall-type estimate, and choosing the threshold $G\geq G_ M >1$ large enough compared to $\|u\|_{L_t^\infty W^{M+2,\infty}},  \, M$ and $\delta_ M^{-1}$ (as defined in \eqref{Chc_del_M}), we achieve the following bound  
 \begin{align}
 \sum_{|i,j,k|=0}^{M}G^{2i}\Phi^{2j}\|\gt f_\nq-S_{t_\star}^{t_\star+\tau}\gts f_\nq\|_{L^2}^2\leq\frac{1}{32}\|f_{\neq}(t_\star)\|_{\Z^ M_{G, \Phi}}^2,\quad \forall \tau\in[0,\delta_{\mathcal Z}^{-1}A^{1/3}]. 
 \end{align}
Since the choice of $\delta_{\mathcal Z}$ \eqref{Chc_del_M} yields that $$ \sum_{|i,j,k|=0}^{M}G^{2i}\Phi^{2j}\|S_{t_\star}^{t_\star+\delta_{ M}^{-1}A^{1/3}}\gts f_\nq\|_2^2\leq \frac{1}{64}\|f_{\neq}(t_\star)\|_{\Z^ M_{G, \Phi}}^2,$$ we have obtained that 
 \begin{align}
 \|f_\nq (t_\star+\delta_{\mathcal Z}^{-1}A^{1/3})\|_{\Z^M_
 {G, \Phi}}^2\leq& 2\sum_{|i,j,k|=0}^{M}G^{2i}\Phi^{2j}\|\gt f_\nq(t_\star+\delta_{\mathcal Z}^{-1}A^{1/3})-S_{t_\star}^{t_\star+\delta_{\mathcal Z}^{-1}A^{1/3}}\gts f_\nq\|_2^2\\
 &+2\sum_{|i,j,k|=0}^{M}G^{2i}\Phi^{2j}\|S_{t_\star}^{t_\star+\delta_{\mathcal Z}^{-1}A^{1/3}}\gts f_\nq\|_2^2\\
 \leq &\lf(\frac{2}{32}+\frac{1}{32}\rg) \|f_\nq(t_\star)\|_{\Z^M_{G, \Phi}}^2\leq \frac{1}{e^2}\|f_\nq(t_\star)\|_{\Z^M_{G, \Phi}}^2.&
 \end{align}
At this point, we have obtained \eqref{Decay_est}. Hence, the enhanced dissipation \eqref{Gld_Rg_ED} is achieved. This concludes the {\bf Step \# 3}. 

\noindent
{\bf Step \# 4: Proof of \eqref{Z_implc}. }Thanks to the definition of the $\mathcal{Z}$-norm \eqref{Gliding_norm} and the estimate \eqref{Gld_Rg_ED}, we have that 
\begin{align}
\sum_{|i,j,k|=0}^{M}G^{2i}\|\palt f_\nq(t)\|_{L^2}^2\leq& C(M)\lf(\Phi^{-2M}e^{-\delta_{\mathcal Z}\frac{t}{A^{1/3}}}\rg)\sum_{|i,j,k|=0}^{M}G^{2i}\|\pa_x^i\pa_y^j\pa_z^k f_{\mathrm{in};\nq}\|_{L^2}^2 e^{-\delta_{\mathcal Z}\frac{t}{A^{1/3}}}\\
\leq &C(M,\delta_{\mathcal Z}^{-1})\sum_{|i,j,k|=0}^{M}G^{2i}\|\pa_x^i\pa_y^j\pa_z^k f_{\mathrm{in};\nq}\|_{L^2}^2 e^{-\delta_{\mathcal Z}\frac{t}{A^{1/3}}}. 
\end{align}
This concludes the proof of Theorem \ref{thm:ED_gldrg}. The only remaining task is to prove Lemma \ref{lem:cm_L_trm}. We collect the proof  below.  

\ifx
\begin{proof}[Proof of Lemma \ref{lem:Estimate_TLR_1}]
To estimate $T_{1}^{L;R}$ term in \eqref{linear_TLR_123}, we decompose it into two terms, 
\begin{align*}
T_{1}^{L;R}=&-\sum_{|i,j,k|=0}^{M}\frac{G^{2i}\Phi^{2j}}{A} 2j\int \palt f_\nq \ B^{(2)}\pa_x \palt f_\nq dV\\
&-\sum_{|i,j,k|=0}^{M}\sum_{\ell=0}^{j-2}\frac{G^{ 2i}\Phi^{2j}}{A}2\left(\begin{array}{rr}j\\ \ell\end{array}\right)\int \palt f_\nq \  B^{(j-\ell+1)} \pa_x^{i+1} \Gamma_y^{\ell+1}\pa_z^k f_\nq dV.
\end{align*}
We observe that the first term is zero through integration by parts. For the second term, we apply the H\"older inequality and Young's inequality to obtain
\begin{align*} 
T_{1}^{L;R}\leq& C \sum_{|i,j,k|=0}^{M}\frac{G^{2i}\Phi^{2j}}{A}\|\palt f_\nq\|_2\ \sum_{\ell=0}^{j-2}t \|\Gamma_y^{\ell+1}\pa_x^{i+1}\pa_z^k f_\nq\|_2\\
\leq&\frac{C}{A^{2/3}G} \sum_{|i,j,k|=0}^{M}\sum_{\ell=0}^{j-2}\left(G^{i}\Phi^{j}\|\palt f_\nq\|_2\right)\|u_y\|_{W^{M,\infty}}\left(\frac{\Phi^{j-\ell-1}t}{A^{1/3}}\right)\left(\Phi^{\ell+1}G^{i+1}\|\pa_x^{i+1} \Gamma_y^{\ell+1}\pa_z^k f_\nq\|_2\right). 
\end{align*}Since $\ell\leq j-2$, we have that $\Phi^{ j-\ell-1 }\leq \Phi $. Recalling the definition of $\Phi$ \eqref{varphi}, we have that\begin{align}
T_1^{L;R}\leq&\frac{C}{A^{2/3}G} \sum_{|i,j,k|=0}^{M}\sum_{\ell=0}^{j-2}\frac{t/A^{1/3}}{ 1+t^{ 3}/A}\left(G^{i}\Phi^{j}\|\palt f_\nq\|_2\right)\left(\Phi^{\ell+1}G^{i+1}\|\pa_x^{i+1}\Gamma_y^{\ell+1}\pa_z^k f_\nq \|_2\right)\\
\leq&\frac{C}{A^{2/3}G}\frac{1}{1+\frac{t^2}{A^{2/3}}}\|f_\nq\|_{\Z^ M_{G;\Phi}}^2.\label{T_1R}
\end{align}
This concludes the proof of the lemma.  
\end{proof}

\begin{proof}[Proof of Lemma \ref{lem:Estimate_TLR_23}]
To estimate the term $T_{2}^{L;R}$ in \eqref{linear_TLR_123}, we single out the $\ell=j-1$ case as follows
\begin{align}
T_{2}^{L;R}=& \sum_{|i,j,k|=0}^{M}\sum_{\ell=0}^{j-1}\frac{G^{2i}\Phi^{2j}}{A}\int \palt f_\nq \ \left(\begin{array}{rr}j\\ \ell\end{array}\right)(\pa_y^{(j-\ell)}(B^{(1)})^2\pa_{xx})\pa_x^i\Gamma_y^\ell\pa_z^k f_\nq dV\\
=& \sum_{|i,j,k|=0}^{M}\frac{G^{2i}\Phi^{2j}}{A}j\int (\pa_y+B^{(1)}\pa_x)(\pa_x^i\Gamma_y^{j-1}\pa_z^k f_\nq) \ \pa_y(B^{(1)})^2\pa_{xx}(\pa_x^i \Gamma_{y}^{j-1}\pa_z^kf_\nq )dV\\
&+ \sum_{|i,j,k|=0}^{M}\sum_{\ell=0}^{j-2}\frac{G^{2i}\Phi^{2j}}{A}\left(\begin{array}{rr}j\\ \ell\end{array}\right)\int \palt f_\nq\ \pa_y^{(j-\ell)}(B^{(1)})^2\ \pa_x^{i+2}\Gamma_y^\ell\pa_z^k f_\nq dV\\
=:&T_{2;1}^{L;R}+T_{2;2}^{L;R}. \label{T_LR_212}
\end{align}
We start with estimating $T_{2;1}^{L;R}$ in \eqref{T_LR_212},

\begin{align}
T_{2;1}^{L;R}=&\sum_{|i,j,k|=0}^{M}\frac{G^{2i}\Phi^{2j}}{A}j\int  \pa_y (\pa_x^i\Gamma_y^{j-1}\pa_z^k f_\nq) \ \pa_y(B^{(1)})^2\ \pa_{xx}(\pa_x^i \Gamma_{y}^{j-1}\pa_z^kf_\nq )dV\\
&+\sum_{|i,j,k|=0}^{M}\frac{G^{2i}\Phi^{2j}}{A}j\int B^{(1)}\ \pa_x (\pa_x^i\Gamma_y^{j-1}\pa_z^k f_\nq) \ \pa_y(B^{(1)})^2 \ \pa_{xx}(\pa_x^i \Gamma_{y}^{j-1}\pa_z^kf_\nq )dV\\
=&-\sum_{|i,j,k|=0}^{M}\frac{G^{2i}\Phi^{2j}}{A}j\int  \frac{1}{2}\pa_y(\pa_x^{i+1}\Gamma_y^{j-1}\pa_z^k f_\nq)^2 \ \pa_y(B^{(1)})^2 dV\\
&+\sum_{|i,j,k|=0}^{M}\frac{G^{2i}\Phi^{2j}}{A}j\int \frac{1}{2}\pa_x(\pa_x^{i+1}\Gamma_y^{j-1}\pa_z^k f_\nq)^2 \  B^{(1)}\ \pa_y(B^{(1)})^2 dV\\
=&\sum_{|i,j,k|=0}^{M}\frac{G^{2i}\Phi^{2j}}{A}j\int  \frac{1}{2}(\pa_x^{i+1}\Gamma_y^{j-1}\pa_z^k f_\nq)^2 \ \pa_{yy}(B^{(1)})^2 dV.
\end{align}
Hence we have that 
\begin{align}
|T_{2;1}^{L;R}|\leq \frac{ C}{A^{1/3}G^2} \ \frac{\frac{t^2}{A^{2/3}}}{1+\frac{t^6}{A^2}}\ \sum_{|i,j,k|=0}^{M}\ G^{2(i+1)}\Phi^{2(j-1)}\ \|\pa_x^{i+1}\Gamma_y^{j-1}\paz^k f_\nq\|_{L^2}^2.\label{T_21LR}
\end{align}Next we estimate the $T_{2;2}^{L;R}$ term,
\begin{align}
|T_{2;2}^{L;R}|\leq &\frac{C}{A^{1/3}G^2} \ \sum_{|i,j,k|=0}^{M}\sum_{\ell=0}^{j-2}\ \frac{  \Phi^{j-\ell}t^2}{A^{2/3}}\ \left(G^{i}\Phi^{j} \|\pa_x^{i}\Gamma_y^{j}\pa_z^k f_\nq\|_2\right)\ \left(G^{i+2}\Phi^{\ell }\|\pa_x^{i+2}\Gamma_y^\ell\pa_z^k f_\nq\|_2\right).
\end{align}Since $\ell\leq j-2$,  $\Phi^{j-\ell} \leq \Phi^2$, and we have \begin{align}
|T_{2;2}^{L;R}|\leq &\ \frac{C}{A^{1/3}G^2} \ \sum_{|i,j,k|=0}^{M}\sum_{\ell=0}^{j-2}\ \frac{  \frac{t^2}{A^{2/3}}}{1+\frac{t^6}{A^2}}\  \left(G^{i}\Phi^{j} \|\pa_x^{i}\Gamma_y^{j}\pa_z^k f_\nq\|_2\right)\ \left(G^{i+2}\Phi^{\ell }\|\pa_x^{i+2}\Gamma_y^\ell\pa_z^k f_\nq\|_2\right)\\
\leq&\ \frac{C}{A^{1/3}G^2 (1+\frac{t^3}{A})}\|f_\nq\|_{\Z^ M_{G,\Phi}}^2.\label{T_22LR}
\end{align}
Combining the estimates \eqref{T_21LR}, \eqref{T_22LR} and \eqref{T_LR_212} yields the estimate of $T_{2}^{L;R}$. 
The treatment of the $T_{3}^{L;R}$ is similar to the treatment of $T_{2}^{L;R}$. So we omit the detail for the sake of brevity. \textcolor{blue}{(Check?! It seems that the term has less $t$ and $\pa_x$. So it should be fine, right?)} 
\end{proof}
\begin{proof}[Proof of Lemma \ref{lem:T123_LD}]
Here we distinguish between the $\ell=j-1$ case and $\ell\leq j-2$ case.  In the $\ell=j-1$ case, the $T_{1}^{L;D}+T_{2}^{L;D}$ has the following simplification 
\begin{align}
&|T_{1}^{L;D}+T_{2}^{L;D}|\\
&=\bigg|-\sum_{|i,j,k|=0}^{M}\frac{G^{2i}\Phi^{2j}}{A}j \int (\palt f_\nq-S_{t_\star}^{t_\star+\tau}\palt  f_\nq)  \ 2B^{(2)}\ \pa_x\pa_y\Gamma_{y}^{j-1}\pa_x^i\pa^k_z f_\nq dV\bigg|\\
&\leq\sum_{|i,j,k|=0}^{M}\frac{G^{2i}\Phi^{2j}}{12 A}\|\pa_y (\palt f_\nq-S_{t_\star}^{t_\star+\tau}\palt  f_\nq)\|_2^2+\sum_{|i,j,k|=0}^{M}\frac{CG^{2i+2}\Phi^{2j-2}(\Phi^2 t^2)}{AG^2}\|\pa_x^{i+1}\Gamma _y^{j-1}\pa_z^k f_\nq\|_2^2\\
&\leq\sum_{|i,j,k|=0}^{M}\frac{G^{2i}\Phi^{2j}}{12A}\|\na(\palt f_\nq-S_{t_\star}^{t_\star+\tau}\palt  f_\nq)\|_2^2+\sum_{|i,j,k|=0}^{M}\frac{C G^{2i+2}\Phi^{2j-2}}{A^{1/3}G^2}\frac{t^2/A^{2/3}}{1+t^6/A^{2}}\|\pa_x^{i+1}\Gamma_y^{j-1}\pa_z^kf_\nq\|_2^2\\
 &\leq\frac{1}{8}\mathfrak D+\frac{C}{A^{1/3}G^2(1+t^3/A)}\sup_{\tau\in[0,\delta_{\mathcal Z}^{-1}A^{1/3}]}\|f_\nq(t_\star+\tau)\|_{\Z^ M_{G,\Phi}}^2.\label{Case_a_T12D} 
\end{align} 

For the $\ell\leq j-2$ case, we estimate the two terms individually. For the $T_{1}^{L;D}$ term, we have
\begin{align}
|T_{1}^{L;D}|=&\bigg|-2\sum_{|i,j,k|=0}^{M}\frac{G^{2i}\Phi^{2j}(t)}{A}\sum_{\ell=0}^{j-2}\left(\begin{array}{rr}j\\\ell\end{array}\right) B^{(j-\ell+1)}\int (\palt f_\nq-S_{t_\star}^{t_\star+\tau}\palt  f_\nq)\ \pa_x^{i+1}\Gamma_y^{\ell+1}\pa_z^k f_\nq dV\bigg|\\
\leq &\sum_{|i,j,k|=0}^{M}\sum_{\ell=0}^{j-2}\frac{C}{A^{2/3}G}\left(\frac{\Phi^{ j-\ell-1}t}{A^{1/3}}\right)\big(G^{ i}\Phi^{j}\|\palt f_{\neq}-S_{t_\star}^{t_\star+\tau}\palt  f_\nq\|_2\big) \ \big(G^{i+1}\Phi^{\ell+1}\|\pa_x^{i+1}\Gamma_y^{\ell+1}\pa_z^kf_\nq\|_2\big)\\
\leq &\frac{C }{(A^{2/3}+t^2)G}\sum_{|i,j,k|=0}^{M}G^{2i}\Phi^{2j}\|\palt f_{\neq}-S_{t_\star}^{t_\star+\tau}\palt f_\nq\|_2^2\\
&+\frac{C}{(A^{2/3}+t^2)G}\sup_{\tau\in[0,\delta_{\mathcal Z}^{-1}A^{1/3}]}\|f_{\nq}(t_\star+\tau)\|_{\Z^M_{G,\Phi}}^2. \label{Case_b_T1D} 
\end{align}

Next we estimate the $T_{2}^{L;D}$ term, which has the $t^2$ factor. We estimate it as follows 
\begin{align}
|T_{2}^{L;D}|\leq& \frac{C}{G^2A^{1/3}}\sum_{|i,j,k|=0}^{M}\sum_{\ell=0}^{j-2}\left(\begin{array}{rr}j\\\ell\end{array}\right)\frac{\Phi^{j-\ell} t^2}{A^{2/3}}(G^{ i}\Phi^{j}\|\palt f_\nq-S_{t_\star}^{t_\star+\tau}\palt  f_\nq\|_2)\ (G^{i+2}\Phi^{\ell}\|\pa_x^{i+2}\Gamma_y^{\ell}\pa_z^k f_\nq\|_2 )\\
\leq&\frac{C }{A^{1/3}G^2(1+t^3/A)}\sum_{|i,j,k|=0}^{M}G^{2i}\Phi^{2j}\|\palt f_{\neq}-S_{t_\star}^{t_\star+\tau}\palt  f_\nq\|_2^2\\
&+\frac{C}{A^{1/3}G^2(1+t^3/A)}\sup_{\tau\in[0,\delta_{\mathcal Z}^{-1}A^{1/3}]}\|f_{\nq}(t_\star+\tau)\|_{\Z^M_{G,\Phi}}^2.\label{Case_b_T2D} 
\end{align}
The estimate of the $T_3^{L;D}$ is as follows:
\begin{align}
|T_3^{L;D}|\leq&\sum_{|i,j,k|=0}^{M}C\sum_{\ell=0}^{j-1}\frac{G^{2i}\Phi^{2j}}{A}\int |\palt f_\nq-S_{t_\star}^{t_\star+\tau}\palt  f_\nq| \ |B^{(j-\ell+2)}|\ |\pa_{x}^{i+1} \Gamma_y^{\ell}\pa_z^k f_\nq | dV\\
\leq &\frac{C}{A^{2/3}G}\sum_{|i,j,k|=0}^{M}(G^{i}\Phi^{j}\|\palt f_\nq-S_{t_\star}^{t_\star+\tau}\palt f_\nq\|_2)  \left(\frac{\Phi^{ j-\ell}t}{A^{1/3}}\right)(G^{i+1}\Phi^{ {\ell} }\|\pa_x^{i+1}\Gamma_y^{j-1}\pa_z^k f_\nq\|_2)\\
\leq&\frac{C}{(A^{2/3}+t^2)G}\left(\sum_{|i,j,k|=0}^{M}G^{2i}\Phi^{2j}\|\palt f_\nq-S_{t_\star}^{t_\star+\tau}\palt f_\nq\|_2^2+\sup_{\tau\in[0,\delta_{\mathcal Z}^{-1}A^{1/3}]}\|f_{\neq}(t_\star+\tau)\|_{ {\Z^ M_{G, \Phi}}}^2\right).&\label{Case_a_T3D}
\end{align}
This concludes the proof. \end{proof}
\fi
\begin{proof}[Proof of Lemma \ref{lem:cm_L_trm}]
Since the result \eqref{T_cm_est} is trivially true for  $j=0$, we always assume that $j\geq 1$ throughout the proof.  According to the commutator relation \eqref{cm_yt_j_yy} and the computation rule \eqref{G_cmp_rl}, we can rewrite the left hand side of \eqref{T_cm_est} in the following fashion,
\begin{align}\label{Tcm_est_123}
\bigg|\frac{G^{2i} \Phi^{2j}}{A}&\int \mf H_{ijk} \ [\Gamma_y^j,\pa_{yy}]\mf G_{i0k} dV\bigg|\\
=&\bigg|\frac{G^{2i}\Phi^{2j}}{A}\int \mf H_{ijk}\  \sum_{\ell=0}^{j-1}\binom{j}{\ell}\lf(-2B^{(j-\ell+1)}\pa_x\Gamma_y+\pa_y^{(j-\ell)}(B^{(1)})^2\pa_{xx}-B^{(j-\ell+2)}\pa_x\rg)\mf G_{i\ell k} dV\bigg|\\
=&: \bigg|T_{1} +T_{2} +T_{3} \bigg|.
\end{align}
Here the quantities $B^{(m)}$ are defined as 
\begin{align}\label{B_m_mt}
B^{(m)}(t,y):=\int_0^t \pa_y^{m}u(s,y)ds.
\end{align} With this notation, the $\Gamma$-derivative can be rewritten as $\pa_y +B^{(1)}\pa_x. $ 

Let us start by estimating the quantity $|T_1+T_2|$. Here we distinguish between the $\ell=j-1$ case and the $0\leq\ell\leq j-2$ case.  In the first case, we expand $\Gamma$ as $\pa_y+B^{(1)}\pa_x$ and observe the following relation
\begin{align}
\lf(-2B^{(2)}\pa_x\Gamma_y+\pa_y(B^{(1)})^2\pa_{xx}\rg)\mG_{i(j-1)k}=&-2B^{(2)} \lf(\pa_{xy} +B^{(1)}\pa_{xx}\rg)\mG_{i(j-1)k}+2B^{(2)}B^{(1)}\pa_{xx}\mG_{i(j-1)k}\\
=&-2B^{(2)}\pa_{xy}\mG_{i(j-1)k}. 
\end{align}Hence we can simplify the $T_{1}+T_{2}$ with the property $\pa_x\mG_{ijk}=\mG_{(i+1)jk}$ and integration by parts,  
\begin{align}
&|T_{1}+T_{2}|\mathbbm{1}_{\ell=j-1}=\bigg|\frac{G^{2i}\Phi^{2j}}{A}j \int \mH \ 2B^{(2)}\ \pa_x\pa_y \mG_{i(j-1)k} dV\bigg|\\
&= \bigg|\frac{G^{2i}\Phi^{2j}}{A}j \int \pa_y\mH \ 2B^{(2)}\  \mG_{(i+1)(j-1)k} dV + \frac{G^{2i}\Phi^{2j}}{A}j \int \mH \ 2  B^{(3)}\  \mG_{(i+1)(j-1)k} dV\bigg|.
\end{align}
Now applications of the H\"older inequality, the Young's inequality   and the definition $\Phi(t)=1+\frac{t^3}{A}$ \eqref{varphi} yield that 
\begin{align}\label{Case_a_T12} 
|T_{1}  +T_{2} |\mathbbm{1}_{\ell=j-1} &\leq \frac{G^{2i}\Phi^{2j}}{12 A}\|\pa_y \mH\|_2^2+ \frac{CG^{2i+2}\Phi^{2j-2}(\Phi^2 t^2)}{AG^2}\|\mG_{(i+1)(j-1)k}\|_2^2\\
&\quad+ \frac{G^{2i}\Phi^{2j}}{G (A^{2/3}+t^2)}\| \mH\|_2^2+\frac{CG^{2i+2}\Phi^{2j-2}\Phi^2 }{G A^{2/3}}\frac{(A^{2/3} t^2+t^4)}{A^{4/3}}\| \mG_{(i+1)(j-1)k}\|_2^2\\
&= \frac{G^{2i}\Phi^{2j}}{12A}\|\pa_y\mH\|_2^2+ \frac{C G^{2i+2}\Phi^{2j-2}}{A^{1/3}G^2}\frac{t^2/A^{2/3}}{(1+t^3/A)^{2}}\|\mG_{(i+1)(j-1)k}\|_2^2\\
&\quad+ \frac{G^{2i}\Phi^{2j}}{G (A^{2/3}+t^2)}\| \mH\|_2^2+ \frac{CG^{2i+2}\Phi^{2j-2} }{G A^{2/3}}\frac{( t^2/A^{2/3}+t^4/A^{4/3})}{(1+t^3/A)^2}\| \mG_{(i+1)(j-1)k}\|_2^2\\
 &\leq\frac{ {G^{2i}\Phi^{2j}}}{12A} \|\pa_y\mH\|_2^2+ \frac{G^{2i}\Phi^{2j}}{G (A^{2/3}+t^2)}\| \mH\|_2^2\\
 &\quad+\lf(\frac{C}{A^{1/3}G^2(1+t^3/A)}+\frac{C}{{G (A^{2/3}+t^2)}}\rg) G^{2(i+1)}\Phi^{2(j-1)}\|\mathfrak G_{(i+1)(j-1)k}\|_{2}^2.
\end{align} 
For the $0\leq\ell\leq j-2$ case, we estimate the two terms individually. For the $T_{1} $ term, we have that
\begin{align} \label{Case_b_T1} 
|T_{1} |\mathbbm{1}_{0\leq \ell\leq j-2}=&\bigg| 2 \frac{G^{2i}\Phi^{2j} }{A}\sum_{\ell=0}^{j-2}\binom{j}{\ell}\int  \mH \  B^{(j-\ell+1)} \pa_x\Gamma_y \mG _{i\ell k}   dV\bigg|\\
\leq &\sum_{\ell=0}^{j-2}\frac{C}{A^{2/3}G}\left(\frac{\Phi^{ j-\ell-1}t}{A^{1/3}}\right)\big(G^{ i}\Phi^{j}\|\mH\|_2\big) \ \big(G^{i+1}\Phi^{\ell+1}\|\mG_{(i+1)(\ell+1)k}\|_2\big)\\
\leq &\frac{C }{(A^{2/3}+t^2)G} G^{2i}\Phi^{2j}\|\mH\|_2^2 +\frac{C}{(A^{2/3}+t^2)G}\sum_{\substack{i'+j'+k'\leq m\\ j'\leq j-1}}G^{2i'}\Phi^{2j'}\|\mG_{i'j'k'}\|_2^2.
\end{align} 
Next we estimate the $T_{2} $ term as follows  
\begin{align}\label{Case_b_T2} 
|T_{2} |\mathbbm{1}_{0\leq\ell\leq j-2}\leq& \frac{C}{G^2A^{1/3}} \sum_{\ell=0}^{j-2}\binom{j}{\ell}\frac{\Phi^{j-\ell} t^2}{A^{2/3}}(G^{ i}\Phi^{j}\|\mH\|_2)\ (G^{i+2}\Phi^{\ell}\|\mG_{(i+2)\ell k}\|_2 )\\
\leq& \frac{C}{G^2A^{1/3}} \sum_{\ell=0}^{j-2}\frac{ t^2/A^{2/3}}{(1+t^3/A)^2}(G^{ i}\Phi^{j}\|\mH\|_2)\ (G^{i+2}\Phi^{\ell}\|\mG_{(i+2)\ell k}\|_2 )\\
\leq&\frac{C }{A^{1/3}G^2(1+t^3/A)}\lf(G^{2i}\Phi^{2j}\|\mH\|_2^2+\sum_{\substack{i'+j'+k'\leq m\\ j'\leq j-1}}G^{2i'}\Phi^{2j'}\|\mG_{i'j'k'}\|_2^2\rg).  
\end{align} 
The estimate of the $T_3$ is as follows:
\begin{align}\label{T3}
|T_3|\leq&C\sum_{\ell=0}^{j-1}\frac{G^{2i}\Phi^{2j}}{A}\int |\mH| \ |B^{(j-\ell+2)}|\ |\mG_{(i+1 ) \ell k}   | dV\\
\leq &\frac{C}{A^{2/3}G}\sum_{\ell=0}^{j-1}(G^{i}\Phi^{j}\|\mH\|_2)  \left(\frac{\Phi^{ j-\ell}t}{A^{1/3}}\right)(G^{i+1}\Phi^{ {\ell} }\|\mG_{(i+1)\ell k}\|_2)\\
\leq&\frac{C}{(A^{2/3}+t^2)G}\lf(G^{2i}\Phi^{2j}\|\mH\|_2^2 +\sum_{\substack{i'+j'+k'\leq m\\ j'\leq j-1}}G^{2i'}\Phi^{2j'}\|\mG_{i'j'k'}\|_2^2 \right).
\end{align}
Combining \eqref{Case_a_T12}, \eqref{Case_b_T1}, \eqref{Case_b_T2}, and \eqref{T3} yields \eqref{T_cm_est}.
\end{proof}

\section{Nonlinear Theory}\label{Sec:ED} 
In this section, we develop the nonlinear theory for the system \eqref{PKS_rsc}. We will prove the propositions stated in Section \ref{Sec_s:nl}. 

First of all, we observe that to derive the estimates  in Proposition \ref{Pro:main} (or other propositions in Section \ref{Sec_s:nl}), there are two types of quantities to consider, namely, the $x$-average of the unknowns $\nz, \, \cz,$ and the remainders of the solutions $n_\nq,\, c_\nq^{t_r},\, d_\nq^{t_r}.$ In Subsection \ref{Sec:NL_xavg}, we collect the lemmas which provide bounds to the $x$-averages. In Subsections \ref{Sec:NLEDreg} - \ref{Sec:NL_EDdc}, we collect estimates of the remainder. The main goal is to prepare necessary bounds to derive the nonlinear  enhanced dissipation. As in the linear case, our general scheme is to fix an arbitrary time $t_\star$ on the time horizon, and then derive the following regularity estimate and the decay estimate on the time intervals $[t_\star,t_\star+\delta^{-1}A^{1/3}]$, i.e.,
\begin{align}
\FM[t_\star+\tau] \ \leq \ &2\FM[t_\star],\quad \forall \tau\in[0, \delta^{-1}A^{1/3}];\label{Rg_est_NL}\\
\FM[t_\star+\delta^{-1}A^{1/3}]\ \leq\ &\frac{1}{2}\FM[t_\star].\label{Decay_est_NL}
\end{align}
We recall that the parameter $\delta$ is chosen as $\delta=\delta(\delta_{\mathcal Z})$ \eqref{Chc_del}, and $\delta_{\mathcal Z}$ is defined in \eqref{Chc_del_M0}.  
We collect theorems involving \eqref{Rg_est_NL} in Subsection \ref{Sec:NLEDreg} and decay estimates associated with \eqref{Decay_est_NL} in Subsection \ref{Sec:NL_EDdc}. Finally, in Subsection \ref{Sec:NL_Con}, we combine results in previous sections to derive Propositions  \ref{pro:prp_FM}, \ref{Pro:2}, and \ref{Pro:main}. 

 \myc{Need to check that if we multiply the $\Phi^2$ to the $c$ part, the estimates of $c_\nq$  and $T_{c;com}$ are still fine. }

\ifx
The goal is to show that for $G\geq G(M, \|u_y\|_{L^\infty_tW_y^{M+1,\infty}})\geq 1$, and $A\geq A_0\geq 1, \quad 0<\delta_{\mathcal Z}\leq 1$, the enhanced dissipation estimate holds:
\begin{align}\label{Goal_F_M_ED}
F_{M}(t)\leq C_{M} F_{M}(0)e^{-\delta_{\mathcal Z}t/A^{1/3}}, \quad\forall t\in[0, A^{1/3+\te}) .
\end{align}
Same as in the linear case, the key is to derive the regularity estimate and the decay  estimate. In order to show the enhanced dissipation, we  first fix an arbitrary starting time $t_\star\in[0,A^{1/3+\te})$ such that $[t_\star,t_\star+\delta_{\mathcal Z}^{-1}A^{1/3}]\subset [0, A^{1/3+\te}] $. 
Next we prove the regularity estimate and the decay estimate on the time interval $[t_\star,t_\star+\delta_{\mathcal Z}^{-1}A^{1/3}]$, i.e.,
\begin{align}
F_M(t_\star+\tau) \leq &2F_M(t_\star),\quad \forall \tau\in[0, \delta_{\mathcal Z}^{-1}A^{1/3}];\label{Regularity_est_NL}\\
F_M(t_\star+\delta_{\mathcal Z}^{-1}A^{1/3})\leq&\frac{1}{2}F_M(t_\star).\label{Decay_est_NL}
\end{align}
\fi

\subsection{The $x$-average Estimates} \label{Sec:NL_xavg}
In this section, we present two lemmas concerning the $x$-average $\lan n\ran,\ \lan \cc\ran$. 

\begin{theorem}\label{thm:<n>_est}
Consider solutions to the system \eqref{ppPKS_<n>}, \eqref{ppPKS_<c>} subject to initial conditions $(\lan n_{\mathrm{in}}\ran,\lan \cc_{\mathrm {in}}\ran)\in H^M\times H^{M+1}$, $M\geq 3$. There exists a constant $C$ depending only on $M$ and $\|\pa_y u_A\|_{L_t^\infty W^{M,\infty}}$ such that the following estimates hold  for all $ t\in [0,A^{1/3+\te}]$, \myr{
\begin{align}\label{<n>_est_dff}
\frac{d}{dt}&\left(\sum_{|j,k|=0}^{ M}\|\pa_y^j\pa_z^k \nz (t)\|_{L_{y,z}^2}^2\right)\\
\leq &\frac{C}{A}\|\nz(t)\|_{H_{y,z}^M}^2\left(\sup_{\tau\in[0, A^{1/3+\te}]}\|\nz(\tau)\|_{H_{y,z}^{M}}^2 +\|\na \lan \cc_{\mathrm{in}}\ran\|_{H_{y,z}^M}^2\right)\\
& +\frac{C}{A}\lf(\sum_{|i,j,k|=0}^{ M+1}\| \Gamma_{y;t}^{ijk} \cc_\nq\|_{L_{x,y,z}^2}^2+ t^2 \sum_{ |i,j,k|=0}^{ M } \| \Gamma_{y;t}^{(i+1)j k} \cc_\nq\|_{L_{x,y,z}^2}^2\rg) \left(\sum_{|i,j,k|=0}^ M\|\Gamma_{y;t}^ {ij k}  n_\nq\|_{L^{2}_{x,y,z}}^2\right).
\end{align}}
\myc{He: In this way, we can reduce the discussion of the zero mode to the discussion of $\cc_\nq. $}
\end{theorem} 
\begin{proof}
We recall the equation \eqref{ppPKS_<n>} and implement direct energy estimate of the Sobolev norm to derive that
\begin{align}\label{T_NL0_12}
\frac{d}{dt}\frac{1}{2}\sum_{|j,k|=0}^ {M}\|\pa_{y}^j\pa_{z}^k\nz\|_{L^2}^2=&-\frac{1}{A}\sum_{|j,k|=0}^ {M}\|\na \pa_{y}^j\pa_{z}^k \nz\|_{L^2}^2+\frac{1}{A}\sum_{|j,k|=0}^ {M}\int \na_{y,z}\pa_{y}^j\pa_{z}^k\nz \cdot \pa_y^j\pa_z^k\lf( \nz\na\cz\rg)dV\\
&+\frac{1}{A}\sum_{|j,k|=0}^ {M}\int \na_{y,z}\pa_y^j\pa_z^k \nz\cdot \pa_y^j\pa_z^k\lf\lan n_\nq \na_{y,z} \cc_{\nq}\rg\ran dV\\
=&-\frac{1}{A}\sum_{|j,k|=0}^ {M}\|\na \pa_{y}^j\pa_{z}^k \nz\|_{L^2}^2+T_1+T_2.
\end{align}
To estimate $T_1$ term in \eqref{T_NL0_12}, we apply the Sobolev product estimate ($M\geq 2$) and the integral estimate \eqref{na_c_0_est} \myc{(We need an $H^M$ version.  Added.)} to obtain
\begin{align}\label{T_NL_0_1_est}
|T_1|\leq &\frac{1}{4A}\sum_{|j,k|=0}^ {M}\|\na\pa_y^j\pa_z^k \nz\|_{L_{y,z}^2}^2+\frac{C}{A}\|\nz\|_{H_{y,z}^M}^2\|\na \cz\|_{H_{y,z}^M}^2\\
\leq &\frac{1}{4A}\sum_{|j,k|=0}^ {M}\|\na\pa_y^j\pa_z^k \nz\|_2^2+\frac{C}{A}\|\nz\|_{H_{y,z}^M}^2\left(\sup_{\tau\in[0, t]}\|\nz(\tau)\|_{H_{y,z}^{M}}^2 +\|\na \lan \cc_{\text{in}}\ran\|_{H_{y,z}^M}^2\right).
\end{align}

To estimate the $T_2$ term in \eqref{T_NL0_12},  we invoke the product estimate \eqref{Prd_est_gld}, and the fact that $\pa_y^j\lan f\ran^x=\lan \Gamma_{y;t}^jf\ran^x$ to estimate
\begin{align}
 \lf|T_2\rg| 
&=\bigg|\frac{1}{A}\sum_{|j,k|=0}^ {M}\int \na_{y,z}\pa_y^j\pa_z^k \nz\cdot \lan\Gamma_{y;t}^j\pa_z^k(n_\nq \na_{y,z}  \cc_{\nq})\ran dV\bigg|\\
&\leq \frac{1}{8A}\sum_{|j,k|=0}^ {M}\|\na\pa_y^j\pa_z^k \nz\|_{L_{y,z}^2}^2+\frac{C}{A}\sum_{|j,k|=0}^ {M}\|\Gamma_{y;t}^{j}\pa_z^k (n_\nq\na_{y,z} \cc_\nq)\|_{L^2_{x,y,z}}^2\\
&\leq\frac{1}{8A}\sum_{|j,k|=0}^ {M}\|\na\pa_y^j\pa_z^k \nz\|_{L_{y,z}^2}^2+\frac{C}{A}\left(\sum_{|i,j,k|=0}^{M}\| \Gamma_{y;t}^{ijk}  \na_{y,z} \cc_\nq \|_{L^2_{x,y,z}}^2\right)\left(\sum_{|i,j,k|=0}^{M}\|\Gamma_{y;t}^{ijk}  n_\nq\|_{L^{2}_{x,y,z}}^2\right).
\end{align} 
Now we apply the estimate of the gradient \eqref{Gld_Reg_grd} to obtain that 
\begin{align}\label{T_NL_0_2_est}
|T_2|
\leq&  \frac{1}{4A}\sum_{|j,k|\leq M}\|\na\pa_y^j\pa_z^k \nz\|_{L_{y,z}^2}^2\\
&+ \frac{C}{A}\bigg(\sum_{|i,j,k|=0}^{ M+1}\|\Gamma_{y;t}^{ijk} \cc_\nq\|_{L_{x,y,z}^2}^2+ {t}^2 \sum_{ |i,j,k|=0}^{ M }  \|\Gamma_{y;t}^{(i+1)jk}\cc_\nq\|_{L_{x,y,z}^2}^2\bigg) \left(\sum_{|i,j,k|=0}^{M}\|\Gamma_{y;t}^{ ij k } n_\nq\|_{L^{2}_{x,y,z}}^2\right).
\end{align}
Now combining the decomposition \eqref{T_NL0_12}, and the estimates \eqref{T_NL_0_1_est}, \eqref{T_NL_0_2_est}, we end up with the result \eqref{<n>_est_dff}. 


\end{proof}
To conclude this section, we present a technical lemma involving the double average of the cell density $n$, which will be applied in the alternating construction. 
\begin{theorem}\label{thm:avg_yz}
Consider solutions to the system \eqref{ppPKS_<n>}, \eqref{ppPKS_<c>}subject to initial conditions $(\lan n_{\mathrm{in}}\ran,\lan \cc_{\mathrm {in}}\ran)\in H^M\times H^{M+1}$, $M\geq 3$. Further recall the double average $\lan\lan f\ran\ran^{\cdot ,\cdot}$ \eqref{f_iota_2}.  There exists constant $C$ such that following estimates hold for all $ t\in [0,A^{1/3+\te}]$,  
\begin{align}\label{n_xz_est_dff}
\frac{d}{dt}&\left(\sum_{j=0}^{ M}\|\pa_y^j  \lan \lan n\ran\ran^{x,z} (t)\|_{L_{y}^2}^2\right)\\
\leq &-\frac{1}{A}\sum_{j=0}^{M}\|\pa_y^{j+1}\lan\lan n\ran\ran^{x,z} \|_{L_y^2}^2+\frac{C}{A}\myr{\|\lan n\ran^{x}(t)\|_{H_{y,z}^M}^2\left(\sup_{s\in[0, A^{1/3+\te}]}\|\lan  n \ran^{x}(s)\|_{H_{y,z}^{M}}^2 +\|\pa_y  \lan \cc_{\mathrm{in}}\ran^{x}\|_{H_{y,z}^M}^2\right)}\\
&+\frac{C}{A}\lf(\sum_{|i,j,k|=0}^{M+1}\|\Gamma_{y;t}^{ijk} \cc_\nq\|_{L_{x,y,z}^2}^2+ t^2 \sum_{ |i,j,k|=0}^{ M } \|\Gamma_{y;t}^{(i+1)jk} \cc_\nq\|_{L_{x,y,z}^2}^2\rg) \left(\sum_{|i,j,k|=0}^{M}\|\Gamma_{y;t}^{ijk}  n_\nq\|_{L^{2}_{x,y,z}}^2\right).
\end{align}
Similarly, {
\begin{align}\label{nxy_est_dff}
\frac{d}{dt}&\left(\sum_{k=0}^{M}\|\pa_z^k  \lan \lan n\ran\ran^{x,y} (t)\|_{L_{z}^2}^2\right)\\
\leq &-\frac{1}{A}\sum_{k=0}^{M}\|\pa_z^{k+1}\lan\lan n\ran\ran^{x,y} \|_{L_z^2}^2+\frac{C}{A}\|\lan n\ran^{x}(t)\|_{H_{y,z}^M}^2\left(\sup_{s\in[0, A^{1/3+\te}]}\| \lan n\ran^{x}(s)\|_{H_{y,z}^{M}}^2 +\|\pa_z \lan  \cc_{\mathrm{in}}\ran ^{x}\|_{H_{y,z}^M}^2\right)\\
&+\frac{C}{A}\bigg(\sum_{|i,j,k|=0}^{M+1}\|\Gamma_{y;t}^{ijk} \cc_\nq\|_{L_{x,y,z}^2}^2+ t^2 \sum_{ |i,j,k|=0}^{ M } \| \Gamma_{y;t}^{(i+1)j k} \cc_\nq\|_{L_{x,y,z}^2}^2\bigg) \left(\sum_{|i,j,k|=0}^{M}\|\Gamma_{y;t}^{ijk}  n_\nq\|_{L^{2}_{x,y,z}}^2\right).
\end{align}} 
\myc{{\bf Remark: }Here the red part involves estimates of the $x$-average. These can be derived from the $x$-average estimates. We note that there is a small parameter $\frac{1}{A}$ in front and the contribution from this term is controllable if we consider time interval of length $\mathcal{O}(A^{1/3+\te})$. }
\end{theorem} 
\begin{proof} We focus on the proof of estimate \eqref{n_xz_est_dff}. The proof of the other inequality is similar\myc{(\bf check!)}. 
First, we decompose the solution $n,\, \cc$ into three parts:
\begin{align}\label{n_cc_dcmp}
n=\lan\lan n\ran\ran^{x,z}+(\lan n\ran^x)^z_\nq+n^x_\nq,\quad \cc =\lan \lan \cc\ran\ran^{x,z}+(\lan \cc\ran^x)^z_\nq+\cc_\nq^x.
\end{align}  Next, we take the $x,z$-average of the nonlinearity in the equation  \eqref{ppPKS_<n>}
\begin{align}
\lan \lan \na\cdot(n\na \cc )\ran\ran^{x,z}=&\lan \na_{y,z}\cdot \lan  n\na_{y,z}\cc\ran^x\ran^{z}= \pa_{y}\lan \lan n\pa_{y}\cc \ran\ran^{x,z}\\
=&\pa_y\lf \lan \lf \lan\ \left[ \lan \lan n\ran\ran^{x,z}+(\lan n\ran^x)^z_\nq+n^x_\nq \right] \, \left[\pa_y\lan \lan \cc\ran\ran^{x,z}+\pa_y(\lan \cc\ran^x)^z_\nq+\pa_y\cc_\nq^x \right]\ \rg\ran \rg\ran^{x,z}\\
=&\pa_y \left( \lan \lan n\ran\ran^{x,z}\ \pa_y\lan \lan \cc\ran\ran^{x,z}\right)+\pa_y\lf\lan  (\lan n\ran^x)^z_\nq\ \pa_y (\lan\cc\ran^x)^z_\nq \rg\ran^{z} +\pa_y\left\lan\left\lan   n_\nq^x\pa_y \cc_\nq^x\right \ran\right\ran^{x,z}.
\end{align}
Hence we end up with the following $\lan\lan n\ran\ran^{x,z}$-equation:
\begin{align}
\pa_t\lan\lan n\ran\ran^{x,z}=\frac{1}{A}\pa_{yy}\lan\lan n\ran\ran^{x,z}-\frac{1}{A}\pa_y  \left( \lan \lan n\ran\ran^{x,z}\ \pa_y\lan \lan \cc\ran\ran^{x,z}\right) -\frac{1}{A}\pa_y\lf\lan  (\lan n\ran^x)^z_\nq\ \pa_y (\lan\cc\ran^x)^z_\nq \rg\ran^{z} -\frac{1}{A}\pa_y \left\lan\left\lan   n_\nq^x\pa_y \cc_\nq^x\right \ran\right\ran^{x,z}. 
\end{align} 
We implement direct energy estimate of the Sobolev norm to derive that
\begin{align}
\label{dblTNL00_123}
\frac{d}{dt}&\frac{1}{2}\sum_{j=0}^M\|\pa_{y}^j\lan\lan n\ran\ran^{x,z}\|_{L_y^2}^2\\
=&-\frac{1}{A}\sum_{j=0}^M\| \pa_{y}^{j+1}  \lan\lan n\ran\ran^{x,z}\|_{L_y^2}^2+\frac{1}{A}\sum_{j=0}^M\int \pa_{y}^{j+1}\lan\lan n\ran\ran^{x,z}\ \pa_y^j\bigg(\lan\lan n\ran\ran^{x,z} \ \pa_y \lan\lan \cc\ran\ran^{x,z}\bigg)dy\\
&+\frac{1}{A}\sum_{j=0}^M\int \pa_y^{j+1} \lan\lan n\ran\ran^{x,z} \ \pa_y^j\lf\lan  (\lan n\ran^x)^z_\nq\ \pa_y (\lan\cc\ran^x)^z_\nq \rg\ran^{z}  dy+\frac{1}{A}\sum_{j=0}^M\int \pa_y^{j+1} \lan\lan n\ran\ran^{x,z} \  \lan  \lan n_\nq^x \ \pa_y \cc_\nq^x  \ran \ran^{x,z} dy \\
=:&-\frac{1}{A}\sum_{j=0}^M\| \pa_{y}^{j+1} \lan\lan n\ran\ran^{x,z} \|_{L_y^2}^2+T_1+T_2+T_3. 
\end{align}
The estimates for the $T_1$ term and $T_2$ term are similar to the estimate \eqref{T_NL_0_1_est}. We apply the H\"older inequality, Young's inequality, Sobolev product estimate ($M\geq 2$),  the integral estimate \eqref{na_c_0_est} \myc{(We need an 1d-$H^M$ version)} and Lemma \ref{lem:F_0andF} to obtain
\begin{align}\label{dblTNL012est}\quad 
|T_1+T_2|\leq &\frac{1}{4A}\sum_{j=0}^M\|\pa_y^{j+1} \lan\lan n\ran\ran^{x,z}\|_{L_{y}^2}^2+\frac{C}{A}\|\lan\lan n\ran\ran^{x,z}\|_{H_{y}^M}^2\|\pa_y \lan\lan \cc \ran\ran^{x,z}\|_{H_{y}^M}^2+\frac{C}{A}\|(\lan n\ran^{x})_\nq^{z}\|_{H_{y,z}^M}^2\|\pa_y (\lan \cc \ran^{x})^z_{\nq}\|_{H_{y,z}^M}^2\\
\leq &\frac{1}{4A}\sum_{j=0}^M\|\pa_y^{j+1}\lan\lan n\ran\ran^{x,z} \|_{L_y^2}^2+\frac{C}{A}\|\nz^x\|_{H_{y,z}^M}^2\left(\sup_{s\in[0, t]}\|\nz^x(s)\|_{H_{y,z}^{M}}^2 +\|\pa_y \lan \cc_{\text{in}}\ran^x\|_{H_{y,z}^M}^2\right).
\end{align}\myc{\begin{align}\sum_{j=0}^M\|\pa_y^j  &\lf\lan  (\lan n\ran^x)^z_\nq\ \pa_y (\lan\cc\ran^x)^z_\nq \rg\ran^{z} \|_{L_y^2}=\sum_{j=0}^M\| \lf\lan \pa_y^j ( (\lan n\ran^x)^z_\nq\ \pa_y (\lan\cc\ran^x)^z_\nq) \rg\ran^{z} \|_{L_y^2}\underbrace{\lesssim}_{\text{Lemma }\ref{lem:F_0andF}} \|   \pa_y^j ( (\lan n\ran^x)^z_\nq\ \pa_y (\lan\cc\ran^x)^z_\nq)   \|_{L_{y,z}^2}\\
\underbrace{\lesssim}_{\text{Product rule}}&\|(\lan n\ran^x)^z_\nq\|_{H^M}\|\pa_y (\lan\cc\ran^x)^z_\nq   \|_{H^M}\underbrace{\lesssim}_{\text{Fourier Side}}\|\lan n\ran^x\|_{H^M}\|\pa_y \lan\cc\ran^x   \|_{H^M}\\ &\underbrace{\lesssim}_{\eqref{na_c_0_est}}\|\lan n\ran^x\|_{H^M}\lf(\sup_{s\in[0,t]}\|\lan n(s)\ran^x\|_{H^M}+\|\pa_y \lan\cc_{\text{in}}\ran^x   \|_{H^M}\rg).
\end{align}
Well, to get the $\|\pa_y\lan \cc_{\text{in}}\ran^x\|_{H^M},  $ we need to use the estimate \eqref{na_c_0_est} with zero initial data, and then propagate the $\pa_y\pa_y^j\pa_z^k\lan \cc_{\text{in}}\ran$ initial condition with the heat equation. }
The estimate of the $T_3$ term in \eqref{dblTNL00_123} is similar to  \eqref{T_NL_0_2_est}. We invoke the product estimate \eqref{Prd_est_gld},  the fact that $\pa_y^j\lan\lan f\ran\ran^{x,z}=\lan\lan \Gamma_{y;t}^jf\ran\ran^{x,z}$, and the gradient estimate \eqref{Gld_Reg_grd} to estimate
\begin{align}  \label{dblTNL03est}
|T_3| =&\bigg|\frac{1}{A}\sum_{j=0}^M\int \pa_y^{j+1} \lan\lan n\ran\ran^{x,z} \ \lan\lan\Gamma_{y;t}^j(n_\nq^x  \pa_{y}  \cc_{\nq}^x)\ran\ran^{x,z} dy\bigg|\leq \frac{1}{8A}\sum_{j=0}^M\lf(\|\pa_y^{j+1} \lan\lan n\ran\ran^{x,z}\|_{L^2}^2+{C}\|\Gamma_{y;t}^{j} (n_\nq^x \pa_{y} \cc_\nq^x)\|_{L^2}^2\rg)\\
&\leq\frac{1}{8A}\sum_{j=0}^M\|\pa_y^{j+1} \lan\lan n\ran\ran^{x,z}\|_{L^2}^2+\frac{C}{A}\left(\sum_{|i,j,k|=0}^{M}\|\Gamma_{y;t}^{ijk}  \pa_{y} \cc_\nq \|_{L^2}^2\right)\left(\sum_{|i,j,k|=0}^{M}\|\Gamma_{y;t}^{ijk}  n_\nq\|_{L^{2}}^2\right)\\
\leq&  \frac{1}{8A}\sum_{j=0}^M\|\pa_y^{j+1} \lan\lan n\ran\ran^{x,z}\|_{L^2}^2+\frac{C}{A}\bigg(\sum_{|i,j,k|=0}^{M+1}\| \Gamma_{y;t}^{ijk} \cc_\nq\|_{L^2}^2+ {t}^2 \sum_{ |i,j,k|=0}^{ M }  \|\Gamma_{y;t}^{(i+1)jk} \cc_\nq\|_{L^2}^2\bigg) \left(\sum_{|i,j,k|=0}^ M\|\Gamma_{y;t}^{ ij k } n_\nq\|_{L^{2}}^2\right).
\end{align}

Now combining the decomposition \eqref{dblTNL00_123}, and the estimates \eqref{dblTNL012est}, \eqref{dblTNL03est}, we end up with the result \eqref{n_xz_est_dff}. 
\end{proof}
\ifx

\myb{ {\bf Previous:}
Now we decompose it into two parts:
\begin{align}
T_{0\nq}=\frac{1}{A}\sum_{|j,k|=0}^ {M}\int \na_{y,z}\pa_y^j\pa_z^k \nz\cdot \pa_y^j\pa_z^k\lan \na  (c_{\nq}+d_\nq) n_\nq\ran dV=:T_{0\nq; c_\nq}+T_{0\nq;d_\nq}.\label{T_0nqcd}
\end{align}
To estimate the term $T_{0\nq;c_\nq}$, 
Now we apply  the estimate of the gradient \eqref{Gld_Reg_grd}, and 
the bootstrap hypothesis \eqref{HypED} to obtain that 
\begin{align}
|T&_{0\nq;c_\nq}|\\
\leq&  \frac{1}{8A}\sum_{|j,k|=0}^ {M}\|\na\pa_y^j\pa_z^k \nz\|_{L_{y,z}^2}^2\\
&+\frac{C   }{A}\bigg(\sum_{|i,j,k|=0}^{M+1}\|\pa_x^i\Gamma_y^j\pa_z^k c_\nq\|_{L_{x,y,z}^2}+\frac{t}{A^{1/3}}\sum_{ |i,j,k|\leq M } A^{1/3}\|\pa_x^{i+1}\Gamma_y^j\pa_z^k c_\nq\|_{L_{x,y,z}^2}\bigg)^2 \left(\sum_{|i,j,k|=0}^{M}\|\pa_x^i\Gamma_y^{j}\pa_z^k  n_\nq\|_{L^{2}_{x,y,z}}^2\right)\\
\leq &   \frac{1}{8A}\sum_{|j,k|=0}^ {M}\|\na\pa_y^j\pa_z^k \nz\|_{L_{y,z}^2}^2 + \frac{C(C_{ED}  ,\myr G)}{A}(\|n_{\neq}^x(0)\|_{H_{x,y,z}^M}^4+\|\cc_\nq^x(0)\|_{H^{M+1}_{x,y,z}}^4+\myr 1)e^{-\frac{4\delta t}{A^{1/3}}} (\Phi(t))^{-4M-\myr 4}.\label{T_0nq_cnq}
\end{align} 
Next we consider the term $T_{0\nq;d_\nq}$ in \eqref{T_0nqcd}. The treatment is similar to the estimate for the $T_{0\nq;c_\nq}$ term. Here we apply the product estimate \eqref{Prd_est_gld}, the fact that $\pa_y^j\lan f\ran^x=\lan \Gamma_y^jf\ran^x$, the estimate of the gradient \eqref{Gld_Reg_grd}, the enhanced dissipation of the $d_\nq$,  and 
 bootstrap hypothesis \eqref{HypED} to derive that
\begin{align}
|T_{0\nq;d_\nq}|\leq&\frac{1}{8A}\sum_{|j,k|=0}^ {M}\|\na\pa_y^j\pa_z^k \nz\|_{L_{y,z}^2}^2+\frac{C}{A}\left(\sum_{|i,j,k|=0}^{M}\|\pa_x^i\Gamma_y^{j}\pa_z^k  \na_{y,z} d_\nq \|_{L^2_{x,y,z}}^2\right)\left(\sum_{|i,j,k|=0}^{M}\|\pa_x^i\Gamma_y^{j}\pa_z^k  n_\nq\|_{L^{2}_{x,y,z}}^2\right)\\
\leq&  \frac{1}{8A}\sum_{|j,k|=0}^ {M}\|\na\pa_y^j\pa_z^k \nz\|_{L_{y,z}^2}^2\\
&+\frac{C }{A}\bigg(\sum_{|i,j,k|=0}^{M+1}\|\pa_x^i\Gamma_y^j\pa_z^k d_\nq\|_{L_{x,y,z}^2}+ {t} \sum_{ |i,j,k|\leq M }  \|\pa_x^{i+1}\Gamma_y^j\pa_z^k d_\nq\|_{L_{x,y,z}^2}\bigg)^2  \left(\sum_{|i,j,k|=0}^{M}\|\pa_x^i\Gamma_y^{j}\pa_z^k  n_\nq\|_{L^{2}_{x,y,z}}^2\right)\\
\leq&\frac{1}{8A}\sum_{|j,k|=0}^ {M}\|\na\pa_y^j\pa_z^k \nz\|_2^2 +\frac{C(C_{ED}  ,\myr G)}{A}t^2\|\cc_{\text{in};\nq}\|_{H^{M+1}}^2(\|n_{\text{in};\nq}\|_{H^M}^2+\|\cc_{\text{in};\nq}\|_{H^{M+1}}^2+\myr 1)e^{-\frac{2\delta t}{A^{1/3}}-\frac{2\delta_d t}{A^{1/3}}}\Phi^{-2M}(t).\label{T_0nq_dnq}
\end{align} 
Combining \eqref{T_0nqcd}, \eqref{T_00_est}, \eqref{T_0nq_cnq}, \eqref{T_0nq_dnq}, we have
\begin{align}
\frac{d}{dt}&\frac{1}{2}\sum_{|j,k|=0}^ {M}\|\pa_{y}^j\pa_{z}^k\nz\|_{L^2}^2\\
 \leq &-\frac{1}{2A}\sum_{|j,k|=0}^ {M}\|\na\pa_y^j\pa_z^k \nz\|_2^2+\frac{C}{A}\|\nz\|_{H_{y,z}^M}^2\left(\sup_{\tau\in[0, t]}\|\nz(\tau)\|_{H_{y,z}^{M}}^2 +\|\na \lan \cc_{\text{in}}\ran\|_{H_{y,z}^M}^2\right)\\
&+ \frac{C(C_{ED} ,\myr G)}{A}(\|n_{\text{in};\neq}^x \|_{H_{x,y,z}^M}^4+\|\cc_{\text{in};\nq}^x\|_{H^{M+1}_{x,y,z}}^4+\myr 1)e^{-\frac{4\delta t}{A^{1/3}}} (\Phi(t))^{-4M-\myr 4}\\
&+\frac{C(C_{ED} ,\myr G)}{A}t^2\|\cc_{\text{in};\nq}^x\|_{H^{M+1}}^2(\|n_{\text{in};\nq}^x\|_{H^M}^2+\|\cc_{\text{in};\nq}^x\|_{H^{M+1}}^2+\myr 1)e^{-\frac{2\delta t}{A^{1/3}}-\frac{2\delta_d t}{A^{1/3}}}\Phi^{-2M-\myr 2 }(t).
\end{align}
Now we integrate the above differential inequality and get
\begin{align}
\sup_{t\in[0, T]}&\|\lan n\ran\|_{H^M_{y,z}}^2\leq \|\lan n_{\text{in}}\ran\|_{H^M_{y,z}}^2+\frac{CT}{A}\sup_{t\in[0,T]}\|\lan n\ran \|_{H^M_{y,z}}^2\left(\sup_{t\in[0,T]}\|\lan n\ran \|_{H^M_{y,z}}^2+\|\na \lan \cc_{\text{in}}\ran\|_{H^M_{y,z}}\right)+\int_0^T \mathcal{G}(s)ds,\quad T\in[0, A^{1/3+\te}], \label{sup_lan_n_ran_HM}\\
\mathcal{G}:=& \frac{C(C_{ED} ,\myr G)}{A}(\|n_{\text{in};\neq}^x \|_{H_{x,y,z}^M}^4+\|\cc_{\text{in};\nq}^x\|_{H^{M+1}_{x,y,z}}^4+\myr 1)e^{-\frac{4\delta t}{A^{1/3}}} (\Phi(t))^{-4M-\myr 4}\\
&+\frac{C(C_{ED} ,\myr G)}{A}t^2\|\cc_{\text{in};\nq}^x\|_{H^{M+1}}^2(\|n_{\text{in};\nq}^x\|_{H^M}^2+\|\cc_{\text{in};\nq}^x\|_{H^{M+1}}^2+\myr 1)e^{-\frac{2\delta t}{A^{1/3}}-\frac{2\delta_d t}{A^{1/3}}}\Phi^{-2M-\myr 2}(t).
\end{align}
Recall that $T\leq A^{1/3+\te}$, so we have 
\begin{align}
\int_0^T\mathcal{G}(s)ds\leq \left(\frac{1}{A^{2/3}}+\|\cc_{\text{in};\nq}^x\|_{H^{M+1}}^2\right)C(C_{ED} , \|n_{\text{in};\nq}^x\|_{H^M}^2+\|\cc_{\text{in};\nq}^x\|_{H^{M+1}}^2,\myr G,\delta^{-1},\delta^{-1}_d ).
\end{align}
Now by choosing $A$ large enough in \eqref{sup_lan_n_ran_HM}, we can apply a barrier argument (?) to see that 
\begin{align}\label{sup_lan_n_ran_HM_goal}
\sup_{t\in[0, T]}&\|\lan n\ran\|_{H^M_{y,z}}^2\\
\leq &C  \|\lan n_{\text{in}}\ran\|_{H^M_{y,z}}^2+ \left(\frac{1}{A^{2/3}}+\|\cc_{\text{in};\nq}^x\|_{H^{M+1}}^2\right)C(C_{ED} ,  \|n_{\text{in};\nq}^x\|_{H^M},\|\cc_{\text{in};\nq}^x\|_{H^{M+1}},\myr G,\delta^{-1},\delta^{-1}_d ).
\end{align}}
\textcolor{red}{In the first alternative shear process, the $\lan n\ran$ becomes large. But after that, $\|\cc_{\nq}(T_\star)\|_{H^{M+1}}\leq \frac{1}{A^{...}}$, and we can use $A$ to get smallness. } {\color{blue} \textbf{Previous:} 
\begin{lem}
Under the assumption \eqref{Hypotheses}, the following estimates are satisfied given that $A^{-1}$ and $\ep$ are chosen small enough
\begin{align}
\|\lan\lan n \ran\ran^{x,z}(t)\|_{H_y^m}+\|\lan \lan n\ran\ran^{x,y}(t)\|_{H_z^m}\leq C(\| \lan\lan n\ran\ran^{x,z}(0)\|_{H_y^m}+\| \lan\lan n\ran\ran^{x,y}(0)\|_{H_z^m}+\ep+\frac{1}{A^{1/12}}),\\
\|\lan \lan \cc\ran\ran^{x,z}(t)\|_{H_y^{m+1}}+\|\lan \lan \cc\ran\ran^{x,y}(t)\|_{H_z^{m+1}}\leq C(\|\lan\lan \cc\ran\ran^{x,z}(0)\|_{H_y^{m+1}}+\|\lan\lan \cc\ran\ran^{x,y}(0)\|_{H_z^{m+1}}+\frac{1}{A^{1/12}}),\quad m\leq M.
\end{align}
\end{lem}
\begin{proof} We decompose the solution $n,\, c$ into three parts:
\begin{align}
n=\lan\lan n\ran\ran^{x,z}+(\lan n\ran^x)^z_\nq+n^x_\nq,\quad \cc =\lan \lan \cc\ran\ran^{x,z}+(\lan \cc\ran^x)^z_\nq+c_\nq^x+d_\nq^x.
\end{align}  We have the following relations
\begin{align}
\lan \lan \na\cdot(\na \cc n)\ran\ran^{x,z}=&\lan \lan \pa_y(\pa_y\cc n)\ran\ran^{x,z}\\
=&\pa_y \left(\pa_y\lan \lan \cc\ran\ran^{x,z} \lan \lan n\ran\ran^{x,z}\right)+\pa_y\lan \pa_y (\lan\cc\ran^x)^z_\nq (\lan n\ran^x)^z_\nq\ran^z +\pa_y\left\lan\left\lan (\pa_y c_\nq^x+\pa_y d_\nq^x) n_\nq^x\right \ran^x\right\ran^z
\end{align}
Similar to the energy estimates carried out in the previous part of the section, we have that the $\lan \lan n\ran\ran^{x,z}$ is bounded in $H^m_y$. 
Through Lemma \ref{...}, we derive the following estimate for  the $\lan \lan c\ran\ran^{x,z}$....
\end{proof}}
\fi

\subsection{The Regularity Estimates of the Remainder}\label{Sec:NLEDreg}
The goal of this subsection is to derive the following theorem which provides the necessary regularity estimates.
\begin{theorem}\label{thm:F_M_reg}
Consider the solutions $n_\nq(t_r+\tau), \, c_\nq^{t_r}(\tau)=\cc_\nq(t_r+\tau)-S_{t_r}^{t_r+\tau}\cc_\nq,\, d_\nq^{t_r}(\tau)=S_{t_r}^{t_r+\tau}\cc_\nq$ to the equations \eqref{ppPKS_neq} initiated from the reference time $t_r\in\{0, T_h=A^{1/3+\te/2}\}$. Then the following estimate holds for all $\underline{ t= t_r+\tau}\in[t_r,A^{\frac{1}{3}+\te}]$,
\begin{align}
\label{F_M_reg}
\frac{d}{d\tau}&\FM [t_r+\tau, n_\nq, c^{t_r}_\nq]
\\ \leq&   \myr{\frac{ C  }{  {G}A^{1/3}}   \left(\frac{G}{A^{1/6}}+\frac{  \Phi(t_r+\tau)}{G}+\frac{A^{1/3}}{A^{2/3}+(t_r+\tau)^2}+\exp\lf\{-\frac{\delta_{\mathcal Z}}{4A^{1/3}} \tau\rg\}\right) \FM}  \\
& \myb{+ \frac{C_{G}}{A^{1/3}}\ \frac{\FM}{A^{1/2}\Phi(t_r+\tau)^{ 4M+4}}\ \bigg(1+{\myc{???}}\FM+ \|\lan n\ran\|_{H^M}^2+ \|\na\lan   \cc\ran\|_{ H^{M}}^2  \bigg)}\\
&\myb{+\frac{C_G}{A^{1/3}{Q^2}}\frac{\exp\lf\{-\frac{\delta_{\mathcal Z}}{A^{1/3}} \tau\rg\}}{ \Phi(t_r+\tau)^{ 4M+4}}\ \lf(Q^2\sum_{|i,j,k|=0}^{M+1}G^{2i}\Phi(t_r)^{2j}\|\pa_x^i\Gamma_{y;t_r}^j\pa_z^k\cc_{\neq}(t_r)\|_{L^{2}}^2\rg) \lf(\FM+\|\lan n\ran\|_{H^M}^2\rg)}. 
\end{align}
Here the constant $C$ only depends on $M, \,\|u_A\|_{L^\infty_tW^{M+3,\infty}}$ and $C_G$ only depends on $G,\ M, \,\|u_A\|_{L^\infty_tW^{M+3,\infty}}$. 

\ifx\myr{For checking purpose only: The last two lines originally have the form:
\begin{align}\ &\myr{+ \frac{\FM [t_r+\tau] }{A^{1/3}}\ C(  {G} )\ \bigg(\frac{\Phi(t_r+\tau)^{-4M-4}}{A^{2/3}}+\Phi^{-2M-4}\frac{\|\lan n\ran^x\|_{H^M}^2+ \|\na\lan   \cc\ran^x\|_{ H^{M}}^2}{A^{1/2 }}  \bigg)}\\
+&\myr{\frac{CG^{6+2M} }{A^{1/3}}\lf(1+\frac{(t_r+\tau)^4}{A^{4/3}}
\rg)e^{-\frac{\delta_{\mathcal Z}}{A^{1/3}} \tau}\ \Phi^{-4M-2}\lf(\sum_{|i,j,k|=0}^{M+1}G^{2i}\Phi(t_r)^{2j}\|\pa_x^i\Gamma_{y;t_r}^j\pa_z^k\cc_{\neq}(t_r)\|_{L^{2}}^2\rg) \FM} \\
+&\myr{\frac{CG^{6+2M} }{A^{1/3}} \lf(1+\frac{(t_r+\tau)^4}{A^{4/3}}\rg) e^{-\frac{\delta_{\mathcal Z}}{A^{1/3}} \tau}\ \Phi^{-2M-2}\lf(\sum_{|i,j,k|=0}^{M+1}G^{2i}\Phi(t_r)^{2j}\|\pa_x^i\Gamma_{y;t_r}^j\pa_z^k\cc_{\neq}(t_r)\|_{L^{2}}^2\rg) \|\lan n\ran\|_{H^M}^2}\\
\end{align}}\fi
\ifx
\myb{{\bf Previous:}
\begin{align}
\frac{d}{dt}\FM &(t_\star+t)\\
\leq&   \frac{1}{  {G}}C     \left(\frac{1}{G A^{1/3}(1+\frac{(t_\star+t)^3}{A})}+\frac{1}{A^{2/3}+(t_\star+t)^2}+e^{-\delta_d (t_\star+t)/A^{1/3}}\right) \FM(t_\star+t)  \\
&+\frac{1}{A^{5/6}}C( \myr {G})   \Phi^{-4M-4}(t_\star+t)\ F_M^2(t_\star+t)    \\
&+ \frac{\FM (t_\star+t) }{A^{1/3}}\ C( \myr{G} )\ \bigg(\frac{\Phi^{-4M-4}(t_\star+t)}{A^{2/3}}+\Phi^{-2M-4}\frac{\|\lan n\ran^x\|_{H^M}^2+ \|\na\lan   c\ran^x\|_{ H^{M+1}}^2}{A^{1/2 }}  +\|\cc_{\text{in};\nq}\|_{H^{M+1}}^2 e^{-\frac{\delta_ d  (t_\star+t)}{8A^{1/3}}}\bigg)\\
\\
+&C( \myr{G} ,\delta_ d^{-1})\frac{\|\lan n\ran\|_{H^M}^2 }{A^{1/3}}e^{-\frac{\delta_ d  (t_\star+t)}{8A^{1/3}}}\min\left\{\ep_I^2\ F_M(t_\star+t),\ \|\cc_{\text{in};\nq}\|_{H^{M+1}}^2  \right\}.
\end{align}}\fi
\end{theorem}
Before presenting the proof of the theorem, we first decompose the time derivative $d\ \FM[t_r+\tau]/ d\tau$ and identify the terms to be estimated. Then we provide several lemmas to provide necessary bounds on these terms. The proofs of these lemmas will be postponed to the end of this subsection. Finally, after all the preparations are ready, the proof of Theorem \ref{thm:F_M_reg} is straightforward.

 The time evolution of the functional $\FM$ \eqref{F_M} on $t=t_r+\tau\in [t_r,T_\star]\subset[0,A^{1/3+\te}]$ has three components:
\begin{align}\label{ddtau_FM}\qquad
\frac{d}{d\tau}\FM=&\frac{d}{d\tau}\lf(\sum_{|i,j,k|=0}^{M}\Phi^{2j} G^{4+2i}\|\Gamma_{y;t_r+\tau}^{ijk} n_\nq(t_r+\tau)\|_{L^2}^2\rg)\\
&+\frac{d}{d\tau}\lf(\sum_{|i,j,k|=0}^{M+1}\Phi^{2j{+2}}G^{2i}(A^{2/3}\mathbbm{1}_{j<M+1}+\mathbbm{1}_{j=M+1})\| \Gamma_{y;t_r+\tau}^{ijk} (\cc_\nq(t_r+\tau)-S_{t_r}^{t_r+\tau}\cc_{\nq})\|_{L^2}^2\rg)\\
&-\frac{\delta_{\mathcal Z} Q^2}{2A^{1/3}}\sum_{|i,j,k|=0}^{M+1}\|\Gamma_{y;t_r}^{ijk} \cc_{\nq}(t_r)\|^2_{L^{2} }\exp\lf\{-\frac{\delta_{\mathcal Z}  \tau}{2A^{1/3}}\rg\}=:\pa F_1+\pa F_2-\pa F_3.
\end{align} For the sake of simplicity, we use the notation $\Gamma_y$ to represent $\Gamma_{y;t_r+\tau}$ and $\Gamma_{y;t}^{ijk}$ to represent $\pa_x^i\Gamma_{y;t_r+\tau}^j\pa_z^k$. We begin by considering the $\pa F_1$ term in \eqref{ddtau_FM}. 
The equation for the higher gliding derivatives of $n_\nq$ can be expressed as follows with the help of the equation \eqref{ppPKS_n_neq},
\begin{align}
\pa_\tau &\gt n_{\neq}+u_A(t_r+\tau,y) \pa_x^{i+1}\Gamma_y^j\pa_z^k n_{\neq}-\frac{1}{A}\de\gt n_{\neq}\\
&=\frac{1}{A}[\Gamma_y^j,\pa_{yy}]\pa_x^i\pa_z^jn_\nq-\frac{1}{A}\gt\lf(\na \cdot( n_{\neq}\na \lan \cc\ran)+\na \cdot( \lan n\ran\na \cc_{\neq} )+\na\cdot(n_{\neq}\na \cc_{\neq})_{\neq}\rg),\qquad t=t_r+\tau. 
\end{align}Combining the equation and a direct $L^2$-energy estimate yields that  
\begin{align} \label{TNL_123_vf} \qquad
\frac{1}{2}\pa F_1\leq & -\frac{1}{A}\sum_{|i,j,k|=0}^{M}G^{4+2i} \Phi^{2j}\|\na (\gt n_{\nq})\|_2^2+\sum_{|i,j,k|=0}^{M}\frac{G^{4+2i} }{A}\Phi^{2j}\int \gt n_\nq \ [\Gamma_y^j,\pa_{yy}]\pa_x^i\pa_z^k n_\nq dV\\
&+\sum_{|i,j,k|=0}^{M}\frac{G^{4+2i} }{A}\Phi^{2j}\int \na\gt n_{\neq}\cdot \palt( n\na \cc_{\neq}+n_\nq \na \lan \cc\ran) dV\\
&-\sum_{|i,j,k|=0}^{M}\frac{G^{4+2i} }{A}\Phi^{2j}\int \gt n_\nq  \ [\Gamma_y^j,\pa_y]\pa_x^i\pa_z^k(n\na \cc_\nq +n_\nq\na\lan \cc\ran)dV\\
=:&-\mathfrak {D}_n+ T_{n;1}^{NL;R}+T_{n;2}^{NL;R}+T_{n;3}^{NL;R}.
\end{align}
Next we recall the definition $c_\nq^{t_r}(\tau)=\cc_\nq(t_r+\tau)-S_{t_r}^{t_r+\tau}\cc_\nq(t_r)$ together with \eqref{ppPKS_c_neq}, and express the equation satisfied by $\gt c_\nq^{t_r}$ as follows,
\begin{align}
\pa_\tau \gt c_{\neq}^{t_r}+u_A(t_r+\tau,y)\Gamma_{y;t}^{(i+1)jk} c_{\neq}^{t_r}-\frac{1}{A}\de\gt c_{\neq}^{t_r}=&\frac{1}{A}[\Gamma_{y;t}^j,\pa_{yy}]\pa_{x}^i\pa_z^k c_\nq^{t_r}+\frac{1}{A}\gt n_{\neq},\\
\quad \pa_x^i\pa_{y;t_r}^j\pa_z^k c_{\neq}^{t_r}(\tau=0)=&0,\quad t=t_r+\tau.
\end{align}
\myc{Because $c_\nq^{t_r}(\tau=0)\equiv 0$, and all the spatial derivatives are identically zero. }Now a direct computation yields an expression for the   $\pa F_2$ term in \eqref{ddtau_FM}:
\begin{align} \label{T_c_com_n}
\frac{1}{2}\pa F_2
\leq&\sum_{|i,j,k|=0}^{M+1}(A^{2/3}\mathbbm{1}_{j\leq M}+\mathbbm{1}_{j=M+1})\bigg(-\frac{\Phi^{2j\myr{+2}}G^{2i}}{A}\|\na \gt c_\nq^{t_r}\|_2^2\\
&\qquad\qquad+\frac{G^{2i}\Phi^{2j\myr{+2}}}{A}\int \gt c_\nq^{t_r}\ [\Gamma_y^j ,\pa_{yy}]\pa_x^i\pa_z^k c_\nq^{t_r} dV+\frac{G^{2i}\Phi^{2j\myr{+2}}}{A}\int \gt c_\nq^{t_r}\ \gt n_\nq dV\bigg)\\
=:&-\mathfrak {D}_c+T_{c;1}^{NL;R}+T_{c;2}^{NL;R}.
\end{align} \myc{There is a negative term: $(\myr{2j+2})\Phi^{\myr {2j+1}}\Phi'G^{2i}\|\palt c_\nq\|_2^2=-CK_c\leq 0$. But I didn't find the use of it. }
\myc{\bf Double check $T_{c;1}^{NL;R}$, there is extra $\Phi^2$!}
We note that the proof of Theorem \ref{thm:F_M_reg} is completed once suitable estimates are provided for the $T_{n;1}^{NL;R}, \ T_{n;2}^{NL;R}, \ T_{n;3}^{NL;R}$ terms in \eqref{TNL_123_vf}, and the $T_{c;1}^{NL;R}, \ T_{c;2}^{NL;R}$ terms in \eqref{T_c_com_n}. Next we collect two lemmas that provides bounds for these terms. The proof of these two lemmas will be postponed to the end of this subsection.  
\begin{lem}\label{lem:est_nNL}[\myb{Lemma of $T_i^{NL;R}$ and $T_i^{NL;D}$.}]
Consider functions $\{\mathfrak H_{ijk}\}_{|i,j,k|\leq M}\subset H^1(\Torus^3)$. \myc{We don't need to put extra constraints on the functions because we have already use the average zero property of the $n_\nq,$ etc., to decompose the chemical gradient into two parts. } The following estimates hold
\begin{align}\label{T_n_2_NL}\qquad 
&\lf|\sum_{|i,j,k|=0}^{M}\frac{G^{4+2i} }{A}\Phi^{2j}\int \na\mathfrak{H}_{ijk} \cdot \gt\lf( n\na\cc_{\neq}+n_\nq \na \lan \cc\ran\rg) dV\rg|\\
&\leq \sum_{|i,j,k|=0}^{M}\frac{G^{4+2i}\Phi^{2j}}{4A}\|\na \mathfrak H_{ijk}\|_2^2+\frac{C(G)}{A\Phi^{4M+4}}\lf( \FM\rg)^2+\frac{ C(G)}{A\Phi^{2M+4}}( \|\lan n\ran\|_{H^{M}}^2+ \|\na\lan \cc\ran\|_{H^{M}}^2)\FM\\
&\quad+ \frac{C}{A^{1/3}Q^2}\left(\FM +\|\lan n \ran\|_{H^{M}}^2\right)\left(Q^2\sum_{|i,j,k|=0}^{M+1}{G^{2i}\Phi(t_r)^{2j}}\| \Gamma_{y;t_r}^{ijk}\cc_\nq(t_r)\|_{L^{2}}^2\right)\frac{\exp\lf\{-\frac{2\delta_{\mathcal Z} \tau}{A^{1/3}}\rg\}}{\Phi^{4M+3}};\\ \label{T_n_3_NL}
&\lf|\sum_{|i,j,k|=0}^{M}\frac{G^{4+2i} }{A}\Phi^{2j}\int\mathfrak H_{ijk} \ [\Gamma_y^j,\pa_y]\pa_x^i\pa_z^k(n\na\cc_\nq +n_\nq\na\lan \cc\ran)dV\rg|\\
&\leq  \sum_{|i,j,k|=0}^{M}\frac{CG^{4+2i}\Phi^{2j}}{A^{1/3}}\lf(\frac{1}{A^{1/6}}+\frac {\exp\lf\{-\frac{\delta_{\mathcal Z}}{A^{1/3}} \tau\rg\}}{G^2}\rg)\|  \mathfrak H_{ijk}\|_2^2+\frac{ C(G) }{A^{5/6}\Phi^{ 4M+4}}\lf(\FM\rg)^2 \\
&\quad+\frac{C(G)}{ A^{5/6}\Phi^{2M+4}}(\norm{n}_{H^M}^2+\norm{\na \lan\cc\ran }_{H^{M}}^2)\FM\\
&\quad+\frac{C(G) \exp\lf\{-\frac{\delta_{\mathcal Z}}{A^{1/3}} \tau\rg\}}{A^{1/3}\Phi^{2M+4}}\ \lf(\sum_{|i,j,k|=0}^{M+1}G^{2i}\Phi(t_r)^{2j}\|\Gamma_{y;t_r}^ {ijk}\cc_{\neq}(t_r)\|_{L^{2}}^2\rg) \lf(\Phi^{-2M}\FM +\|\lan n\ran\|_{H^M}^2\rg).
\end{align}
Here $\Phi=\Phi(t)=\Phi(t_r+\tau). $ 

\end{lem}

\begin{lem}\label{lem:est_ctrm}
Consider functions $\{\mathfrak H_{ijk}\}_{|i,j,k|\leq M+1}\subset H^1$. There exists a constant $C$, which depends only on $M, \, \myr{\|\pa_yu_A\|_{L^\infty_t W^{M+2,\infty}}}$, such that the following estimates  hold
\ifx
\footnote{
\myb{Siming: The following part is already introduced in the main text: It will be ideal to prove the estimate like this in the linear part: 
\begin{align}
\lf|T_{c;1}^{NL}\rg|:=&\lf|\sum_{|i,j,k|=0}^{M}\frac{G^{2i}\Phi^{2j}}{A}\int \mH\ [\Gamma_y^j,\pa_{yy}]\mathfrak{G}_{ik} dV\rg|\\
\leq & \frac{1}{8A}\sum_{|i,j,k|=0}^{M}{G^{2i}\Phi^{2j}}\|\na \mH \|_2^2 \ + \frac{C }{GA^{1/3}(1+\frac{t^2}{A^{2/3}})} \sum_{|i,j,k|=0}^{M}G^{2i}\Phi^{2j}\|\mH\|_2^2\\
&+\left(\frac{C\Phi}{A^{1/3}G^2}+\frac{C}{G(A^{2/3}+t^2)}\right) \sum_{|i,j,k|=0}^{M}G^{2i}\Phi^{2j}\|\Gamma_y^j \mathfrak G_{ik} \|_{2}^2. 
\end{align}}}\fi
\begin{align} 
\label{T_c_com}\ \ &\sum_{|i,j,k|=0}^{M+1}(A^{2/3}\mathbbm{1}_{j\leq M}+\mathbbm{1}_{j=M+1})\frac{G^{2i}\Phi^{2j+2}}{A}\bigg|\int \mf H_{ijk}\ [\Gamma_{y;t}^j ,\pa_{yy}]\pa_x^i\pa_z^k c_\nq^{t_r} dV\bigg|\\
 &\leq\frac{1}{8A}\sum_{|i,j,k|=0}^{M+1}(A^{2/3}\mathbbm{1}_{j\leq M}+\mathbbm{1}_{j=M+1}) {G^{2i}\Phi^{2j+2}} \|\na \mf H_{ijk}\|_{L^2}^2\\
&\quad+  \left(\frac{C\Phi}{G^2 A^{1/3}}+\frac{C}{G(A^{2/3}+  t^2)}\right)\sum_{|i,j,k|=0}^{M+1}(A^{2/3}\mathbbm{1}_{j\leq M}+\mathbbm{1}_{j=M+1}) {G^{2i}\Phi^{2j+2}}(\|\mf H_{ijk}\|_{L^2}^2+\|\gt c_\nq^{t_r}\|_{L^2}^2);  
\end{align}
\begin{align}
\label{est_c_term}
\sum_{|i,j,k|=0}^{M+1}&(A^{2/3}\mathbbm{1}_{j\leq M}+\mathbbm{1}_{j=M+1})  \frac{ G^{2i}\Phi^{2j+2}}{A}\lf|\int \mH \ \gt n_\nq dV \rg|\\
\leq &\frac{1}{4A}\sum_{|i,j,k|=0}^{M+1}(A^{2/3}\mathbbm{1}_{j\leq M}+\mathbbm{1}_{j=M+1})G^{2i}\Phi^{2j+2}\|\na \mH\|_2^2+\frac{C\Phi^2}{G^2A^{1/3}}\FM.
\end{align}\myc{Is the first estimate sensitive to the starting time?} \myc{This is indeed a valid problem. The $c_\nq^{t_r}(\tau)$ starts from $\tau=0\Leftrightarrow t_r$ and the corresponding vector fields that acts on it is $\Gamma_{y;t_r+\tau}$. But the $B^{(m)}$'s in the commutators are always using the time $t=t_r+\tau$, which is the time of the weight $\Phi(t)$. In the proof, we use the $\Phi's$ to control the extra $t$'s coming out from the $[\pa_{yy}, \Gamma_{y;t}^j]$. So the resulting bound should also be $\Phi(t),\, \frac{1}{A^{2/3}+t^2}$. So the first estimate seems not to be sensitive to the starting time.}
Here $\Phi=\Phi(t_r+\tau)$.
\end{lem}
With these two lemma, we can complete the proof of Theorem \ref{thm:F_M_reg}.
\begin{proof}[Proof of Theorem \ref{thm:F_M_reg}]
We estimate each term in \eqref{TNL_123_vf} and \eqref{T_c_com_n}. 
We note that the commutator term $T_{n;1}^{NL;R} $ in \eqref{TNL_123_vf} is estimated in Lemma \ref{lem:cm_L_trm}. By setting $\mH=\gt n_\nq$ in \eqref{T_n_2_NL}, \eqref{T_n_3_NL}, we obtain the estimates for the $T_{n;2}^{NL;R}$ and $T_{n;3}^{NL;R}$ terms in \eqref{TNL_123_vf}. By setting $\mH=\gt c_\nq^{t_r}$ in \eqref{T_c_com} and \eqref{est_c_term}, we obtain the bounds for $T_{c;1}^{NL;R},\ T_{c;2}^{NL;R}$ terms in \eqref{T_c_com_n}. Combining the estimates stated above and recalling the decomposition \eqref{ddtau_FM}, we obtain \eqref{F_M_reg}. 
\end{proof}  
\begin{proof}[Proof of Lemma \ref{lem:est_nNL}] 
{\bf Proof of \eqref{T_n_2_NL}:}
We decompose the left hand side of \eqref{T_n_2_NL} into four parts:
\begin{align}\label{T2_NL_1234}
\sum_{|i,j,k|=0}^{M}&\frac{G^{4+2i} }{A}\Phi^{2j}\int \na\mH\cdot \gt( n\na c_{\neq}+n\na d_{\neq}+n_\nq \na \lan \cc\ran) dV 
=: T_{1}+T_{2}+T_{3}. 
\end{align}
Here we drop the $(\cdots)^{t_r}$ in the $c_\nq^{t_r}$ and $d_\nq^{t_r}. $

We first estimate the $T_{1},\, T_2$ terms with the product estimate \eqref{Prd_est_gld} \myc{(The estimate seems to apply equally well for general $f,\ g$ because we include all the lower order terms. One does not need to enforce the average zero condition?)} as follows
\begin{align}
|T_{1}|\leq&C\sum_{|i,j,k|=0}^{M}\frac{G^{4+2i}\Phi^{2j}}{A}\|\na\mH\|_2\left(\sum_{|i,j,k|=0}^{M}\|\gt\na c_\nq \|_2\right)\left(\sum_{|i,j,k|=0}^{M}\|\gt n\|_2\right)
\end{align}Next we recall the estimate \eqref{Gld_Reg_grd}, the fact that $\frac{t^2}{A^{2/3}}\leq C\Phi^{-2/3}$ \eqref{varphi}, and the definition of $\FM$ \eqref{F_M} to derive that
\begin{align}\label{T_21} 
|T_{1}|\leq &\sum_{|i,j,k|=0}^{M}\frac{G^{4+2i}\Phi^{2j}}{12A}\|\na \mH\|_2^2\\
&+\frac{C(G)}{A\Phi^{4M+4}}\left(\sum_{|i,j,k|=0}^{M+1}G^{2i}\Phi^{2j\myr{+2}}\|\gt c_\nq\|_2^2+ \Phi^{-2/3}\sum_{|i,j,k|=0}^{M}A^{2/3}G^{2i}\Phi^{2j\myr{+2}}\|\Gamma_{y;t}^{(i+1)jk} c_\nq\|_2^2\right)\\
&\qquad\qquad\qquad \times \left(\sum_{|i,j,k|=0}^{M} {G^{4+2i}\Phi^{2j}}  \|\gt n_\nq\|_2^2+\Phi^{2M}\|\lan n\ran\|_{H^M}^2\right)\\
\leq &\sum_{|i,j,k|=0}^{M}\frac{G^{4+2i}\Phi^{2j}}{12A}\|\na \mH\|_2^2\ +\ \frac{C( G)}{A\Phi^{4M+4}}\lf( \FM\rg)^2 \ +\ \frac{C(G)  \|\lan n\ran\|_{H^{M}}^2}{A\Phi^{2M+4}}\ \FM.
\end{align}
\ifx We estimate the term $T_{2}$ using the product estimate \eqref{Prd_est_gld} (?) and the estimate \eqref{Gld_Reg_grd} as follows, 
 \begin{align}\label{T_24}
|T_{2}| 
\leq&\sum_{|i,j,k|=0}^{M}\frac{CG^{4+2i}}{A}\Phi^{2j}\|\na \mH\|_2\left(\sum_{|i,j,k|=0}^{M}\|\palt  \na c_\nq\|_2\right)\|\lan n\ran\|_{H^M}\\
\leq&\sum_{|i,j,k|=0}^{M}\frac{G^{4+2i}}{12A}\Phi^{2j}\|\na \mH\|_2^2\\
&+\frac{C ( G)}{A}\frac{\|\lan n\ran\|_{ H^M }^2}{ \Phi^{2M\myr{+2}}} \ \bigg(\Phi^{-2}\sum_{|i,j,k|=0}^{M+1}G^{ 2i}\Phi^{2j\myr{+2}} \|\pa_x^i\Gamma_y^{j}\pa_z^k c_\nq\|_2^2 \\
&\quad\quad\quad\quad\quad\quad\quad\quad\quad\quad\quad\quad+\frac{(t_r+\tau)^2}{A^{2/3}}\sum_{|i,j,k|=0}^{M} A^{2/3}G^{2(i+1)}\Phi^{2j\myr{+2}}\|\pa_x^{i+1}\Gamma_y^{j}\pa_z^{k} c_\nq\|_2^2\bigg)\\
\leq & \sum_{|i,j,k|=0}^{M}\frac{G^{4+2i}}{12A}\Phi^{2j}\ \|\na \mH\|_2^2\ +\ \frac{C(G)  \|\lan n\ran\|_{H^{M}}^2}{A\Phi^{2M+4}}\ \FM.
\end{align}
\fi 
Next we consider the $T_{2}$ term in \eqref{T2_NL_1234}, which contains the $\na d_\nq=\na S_{t_r}^{t_r+\tau}\cc_\nq$. We apply the linear  estimate \eqref{Gld_Rg_ED} (with $G$ chosen large enough), and the gradient estimate  \eqref{Gld_Reg_grd} to obtain 
\begin{align}\label{palt_na_d_nq}  
\sum_{|i,j,k|=0}^{M} \|\gt \na S_{t_r}^{t_r+\tau}\cc_\nq\|_2
  \leq& \frac{C}{(\Phi(t_r+\tau))^{M}}\bigg(\frac{1}{\Phi(t_r+\tau)}\sum_{|i,j,k|=0}^{M+1}G^{i}(\Phi(t_r+\tau))^{j}\|\Gamma_{y;t_r+\tau}^{ijk} S_{t_r}^{t_r+\tau}\cc_\nq\|_2 \\
&\qquad\qquad\qquad\quad+\ (t_r+\tau)\  \sum_{|i,j,k|=0}^{M} G^{i+1}(\Phi(t_r+\tau))^{j} \| \Gamma_{y;t_r+\tau}^{(i+1)jk} S_{t_r}^{t_r+\tau}\cc_\nq\|_2\bigg)\\
\leq& C\frac{1+t_r+\tau}{(\Phi(t_r+\tau))^{M+1}} \sum_{|i,j,k|=0}^{M+1}G^i\Phi(t_r)^j\|\Gamma_{y,t_r}^{ijk}\cc_\nq(t_r)\|_{L^{2}}\exp\lf\{-\frac{\delta_{\mathcal Z} }{A^{1/3}}\tau\rg\}. 
\end{align}
Here the notation $\Phi(\cdots)$ represents the value of the function $\Phi$ at the given time. 
This estimate, when combined with the product estimate \eqref{Prd_est_gld} \myc{(??)}, and the H\"older, Young inequalities, yields that 
\begin{align}\label{T_22}  \qquad
|T_{2}|\leq & \sum_{|i,j,k|=0}^{M}\frac{G^{4+2i}\Phi^{2j}}{12A}\|\na \mH\|_2^2\\
&+ \frac{C}{Q^2 A^{1/3}}\left(\sum_{|i,j,k|=0}^{M}G^{4+2i}\Phi^{2j}\|\gt n_\nq\|_{2} ^2 +\|\lan n \ran\|_{H^{M}}^2\right)\\
&\quad\times Q^2\left(\sum_{|i,j,k|=0}^{M+1}\myr{G^{2i}\Phi(t_r)^{2j}}\|\Gamma_{y;t_r}^{ ijk}\cc_\nq(t_r)\|_{L^{2}}^2\right)\exp\lf\{-\frac{2\delta_{\mathcal Z}}{A^{1/3}} \tau\rg\}\frac{(1+t_r+\tau)^2}{A^{2/3}}\Phi^{-4M-2}\\
\leq &  \sum_{|i,j,k|=0}^{M}\frac{G^{4+2i}\Phi^{2j}}{12A}\|\na \mH\|_2^2\\
&+\frac{C(G)\exp\lf\{-\frac{2\delta_{\mathcal Z}}{A^{1/3}} \tau\rg\}}{A^{1/3}Q^2\Phi^{4M+3}}\left(Q^2\sum_{|i,j,k|=0}^{M+1}\myr{G^{2i}\Phi(t_r)^{2j}}\|\Gamma_{y;t_r}^{ijk}\cc_\nq(t_r)\|_{L^{2}}^2\right)\lf(\FM+\|\lan n\ran\|_{H^M}^2\rg) .
\end{align}
\myc{$(\frac{1+t}{A^{1/3}})^m\leq (1+t/A^{1/3})^m\leq C_m (1+(t/A^{1/3})^m)\leq C_m (\Phi^{-m/3}+(1+t^3/A)^{m/3})=C_m \Phi^{-m/3}$.}
Similar to the estimate of $T_1$, we estimate $T_{3}$ term with the product estimate \eqref{Prd_est_gld} \myc{(?)} as follows 
\begin{align}\label{T_23} 
|T_{3}|
\leq&\sum_{|i,j,k|=0}^{M}\frac{G^{4+2i}}{A}\Phi^{2j}\|\na \mH\|_2\ \|\na  \lan  \cc\ran\|_{H^M} \bigg( \sum_{|i,j,k|=0}^{M}\|\gt n_\nq\|_{ 2}\bigg)\\
\leq&\sum_{|i,j,k|=0}^{M}\frac{G^{4+2i}}{12A}\Phi^{2j}\|\na \mH\|_2^2+\frac{C(G)} {A \Phi^{2M}}\|\na  \lan  \cc\ran\|_{H^M} ^2 \left( \sum_{|i,j,k|=0}^{M}G^{4+2i}\Phi^{2j}\|\gt n_\nq\|_2^2\right) \\
\leq & \sum_{|i,j,k|=0}^{M}\frac{G^{4+2i}}{12A}\Phi^{2j}\ \|\na \mH\|_2^2\ +\ \frac{C(G)}{A  \Phi^{ 2M}}\  \|\na\lan \cc\ran\|_{H^{M}}^2\FM. 
\end{align} 

Combining \eqref{T2_NL_1234} and \eqref{T_21}, \eqref{T_22}, \eqref{T_23}, we obtain the estimate \eqref{T_n_2_NL}.
\ifx
\myb{Now we choose the $\delta_\dagger$ small such that $\delta_\dagger\leq \frac{1}{8}\delta_d$, then
\begin{align}
|T_{24}^{NL;R}|
\leq &\sum_{|i,j,k|=0}^{M}\frac{G^{4+2i}\Phi^{2j}}{12A}\|\na \palt n_\nq\|_2^2
 + \frac{1}{A^{1/3}} C(\delta_d^{-1})\|\lan n\ran\|_{H^M}^2 e^{-\frac{\delta_d (t_\star+t)}{8A^{1/3}}}\min\left\{ \ep_I^2 F_M, \ \|\cc_{\text{in};\nq}\|_{H^{M+1}}^2 \right\} \\ & 
  +\frac{1}{A^{1/3}}C( G, \delta_d^{-1})e^{-\frac{\delta_d (t_\star+t)}{8A^{1/3}}}\|\cc_{\text{in};\nq}\|_{H^{M+1}}^2 F_M.
\end{align}}\fi

\noindent 
{\bf Proof of \eqref{T_n_3_NL}:}
Now we estimate the left hand side of \eqref{T_n_3_NL}. We decompose it as three terms 
\begin{align} \label{TNL_3i} 
-\sum_{|i,j,k|=0}^{M}\frac{G^{4+2i} }{A}\Phi^{2j}\int \mH   \ [\Gamma_y^j,\pa_y]\pa_x^i\pa_z^k(n\na c_\nq +n\na d_\nq+n_\nq\na\lan \cc\ran)dV=:\sum_{\ell= 4}^6T_{\ell}.
\end{align}
The $T_{4}$ term can be estimated with the commutator relation \eqref{cm_yt_j_y}, the regularity estimate \eqref{Gld_Reg_grd}, and the product estimate \eqref{Prd_est_gld} as follows:
 \begin{align}
| T_{4}|=&\bigg|\sum_{|i,j,k|=0}^{M}\frac{G^{4+2i}}{A}\Phi^{2j}\sum_{\ell=0}^{j-1}\binom{j} {\ell}\int \mH \ B^{(j-\ell+1)}\pa_x\Gamma_y^{\ell}\pa_x^{i}\pa_z^k(n \na c_\nq)dV\bigg|\\
\leq&\frac{C(G)}{ A}\sum_{|i,j,k|=0}^{M}G^{4+2i}\Phi^{2 j}\ \|  \mH\|_2  \ \myr{(t_r+\tau)}\ \sum_{\ell=0}^{j-1} \|\Gamma_{y;t}^{{(i+1)}\ell k}( n \na c_\nq)\|_2.
\end{align}
To estimate the last factor, we apply the product estimate \eqref{Prd_est_gld} and the fact that $\|f_\nq \|_2\leq \|\pa_x f_\nq\|_2$ to derive the following general estimate: 
\begin{align}
&\mathbbm{1}_{|i,j,k|\leq M}\sum_{\ell =0}^{ j-1 }\|\Gamma_{y;t}^{(i+1)\ell k} (n \na f_\nq)\|_2\leq\mathbbm{1}_{|i,j,k|\leq M}\sum_{\ell =0}^{ j-1 }\lf(\|\Gamma_{y;t}^{i \ell k } (\pa_x n_\nq \na f_\nq)\|_2+\| \Gamma_{y;t}^{i\ell k} (n \pa_x \na f_\nq)\|_2\rg)\\
 &\leq C\left(\sum_{|i,j,k|=0}^{M-1}\|\gt \pa_x n_\nq\|_2\right)\left(\sum_{|i,j,k|=0}^{M-1}\|\gt \na  f_\nq\|_2\right)+C\left(\sum_{|i,j,k|=0}^{M-1}\|\gt n\|_2\right)\left(\sum_{|i,j,k|= 0}^{M-1}\|\gt \pa_x \na f_\nq\|_2\right)\\
&\leq C\left(\sum_{|i,j,k|=0}^{M}\|\gt  n_\nq\|_2+\|\nz\|_{H^{M-1}}\right)\left(\sum_{|i,j,k|=0}^{M -1}\| \Gamma_{y;t}^{(i+1)jk}  \na  f_\nq\|_2\right). 
\end{align}
Now we invoke the gradient estimate \eqref{Gld_Reg_grd} to derive the following
\begin{align}\label{n_naf_est} \mathbbm{1}_{|i,j,k|\leq M}\sum_{\ell =0}^{ j-1 }\|\Gamma_{y;t}^{(i+1)\ell k} (n \na f_\nq)\|_2\leq C\left(\sum_{|i,j,k|=0}^{M}\|\gt  n_\nq\|_2+\|\nz\|_{H^{M-1}}\right)\left(\sum_{|i,j,k|=0}^{M }\| \Gamma_{y;t}^{(i+1)jk}  f_\nq\|_2\right)\myr{(1+ t)}.
\end{align}
\myc{Here we rewrite $\|\Gamma_{y;t}^{(i+1)jk}\na f_\nq\|_2$ as $\|\Gamma_{y;t}^{ i jk}\na (\pa_xf_\nq)\|_2$. Then we apply the result \eqref{Gld_Reg_grd} to obtain that 
\begin{align*}
\sum_{|i,j,k|=0}^{M-1}\lf\|\Gamma_{y;t}^{ijk}\na(\pa_x f_\nq)\rg\|_{L^2}
\leq& \sum_{|i,j,k|=0}^{ M}\lf\|\Gamma_{y;t}^{ijk}(\pa_x f_\nq)\rg\|_{L^2}+C t  \sum_{ |i,j,k|=0}^ { M-1} \lf \|  \Gamma_{y;t}^{(i+1)jk} (\pa_x f_\nq)\rg\|_{L^2}\\
\leq & C(1+t)\sum_{|i,j,k|=0}^{M}   \lf \|  \Gamma_{y;t}^{(i+1)jk}   f_\nq \rg\|_{L^2}.
\end{align*}
}
Combining the estimate \eqref{n_naf_est}, the bound $(\frac{1+t}{A^{1/3}})^4\leq C\Phi^{-2}$ \eqref{varphi}, and the definition of $\FM$ \eqref{F_M} yields the bound 
\begin{align}
\label{T_31}  |T_{4}|
\leq& \sum_{|i,j,k|=0}^{M}\frac{G^{4+2i}}{A^{1/2}}\Phi^{2j}\|  \mH\|_2^2+\frac{C( G )\myr{(1+t)^4}}{A^{{3/2}+2/3} \Phi^{4M+2}}\bigg(\sum_{|i,j,k|=0}^ {M}G^{2i}\Phi^{2j\myr{+2}}A^{2/3}\|\Gamma_{y;t}^{(i+1)jk} c_\nq\|_2^2\bigg) \\
&\qquad\qquad\qquad\qquad\qquad\qquad\times \bigg( \sum_{|i,j,k|=0}^{M}G^{4+2i}\Phi^{2j}\|  \gt n_\nq\|_2^2+\Phi^{2M}\|\nz\|_{H^M}^2\bigg)\\
\leq& \sum_{|i,j,k|=0}^{M}\frac{G^{4+2i}}{A^{1/2}}\Phi^{2j}\|  \mH\|_2^2+\frac{ C(  G   )}{A^{5/6} \Phi^{4M+4}}\lf(\FM\rg)^2+\frac{C( G)}{A^{5/6}\Phi^{2M\myr{+ 4 }}}\ \FM\ \|\lan n\ran\|_{H^M}^2 .
 \end{align}

\myc{Next we recall the commutator relation \eqref{cm_yt_j_y}, the product estimate \eqref{Prd_est_gld} (?), the regularity of function gradient \eqref{Gld_Reg_grd}, and estimate as follows
\begin{align} \qquad
&\bigg|\sum_{|i,j,k|=0}^{M}\frac{G^{4+2i}}{A}\Phi^{2j}\sum_{\ell=0}^{j-1}\left(\begin{array}{rr}j\\ \ell\end{array}\right)\int \mH \ B^{(j-\ell+1)}\pa_x \Gamma_y^{\ell}\pa_x^{i}\pa_z^k(\lan n\ran \na c_\nq)dV\bigg|\\
\leq&C \sum_{|i,j,k|=0}^{M}\frac{  G^{4+2i}\Phi^{j }}{ A}\sum_{\ell=0}^{j-1}\ \Phi^{j} \|  \mH\|_2\ \myr{(t_r+\tau)} \ \|\Gamma_y^{\ell}\pa_z^k(\lan n\ran \na \pa_{x}^{i+1}c_\nq)\|_2\\
\leq& \sum_{|i,j,k|=0}^{M}\frac{G^{4+2i}}{A^{1/2}}\Phi^{2j}\|  \mH\|_2^2+\frac{C( G)}{\Phi^{2M\myr{+2} }}\sum_{|i,j,k|=0}^{M}\frac{ (t_r+\tau)^4G^{2i+2}\Phi^{2j\myr{+2}}}{A^{{3/2}+2/3}}\ \lf(A^{2/3}\|\pa_x^{i+1}\Gamma_y^j\pa_z^k c_\nq\|_2^2\rg)\ \|\lan n\ran\|_{H^M}^2\\
\leq& \sum_{|i,j,k|=0}^{M}\frac{G^{4+2i}}{A^{1/2}}\Phi^{2j}\|  \mH\|_2^2+\frac{C( G)}{A^{5/6}\Phi^{2M\myr{+ 4 }}}\ \FM\ \|\lan n\ran\|_{H^M}^2.
\end{align}}
Now we recall the estimates $\eqref{n_naf_est}_{f_\nq=d_\nq}$, \eqref{Gld_Reg_grd}, the enhanced dissipation \eqref{palt_na_d_nq}, and obtain the bound
\begin{align}\label{T_32}\quad
|T_{5}|\leq& \sum_{|i,j,k|=0}^{M}\frac{G^{4+2i}}{A^{1/3}G^2}\Phi^{2j}\|  \mH\|_2^2e^{-\frac{\delta_{\mathcal Z}}{A^{1/3}} \tau}  +\frac{C( G )\myr{(1+t)^4}}{A^{5/3} \Phi^{4M+2}}\bigg(\sum_{|i,j,k|=0}^ {M+1}G^{2i}\Phi^{2j} \|\Gamma_{y;t}^{ijk} S_{t_r}^{t_r+\tau}\cc_\nq\|_2^2\bigg)e^{\frac{\delta_{\mathcal Z}}{A^{1/3}} \tau} \\
&\qquad\qquad\qquad\qquad\qquad\qquad\times \bigg( \sum_{|i,j,k|=0}^{M}G^{4+2i}\Phi^{2j}\|  \gt n_\nq\|_2^2+\Phi^{2M}\|\nz\|_{H^M}^2\bigg)\\
\leq &\sum_{|i,j,k|=0}^{M}\frac{G^{4+2i}\Phi^ {2j}}{G^2A^{1/3}}\| \mH\|_{L^2}^2 \exp\lf\{-\frac{\delta_{\mathcal Z}}{A^{1/3}} \tau\rg\} \\
\quad&+\frac{C(G) \exp\lf\{-\frac{\delta_{\mathcal Z}}{A^{1/3}} \tau\rg\}}{Q^2A^{1/3}\Phi^{2M+4}}\ \lf(Q^2\sum_{|i,j,k|=0}^{M+1}G^{2i}\Phi(t_r)^{2j}\|\Gamma_{y;t_r}^{ijk}\cc_{\neq}(t_r)\|_{L^{2}}^2\rg) \lf(\Phi^{-2M}\FM +\|\lan n\ran\|_{H^M}^2\rg).
\end{align}

Similar to the estimate of $T_4$, we estimate the $T_{6}$ as follows{}
\begin{align}\label{T_33} \qquad
|T_{6}|
\leq&C\sum_{|i,j,k|=0}^{M}\frac{  G^{4+2i}}{A}\Phi^{2j}\|\mH\|_2 \sum_{\ell=0}^{j-1}\binom{j}{\ell}\ \myr{(t_r+\tau)}\ \|\pa_x^{i+1}\Gamma _y^\ell \pa_z^k(n_\nq \na \lan \cc\ran )\|_2\\
\leq&\sum_{|i,j,k|=0}^{M}\frac{G^{4+2i}}{A^{1/2}}\Phi^{2j}\|\mH\|_2^2+ \frac{C }{A^{3/2}}(t_r+\tau)^2\Phi^{-2M}\left(\sum_{|i,j,k|=0}^{M}G^{4+2i}\Phi^{2j}\|\gt n_\nq\|_2^2\right)\|\na \lan\cc\ran\|_{H^{M}}^2\\
\leq&\sum_{|i,j,k|=0}^{M}\frac{G^{4+2i}}{A^{1/2}}\Phi^{2j}\|\mH\|_2^2+\frac{C( G)}{A^{5/6} \Phi^{2M+2}}\ \lf(\sum_{|i,j,k|=0}^{M} G^{4+2i} \Phi^{2j}\|\gt n_\nq\|_2^2\rg)\ \norm{\na \lan\cc\ran }_{\myr{H^{M}}}^2.&
\end{align}
\ifx
 \myr{last line previous\begin{align}
&+\frac{CG^{6+2M} }{A^{1/3}}\frac{(t_r+\tau)^4}{A^{4/3}} e^{-\frac{\delta_{\mathcal Z}}{A^{1/3}} \tau}\ \Phi^{-2M-2}\lf(\sum_{|i,j,k|=0}^{M+1}G^{2i}\Phi(t_r)^{2j}\|\pa_x^i\Gamma_{y;t_r}^j\pa_z^k\cc_{\neq}(t_r)\|_{L^{2}}^2\rg) \lf(\Phi^{-2M}\FM\rg) \\
&+\myr{\frac{CG^{6+2M} }{A^{1/3}} \frac{(t_r+\tau)^4}{A^{4/3}} e^{-\frac{\delta_{\mathcal Z}}{A^{1/3}} \tau}\ \Phi^{-2M-2}\lf(\sum_{|i,j,k|=0}^{M+1}G^{2i}\Phi(t_r)^{2j}\|\pa_x^i\Gamma_{y;t_r}^j\pa_z^k\cc_{\neq}(t_r)\|_{L^{2}}^2\rg) \|\lan n\ran\|_{H^M}^2}.
\end{align}}\fi
Combining \eqref{T_31}, \eqref{T_32}, \eqref{T_33} and the decomposition \eqref{TNL_3i}, we obtain the estimate \eqref{T_n_3_NL}. This concludes the proof.
\ifx
\footnote{\bf\color{red} To extend it to all initial data, we need to track the time for the size of $n_\nq, d_\nq$ impact the dynamics. Maybe on the time interval $[0, A^{1/3+\te}]$, it is OK. The observation is the following. Even though the $d_\nq$ will have nontrivial impact on the $\lan n\ran$, but the size of $\lan n\ran$ will take more time to influence the $\na \lan c\ran$ due to the linear integration formula. So on time interval $[0, A^{1/3+\ep}]$, $\|\na \lan c\ran\|_2\approx C(n_{in}, c_{in})$. Then we can choose the $G$ to be depending on the initial data to control it. In this case, we have no dependence of $G$ on the $C_{\lan c\ran}$?? Also, if $\lan n\ran \sim t^3/A$, but if we plug it into the equation for $n_\nq$, it might be fine. Because we see expressions like $\frac{\varphi^j\|\na \cc _{in}\|^4t^2}{A}\frac{t^{3}}{A}e^{-t\delta/A^{1/3}}+\frac{\|\lan n_{in}\ran\|^2\|\na \cc_{in}\|^2  t^2}{A}e^{-\delta t/A^{1/3}}$. But there might be $\|(\pa_y^t)^M n_\nq\|_2$ here?}\fi
\end{proof}

\begin{proof}[Proof of Lemma \ref{lem:est_ctrm}]
Application of the estimates in Lemma \ref{lem:cm_L_trm} yield that,  
\begin{align} 
&\sum_{|i,j,k|=0}^{M+1}(A^{2/3}\mathbbm{1}_{j\leq M}+\mathbbm{1}_{j=M+1})\frac{G^{2i}\Phi^{2j+2}}{A}\bigg|\int \mf H_{ijk}\ [\Gamma_{y;t}^j ,\pa_{yy}]\pa_x^i\pa_z^k c_\nq^{t_r} dV\bigg|\\
& \leq   \frac{A^{2/3}}{8A}\sum_{\substack{|i,j,k|\leq M+1\\ j\leq M}}G^{2i}\Phi^{2j+2}\|\na \mf H_{ijk}\|_{L^2}^2+ \frac{1}{8A}\Phi^{2M+4}\|\na \mf H_{0(M+1)0}\|_{L^2}^2\\
&\quad+\sum_{\substack{|i,j,k|\leq M+1\\ \myr{j\leq M}}}A^{2/3}\left(\frac{C\Phi}{G^2 A^{1/3} }+\frac{C}{G(A^{2/3}+t^2)}\right)G^{2i}\Phi^{2j+2}(\| \mf H_{ijk}\|_{L^2}^2+\| \gt c_\nq^{t_r}\|_{L^2}^2)\\
&\quad +\left(\frac{C\Phi}{G^2 A^{1/3} }+\frac{C}{G(A^{2/3}+t^2)}\right)\Phi^{2M+4}\| \mf H_{0(M+1)0}\|_{L^2}^2\\
&\quad+\sum_{\substack{|i,j,k|\leq M+1\\ j\leq M}}\left(\frac{C\Phi}{G^2 A^{1/3} }+\frac{C}{G(A^{2/3}+t^2)}\right)G^{2i}\Phi^{2j+2}\| \gt c_\nq^{t_r}\|_{L^2}^2\\
\leq&\frac{1}{8A}\sum_{|i,j,k|=0}^{M+1}(A^{2/3}\mathbbm{1}_{j\leq M}+\mathbbm{1}_{j=M+1}) {G^{2i}\Phi^{2j+2}} \|\na \mf H_{ijk}\|_{L^2}^2\\
&\quad+  \left(\frac{C\Phi}{G^2 A^{1/3}}+\frac{C}{G(A^{2/3}+t^2)}\right)\sum_{|i,j,k|=0}^{M+1}(A^{2/3}\mathbbm{1}_{j\leq M}+\mathbbm{1}_{j=M+1}) {G^{2i}\Phi^{2j+2}}(\|\mf H_{ijk}\|_{L^2}^2+\|\gt c_\nq^{t_r}\|_{L^2}^2).  
\end{align}  This is consistent with \eqref{T_c_com}.

Now we  focus on \eqref{est_c_term}. We distinguish between two cases, $j=M+1$ and $j\leq M$. In the first case, we recall the definitions of $\Phi(t_r+\tau)=(1+(t_r+\tau)^3/A)^{-1}$ \eqref{varphi} and $\FM$ \eqref{F_M}, and observe  that no $\pa_x$ is present ($i=0$). Hence, the left hand side of \eqref{est_c_term} can be estimated as follows 
\begin{align}\label{T_c_j=M+1}\quad \
&\bigg|\frac{\Phi^{\myr{2M+4}}}{A}\int  \mathfrak{H}_{0(M+1)0}\ \Gamma_y^{M+1} n_\nq dV\bigg| 
= \bigg|- \frac{\Phi^{2M+4}}{A}\int (\pa_y+B^{(1)}\pa_x) \mathfrak{H}_{0(M+1)0} \ \Gamma_y^M n_\nq dV\bigg| \\
&\leq\frac{\Phi^{\myr{2M+4}}}{12A}\|\na \mathfrak{H}_{0(M+1)0}\|_2^2+\frac{C  \Phi^4 }{G^4A^{1/3}}\frac{(t_r+\tau)^2}{A^{2/3}}(G^4\Phi^{2M}\|\Gamma_y^M n_\nq\|_2^2)
\leq \frac{\Phi^{\myr{2M+4}}}{12A}\|\na  \mathfrak{H}_{0(M+1)0}\|_2^2+\frac{C\Phi^3}{G^4 A^{1/3}}\FM.
\end{align} 
\ifx\myr{Previous: \begin{align}
|\mathbbm{1}_{j=M+1}T_{c;2}^{NL;R}|=&\bigg|\frac{\Phi^{\myr{2M+4}}}{A}\int \Gamma_y^{M+1}c_\nq\ \Gamma_y^{M+1} n_\nq dV\bigg| 
= \bigg|- \frac{\Phi^{2M+4}}{A}\int (\pa_y+B^{(1)}\pa_x)\Gamma_y^{M+1}c_\nq \ \Gamma_y^M n_\nq dV\bigg| \\
\leq&\frac{\Phi^{\myr{2M+4}}}{12A}\|\na \Gamma_y^{M+1}c_\nq\|_2^2+\frac{C  \Phi^4 }{G^4A^{1/3}}\frac{(t_r+\tau)^2}{A^{2/3}}(G^4\Phi^{2M}\|\Gamma_y^M n_\nq\|_2^2)\\
\leq&\frac{\Phi^{\myr{2M+4}}}{12A}\|\na \Gamma_y^{M+1}c_\nq\|_2^2+\frac{C\Phi^3}{G^4 A^{1/3}}\FM.  
\end{align} }\fi
In the $j\leq M$ case, we invoke the definitions of $\Phi$ \eqref{varphi}, $\FM$ \eqref{F_M}, and distinguish between the $i=0$ and $i\neq 0$ cases. As a result, the left hand side of \eqref{est_c_term} is bounded as follows 
\begin{align}\label{T_c_j_leq_M}
\bigg|\sum_{|i,j,k|=0}^{M+1}&\mathbbm{1}_{j< M+1}\frac{G^{2i }\Phi^{2j\myr{+2}}}{A^{1/3}}\int\bigg(\mathbbm{1}_{i\neq 0}\,\pa_x\mH \,\pa_x^{i-1}\Gamma_y^j\pa_z^k n_\nq+\mathbbm{1}_{i=0}\,\pa_z\mH\, \pa_x^i\Gamma_y^j\pa_z^{k-1} n_\nq\bigg)dV\bigg| \\
 \leq &\sum_{\substack{|i,j,k|\leq M+1\\ j< M+1}}\frac{G^{2i }\Phi^{2j\myr{+2}}}{12A^{1/3}}\|\na \mH\|_2^2\\
 &+\frac{C}{G^2A^{1/3}}\myr{\Phi^2}\sum_{\substack{|i,j,k|\leq M+1\\ j< M+1}}\left(G^{4+2(i-1)}\Phi^{2j}\|\pa_x^{i-1}\Gamma_y^j\pa_z^{k} n_\nq\|_2^2+G^{4 +2i}\Phi^{2j}\|\pa_x^{i}\Gamma_y^j\pa_z^{k-1} n_\nq\|_2^2\right)\\
 \leq &\sum_{|i,j,k|=0}^{ M+1}\frac{G^{2i }\Phi^{2j\myr{+2}}}{12A}A^{2/3}\mathbbm{1}_{j< M+1}\|\na \mH\|_2^2+\frac{C\myr{\Phi^2}}{G^2A^{1/3}} \FM. 
\end{align}
Combining \eqref{T_c_j=M+1} and \eqref{T_c_j_leq_M} yields \eqref{est_c_term}.\ifx
\myr{Previous:\begin{align}
|\mathbbm{1}&_{j< M+1}T_{c;2}^{NL;R}|
\\=&\bigg|\sum_{|i,j,k|=0}^{M+1}\mathbbm{1}_{j< M+1} \frac{G^{2i }\Phi^{2j\myr{+2}}A^{2/3}}{A}\int \palt c_\nq \ \palt n_\nq dV\bigg|\\
=&\bigg|\sum_{|i,j,k|=0}^{M+1}\mathbbm{1}_{j< M+1}\frac{G^{2i }\Phi^{2j\myr{+2}}}{A^{1/3}}\int\bigg(\mathbbm{1}_{i\neq 0}\,\pa_x^{i+1}\Gamma_y^j\pa_z^k c_\nq \,\pa_x^{i-1}\Gamma_y^j\pa_z^k n_\nq+\mathbbm{1}_{i=0}\,\pa_x^i\Gamma_y^j\pa_z^{k+1}c_\nq\, \pa_x^i\Gamma_y^j\pa_z^{k-1} n_\nq\bigg)dV\bigg| \\
 \leq &\sum_{\substack{|i,j,k|\leq M+1\\ j< M+1}}\frac{G^{2i }\Phi^{2j\myr{+2}}}{12A^{1/3}}\|\na (\palt c_\nq)\|_2^2\\
 &+\frac{C}{G^2A^{1/3}}\myr{\Phi^2}\sum_{\substack{|i,j,k|\leq M+1\\ j< M+1}}\left(G^{4+2(i-1)}\Phi^{2j}\|\pa_x^{i-1}\Gamma_y^j\pa_z^{k} n_\nq\|_2^2+G^{4 +2i}\Phi^{2j}\|\pa_x^{i}\Gamma_y^j\pa_z^{k-1} n_\nq\|_2^2\right)\\
 \leq & \frac{1}{12}D_c+\frac{C}{G^2A^{1/3}}\myr{\Phi^2} \FM. 
\end{align}}\fi
\myc{ There is an implicit choice of $G$ here. But since the constant in the second term depends only on $M$, the choice of $G$ is still independent of the data $n,\cc$.}
\end{proof}
\ifx
Combining all the estimates we developed so far (\eqref{TNL_123_vf}, \eqref{T2_NLR_1234}, \eqref{T_NLR_21}, \eqref{T_22R}, \eqref{T_23R}, \eqref{T_24R}, \eqref{T_31R}, \eqref{T_32R}, \eqref{T_33R}, \eqref{T_34R}, \eqref{T_c_com_n}, \eqref{T_c_com},   \eqref{T_c_j=M+1}, \eqref{T_c_j_leq_M}), we have obtained \eqref{F_M_reg}.

We can take the $G$ large compared to $M,\ \|u_y\|_{L^\infty_t W_y^{M+1, \infty}}, \ \delta_d^{-1}$, then $\ep_I$  large compared to $G, \ M,\|u_y\|_{L^\infty_t W_y^{M+1, \infty}},\ \delta_d^{-1}$ and $\|\lan n\ran \|_{H^{M}}\leq 4C_{\lan n\ran ; H^M}$, and finally take $\|\cc_{\text{in};\nq}\|_{H^{M+1}}^{-1}, A$  large enough compared to all the previous chosen or specified constants, and  bootstrap constants in \eqref{Hypotheses} so that 
\begin{align}
F_M(t_\star+t)\leq 2F_M(t_\star),\quad \forall t\in[0, \delta_{\mathcal Z}^{-1}A^{1/3}]. 
\end{align} \fi
\subsection{The Decay Estimates of the Remainder}
\label{Sec:NL_EDdc} In this subsection, we fixed an arbitrary time $t_\star\in [t_r, T_\star-\delta^{-1}A^{1/3}]$ and estimate the deviation between the solutions to the system \eqref{ppPKS_neq} and the passive scalar solutions $\lf(S_{t_\star}^{t_\star+\tau}\lf(\Gamma_{y;t_\star}^{ijk} n_\nq\rg),\, S_{t_\star}^{t_\star+\tau}\lf(\Gamma_{y;t_\star}^{ijk}  c_\nq^{t_r}\rg)\rg)$ initiated from time $t_\star$ \myr{on the time interval $t_\star+\tau\in [t_\star, t_\star+\delta^{-1}A^{1/3}]\subset [t_r, T_\star]$}. To this end, we define the variations and the functional that measures them 
\begin{align}
\mathbb{V}_{ijk}^{n}(t_\star+\tau):=&  \Gamma_{y;t_\star+\tau}^{ijk} n_\nq-S_{t_\star}^{t_\star+\tau}\lf( \Gamma_{y;t_\star}^{ijk} n_\nq(t_\star)\rg);\label{V_n_ijk}\\
\mathbb{V}_{ijk}^c(t_\star+\tau):=&\Gamma_{y;t_\star+\tau}^{ijk} c_\nq^{t_r}- \myr{S_{t_\star}^{t_\star+\tau}\lf(\Gamma_{y;t_\star}^{ijk} c_\nq^{t_r}(t_\star-t_r)\rg)};\label{V_c_ijk}\\
\mathbb{D}_{G,Q}^{t_r;M}[t_\star+\tau]:=&\sum_{|i,j,k|=0}^{M}G^{4+2i}\Phi^{2j} \lf\|\mathbb V_{ijk}^n\rg\|_2^2+\sum_{ |i,j,k|=0}^{ M+1}G^{2i}\lf(A^{2/3}\mathbbm{1}_{j\leq M}+\mathbbm{1}_{j=M+1}\rg)\Phi^{2j+2}\lf\|\mathbb{V}_{ijk}^c\rg\|_2^2.\label{D_M}
\end{align}
Here the parameter $\tau$ is the time increment from $t_\star$ (instead of $t_r$). Since the argument in the function $c_\nq^{t_r}$ is the time increment from the reference time $t_r$, we have to write $c_\nq^{t_r}(t_\star-t_r)$ when invoking the function value at time $t_\star$.  
 \myc{We have to be careful when we use the $c_\nq^{t_r}$. We have to specify the reference time. The argument is the increment with respect to the reference time. Since we cannot use $\tau$ here, we have to use $t_\star-t_r$. }
To estimate the quantity $\mathbb{D}_{G, Q}^{t_r;M}$  \eqref{D_M}, we write down the equations of the variations $\mathbb{V}_{ijk}^{n}\ \eqref{V_n_ijk}, \, \mathbb{V}_{ijk}^{c}$ \eqref{V_c_ijk}:
\begin{subequations}\label{V_eq}
\begin{align}\label{Vn_eq}
\pa_\tau 
\mathbb{V}_{ijk}^{n}+u_A(t_\star+\tau,y)\pa_x
\mathbb{V}_{ijk}^{n}-\frac{1}{A}\de\mathbb{V}_{ijk}^{n}
=& \frac{1}{A}[\Gamma_y^j, \pa_{yy}]\pa_x^i\pa_z^k n_\nq-\frac{1}{A}\gt\lf( \na\cdot(n\na \cc)\rg)_\nq,\\
\label{Vc_eq}\pa_\tau 
\mathbb{V}_{ijk}^{c}+u_A(t_\star+\tau,y)\pa_x
\mathbb{V}_{ijk}^{c}-\frac{1}{A}\de\mathbb{V}_{ijk}^{c}
=& \frac{1}{A}[\Gamma_y^j, \pa_{yy}]\pa_x^i\pa_z^k c_\nq+\frac{1}{A}\gt n_\nq\\
\mathbb{V}_{ijk}^n(\tau=0)=\mathbb{V}_{ijk}^c(\tau=0)=&0,\quad [t_\star, t_\star+\tau]\subset[t_\star,t_\star+\delta^{-1}A^{1/3}].
\end{align}\myr{Here $\delta$ is defined in \eqref{Chc_del} and the current time is $t=t_\star+\tau$.}
\end{subequations}
The estimate of the functional $\DD$ is collected in the next theorem. 
\begin{theorem}\label{thm:DD_est} There exist constants $C$ and $C_G$ such that the following estimate holds
\begin{align}\label{DD_est}\qquad
\frac{d}{d\tau}&\mathbb{D}_{G,Q}^{t_r;M} [t_\star+\tau,
\mathbb{V}_{ijk}^{n},\mathbb{V}_{ijk}^{c}]\\
\leq&   \frac{C}{  {G}}  \left(\frac{\Phi}{G A^{1/3}}+\myr{\frac{G}{A^{1/2}}}+\frac{1}{A^{2/3}+(t_\star+\tau)^2}+\frac{1}{A^{1/3}}\exp\lf\{-\frac{\delta_{\mathcal Z}}{4A^{1/3}} (t_\star+\tau-\myr{t_r})\rg\}\right) \mathbb{D}_{G,Q}^{t_r;M} \\
& +\frac{C}{  {G}}  \left(\frac{\Phi}{G A^{1/3}}+\frac{1}{A^{2/3}+(t_\star+\tau)^2}\right)\FM\\ 
&+ \frac{C_{G}}{A^{1/3}}\ \frac{\FM}{A^{1/2}\Phi^{ 4M+4}}\ \bigg(1+{\myc{??}}\FM+ \|\lan n\ran\|_{H^M}^2+ \|\na\lan   \cc\ran\|_{ H^{M}}^2  \bigg)
\\ &+\frac{C_G}{A^{1/3}}\frac{\exp\lf\{-\frac{\delta_{\mathcal Z}}{A^{1/3}} (t_\star+\tau-t_r)\rg\}}{ \Phi^{ 4M+4}}\ \lf(\sum_{|i,j,k|=0}^{ M+1}G^{2i}\Phi(t_r)^{2j}\|\Gamma_{y;t_r}^{ijk}\cc_{\neq}(t_r)\|_{L^{2}}^2\rg) \lf(\FM+\|\lan n\ran\|_{H^M}^2\rg).
\end{align}
Here $\Phi=\Phi(t_\star+\tau), \, \FM=\FM[t_\star+\tau],\, \mathbb{D}_{G,Q}^{t_r;M}=\mathbb{D}_{G,Q}^{t_r;M}  [t_\star+\tau]   $. 
Here the constant $C$ only depends on $M, \,\|u_A\|_{L^\infty_tW^{M+3,\infty}}$ and $C_G$ only depends on $G,\ M, \,\|u_A\|_{L^\infty_tW^{M+3,\infty}}$. 
\end{theorem}\begin{proof}
We first decompose the time evolution as follows 
 \begin{align}\label{ddtau_DD}\qquad
\frac{d}{d\tau}\DD=&\frac{d}{d\tau}\lf(\sum_{|i,j,k|=0}^ M\Phi^{2j} G^{4+2i}\|\mathbb{V}_{ijk}^{n}\|_{2}^2+\sum_{|i,j,k|=0}^ { M+1}\Phi^{2j{+2}}G^{2i}\lf(A^{2/3}\mathbbm{1}_{j<M+1}+\mathbbm{1}_{j=M+1}\rg)\|
\mathbb{V}_{ijk}^{c}\|_{2}^2\rg)\\
 =:&\pa D_1+\pa D_2.
\end{align}
Application of the $\VV^n$-equation \eqref{Vn_eq} yields that the $\pa D_1$ term in \eqref{ddtau_DD} can be decomposed as follows
\begin{align}\label{pa_D_1_123}
\frac{1}{2}\pa D_1\leq &-\sum_{|i,j,k|=0}^{M}\frac{G^{4+2i} \Phi^{2j}}{A}\|\na \VV^n\|_2^2+\sum_{|i,j,k|=0}^{M}\frac{G^{4+2i} }{A}\Phi^{2j}\int \VV^n \  [ \Gamma_y^j,\pa_{yy}]\pa_x^i\pa_z^k n_\nq dV\\
&+\sum_{|i,j,k|=0}^{M}\frac{G^{4+2i} }{A}\Phi^{2j}\int \na\VV^n \cdot \palt\lf( n\na  \cc_{\neq}+n_\nq \na \lan \cc\ran\rg) dV\\
&-\sum_{|i,j,k|=0}^{M}\frac{G^{4+2i} }{A}\Phi^{2j}\int \VV^n   \ [\Gamma_y^j,\pa_y]\pa_x^i\pa_z^k\lf(n\na\cc_\nq +n_\nq\na\lan \cc\ran\rg)dV\\
=:&-\mf D_{\mathbb V^n}+T_{n;1}^{NL;D}+T_{n;2}^{NL;D}+T_{n;3}^{NL;D}.
\end{align}
\myc{$\cc_\nq(t_\star+\tau)= c^{t_r}_\nq(t_\star+\tau-t_r)+ S_{t_r}^{t_\star+\tau}\cc_\nq(t_r)$.}
Similarly, the $\pa D_2$-term in \eqref{ddtau_DD} can be decomposed with the equation \eqref{Vc_eq} as follows
\begin{align} \label{pa_D_2_12}
\frac{1}{2}\pa D_2 \leq&\sum_{|i,j,k|=0}^ { M+1}(A^{2/3}\mathbbm{1}_{j<M+1}+\mathbbm{1}_{j=M+1})\bigg(-\frac{\Phi^{2j{+2}}G^{2i}}{A}\|\na  \VV^c\|_2^2\\
&+\frac{G^{2i}\Phi^{2j\myr{+2}}}{A}\int \VV^c\ [\Gamma_y^j ,\pa_{yy}]\pa_x^i\pa_z^k c_\nq dV+\frac{G^{2i}\Phi^{2j\myr{+2}}}{A}\int \VV^c\ \palt n_\nq dV\bigg)\\
=:&-\mf D_{\mathbb{V}^c}+T_{c;1}^{NL;D}+T_{c;2}^{NL;D}.
\end{align}
First of all, we invoke Lemma  \ref{lem:cm_L_trm} and  Lemma \ref{lem:est_ctrm} to show that  the commutator terms $T_{n;1}^{NL;D}$ in \eqref{pa_D_1_123} and $T_{c;1}^{NL;D}$ in \eqref{pa_D_2_12}  are bounded as follows:
\begin{align}\label{T_nc1_NLD}
|T_{n;1}^{NL;D}|+|T_{c;1}^{NL;D}|\leq\frac{1}{8} \mf D_{\mathbb{V}^n}+\frac{1}{8}\mf D_{\mathbb{V}^c}+\lf(\frac{\Phi}{GA^{1/3}}+\frac{1}{A^{2/3}+(t_\star+\tau)^2}\rg)\myr{(\mathbb{D}_{G, Q}^{t_r;M}+\FM)}. 
\end{align}
Now application of Lemma \ref{lem:est_nNL} yields that the $T_{n;2}^{NL;D}+T_{n;3}^{NL;D}$ terms are bounded as follows  
\begin{align}\label{T_n23_NLD}\qquad
&|T_{n;2}^{NL;D}+T_{n;3}^{NL;D}|\\
&\leq  \sum_{|i,j,k|=0}^{M}\frac{G^{4+2i}\Phi^{2j}}{4A}\|\na \mathbb V^n_{ijk}\|_2^2 +\sum_{|i,j,k|=0}^{M}\frac{CG^{4+2i}\Phi^{2j}}{A^{1/3}}\lf(\frac{1}{A^{1/6}}+\frac {\exp\lf\{-\frac{\delta_{\mathcal Z}}{A^{1/3}} (t_\star+\tau-t_r)\rg\}}{G^2}\rg)\|  \mathbb V^n_{ijk}\|_{L^2}^2\\ 
 & \quad +\frac{C(G)}{ A^{5/6}\Phi^{4M+4}}\lf( \FM\rg)^2+\frac{C(G)}{ A^{5/6}\Phi^{2M+4}}(\norm{n}_{H^M}^2+\norm{\na \lan\cc\ran }_{H^{M}}^2)\FM  \\
&\quad+\frac{C(G) \exp\lf\{-\frac{\delta_{\mathcal Z}}{A^{1/3}} (t_\star+\tau-t_r)\rg\}}{A^{1/3}\Phi^{4M+4}}\ \lf(\sum_{|i,j,k|=0}^{M+1}G^{2i}\Phi(t_r)^{2j}\|\Gamma_{y;t_r}^{ijk}\cc_{\neq}(t_r)\|_{L^{2}}^2\rg) \lf(\FM +\|\lan n\ran\|_{H^M}^2\rg).
\end{align}
This is consistent with the estimate \eqref{DD_est}. 

Application of Lemma \ref{lem:est_ctrm} yields that the $|T_{c;2}^{NL;D}|$ is bounded as follows:
\begin{align}\label{T_c2_NLD}
|T_{c;2}^{NL;D}|
\leq &\frac{1}{4}\mf D_{\mathbb{V}^c}+\frac{C\Phi^2}{G^2A^{1/3}}\FM. 
\end{align}
Combining the decomposition \eqref{pa_D_1_123}, \eqref{pa_D_2_12} and the estimates \eqref{T_nc1_NLD}, \eqref{T_n23_NLD}, \eqref{T_c2_NLD}, we obtain \eqref{DD_est}.
\end{proof} 

\subsection{Conclusion}\label{Sec:NL_Con}
In this section, we prove Proposition \ref{Pro:main}, Proposition \ref{Pro:2}, and Proposition \ref{pro:prp_FM}. We start with the proof of Proposition \ref{Pro:main}.
\begin{proof}[Proof of Proposition \ref{Pro:main}]
First of all, we specify the reference time $t_r=0$. To prove the  proposition, we use a bootstrap argument.  Assume that $[0,T_\star]\subset[0, A^{1/3+\te}]$ is the largest interval on which the following hypotheses hold:\begin{subequations}
\begin{align}
\FM[t, n_\nq,\cc_\nq]\leq& 2C_{ED}\ \FM [0,n_\nq,\cc_\nq]\ \exp\lf\{-\frac{2\delta t}{A^{1/3}}\rg\},\quad\forall\ t\in [0, T_\star],\label{HypED_1}\\
\|\lan n\ran\|_{L_t^\infty([0,T_\star];H_{y,z}^M)}\leq& 2\BB_{\lan n\ran;H^M},\quad
\|\na \lan \cc\ran\|_{L_t^\infty([0,T_\star];H_{y,z}^{M})}\leq 2\BB_{\lan \cc\ran;H^{M+1}}.\label{Hyp_<n>_<c>}
\end{align}\end{subequations}
Here,
\begin{align}\label{BB<n><c>}
\BB_{\lan n\ran;H^M}:= 2\|\lan n_{\text{in}}\ran\|_{H^M}+2,\quad\BB_{\lan \cc\ran;H^{M+1}}:=\myr{  {A^{-1/5}}+2\|\na\lan \cc_{\mathrm{in}}\ran\|_{H^{M}}}. 
\end{align}
The parameter $\delta$ is chosen in \eqref{Chc_del}. By the local well-posedness of the equation \eqref{ppPKS} in $H^M,\, M\geq 3$ (Theorem \ref{thm:lcl_exst}), we have that the interval $[0,T_\star]$ is non-empty.  The \underline{\emph{goal}} is to prove that  given the assumption \eqref{chc_GAQ} and the hypotheses \eqref{HypED_1}, \eqref{Hyp_<n>_<c>}, the following stronger estimates hold\begin{subequations}
\begin{align}\label{ConED_1}
\FM[t, n_\nq,\cc_\nq]\leq C_{ED}\ \FM [0,n_\nq,\cc_\nq]& \ \exp\lf\{-\frac{2\delta t}{A^{1/3}}\rg\} \leq  C(\|n_{\mathrm{in};\neq}\|_{H^M}^2+ 1)\exp\lf\{-\frac{2\delta t}{A^{1/3}}\rg\},\quad\forall\ t\in [0, T_\star],\\
\|\lan n\ran\|_{L_t^\infty([0,T_\star];H_{y,z}^M)}\leq& \BB_{\lan n\ran;H^M},\qquad 
\| \na\lan \cc\ran\|_{L_t^\infty([0,T_\star];H_{y,z}^{M})}\leq\BB_{\lan \cc\ran;H^{M+1}}. \label{Con_<n><c>}
\end{align}\end{subequations}
As a consequence of this bootstrap argument and  Theorem \ref{thm:lcl_exst}, the estimates \eqref{ConED},  \eqref{Con_<n>}, \eqref{Con_<grd_c>} hold on the time horizon $[0, A^{1/3+\te}]$. 

First we prove the enhanced dissipation estimate \eqref{ConED_1}. We choose an arbitrary starting time $t_\star\in[0, A^{1/3+\te}]$ and consider a time interval $t=t_\star+\tau\in [t_\star, t_\star+\delta^{-1}A^{1/3}]\subset [0, A^{1/3+\te}]$.  It is enough to derive the regularity estimate \eqref{Rg_est_NL} and the decay estimate \eqref{Decay_est_NL}.

To prove the regularity estimate \eqref{Rg_est_NL}, we invoke Theorem \ref{thm:F_M_reg}. Combining the bootstrap hypothesis \eqref{HypED_1}, \eqref{Hyp_<n>_<c>},  we refine the estimate \eqref{F_M_reg} as follows
\begin{align}\label{ddt_FM_lg}\qquad
\frac{d}{dt}&\FM [t]
\\ \leq&     \frac{ C  }{  {G}A^{1/3}}   \left(\frac{  \Phi(t)}{G}+\myr{\frac{G}{A^{1/6}}}+\frac{A^{1/3}}{A^{2/3}+t^2}+\exp\lf\{-\frac{\delta_{\mathcal Z}}{4A^{1/3}} t\rg\}\right) \FM  \\ 
& {+ \frac{C _{G} }{A^{1/3}}\ \frac{\FM  }{A^{1/2}\Phi ^{ 4M+4}}\ \bigg(1+ C_{ED}\FM[0] + \BB_{\lan n\ran;H^M}^2+ \BB_{ \lan   \cc\ran;H^{M+1}}^2  \bigg)}\\
& +\frac{C_G}{Q^2A^{1/3}}\myr{ \exp\lf\{-\frac{\delta_{\mathcal Z}}{4A^{1/3}} t\rg\}  }\ \underbrace{ \lf(Q^2\sum_{|i,j,k|=0}^ { M+1}  \|\Gamma_{y;0}^{ijk}\cc_{\text{in};\neq}\|_{L^{2}}^2\myr{\exp\lf\{-\frac{\delta_{\mathcal Z} }{2A^{1/3}} t\rg\}}\rg)}_{\leq \FM[t],\quad\text{by } \eqref{F_M}_{t_r=0} } \lf(C_{ED}\FM[0]+\BB_{\lan n\ran;H^M}^2\rg).
\end{align}
\myc{Based on the new definition of $\FM,$ I adjust the last line. }
Note that the choice of $\te$ \eqref{te_M_chc} guarantees that $\frac{1}{A^{1/2 }\Phi^{4M+4}}\bigg|_{t\leq A^{1/3+\te}}\leq \frac{1}{A^{1/4}}$. We take the parameters $A$ and $Q$ large enough \myc{compared to the bootstrap bounds in \eqref{HypED_1}, \eqref{Hyp_<n>_<c>} } such that the following estimate holds 
\begin{align}\label{ddt_FM_smp}\frac{d}{dt}\FM[t]\leq \frac{ C  }{  {G}A^{1/3}}   \left( {  \Phi(t)} +\frac{A^{1/3}}{A^{2/3}+t^2}+\exp\lf\{-\frac{\delta_{\mathcal Z}}{4A^{1/3}} t\rg\}+\frac{1}{A^{1/6}}\right) \FM[t]. 
\end{align}
Now we have that if $G$ is chosen larger than a constant depending only on $M, \, \|u_A\|_{L_t^\infty W^{M+3,\infty}},\ \delta_{\mathcal Z}^{-1}$, the following inequalities hold
\begin{align}  \label{FM_reg_est} \qquad\FM[t_\star+\tau]\leq&\ 2\ \FM[t_\star],\quad \forall \tau\in[0, \delta^{-1}A^{1/3}];\quad
\FM[t]\leq  \ 2\ \FM[0], \quad \forall t\in[0,T_\star)\subset[0, A^{1/3+\te}].
\end{align}
These are the regularity estimates that we are after.

Next we derive the decay estimate \eqref{Decay_est_NL}. To this end we apply Theorem \ref{thm:DD_est}. We observe that the functional $\FM$ can be decomposed as follows:
\begin{align} \label{pf_Decay_st1} \ \
\FM [t_\star+\tau]
\leq &2\DD[t_\star+\tau]+\bigg( 2\sum_{|i,j,k|=0}^{M}G^{2i+4}\Phi(t_\star+\tau)^{2j}\| S_{t_\star}^{t_\star+\tau} (\pa_x^i\Gamma_{y;t_\star}^j\pa_z^k n_\nq(t_\star))\|_2^2\\
&+ 2\sum_{|i,j,k|=0}^{M+1}\lf(A^{2/3}\mathbbm{1}_{j<M+1}+\mathbbm{1}_{j=M+1}\rg)G^{2i}\Phi(t_\star+\tau)^{2j+2}\| S_{t_\star}^{t_\star+\tau} (\Gamma_{y;t_\star}^{ijk} c_\nq^{t_r}(t_\star-t_r))\|_2^2\\
&+ \myr{Q^2}\sum_{|i,j,k|=0}^{M+1}\|\pa_x^i\Gamma_{y;t_r}^j\pa_z^k \cc_{\text{in};\nq}\|^2_{{2} }\myr{\exp\lf\{- \frac{ \delta_{\mathcal Z}}{2A^{1/3}}\tau\rg\}}\bigg)=:2\DD[t_\star+\tau]+\mathcal{R}[t_\star+\tau].
\end{align}
Here we recall that the reference time is $t_r=0$.  \myc{We need to check the $c_\nq^{t_r=T_h}(t_\star-t_r)$ in the proof of the gluing proposition \ref{Pro:2}. } 
We note that thanks to the $L^2$-enhanced dissipation estimate \eqref{ED_tdsh_intr}, as long as we set $\tau=\delta^{-1}A^{1/3}$, with $ \delta\in(0,\delta_{\mathcal Z})$ chosen small enough \eqref{Chc_del} \myc{(Check once)}, we have that 
\begin{align} \label{pf_Decay_st2}
\mathcal{R}[t_\star+\delta^{-1}A^{1/3}]\leq \frac{1}{16}\FM[t_\star]. 
\end{align}
Hence it is enough to estimate the $\DD[t_\star+\tau]{}$-term in  \eqref{pf_Decay_st1}. To this end, we recall the estimate \eqref{DD_est} from Theorem \ref{thm:DD_est}. Since the main portion of the terms in \eqref{DD_est} and \eqref{ddt_FM_lg} are similar, we choose the parameters as in \eqref{ddt_FM_smp} and apply the regularity estimate \eqref{FM_reg_est} to refine the estimate as follows
\begin{align}
 {\frac{d}{d\tau}}&\mathbb{D}_{G,Q}^{t_r;M} [t_\star+\tau]\\
\leq&  \frac{ C  }{  {G}A^{1/3}}   \left( {  \Phi} +\frac{G}{A^{1/6}}+\frac{A^{1/3}}{A^{2/3}+(t_\star+\tau)^2}+\exp\lf\{-\frac{\delta_{\mathcal Z}}{4A^{1/3}} (t_\star+\tau-\myr{t_r})\rg\}\right)\lf(\mathbb{D}_{G,Q}^{t_r;M}[t_\star+\tau]+ \FM[t_\star]\rg) .
\end{align}
\myc{
\begin{align}
\leq&   \frac{C}{  {G}}  \left(\frac{\Phi}{G A^{1/3}}+\frac{G}{A^{1/2}}+\frac{1}{A^{2/3}+(t_\star+\tau)^2}+\frac{1}{A^{1/3}}\exp\lf\{-\frac{\delta_{\mathcal Z}}{4A^{1/3}} (t_\star+\tau-\myr{t_r})\rg\}\right) \mathbb{D}_{G,Q}^{t_r;M}[t_\star+\tau]  \\ 
&+\frac{ C  }{  {G}A^{1/3}}   \left( {  \Phi(t)} +\frac{G}{A^{1/6}}+\frac{A^{1/3}}{A^{2/3}+(t_\star+\tau)^2}+\lf(1+\frac{1}{Q}\rg)\exp\lf\{-\frac{\delta_{\mathcal Z}}{4A^{1/3}} (t_\star+\tau-\myr{t_r})\rg\}\right) \FM[t_\star].
\end{align}}
Now we observe that as long as $G=G(M,\myr{\myc{?}\|u_A\|_{L_t^\infty W^{M+3,\infty}}},\delta_{\mathcal Z}^{-1})$ and $A$ are chosen large enough, we have 
\begin{align} \label{pf_Decay_st3}
&\DD[t_\star+\delta^{-1}A^{1/3}]\\
&\leq\int_0^{\delta^{-1} A^{1/3}} \exp\lf\{\frac{C}{  {G}} \int_{s}^{\delta^{-1} A^{1/3}} \left(\frac{\Phi(t_\star+\tau)}{ A^{1/3}}+\frac{G}{A^{1/2}}+\frac{1}{A^{2/3}+(t_\star+\tau)^2}+\frac{1}{A^{1/3}}\exp\lf\{-\frac{\delta_{\mathcal Z}}{4A^{1/3}} (t_\star+\tau-\myr{t_r})\rg\}\right)d\tau \rg\} \\
&\quad\quad\quad\quad\quad\times\frac{ C  }{  {G}A^{1/3}}   \left( {  \Phi(t_\star+s)} +\frac{G}{A^{1/6}}+\frac{A^{1/3}}{A^{2/3}+(t_\star+s)^2}+\exp\lf\{-\frac{\delta_{\mathcal Z}}{4A^{1/3}} (t_\star+s-\myr{t_r})\rg\}\right)ds \ \FM[t_\star]\\
 &\leq\frac{1}{16}\FM[t_\star].
\end{align}
Combining \eqref{pf_Decay_st1}, \eqref{pf_Decay_st2} and \eqref{pf_Decay_st3}, we have that
\begin{align}\label{FM_dec_est}
\FM [t_\star+\delta^{-1}A^{1/3}]\leq\frac{1}{2}\FM[t_\star]. 
\end{align}
Combining \eqref{FM_reg_est} and \eqref{FM_dec_est}, an argument as in \eqref{ED_argument} yields \eqref{ConED_1}. 

To conclude the proof of \eqref{ConED_1}, we further choose the $\ep$ \eqref{smll_c_nq} to be small enough compared to $Q$ and recall that $G=G(M,\|u_A\|_{L_t^\infty W^{M+3,\infty}},\delta_{\mathcal Z}^{-1})$ to obtain 
\begin{align}
\FM[0]=\sum_{|i,j,k|=0}^{M}G^{4+2i}\|\pa_x^i\pa_{y}^j\pa_z^k n_{\text{in};\nq}\|_{L^2}^2+ \myr{Q^2}\ep^2\leq C(M,\myr{\|u_A\|_{L_t^\infty W^{M+3,\infty}}},\delta_{\mathcal Z}^{-1})(\|n_{\text{in};\nq}\|^2_{H^M}+1). 
\end{align}

We conclude the proof by proving the estimates \eqref{Con_<n><c>}. Recall the main conclusions from Theorem \ref{thm:<n>_est},
\begin{align}\frac{d}{d t} \left(\sum_{|j,k|=0}^ {M}\|\pa_y^j\pa_z^k \nz (t)\|_{L ^2}^2\right) 
\leq &\frac{C}{A}\|\nz(t)\|_{H ^M}^2\left(\sup_{s\in[0, A^{1/3+\te}]}\|\nz(s)\|_{H ^{M}}^2 +\|\na \lan \cc_{\mathrm{in}}\ran\|_{H ^M}^2\right)\\
&+\frac{C}{A}\lf(\sum_{|i,j,k|=0}^{M+1}\|\Gamma_{y;t}^{ijk} \cc_\nq\|_{L^2}^2+ t^2\sum_{|i,j,k|=0}^{M} \|\Gamma_{y;t}^{(i+1)jk} \cc_\nq\|_{L^2}^2\rg) \left(\sum_{|i,j,k|=0}^{M}\|\Gamma_{y;t}^{ijk}  n_\nq\|_{L^{2}}^2\right).
\end{align}
Now we decompose the chemical density into two parts, i.e., $\cc_\nq(t)=c_\nq(t)+d_\nq(t), \quad d_\nq(t)=S_0^t \cc_{\text{in};\nq}$. Then we recall the definition of the functional $\FM$ \eqref{F_M} and estimate the above expression using the fact that $ {t^2}\leq C\Phi^{-1}{A^{2/3}}$ as follows:
\begin{align}\label{ddt_nz_HM}
\frac{d}{d t}&\left(\sum_{|j,k|=0}^ {M}\|\pa_y^j\pa_z^k \nz (t)\|_{L ^2}^2\right)\\
\leq &\frac{C}{A}\|\nz(t)\|_{H ^M}^2\left(\sup_{s\in[0, A^{1/3+\te}]}\|\nz(s)\|_{H ^{M}}^2 +\|\na \lan \cc_{\mathrm{in}}\ran\|_{H ^M}^2\right)\\
&+\frac{C}{A\Phi^{4M+4}}\lf(\sum_{|i,j,k|=0}^{M+1}G^{2i}\Phi^{2j+2}\|\Gamma_{y;t}^{ijk} c_\nq^{t_r}\|_{L^2}^2+\sum_{|i,j,k|=0}^{M} A^{2/3} G^{2i+2}\Phi^{2j+2}\|\Gamma_{y;t}^{(i+1)jk} c_\nq^{t_r}\|_{L^2}^2\rg) \FM\\
&+\frac{C}{A\Phi^{4M+2}}\lf(\sum_{|i,j,k|=0}^{M+1}G^{2i}\Phi^{2j}\|\Gamma_{y;t}^{ijk} d_\nq^{t_r}\|_{L^2}^2+ t^2\sum_{|i,j,k|=0}^{M}G^{2i}\Phi^{2j} \|\Gamma_{y;t}^{(i+1)jk} d_\nq^{t_r}\|_{L^2}^2\rg) \FM\\
\leq  &\frac{C}{A}\|\nz(t)\|_{H ^M}^2\left(\sup_{s\in[0, A^{1/3+\te}]}\|\nz(s)\|_{H ^{M}}^2 +\|\na \lan \cc_{\mathrm{in}}\ran\|_{H ^M}^2\right)+\frac{C}{A\Phi^{4M+4}}\lf(\FM\rg)^2\\
&+\frac{C}{A\Phi^{4M+2}}\lf(\sum_{|i,j,k|=0}^{M+1}G^{2i}\Phi^{2j}\|\Gamma_{y;t}^{ijk} d_\nq^{t_r}\|_{L^2}^2+ t^2\sum_{|i,j,k|=0}^{M}G^{2i}\Phi^{2j} \|\Gamma_{y;t}^{(i+1)jk} d_\nq^{t_r}\|_{L^2}^2\rg) \FM.
\end{align}  
Now we apply the linear enhanced dissipation estimate \eqref{Gld_Rg_ED} (given that $G$ is large enough), the smallness assumption \eqref{smll_c_nq} and the bootstrap assumptions \eqref{HypED_1}, \eqref{Hyp_<n>_<c>} to obtain that
\begin{align}\label{ddtnz_HM1}\ \ \ \
\frac{d}{dt}\left(\sum_{|j,k|=0}^ {M}\|\pa_y^j\pa_z^k \nz (t)\|_{L ^2}^2\right)\leq& \frac{C}{A}\mathcal{B}_{\lan n\ran;H^M}^2\lf(\mathcal{B}_{\lan n\ran;H^M}^2+\|\na \cc_{\text{in};\nq}\|_{H^M}^2\rg)+\frac{C\lf(\FM\rg)^2}{A\Phi^{4M+4}}\\
&+\frac{C}{A^{1/3}\Phi^{4M+3}} \lf(\sum_{|i,j,k|=0}^{ M+1}G^{2i}\|\Gamma_{y;0}^{ijk}\cc_{\text{in};\nq}\|_{2}^2\exp\lf\{-\delta_{\mathcal Z}\frac{t}{A^{1/3}}\rg\}\rg)\FM\\
\leq & {A^{-2/3}}+\frac{C(G,\delta_{\mathcal Z}^{-1})C_{ED}}{A^{1/3} } \ep^2\exp\lf\{-\frac{\delta_{\mathcal Z}}{2A^{1/3}}t\rg\}\FM[0].
\end{align}
Here in the last line, we have invoked the facts that  $M\leq \mathbb{M}$, $\|\Phi^{-1}\|_{L^\infty_t(0, A^{1/3+\te})}\leq C A^{3\te}$, and$\, 12\te(M+2)\leq \frac{1}{9}$  \eqref{te_M_chc}. Then we choose the $A$ to be large compared to the bootstrap bounds in \eqref{HypED_1}, \eqref{Hyp_<n>_<c>} and $\|\na \cc_{\text{in};\nq}\|_{H^{M}}$ to achieve the bound in \eqref{ddtnz_HM1}. Next we take the $A^{-1}$ and $\ep$ small enough compared to  the bootstrap bounds $C_{ED}, \, \FM[0]$ and other constants appeared in \eqref{ddtnz_HM1}, and integrate this differential inequality directly to obtain that 
\begin{align}
\sup_{t\in[0, A^{1/3+\te}]}\|\lan n(t)\ran\|_{H^M}^2\leq 2\|\lan n_{\text{in}}\ran\|_{H^M}^2+1. 
\end{align}
Since we choose the $\BB_{\lan n\ran; H^M}=2\|\lan n_{\text{in}}\ran\|_{H^M}+2$ in \eqref{BB<n><c>}, the first bootstrap conclusion in  \eqref{Con_<n><c>} follows. 

The second bootstrap conclusion in  \eqref{Con_<n><c>} is a direct consequence of the bootstrap hypothesis \eqref{Hyp_<n>_<c>}, the length of the time interval ($0\leq t\leq T_\star\leq A^{1/3+\te} \leq A^{1/2} $) and the estimate \eqref{Winklersgp2} \myc{Check!}
\begin{align}
\label{nacz_HM}
\sum_{|j,k|=0}^ {M}&\|\pa_y^j\pa_z^k\na \lan \cc\ran(t)\|_{2}^2
\\
\leq &
 2\sum_{|j,k|=0}^ {M}\lf\|\exp\lf\{\frac{t}{A}\de_{\Torus^2}\rg\}\pa_y^j\pa_z^k\na \lan \cc\ran(0)\rg\|_2^2+\frac{2}{A^2}
\sum_{|j,k|=0}^ {M}\lf\| \int_0^t \na \exp\lf\{\frac{t-s}{A}\de_{\Torus^2} \rg \}\pa_y^j\pa_z^k (\lan n\ran(s)-\overline{n}) ds\rg\|_2^2\\
\leq& 
2\sum_{|j,k|=0}^ {M}\|\pa_y^j\pa_z^k \na \lan \cc_{\text{in}}\ran\|_2^2+\frac{2}{A^2}
\sum_{|j,k|=0}^ {M}\lf(\int_0^{t }\lf\|\na\exp\lf\{\frac{t-s}{A}\de_{\Torus^2} \rg \}\rg\|_{L^2\rightarrow L^2}\sup_{s\in[0,T_\star)}\|\pa_y^j\pa_z^k(\lan n\ran(s)-\overline n)\|_{2} ds\rg)^2\\
\leq & 
2\sum_{|j,k|=0}^ {M}\|\pa_y^j\pa_z^k \na \lan \cc_{\text{in}}\ran\|_2^2+C \BB_{\lan n\ran; H^M}^2\lf(\frac{1}{A}\int_0^{t}\lf(\frac{t-s}{A}\rg)^{-1/2} ds\rg)^2\\
\leq & 
2\sum_{|j,k|=0}^ {M}\|\pa_y^j\pa_z^k \na \lan \cc_{\text{in}}\ran\|_2^2+C\BB_{\lan n\ran; H^M}^2\lf(\frac{t}{A}\rg)\leq  
2\|\na \lan \cc_{\text{in}}\ran\|_{H^M}^2+\frac{C\BB_{\lan n\ran; H^M}^2}{A^{1/2}}.
\end{align} \myc{$\leq \lf(2\|\na \lan \cc_{\text{in}}\ran\|_{H^M}+\frac{C\BB_{\lan n\ran; H^M}}{A^{1/4}}\rg)^2$}
Hence by taking $A$ large enough, we have obtained the second bound in \eqref{Con_<n><c>}.

\end{proof} 

\begin{proof}[Proof of Proposition \ref{Pro:2}\myc{ (!? Not sure.)}] The proof of the proposition is similar to the proof of Proposition \ref{Pro:main}. Hence we only highlight the main differences here.

First of all, we set the reference time $t_r=T_h=A^{1/3+\te/2}$ and consider the functional $\mathbb{L}_{G}^{M}:=\mathbb{F}_{G, A^{1/4}}^{T_h;M}$. Another adjustment in the proof is that the $c_\nq^{t_r}$ and $d_\nq^{t_r}$ are redefined with $t_r$ being $T_h=A^{1/3+\te/2}$.  Since we assume $M\leq \mathbb{M}$, the choice of $\zeta$  \eqref{te_M_chc} yields the following
\begin{align}\label{te_choice_M}
12\te(2+M)\leq \frac{1}{9}.
\end{align}  This estimate, together with the ``gluing'' condition \eqref{Glu_con},  yields that $\mathbb{L}_{G}^{M}[T_h]\leq C(\BB_1)$. \myc{$A^{1/4-1/3 }A^{3\te \times (M+2)}=A^{1/36-1/12}$} Now we apply the estimate \eqref{F_M_reg} to obtain that 
\begin{align}\label{ddt_FM_lg_2}
\frac{d}{d\tau}&\mathbb{L}_{G}^{M}[T_h+\tau]
\\ \leq&     \frac{ C  }{  {G}A^{1/3}}   \left(\frac{  \Phi(T_h+\tau)}{G}+\frac{G}{A^{1/2}}+\frac{A^{1/3}}{A^{2/3}+(T_h+\tau)^2}+\exp\lf\{-\frac{\delta_{\mathcal Z}}{4A^{1/3}} \tau\rg\}\right) \mathbb{L}_{G}^{M}    \\ 
& {+ \frac{C( {G} )}{A^{1/3}}\ \frac{\mathbb{L}_{G}^{M}  }{A^{1/2}\Phi ^{ 4M+4}}\ \bigg(1+ \mathbb{L}_{G}^{M}+\BB_{\lan n\ran^x;H^M}^2+ \BB_{ \lan   \cc\ran^x;H^{M+1}}^2  \bigg)}\\
& +\frac{C(G)}{A ^{5/6}}\frac{\exp\lf\{-\frac{\delta_{\mathcal Z}}{A^{1/3}} \tau\rg\}}{ \Phi^{ 4M+4}}\ \lf(A^{1/2}\sum_{|i,j,k|=0}^{M+1}\|\pa_x^i\Gamma_{y;T_h}^j\pa_z^k\cc_{\neq}(T_h)\|_{L^{2}}^2\rg) \lf(\mathbb{L}_{G}^{M} +\|\lan n\ran\|_{H^M}^2\rg)\\
 \leq&     \frac{ C  }{  {G}A^{1/3}}   \left(\frac{  \Phi(T_h+\tau)}{G}+\frac{G}{A^{1/6}}+\frac{A^{1/3}}{A^{2/3}+(T_h+\tau)^2}+\exp\lf\{-\frac{\delta_{\mathcal Z}}{4A^{1/3}} \tau\rg\}\right) \mathbb{L}_{G}^{M}   \\ 
& {+ \frac{C( {G} )}{A^{1/2}}\ {\mathbb{L}_{G}^{M}  }\ \bigg(1+\mathbb{L}_{G}^{M} + \BB_{\lan n\ran;H^M}^2+ \BB_{ \lan   \cc\ran;H^{M+1}}^2  \bigg)} +\frac{1}{A ^{1/2}}\exp\lf\{-\frac{\delta_{\mathcal Z}}{4A^{1/3}} \tau\rg\} \mathbb{L}_{G}^{M} 
 \lf(\mathbb{L}_{G}^{M} +\|\lan n\ran\|_{H^M}^2\rg).
\end{align}
\myc{Here do we need more on $\te?$ We have from \eqref{te_M_chc} that $12\te (M+2)\leq \frac{1}{9}, \, \mathbb{M}\geq M$. Hence we have that $A^{-1/3}\times \Phi^{-4M-4}\leq C_M A^{-1/3}A^{3\te(4M+4)}\leq C_M A^{-1/3+1/9}$. So we have extra inverse power of $A$ to absorb other factors. }Now a similar argument as the one in the previous proof of Proposition \ref{Pro:main} yields the regularity bound. The  estimate \eqref{FM_dec_est} can be derived in a similar fashion as before. When one derive the estimate for $\nz,$ the $d_\nq^{t_r=T_h}-$contribution in \eqref{ddt_nz_HM} needs to be estimated differently. {Combining the ``gluing'' condition \eqref{Glu_con}, the definition $d_\nq^{T_h}(\tau)=S_{T_h}^{T_h+\tau}[\cc_{\nq}(T_h)]$, and the linear enhanced dissipation \eqref{Gld_Rg_ED} yields that 
\begin{align}\frac{C}{A\Phi^{4M+2}}&\lf(\sum_{|i,j,k|=0}^{M+1}G^{2i}\Phi^{2j}\|\Gamma_{y;t}^{ijk} d_\nq^{t_r}\|_{L^2}^2+ t^2\sum_{|i,j,k|=0}^{M}G^{2i}\Phi^{2j} \|\Gamma_{y;t}^{(i+1)jk} d_\nq^{t_r}\|_{L^2}^2\rg)\mathbb{L}_{G}^{M} \\
\leq&\frac{CA^{3\te(4M+2)}}{A^{1/3}}\frac{1+(T_h+\tau)^2}{A^{2/3}}\sum_{|i,j,k|=0}^{M+1}G^{2i}\Phi(T_h+\tau)^{2j}\|\Gamma_{y;T_h+\tau}^{ijk} S_{T_h}^{T_h+\tau}[\cc_\nq(T_h)]\|_{L^2}^2\mathbb{L}_{G}^M\\
\leq&C \frac{A^{3\te(4M+3)}}{A^{1/3}}\sum_{|i,j,k|=0}^{M+1}G^{2i}\Phi(T_h)^{2j}\|\Gamma_{y;T_h}^{ijk}\cc_{\nq}(T_h)\|_{L^2}^2\exp\lf\{-\frac{\delta_{\mathcal Z}}{A^{1/3}}\tau \rg\}\mathbb{L}_{G}^M\\
\leq&\frac{C(G)A^{3\te(4M+5)}}{A }\lf(A^{2/3}\sum_{|i,j,k|=0}^{M+1}\Phi(T_h)^{2j+2}\|\Gamma_{y;T_h}^{ijk}\cc_{\nq}(T_h)\|_{L^2}^2\rg)\exp\lf\{-\frac{\delta_{\mathcal Z}}{A^{1/3}}\tau \rg\}\mathbb{L}_{G}^M\\
\leq &\frac{\mathcal{B}_1^2 }{A^{2/3}}\exp\lf\{-\frac{\delta_{\mathcal Z}}{A^{1/3}}\tau \rg\}\mathbb{L}_{G}^M.
\end{align}Here in the last line, we have used the choice of $\te$ \eqref{te_M_chc} and $M\leq \mathbb M$.   Hence, by picking $A$ large enough, we observe that the time integral contribution from the $d_\nq^{t_r}(\tau)=S_{t_r}^{t_r+\tau}\cc_{\nq}(t_r)$ is small. }The other arguments are similar and we omit the details.
\end{proof}
\begin{proof}[Proof of Proposition \ref{pro:prp_FM}] Here we observe that $t_r=0, \, t=0+\tau\in[0, A^{1/3+\te}]$, and $\|\Phi^{-4M-4}\|_{L^\infty_t(0, A^{1/3+\te})}\leq C A^{12\te (M+1)}\leq CA^{1/9}$ \eqref{te_M_chc}. 
Combining the estimates \eqref{F_M_reg}, \eqref{ddt_nz_HM}, and \eqref{ddtnz_HM1}, we have that
\begin{align}\label{dt_FM_<n>}\qquad
\frac{d}{dt} \lf(\|\lan n\ran\|_{H^M}^2+\FMZ\rg) 
\leq &\frac{C_1}{A}\|\nz\|_{H^M}^2\left(\sup_{s\in[0,t]}\|\nz(s)\|_{H^{M}}^2 +\|\na \lan \cc_{\mathrm{in}}\ran\|_{H^M}^2\right)\\
&+ {\frac{ C_2  }{  {G}A^{1/3}}   \left(\frac{  \Phi}{G}\myr{+\frac{G}{A^{1/6}}}+\frac{A^{1/3}}{A^{2/3}+t^2}+\exp\lf\{-\frac{\delta_{\mathcal Z}}{4A^{1/3}} t\rg\}\right) \FMZ}  +\frac{C_3({G}) }{A^{5/6}}  \frac{\left(\FMZ \right)^2   }{\Phi^{4M+4}}\  \\
&+ \frac{C_4( {G} )}{A^{1/3}}\ \frac{\FMZ  }{A^{1/2}\Phi^{ 4M+4}}\ \bigg(1+ \|\lan n\ran\|_{H^M}^2+\sup_{s\in[0, t]}\|\nz(s)\|_{H^{M}}^2 +\|\na \lan \cc_{\mathrm{in}}\ran\|_{H^M}^2  \bigg){}\\
&\myb{+\frac{C_5(\delta_{\mathcal Z}^{-1},G)}{A^{1/3}} \exp\lf\{-\frac{\delta_{\mathcal Z}}{2A^{1/3}}t \rg\} \| \cc_{\text{in};\neq}\|_{H^{M+1}}^2 \lf(\FMZ+\|\lan n\ran\|_{H^M}^2\rg)}.
\end{align}
\myc{Previously, \begin{align*}
\cdots\leq &\frac{C}{A}\|\nz(t)\|_{H^M}^2\left(\sup_{s\in[0, A^{1/3+\te}]}\|\nz(s)\|_{H^{M}}^2 +\|\na \lan \cc_{\mathrm{in}}\ran\|_{H^M}^2\right)+\frac{C}{A\Phi^{4M+4}}\lf(\FMZ\rg)^2\\
&+\frac{C}{A\Phi^{4M+2}}\lf(\sum_{|i,j,k|=0}^{M+1}G^{2i}\Phi^{2j}\|\Gamma_{y;t}^{ijk} d_\nq\|_{L^2}^2+ t^2\sum_{|i,j,k|=0}^{M}G^{2i}\Phi^{2j} \|\Gamma_{y;t}^{(i+1)jk} d_\nq\|_{L^2}^2\rg) \FMZ\\
&+ {\frac{ C  }{  {G}A^{1/3}}   \left(\frac{  \Phi(t)}{G}\myr{+\frac{G}{A^{1/6}}}+\frac{A^{1/3}}{A^{2/3}+t^2}+\exp\lf\{-\frac{\delta_{\mathcal Z}}{4A^{1/3}} t\rg\}\right) \FMZ[t]}  +\frac{C({G}) }{A^{5/6}}  \frac{\left(\FMZ[t] \right)^2   }{\Phi^{4M+4}}\  \\
&{+ \frac{C( {G} )}{A^{1/3}}\ \frac{\FMZ [t] }{A^{1/2}\Phi^{ 4M+4}}\ \bigg(1+ \|\lan n\ran\|_{H^M}^2+ \|\na\lan   \cc\ran\|_{ H^{M}}^2  \bigg)}\\
&{+\frac{C(G)}{A^{1/3}}\frac{\exp\lf\{-\frac{\delta_{\mathcal Z}}{A^{1/3}}t \rg\}}{ \Phi^{ 4M+4}}\ \lf(\sum_{|i,j,k|=0}^{M+1}G^{2i} \|\pa_x^i\pa_{y}^j\pa_z^k\cc_{\text{in},\neq}\|_{L^{2}}^2\rg) \lf(\FMZ+\|\lan n\ran\|_{H^M}^2\rg)} \leq...\
\end{align*}But by quoting \eqref{F_M_reg}, \eqref{ddt_nz_HM}, and \eqref{ddtnz_HM1}, we can avoid the $d_\nq. $}
{Here in the last estimate, we employ the chemical gradient $\na \cz$ estimate \eqref{na_c_0_est}. To prove the boundedness of the solution on $[0,A^{1/3+\te}]$, we apply a bootstrap argument. Assume that $[0,\wt T)\in[0,A^{1/3+\te}]$ is the largest time interval such that the following estimate holds 
\begin{align}\label{hyp}
\FMZ[t]+\|\lan n\ran(t)\|_{H^M}^2\leq 2(1+\FMZ[0]+\|\lan n_\text{in}\ran\|_{H^M}^2) \exp\lf\{1+\frac{2C_5(\delta_{\mathcal Z}^{-1},G)}{\delta_{\mathcal Z}}  \|\cc_{\text{in};\neq}\|_{H^{M+1}}^2\rg\}.
\end{align} 
We first take $A$ large compared to the right hand side of \eqref{hyp} and various constants in \eqref{dt_FM_<n>} to obtain that
\begin{align}
\frac{d}{dt}\lf(\FMZ+  \| \lan n\ran\|_{H^M}^2\rg)
 \leq&\lf(\FMZ+  \| \lan n\ran\|_{H^M}^2\rg)\bigg(\frac{\sum_{\ell=1}^4C_\ell 
 }{A^{1/2}}+\frac{C_2}{GA^{1/3}} \left(\frac{  \Phi(t)}{G}+\frac{A^{1/3}}{A^{2/3}+t^2}+\exp\lf\{-\frac{\delta_{\mathcal Z}}{4A^{1/3}} t\rg\}\right) \\
&\quad\quad\quad\quad\quad\quad\quad\quad\quad\quad+\frac{C_5}{A^{1/3}} \exp\lf\{-\frac{\delta_{\mathcal Z}}{2A^{1/3}}t \rg\} \|\cc_{\text{in};\neq}\|_{H^{M+1}}^2\bigg).
\end{align}
By integrating on $[0, \wt T]$, and taking $G\geq G_\mathbb{H}(M, \|u_A\|_{L_t^\infty W_y^{M+3,\infty}}), \, A\geq A_\mathbb{H}(M, \|u_A\|_{L_t^\infty W_y^{M+3}}, G,\|n_{\mathrm{in}}\|_{H^M},\|\cc_{\mathrm{in}}\|_{H^{M+1}})$ large enough, we have that
\begin{align}
\FMZ[t]+  \| \lan n\ran(t)\|_{H^M}^2\leq (\FMZ[0]+\|\lan n_\text{in}\ran\|_{H^M}^2) \exp\lf\{1+\frac{2C_5}{\delta_{\mathcal Z}}  \|\cc_{\text{in};\neq}\|_{H^{M+1}}^2\rg\}.
\end{align}
Since this estimate is stronger than \eqref{hyp} and $G$ can be chosen depending only on $M, \|u_A\|_{L_t^\infty W_y^{M+3,\infty}},\delta_{\mathcal Z}^{-1}$, standard bootstrap argument yields  the following estimate on $[0, A^{1/3+\te}]$
\begin{align}
\FMZ[t]+  \| \lan n\ran(t)\|_{H^M}^2\leq C(G)(\|n_\text{in}\|_{H^M}^2+\|\cc_{\mathrm{in};\nq}\|_{H^{M+1}}^2) \exp\lf\{1+{C(M, \|u_A\|_{L_t^\infty W_y^{M+3,\infty}},\delta_{\mathcal Z}^{-1})}  \|\cc_{\text{in};\neq}\|_{H^{M+1}}^2\rg\}.
\end{align} Combining this estimate with gradient estimate \eqref{na_c_0_est} and the argument in \eqref{nacz_HM} yields \eqref{est_pro}.
}

\end{proof}\myc{
\subsection{Previous}. }
\ifx
We choose $G$ large compared to $\|u_y\|_{L^\infty_t W^{M+1,\infty}},\ M,\ \delta^{-1}_d$, then $\ep_I$ small compared to $G,\ \|u_y\|_{L^\infty_t W^{M+1,\infty}},\ M,\ \delta^{-1}_d$ and $\|\lan n\ran \|_{H^{M}}\leq 4C_{\lan n\ran ; H^M}$. Finally we pick $A^{-1}, \|\cc_{\text{in};\nq}\|_{H^{M+1}}$ small compared to all the previously chosen/specified constants, and the bootstrap constants in \eqref{Hypotheses}. 
Direct calculation yields that 
\begin{align}\mathcal D_M (t_\star+\tau)\leq \frac{1}{16 }F_{M}(t_\star),\quad \tau\in[0,\delta_{\mathcal Z}^{-1}A^{1/3}].
\end{align} Hence we have 
\begin{align}
F_M(&t_\star+\delta_{\mathcal Z}^{-1}A^{1/3})\\
\leq& \mathcal D_M(t_\star+\delta_{\mathcal Z}^{-1}A^{1/3})+\sum_{|i,j,k|=0}^{M}G^{4+2i}\Phi^{2j} \| S_{t_\star}^{t_\star+\delta_{\mathcal Z}^{-1}A^{1/3}}\palt n_\nq\|_2^2\\
&+\sum_{ {|i,j,k|\leq M+1}}G^{2i}(\mathbbm{1}_{j\neq M+1}A^{2/3}+\mathbbm{1}_{j= M+1})\Phi^{2j+2}\| S_{t_\star}^{t_\star+\delta_{\mathcal Z}^{-1}A^{1/3}}\palt c_\nq\|_2^2+\frac{1}{\ep_I^2}\|\cc_{\mathrm{in};\nq}\|^2_{H^{M+1} }e^{-\delta_\dagger t_\star/A^{1/3}} e^{-\delta_{\mathcal Z}^{-1}\delta_\dagger}\\
\leq &\frac{1}{2}F_M(t_\star).
\end{align}
This concludes the proof of \eqref{Decay_est_NL}. 
\ifx
For the deviation estimate, we can carry out similar argument as in the $H^1$ case. Specifically speaking, the deviation follows the estimate 
\begin{align}
\frac{d}{dt}&\frac{1}{2}G^{4+2i} \varphi^j\|\pa_x^i(\pa_y^t)^j\pa_z^k n_\nq-S_{t_\star}^{t_\star+\tau}\pa^t_{ijk} n_\nq\|_2^2\\
=&-\frac{G^{4+2i} }{A}\varphi^j\|\na (\pa_x^i(\pa_y^t)^j\pa_z^k n_\nq-S_{t_\star}^{t_\star+\tau}\pa^t_{ijk} n_\nq)\|_2^2 +\frac{G^{4+2i}j\varphi^{j-1}\varphi'}{2} \|\na (\pa_x^i(\pa_y^t)^j\pa_z^k n_\nq-S_{t_\star}^{t_\star+\tau}\pa^t_{ijk} n_\nq)\|_2^2 \\
&-\frac{1}{A}G^{4+2i}\varphi^j \int (\pa_x^i(\pa_y^t)^j\pa_z^k n_\nq-S_{t_\star}^{t_\star+\tau}\pa^t_{ijk} n_\nq) \pa_x^i(\pa_y^t)^j\pa_z^k \na \cdot( (\na c_\nq+\na d_\nq) (\lan n\ran+n_\nq)+\na \lan \cc\ran n_\nq)dV\\
\leq & -\frac{1}{A}G^{4+2i} \varphi^j\|\na (\pa_x^i(\pa_y^t)^j\pa_z^k n_{\nq}-S_{t_\star}^{t_\star+\tau}\pa^t_{ijk} n_\nq)\|_2^2+\frac{G^{4+2i} \varphi^j }{A}\int (\palt n_\nq-S_{t_\star}^{t_\star+\tau}\pa^t_{ijk} n_\nq) [(\pa_y^t)^j,\pa_{yy}]\pa_x^i\pa_z^k n_\nq dV\\
&+\frac{G^{4+2i} }{A}\varphi^j\int \na_{x,y,z}(\palt n_{\neq}-S_{t_\star}^{t_\star+\tau}\pa^t_{ijk} n_\nq)\cdot \palt( n\na_{x,y,z}c_{\neq}+n_\nq \na_{y,z}\lan \cc\ran+n\na_{x,y,z}d_{\neq}) dV\\
&-\frac{G^{4+2i} }{A}\varphi^j\int (\palt n_\nq-S_{t_\star}^{t_\star+\tau}\pa^t_{ijk} n_\nq)  [(\pa_y^t)^j,\pa_y]\pa_x^i\pa_z^k(n\na c_\nq +n\na d_\nq+n_\nq\na\lan \cc\ran)dV\\
=:&-\frac{1}{A}G^{4+2i} \varphi^j\|\na(\palt n_\nq-S_{t_\star}^{t_\star+\tau}\pa^t_{ijk} n_\nq)\|_2^2+\sum_{i=1}^3T_i'. \label{palt_n-eta_T_123_vector_field}
\end{align} 
\ifx
We write the first term $T_1'$ as
\begin{align}
T_1'=&\frac{G^{4+2i}\varphi^j }{A}\int (\palt n_\nq-S_{t_\star}^{t_\star+\tau}\pa^t_{ijk} n_\nq) \sum_{\ell=0}^{j-1}\left(\begin{array}{rr}j\\ \ell\end{array}\right)(-2B^{(j-\ell+1)}\pa_x\pa_y^t+\pa_y^{(j-\ell)}(B^{(1)})^2\pa_{xx}-B^{(j-\ell+2)}\pa_x)(\pa_y^t)^\ell \pa_x^i\pa_z^k n_\nq dV\\
=:&T_{11}'+T_{12}'+T_{13}'.\label{T'_1123}
\end{align}
{\bf Here we need to distinguish between the $\ell=j-1$ and $\ell\leq j-2$ case. In the $\ell=j-1$ case, the $T_{11}'+T_{12}'$ has the following simplification
\begin{align}
T_{11}'+T_{12}'=&-\frac{G^{4+2i}\varphi^j}{A}\int (\palt n_\nq-S_{t_\star}^{t_\star+\tau}\pa^t_{ijk} n_\nq)j2B^{(2)}\pa_x\pa_y(\pa_{y}^t)^{j-1}\pa_x^i\pa^k_zn_\nq dV\\
\leq&\frac{G^{4+2i}\varphi^{j}}{12 A}\|\pa_y (\palt n_\nq-S_{t_\star}^{t_\star+\tau}\pa^t_{ijk} n_\nq)\|_2^2+\frac{CG^{4+2i+2}\varphi^{j-1}(\varphi t^2)}{AG^2}\|\pa_x^{i+1}(\pa_y^t)^{j-1}\pa_z^k n_\nq\|_2^2\\
\leq&\frac{G^{4+2i}\varphi^j}{12A}\|\na(\palt n_\nq-S_{t_\star}^{t_\star+\tau}\pa^t_{ijk} n_\nq)\|_2^2+\frac{CG^{4+2i+2}\varphi^{j-1}}{A^{1/3}G^2}\frac{t^2/A^{2/3}}{1+t^2/A^{2/3}}\|\pa_x^{i+1}(\pa_y^t)^{j-1}\pa_z^kn_\nq\|_2^2.
\end{align}
}
\ifx 
The $T_{11}'$ term need to be treated as follows
\begin{align}
T_{11}'=&\frac{G^{4+2i}\varphi^j}{A}\int (\palt n_\nq-\eta_\nq) j 2B^{(2)}\pa_x(\pa_y+\int_0^tu^{(1)}ds\pa_x )(\pa_y^t)^{j-1}\pa_x^i\pa_z^kn_\nq dV\\
&+\sum_{\ell=0}^{j-2}\left(\begin{array}{rr}j\\ \ell\end{array}\right)(-2B^{(j-\ell+1)})\pa_x^{i+1}(\pa_y^t)^\ell\pa_z^k  n_\nq dV\\
=& -\frac{G^{4+2i}\varphi^j}{A}j\int \pa_y(\palt n_\nq-\eta_\nq)(2B^{(2)}\pa_x^{i+1}(\pa_y^t)^{j-1}\pa_z^k n_\nq)dV\\
& +\frac{G^{4+2i}\varphi^j}{A}j\int (\palt n_\nq-\eta_\nq)(-2B^{(3)}\pa_x^{i+1}(\pa_y^t)^{j-1}\pa_z^k n_\nq+\pa_y(B^{(1)})^2\pa_{x}\pa_x^{i+1}(\pa_y^t)^{j-1}\pa_z^k n_\nq)dV\\
&+\sum_{\ell=0}^{j-2}\left(\begin{array}{rr}j\\ \ell\end{array}\right)(-2B^{(j-\ell+1)})\pa_x^{i+1}(\pa_y^t)^\ell\pa_z^k  n_\nq dV\\
\leq&\frac{1}{BA}G^{4+2i}\varphi^j\|\na(\palt n_\nq-\eta_\nq)\|_2^2+\frac{C(j,u_y)G^{4+2i}\varphi^{j-1}}{A}(\varphi t^2\|\pa_x^{i+1}(\pa_y^t)^{j-1}\pa_z^k n_\nq(0)\|_2^2)\\ 
&+\frac{1}{AG^2}\varphi^{j-1}G^{4+2i+2}C(\varphi t^2)\|\palt n_\nq-\eta_\nq\|_2\|\pa_x^{i+2}(\pa_y^t)^{j-1}\pa_z^kn_\nq(0)\|_2e^{-\delta t/A^{1/3}}...
\end{align}
But $\frac{1}{A}\int_{T}^\infty t^2 e^{-\delta t/A^{1/3}}dt\leq C e^{-\delta T/(2A^{1/3})}$. We might tune the number in the functional to get decay? But before the $\pa_y^t$, we need two copies of $\pa_x?$ 
\fi
For the $\ell\leq j-2$ case, we invoke the estimate already derived before. Here we first recall the Gagliardo-Nirenberg inequality
\begin{align}
\|\palt f\|_{L^4}\leq C\|\palt f\|_2^{1/4}(\|\pa_x\palt f\|_2+\|\pa_y^t\palt f\|_2+\|\pa_z\palt f\|_2 )^{3/4}.
\end{align}
Here, in order to rigorously justify the inequality, we can consider the new coordinate we use before. In the new coordinate system, the $L^p$-norms are preserved. But the $v$-derivative ($\pa_v$) of the system is the profile derivative $\pa_y^t$ here. The Gagliardo-Nirenberg inequality in the new coordinate system is the inequality above. 
There the $L^p$ norms are preserved, and $\pa_y^t$ is the $\pa_v$. 
Now we estimate the mixed term as follows
\begin{align}
T_{11}'=&-2\frac{G^4G^{2i}\varphi^{j}(t)}{A}\sum_{\ell=0}^{j-2}C^j_\ell B^{(j-\ell+1)}\int (\palt n_\nq-S_{t_\star}^{t_\star+\tau}\pa^t_{ijk} n_\nq)\pa_x^{i+1}(\pa_y^t)^{\ell+1}\pa_z^k n_\nq dV\\
\leq &\frac{C(u_y)G^{4+2i}\varphi^j}{A}t\sum_{\ell=0}^{j-2}\|\palt n_{\neq}-S_{t_\star}^{t_\star+\tau}\pa^t_{ijk} n_\nq\|_2\|\pa_x^{i+1}(\pa_y^t)^{\ell+1}\pa_z^kn_\nq\|_2\\
=&\frac{C(u_y)\varphi^{1/2}}{GA}t(\varphi^{j/2}G^{2+i}\|\palt n_\nq-S_{t_\star}^{t_\star+\tau}\pa^t_{ijk} n_\nq\|_2)(G^{2+i+1}\varphi^{j/2-1/2}\|\pa_x^{i+1}(\pa_y^t)^{j-1}\pa_z^kn_\nq\|_2)\\
&+\frac{C(u_y)G^{2+i}\varphi^{j/2}}{A}t(\varphi^{j/2}G^{2+i}\|\palt n_\nq-S_{t_\star}^{t_\star+\tau}\pa^t_{ijk} n_\nq\|_2)\sum_{i'+j'+k'\leq |i,j,k|-1}\|\pa_x^{i'}(\pa_y^t)^{j'}\pa_z^{k'}n_{\neq}(0)\|_2e^{-\delta t/A^{1/3}}.
\end{align}

The more dangerous term is the $T_{12}'$, it has the $t^2$, we estimate as follows ({\bf\color{red} From this term, we see that we might need the $\varphi$?})
\begin{align}
T_{12}'\leq& \frac{C\varphi^{1}t^2}{G^2A}(G^{2+i}\varphi^{j/2}(t)\|\palt n_\nq-S_{t_\star}^{t_\star+\tau}\pa^t_{ijk} n_\nq\|_2)(G^{2+i+2}\varphi^{j/2-1}\|\pa_x^{i+2}(\pa_y^t)^{j-2}\pa_z^kn_\nq\|_2+...)\\
\leq&\frac{C t^2}{G^2A(1+t^2A^{-2/3})}(G^{2+i}\varphi^{j/2}(t)\|\palt n_\nq-S_{t_\star}^{t_\star+\tau}\pa^t_{ijk} n_\nq\|_2)(G^{2+i+2}\varphi^{j/2-1}\|\pa_x^{i+2}(\pa_y^t)^{j-2}\pa_z^kn_\nq\|_2+...)\\
\leq & \frac{C}{G^2A^{1/3}}(G^{2+i}\varphi^{j/2}(t)\|\palt n_\nq-S_{t_\star}^{t_\star+\tau}\pa^t_{ijk} n_\nq\|_2)(G^{2+i+2}\varphi^{j/2-1}\|\pa_x^{i+2}(\pa_y^t)^{j-2}\pa_z^kn_\nq\|_2)....
\end{align}
 \fi
 Here we distinguish between two cases, i.e., $j=M$ and $j<M$. In the $j=M$ case, we estimate the top order contribution as follows:
\begin{align}
\frac{G^{4}}{A}&\int \pa_y( (\pa_y^t)^M n_\nq-S_{t_\star}^{t_\star+\tau}(\pa_y^t)^M n_\nq)((\pa_y^t)^M\pa_y c_\nq) n_\nq dV-\frac{G^{4+2(i+k)}}{A}\int((\pa_y^t)^Mn_\nq-S_{t_\star}^{t_\star+\tau}(\pa_y^t)^M n_\nq)([\pa_y^M,\pa_y]\pa_y c_{\neq}) n_\nq dV\\
=&\frac{G^{4}}{A}\int \pa_y ((\pa_y^t)^M n_\nq-S_{t_\star}^{t_\star+\tau}(\pa_y^t)^M n_\nq)\left([(\pa_y^t)^M,\pa_y]c_\nq +(\pa_y^t)^{M+1}c_\nq-B^{(1)}\pa_x( \pa_y^t)^M c_\nq\right) n_\nq dV\\
&+\frac{G^{4}}{A}\int((\pa_y^t)^Mn_\nq-S_{t_\star}^{t_\star+\tau}(\pa_y^t)^M n_\nq)([[(\pa_y^t)^M,\pa_y],\pa_y] c_{\neq}+\pa_y^t[(\pa_y^t)^M,\pa_y]c_\nq-B^{(1)}\pa_x[(\pa_y^t)^M, \pa_y]c_\nq)) n_\nq dV\\
=: &\sum_{i=1}^6 T_{2i}.
\end{align}
The main term is the $T_{22}$. For other term, we have the $A^{1/3}$ weight in front of $\|\pa_x(\pa_y^t)^jc_\nq\|_2$, hence the $B^{(i)}$ growth is OK. The term $T_{22}$ will be controlled as follows:
\begin{align}
T_{22}\leq &\frac{G^{4}}{A}\|\na((\pa_y^t)^Mn_\nq-S_{t_\star}^{t_\star+\tau}(\pa_y^t)^M n_\nq)\|_2\|{(\pa_y^t)}^{M+1}c_\nq\|_2\|n_\nq\|_\infty\\
\leq&\frac{G^{4}}{8A}\|\na((\pa_y^t)^Mn_\nq-S_{t_\star}^{t_\star+\tau}(\pa_y^t)^M n_\nq)\|_2^2+\frac{CG^4}{A}(1+\frac{t^2}{A^{2/3}})\varphi(t)\|(\pa_y^t)^{M+1}c_\nq\|_2^2\|n_\nq\|_\infty^2.
\end{align}
Some other terms are estimated in the following way ( We should already have some information about the ED in lower regularity?)
\begin{align}
T_{2.}\leq &\frac{ C(u_y)G^{4}}{A}\|(\pa_y^t)^Mn_\nq-S_{t_\star}^{t_\star+\tau}(\pa_y^t)^M n_\nq\|_2\frac{t^2}{A^{1/3}}(A^{1/3}\|\pa_x(\pa_y^t)^{M...}c_\nq\|_2)\|n_\nq\|_\infty
\end{align} 
For the $T_2'$ term, I first estimate the $ n_\nq\na c_\nq$ term as follows
\begin{align}
T_2'&\\
=&\frac{G^{2i}\varphi^j}{A}\sum_{(i',j',k')\leq (i,j,k)}C_{i',j',k'}^{i,j,k}
\int \na(\palt n_\nq-S_{t_\star}^{t_\star+\tau}\palt n_\nq)\cdot \left(\na^t(\pa_{i',j',k'}^t)c_{\nq}-\left(\begin{array}{cc}0\\ B^{(1)}\pa_x\\ 0\end{array}\right)\pa_{i',j',k'}^t c_\nq \right)\pa_{i-i',j-j',k-k'}^t n_\nq dV\\
\leq&\frac{CG^{2i}\varphi^j}{A}\|\na(\palt n_\nq-S_{t_\star}^{t_\star+\tau}\palt n_\nq)\|_2(\|\na^t \palt c_{\neq}\|_2+\frac{t}{A^{1/3}}A^{1/3}\|\pa_x\palt c_\nq\|_2)\|n_\nq\|_\infty\\
&+\frac{CG^{2i}\varphi^j}{A}\|\na(\palt n_\nq-S_{t_\star}^{t_\star+\tau}\palt n_\nq)\|_2\|\na c_\nq\|_\infty\|\palt n_\nq\|_2\\
&+\frac{CG^{2i}\varphi^j}{A}\sum_{1\leq|i',j',k'|\leq |i,j,k|-1\atop (i',j',k')\leq (i,j,k)}
\|\na(\palt n_\nq-S_{t_\star}^{t_\star+\tau}\palt n_\nq)\|_2\|\na^t \pa_{i',j',k'}^t c_\nq\|_2^{1/4}\|(\na^t)^2\pa_{i',j',k'}^t c_\nq\|_2^{3/4} \\
&\times\|\pa^t_{i-i',j-j',k-k'}n_\nq\|_2^{1/4}\|\na^t\pa_{i-i',j-j',k-k'}^t n_\nq\|_2^{3/4} \\
&+\frac{CG^{2i}\varphi^j}{A}\sum_{\substack{1\leq |i',j',k'|\leq |i,j,k|-1\\ (i',j',k')\leq(i,j,k)}}\|\na(\palt n_\nq-S_{t_\star}^{t_\star+\tau}\palt n_\nq)\|_2 \frac{t}{A^{1/3}}(A^{1/3}\|\pa_{i'+1,j',k'}^t c_\nq\|_2^{1/4}\|\na^t \pa_{i'+1,j',k'}^tc_\nq\|_2^{3/4})\\
&\times \|\pa^t_{i-i',j-j',k-k'}n_\nq\|_2^{1/4}\|\na^t\pa_{i-i',j-j',k-k'}^t n_\nq\|_2^{3/4}\\
\leq &\frac{G^{2i}\varphi^{j}}{A 12}\|\na(\palt n_\nq-S_{t_\star}^{t_\star+\tau}\palt n_\nq)\|_2^2+(?)
\end{align}
Here I think we can use the estimate $\|fg\|_{\dot H^s}\leq C\|f\|_{H^s}\|g\|_\infty+C\|g\|_{H^s}\|f\|_\infty $ (here the norm should be the $\pa_x,\pa_y^t,\pa_z$ norm). Because for $\na f,\na g$, their average is zero, so we have $\|\na f\|_{H^s}\leq C\|\na f\|_{\dot H^s}$. ({\bf check whether the constant is the same as/similar to  $\Torus^3$.} The $G,\varphi^j$ power OK? The $H^s$ norm might contain a lot of combinations?) The main problem might comes from the $(\na_t)^2$ case. 
{\color{blue}There, we might need to use the nonlinear blow-up type ODE to estimate the regularity of the solutions on the time interval $[t_\star,t_\star+\delta^{-1} A^{1/3}]$, then use the regularity bound to show that the deviation is tinny. To estimate that $|i,j,k|=2$ and $i'+j'+k'=1$ case, we do the following
\begin{align}
&T_2'\\
\leq &....\frac{G^{2i+4}\varphi^j}{A}\|\na (\palt n_\nq)\|_2(\|\na^t\pa^t_{i',j',k'} c_\nq\|_4+Ct\|\pa_{i'+1,j',k'}^t c_\nq\|_4)\|\pa_{i-i',j-j',k-k'}^t n_\nq\|_4\\
\leq&...\frac{CG^{2i+4}\varphi^j}{A}\|\na (\palt n_\nq)\|_2(\|(\na^t)^2\pa^t_{i',j',k'} c_\nq\|_2^{3/4}\|\na^t\pa^t_{i',j',k'}c_\nq\|_2^{1/4}+C\frac{A^{1/3}t}{A^{1/3}}\|\na^t\pa_{i'+1,j',k'}^t c_\nq\|_2^{3/4}\|\pa_{i'+1,j',k'}^t c_\nq\|_2^{1/4})\\
&\times\|\na^t\pa_{i-i',j-j',k-k'}^t n_\nq\|_2^{3/4}\|\pa_{i-i',j-j',k-k'}n_\nq\|_2^{1/4}\\
\leq&...\frac{G^{2i+4}\varphi^j}{12A}\|\na (\palt n_\nq)\|_2+\frac{CG^{4+2i}\varphi^j}{A}\|(\na^t)^2\pa^t_{i',j',k'}c_\nq\|_2^{3/2}\|\na^t\pa^t_{i-i',j-j',k-k'}n_\nq\|_2^{3/2}C_{2;\infty}^3 
\end{align}\begin{align}+\frac{CG^{4+2i}\varphi^j}{A}\frac{t^2}{A^{2/3}}(A^{1/3}\|\na^t\pa_{i'+1,j',k'}^tc_\nq\|_2)^{3/2}\|\na^t \pa_{i-i',j-j',k-k'} ^t n_\nq\|_2^{3/2} F_{1}[n_\nq,c_\nq]^{3/2}(t). 
\end{align}
Now we get regularity estimate. \textcolor{red}{Maybe similar argument with $A$ large yields $\|\palt n_\nq-S_{t_\star}^{t_\star+\tau}n_\nq\|_2\ll F_M(t_\star)$. }
 For the $\|\palt n_\nq-S_{t_\star}^{t_\star+\tau}n_\nq\|_2^2$, we should use the $A$ to control. } Or we need to estimate the $L^4$ norm? 
Next we comment on the estimate of  the $\varphi\|(\pa_y^t)^{M+1}c_\nq\|_2^2$ and $A^{2/3}\|\pa_x(\pa_y^t)^{M} c_\nq\|_2^2$. For the first term, we estimate it as follows
\begin{align}
\frac{d}{dt}&\frac{1}{2}\varphi^{M+1}\|(\pa_y^t)^{M+1}c_\nq-S_{t_\star}^{t_\star+\tau}(\pa_y^t)^{M+1}c_\nq\|_2^2\\
=&\frac{(M+1){\varphi^{M}\varphi'}}{2}\|(\pa_y^t)^{M+1}c_\nq-S_{t_\star}^{t_\star+\tau}(\pa_y^t)^{M+1}c_\nq\|_2^2-\frac{\varphi^{M+1}}{A}\|\na((\pa_y^t)^{M+1}c_\nq-S_{t_\star}^{t_\star+\tau}(\pa_y^t)^{M+1}c_\nq)\|_2^2\\
&+\frac{1}{A}\varphi^{M+1}\int ((\pa_y^t)^{M+1}c_\nq-S_{t_\star}^{t_\star+\tau}(\pa_y^t)^{M+1}c_\nq) [(\pa_y^t)^{M+1},\pa_{yy}]c_\nq d V\\
&+\frac{\varphi^{M+1}}{A}\int ((\pa_y^t)^{M+1} c_\nq-S_{t_\star}^{t_\star+\tau}(\pa_y^t)^{M+1}c_\nq)\pa_y^t(\pa_y^t)^M n_\nq dV\\
\leq&...+\frac{\varphi^{M+1}}{10A}\|\na ((\pa_y^t)^{M+1} c_\nq-S_{t_\star}^{t_\star+\tau}(\pa_y^t)^{M+1}c_\nq)\|_2^2+\frac{1}{A}\varphi^M\frac{t^2}{1+t^2/A^{2/3}}C\|(\pa_y^t)^M n_\nq\|_2^2.
\end{align} 
\begin{align}
\frac{d}{dt}&A^{2/3}G^{2}\varphi^M\|\pa_x(\pa_y^t)^M c_\nq-S_{t_\star}^{t_\star+\tau}\pa_x(\pa_y^t)^{M}c_\nq\|_2^2\\
\leq&-\frac{G^{2}\varphi^M}{A^{1/3}}\|\na(\pa_x(\pa_y^t)^Mc_\nq-S_{t_\star}^{t_\star+\tau}\pa_x(\pa_y^t)^{M}c_\nq)\|_2^2+\frac{G^{2}\varphi^M}{A^{1/3}}\int (\pa_x(\pa_y^t)^Mc_\nq-S_{t_\star}^{t_\star+\tau}\pa_x(\pa_y^t)^{M}c_\nq)\pa_x (\pa_y^t)^M n_\nq dV\\
\leq&-\frac{1}{8A^{1/3}}G^{2}\varphi^M\|\na(\pa_x(\pa_y^t)^Mc_\nq-S_{t_\star}^{t_\star+\tau}\pa_x(\pa_y^t)^{M}c_\nq)\|_2^2+\frac{1}{G^2A^{1/3}}\varphi^MG^{4}\|(\pa_y^t)^Mn_\nq\|_2^2. 
\end{align}

{\bf $G$ should be independent of data? We might use $A, \ep$ to control nonlinear term.} Now we estimate the general terms. They should be similar to previous estimates?

For the $A^{2/3}G^2\|\pa_x\palt c_\nq\|_2^2$, we can use the previous argument to estimate\begin{align}\frac{A^{2/3}G^4}{AG^2}\|\palt n_\nq\|_2^ 2.
\end{align}
 For the $\|\pa_y\palt c_\nq\|_2^2$ term, there will be an extra term 
\begin{align*}
\|u_y\|_\infty\|\pa_x \palt c_\nq\|_2\|\pa_y \palt c_\nq\|_2\leq \frac{1}{A^{1/3}G}\|u_y\|_\infty A^{1/3}G\|\pa_x^{}\palt c_\nq\|_2\|\pa_y\palt c_\nq\|_2 .
\end{align*}\fi
  
{\color{blue}
\noindent
\textbf{The proof of Proposition \ref{Pro:2}.} We set the time $t_0$ to be $t_0:=A^{1/3+\te/2}$. At this time instance, the estimate \eqref{Glu_con} holds. Moreover, we decompose the chemical density into the following two components
\begin{align}
\cc_\nq(t_0+\tau)= S_{t_0,t_0+\tau} \cc_\nq(t_0)+c_\nq(t_0+\tau)=:d_\nq(t_0+\tau)+c_\nq(t_0+\tau).
\end{align}
Thanks to the constraint \eqref{Glu_con} and the enhanced dissipation estimate \eqref{Gld_Rg_ED}, we have that 
\begin{align}
\sum_{|i,j,k|=0}^{M}G^{2i}\Phi^{2j}(t_0+\tau)\|\pa_x^i\Gamma_{y;t_0+\tau}^j\pa_z^kd_\nq(t_0+\tau)\|_2^2\leq \frac{C}{A^{2/3}}e^{-2\delta_{\mathcal Z}\frac{\tau}{A^{1/3}}}.\label{d_nq_est_t_0+tau}
\end{align}
Next we comment on the necessary adjustments to the arguments in the proof of Proposition \ref{Pro:main}. The main adjustments come from the estimate of the $T_{24}^{NL;R}$-term \eqref{T_24R} and the $T_{34}^{NL;R}$-term \eqref{T_34R} ({\bf Also need to double check the regularity part}). For the $T_{24}^{NL;R}$-term, we estimate it  as follows. First, by the estimate \eqref{Gld_Reg_grd}, the smallness condition \eqref{d_nq_est_t_0+tau}, and the enhanced dissipation of the linear passive scalar equation \eqref{Gld_Rg_ED}, we have that
\begin{align}
\sum_{|i,j,k|=0}^{M}\|\pa_{x}^i\Gamma_{y;t_0+\tau}^j\pa_{z}^k \na d_\nq(t_0+\tau)\|_2\leq& \sum_{|i,j,k|=0}^{M+1}\|\pa_{x}^i\Gamma_{y;t_0+\tau}^j\pa_z^k d_\nq\|_2+C(t_0+\tau)\sum_{|i,j,k|=0}^{M}\|\pa_x^{i+1}\Gamma_{y;t_0+\tau}^j\pa_z^k d_\nq\|_2\\
\leq&C(G)\frac{(t_0+\tau)}{A^{1/3}}\Phi(t_0+\tau)^{-M} e^{- \frac{\delta_{\mathcal Z}}{A^{1/3}}(\tau)}.\label{post_gluing_d_nq_est}
\end{align}
Now we implement the same estimate as in \eqref{T_24R} to obtain that for $\te$ chosen small enough,
\begin{align}
T_{24}^{NL;R} \leq &\sum_{|i,j,k|=0}^{M}\frac{G^{4+2i}\Phi^{2j}}{12A}\|\na \pa_{x}^i\Gamma_{y;t_0+\tau}^j\pa_{z}^k  n_\nq\|_2^2\\
&+ \frac{ C(G )}{A }\left(\sum_{|i,j,k|=0}^{M}G^{4+2i}\Phi^{2j}\|\pa_{x}^i\Gamma_{y;t_0+\tau}^j\pa_{z}^k  n_\nq\|_{2} ^2 +\|\lan n \ran\|_{H^{M}}^2\right)\frac{(t_0+\tau)^2}{A^{2/3}} e^{-\frac{2\delta_{\mathcal Z}  }{A^{1/3}}\tau} \Phi^{\myr{-4M-2 ?}}\\
\leq &\sum_{|i,j,k|=0}^{M}\frac{G^{4+2i}\Phi^{2j}}{12A}\|\na\pa_{x}^i\Gamma_{y;t_0+\tau}^j\pa_{z}^k  n_\nq\|_2^2 + \frac{ C(G) }{A^{2/3} }\|\lan n\ran\|_{H^M}^2  +\frac{ C( G )}{A^{2/3}} F_M.
\end{align}
Next, we invoke the estimate of $d_\nq$  \eqref{post_gluing_d_nq_est}, choose $\te,\ A^{-1}$ small enough and estimate the $T_{34}^{NL,R}$ term with the commutator relation \eqref{cm_yt_j_y} as in \eqref{T_34R},
\begin{align}
|T_{34}^{NL;R}|\leq &\sum_{|i,j,k|=0}^{M}\frac{G^{4+2i}\Phi^{2j}}{G^2A^{1/2}}\| \pa_{x}^i\Gamma_{y;t_0+\tau}^j\pa_{z}^k  n_\nq\|_2^2\\
& + \frac{C(G)}{A^{1/2}}\sum_{|i,j,k|=0}^{M}G^{4+2i}\Phi^{2j} \frac{ (t_0+t)^4}{A^{3/2}}  \bigg(\sum_{|i,j,k|=0}^{M}\|\pa_{x}^i\Gamma_{y;t_0+\tau}^j\pa_{z}^k n_\nq\|_2^2+\| \lan n\ran\|_{H^M}^2\bigg)\\
\leq &C(G)\frac{F_M}{ A^{1/2}}   +\frac{C(G) }{A^{1/2}}   \|\lan n\ran\|_{H^M}^2 .
\end{align}
Now we observe that the time integral of these two terms are small on the time interval of size $O(A^{1/3})$. 

Next we consider the term $T_{0\nq;d_\nq}$ in \eqref{T_0nqcd}. The treatment is similar to the estimate \eqref{T_0nq_dnq}. Here we apply the product estimate \eqref{Prd_est_gld}, the fact that $\pa_y^j\lan f\ran^x=\lan \Gamma_y^jf\ran^x$, the estimate of the gradient \eqref{Gld_Reg_grd}, the enhanced dissipation of the $d_\nq$ \eqref{d_nq_est_t_0+tau},  and 
 bootstrap hypothesis \eqref{HypED} to derive that
\begin{align}
|T_{0\nq;d_\nq}|\leq&\frac{1}{8A}\sum_{|j,k|=0}^ {M}\|\na\pa_y^j\pa_z^k \nz\|_{L_{y,z}^2}^2+\frac{C}{A}\left(\sum_{|i,j,k|=0}^{M}\|\pa_x^i\Gamma_{y;t_0+\tau}^{j}\pa_z^k  \na_{y,z} d_\nq \|_{L^2_{x,y,z}}^2\right)\left(\sum_{|i,j,k|=0}^{M}\|\pa_x^i\Gamma_{y;t_0+\tau}^{j}\pa_z^k  n_\nq\|_{L^{2}_{x,y,z}}^2\right)\\
\leq&  \frac{1}{8A}\sum_{|j,k|=0}^ {M}\|\na\pa_y^j\pa_z^k \nz\|_{L_{y,z}^2}^2\\
&+\frac{C }{A}\bigg(\sum_{|i,j,k|=0}^{M+1}\|\pa_x^i\Gamma_{y;t_0+\tau}^j\pa_z^k d_\nq\|_{L_{x,y,z}^2}+ ({t_0+\tau}) \sum_{ |i,j,k|\leq M }  \|\pa_x^{i+1}\Gamma_{y;t_0+\tau}^j\pa_z^k d_\nq\|_{L_{x,y,z}^2}\bigg)^2 \\
&\quad\quad\times \left(\sum_{|i,j,k|=0}^{M}\|\pa_x^i\Gamma_{y;t_0+\tau}^{j}\pa_z^k  n_\nq\|_{L^{2}_{x,y,z}}^2\right)\\
\leq&\frac{1}{8A}\sum_{|j,k|=0}^ {M}\|\na\pa_y^j\pa_z^k \nz\|_{L^2_{y,z}}^2 +\frac{C(  G)}{A^{1/2}}F_M.
\end{align}
Application of a similar argument as in the proof of Proposition \ref{Pro:main} should yield the result. }

In the last section, we estimate the solutions on the transient period. The goal is to show that the $F_{M+3}$ functional is  bounded independent of $A$ on $[0, A^{1/3+\te}]$. Even though the top regularity $\Phi^{2M+10}\|\Gamma_{y}^{M+4}c_\nq\|_2^2$ might not be small, we can use the next level regularity. And there is a competition between $\Phi^{2M}$ and $A^{2/3}$. The goal is to show that there exists a time $A^{1/3+\te}$ such that $\sum_{|i,j,k|\leq M+3}\|\palt  c_\nq\|_2\leq A^{-1/12}$. Then together with the enhanced dissipation of $d_\nq$, we have that the chemical gradient is small $\sum_{|i,j,k|\leq M+3}\|\palt \cc_{\nq}\|_2\leq \frac{1}{A^{12}}$.   
\ifx
In this section, we extend the main result Theorem \ref{thm_1}. The main idea is that for more general initial condition $(n_{\text{in}},\cc_{\text{in}})\times H^{M+3}\times H^{M+4}$, we can derive the smallness of $\cc$ in $H^{M+1}$ at time $T_{11}=3A^{1/3+\te}$  if we tune the shear more carefully.  
 The main adjustment is in the estimate of \eqref{T_34R}. There $\ep$ depends on $G$ and $C_{\lan n\ran}$. But does  $C_{\lan n\ran}$ depends on $\ep$? This dependence will restrict the size of the data?? Can we use the delay in the $\lan \cc\ran$ equation. $\ep$ can be small compared to the initial data, but the contribution from the later $\|\na d_\nq\|_2$ might have less impact? \fi
{\color{red}\textcolor{blue}{\ifx 
We need to use different arguments in the transient growth period $[0,A^{1/3+\te/2}]$ and the decaying period $[A^{1/3+\te/2}, A^{1/3+\te}]. $ In the transient growth periods, by the structure of the nonlinearity, we have that
\begin{align}
\frac{d}{dt}\|n_\nq\|_{H^{M+3}}^2\sim&...+\frac{1}{A}\|\na n_\nq\|\|\lan n\ran\|_{H^{M+3}}\|\na d_\nq\|_{\mathcal{Z}^{M+3}}\\
\leq &...\frac{1}{A}t^2 e^{-\delta_d \frac{t}{A^{1/3}}}\|\cc_{\nq}(T_{01})\|_{\mathcal{Z}^{M+4}}^2\|\lan n\ran\|_{H^{M+3}}^2\\
\frac{d}{dt}\|\lan n\ran\|_{H^{M+3}}^2\sim&...+\frac{1}{A}\|\na \lan n\ran\|_{...}\| n_\nq \|_{H^{M+3}} \|\na c_\nq\|_{\mathcal{Z}^{M+3}}+\frac{1}{A}\|\na \lan n\ran\|_{...}\| n_\nq \|_{H^{M+3}}(\|\na d_\nq\|_{\mathcal{Z}^{M+3}})\\
\leq &...\frac{1}{A}\| n_\nq \|_{H^{M+3}}^2\|\na c_\nq\|_{\mathcal{Z}^{M+3}}^2+\frac{1}{A}t^2 e^{-\delta_d \frac{t}{A^{1/3}}}\|\cc_{\nq}(T_{01})\|_{\mathcal{Z}^{M+4}}^2\| n_\nq \|_{H^{M+3}}^2.
\end{align}For the first term, we can write everything in terms of $F_{M+3}$ and gain from the time scale. From the argument we have for $\|\na c_\nq\|_\infty$, we see that the $\na c_\nq$ is small. And the $\|\na\palt c_\nq\|$ should also be small for $[0, A^{1/3+\te/2}$. Hence the first term can be reduced to something like $\frac{1}{A}\|n_\nq,\lan n\ran\|_{...}^2$. The time integral on $[0, A^{1/3+\te/2}]$ will be fine.\fi
As a result of the arguments in the last two sections, we have that
\begin{align}
\frac{d}{dt}&(F_{M+3}+\|\lan n\ran\|_{H^{M+3}}^2)\\
\leq&\frac{1}{GA^{1/3}(1+t^3/A)} F_{M+3}+ \frac{1}{A^{5/6}}C(M,\myr {G})   \Phi^{-4M-16} (F_{M+3}+\|\lan n\ran\|_{H^{M+3}}^2)^2\\
&+\frac{1}{A}t^2e^{-\delta_d t/A^{1/3}}\|\cc_{\nq}(T_{01})\|_{H^{M+4}}^2(F_{M+3}+\|\lan n\ran\|_{H^{M+3}}^2)
\end{align}
We define the following quantity
\begin{align}
\mathcal{G}(t)=\int_0^t \frac{1}{GA^{1/3}(1+s^3/A)}+\frac{1}{A}s^2e^{-\delta_d s/A^{1/3}}\|\cc_{\nq}(T_{01})\|_{H^{M+4}}^2 ds\leq C(1+\|\cc_{\nq}(T_{01})\|_{H^{M+4}}^2).  
\end{align}
Hence we can estimate the time evolution as follows
\begin{align}
\frac{d}{dt}(F_{M+3}+\|\lan n\ran \|_{H^{M+3}}^2)e^{-\mathcal{G} (t)}\leq \frac{1}{A^{5/6}}C(M,\myr {G})   \Phi^{-4M-16} e^{\|\mathcal{G}\|_{L^\infty_t}} \left(F_{M+3}+\|\lan n\ran \|_{H^{M+3}}^2\right)^2e^{-2\mathcal G(t)}.
\end{align}
Solving the differential inequality yields that on the time interval $[0,A^{1/3+\te/2}]$
\begin{align}
\sup_{t\in[0, A^{1/3+\te/2}]}F_{M+3}+\|\lan n\ran \|_{H^{M+3}}^2\leq Ce^{\|\mathcal G\|_{L^\infty_t}} (F_{M+3}(0)+\|\lan n\ran (0)\|_{H^{M+3}}^2)\leq Ce^{\|\cc_{\nq}(T_{01})\|_{H^{M+4}}^2} (F_{M+3}(0)+\|\lan n\ran (0)\|_{H^{M+3}}^2). 
\end{align}
}
}

\ifx
\begin{align}
\Rightarrow&\\
\frac{d}{dt}&[(F_{M+3}+\|\lan n\ran\|_{H^{M+3}}^2)e^{-\int_0^t \frac{s^2e^{-\delta_d\frac{s	}{A^{1/3}}}}{A}\|\cc_{\nq}(T_{01})\|_{H^{M+4}}^2 ds}]\\
\leq &\frac{1}{A^{3/4}}C[(F_{M+3}+\|\lan n\ran\|_{H^{M+3}}^2)e^{-\int_0^t \frac{s^2e^{-\delta_d\frac{s}{A^{1/3}}}}{A}\|\cc_{\nq}(T_{01})\|_{H^{M+4}}^2 ds}]^2e^{\int_0^t \frac{s^2e^{-\delta_d\frac{s}{A^{1/3}}}}{A}\|\cc_{\nq}(T_{01})\|_{H^{M+4}}^2 ds}.
\end{align}
Now we have that on the time interval $[0,A^{1/3+\te/2}]$, the following quantities are bounded:
\begin{align}
\sum_{|i,j,k|\leq M+3}\|\palt n_\nq\|_{L^2}^2+\sum_{|i,j,k|\leq M+4}\varphi^{j}G^{2i}(A^{2/3}\mathbbm{1}_{j<M+4}+\mathbbm{1}_{j=M+4})\|\palt c_\nq\|_{L^2}^2 \leq C(\|n_{\text{in}}\|_{H^{M+3}}, \|\cc_\text{in}\|_{H^{M+4}}).
\end{align} 
Now we have that
\begin{align}
\sum_{|i,j,k|\leq M+3}G^{2i}A^{2/3}\|\palt c_\nq\|_{L^2}^2\leq C(\|n_\text{in}\|_{H^{M+3}}, \|\cc_{\text{in}}\|_{H^{M+4}}) \varphi^{-M-4}(t). 
\end{align}
\fi
Hence we have smallness
\begin{align}
\sum_{|i,j,k|\leq M+3}& \|\palt \cc_\nq(T_{01}+A^{1/3+\te/2})\|_{L^2}^2\leq C
\sum_{|i,j,k|\leq M+3} \|\palt c_\nq\|_{L^2}^2+
\sum_{|i,j,k|\leq M+3} \|\palt d_\nq\|_{L^2}^2\\
\leq& C(\frac{1}{A^{2/3}}\Phi(A^{1/3+\te/2})^{-4M-16}+e^{-A^{\te/2}/C}).
\end{align}
On the interval $[A^{1/3+\te/2}, A^{1/3+\te}]$, we can use the same argument before to prove the exponential decay. On $[T_{02}, T_{03}]$ and $[T_{03}, T_{11}]$, we pay extra derivative to gain that at $T_{11}$, $\|\cc\|_{H^{M+1}}\leq C A^{-\gamma}$. \fi
\ifx 
\subsection{Previous Argument}
\begin{lem}[Short time estimate is enough]
Assume  the entropy and chemical gradient bound \eqref{PPKSwCDEntropyBound} and the bootstrap hypotheses \eqref{Hypotheses}. If the magnitude $A$ is chosen large enough compared to the constants in the bootstrapping hypothses, then the following estimates on the $L^4 $ norm of $n_0$ and $L^\infty$ norm of the chemical gradient $\na_y c_0$ hold
\begin{align}
\|n_0(t)\|_{L_{y}^4}\leq& C_{n_0;L^4}(M_0,\|n_{\mathrm{in};0}\|_{L_y^4}, E[n_{\mathrm{in};0}, c_{\mathrm{in};0}])<\infty,\quad \forall t\in [0,T_\star];\\
\|\na_y c_0(t)\|_{L_y^\infty}\leq& C_{\na c_0;L^\infty}^\star(M_0,\|n_{\mathrm{in};0}\|_{L_y^4},\| \na_y c_{\mathrm{in};0}\|_{  L_y^\infty},E[n_{\mathrm{in};0}, c_{\mathrm{in};0}])<\infty, \quad \forall t\in[0,T_\star].
\end{align}
\begin{rmk}
By choosing the constant $C_{\na_y c_0;L^\infty}$ in the bootstrapping hypothesis \eqref{Hyp_<c>} large enough compared to $C_{\na c_0;L^\infty}^\star$ in the lemma, we have obtained the improvement \eqref{Con_<grd_c>}. 
\end{rmk}
\end{lem} 
\begin{proof}
By the same type of energy estimate as in the $L^2$-case, we apply the Nash  inequality, the $L^2$-bound \eqref{L2case}, and the boostrap hypotheses \eqref{Hypotheses}  to obtain that 
\begin{align}
\frac{d}{dt}\|n_0\|_4^4\leq& -\frac{1}{12C_{N}A(C_{n_0;L^2}^\star)^{4}}\|n_0\|_4^{8}+\frac{{C(M_0, C_{n_0;L^2}^\star)}}{A}(1+A^2\|\pa_t c_0 \|_2^2)\|n_0\|_4^4+\frac{ {C(M_0,C_{n_0;L^2}^\star)}}{A}(1+A^2\|\pa_t c_0 \|_2^2)\\
&+\frac{C_{ED}C(C_{n;L^\infty}, C_{\na_y c_0;L^\infty}, C_{\na c_{\neq};L^\infty})}{A^{1-4\eta}}(\|n_{\mathrm{in};\neq}\|_2^2+\norm{\na c_{\mathrm{in};\neq}}_{L^2}^2+1)e^{-\frac{\delta t}{A^{1/3}}}.
\end{align}
Now we define 
\begin{align}
G_2(t):=\int_0^t\frac{C_{ED}C(C_{n;L^\infty}, C_{\na_y c_0;L^\infty}, C_{\na c_{\neq};L^\infty})}{A^{1-4\eta}}(\|n_{\mathrm{in};\neq}\|_2^2+\norm{\na c_{\mathrm{in};\neq}}_{L^2}^2+1)e^{-\frac{\delta s}{A^{1/3}}}ds,\quad\forall t\in[0,T_\star].
\end{align}
By taking $A$ large enough, we see that $G_2(t)\leq 1$. Now apply similar ODE calculation as in the proof of previous lemma, we obtain that $L^4$ norm of the cell density $n_0$ is bounded on the time interval $[0,T_\star]$. 

The estimate on the chemical density follows from the $L^4$ estimation of $n_0$ and the estimate \eqref{na_c_0_est}. 
\end{proof}
\begin{lem}
Assume the entropy and chemical gradient bound \eqref{PPKSwCDEntropyBound}, and the bootstrapping hypotheses \eqref{Hypotheses}. If the magnitude $A$ is chosen large enough, then the following estimate on the $\dot H_{y,z}^1 $ norm of $\nz_x$ holds
\begin{align}
\|\na_y \nz_x(t)\|_{L_y^2}\leq C^\star_{n_0;\dot H^1}(\CC)<\infty,\quad \forall t\in [0,A^{1/3+\ep}].
\end{align}
\end{lem} 
\begin{rmk}
By choosing the constant $C_{n_0;\dot H^1}$ in the bootstrapping hypothesis \eqref{Hyp_<c>} large enough compared to $C_{\na c_0}$ in the lemma, we have obtained the improvement \eqref{Con_<grd_c>}. 
\end{rmk}
\begin{proof}
Now we estimate the $\|\na_y n_0\|_2$. Estimating the time evolution of $\|\na_y n_0\|_2$ using Gagliardo-Nirenberg-Sobolev inequality, Young's inequality, Minkowski inequality, Lemma \ref{Lem:gc0Lpest} and the time integral estimate of $\|\na^2 c_{\neq}\|_2^2$ \eqref{Hyp1_c}, we have that 
\begin{align}
\frac{1}{2}\frac{d}{dt}\|\na_{y,z} \nz_x\|_2^2
\leq &-\frac{\|\na_{y,z}^2 \nz\|_2^2}{2A}+\frac{4\|\de_{y,z} \cz\|_2^2\|\nz\|_\infty^2}{A}+\frac{4\|\na_{y,z} \cz\|_\infty^2\|\na_{y,z} \nz\|_2^2}{A}\\
&+\frac{4}{A}\int \na_{y,z}\nz\cdot \na_{y,z}\lan(\pa_y^t,\pa_z)\cdot((\pa_y,\pa_z) \cc_{\neq}n_{\neq})\ran_xdydz\nonumber\\
\leq&-\frac{\|\na_{y,z}^2 \nz\|_2^2}{4A}+\frac{4(\sup_{0\leq s\leq t}\|\na_{y,z} \nz(s)\|_2+C(\na \cz_{in}) )^2\|\nz\|_\infty^2}{A}+\frac{C(C_{n,\infty})\|\na_{y,z} \nz\|_2^{2}}{A}\nonumber\\
&+\frac{4}{A}\left(\|(\na_{y,z}^t)^2 c_{\neq}\|_{L_{x,y}^2}^2+\frac{Ct^2}{A^{2/3}}A^{2/3}\|\pa_x\na_{y,z}^t c_\nq\|_2^2+Ct^2 \|\pa_x\na_{y,z}^t d_\nq\|_2^2+C\|(\na_{y,z}^t)^2d_\nq\|_2^2\right)\|n_{\neq}\|_{L_{x,y}^\infty}^2\\
&+\frac{4}{A}(\|\na_{y,z} c_{\neq}\|_{L^\infty}^2+Ct^2\|\pa_x d_\nq\|_{L^\infty}^2+\|\na_{y,z}^t d_\nq\|_{L^\infty}^2)\|\na_{y,z}^t n_{\neq}\|_{L^2}^2\nonumber\\
\leq&-\frac{\|\na_{y,z}^2 \nz\|_2^2}{4A}+\frac{4(\sup_{0\leq s\leq t}\|\na_{y,z} \nz(s)\|_2+C(\na \cz_{in}) )^2\|\nz\|_\infty^2}{A}+\frac{C(C_{n,\infty})\|\na_{y,z} \nz\|_2^{2}}{A}\nonumber\\
&+NZ.\label{time evolution of na f L2}
\end{align}
We define the function 
\begin{equation}\label{G3}
G_3(t):=\int_0^t NZ(s)ds.
\end{equation}
Note that by the hypotheses \eqref{Hyp1}, \eqref{Hyp1_c}, \eqref{Hyp_n_Linf}, \eqref{Hyp_na_c_neq_Linf}, we have that
\begin{align}
G_3(t)\leq \frac{1}{A^{1/3}}+C\ep^2 , \quad \forall t\in[0,T_\star].
\end{align} By an  ODE argument (Similar to the 2d case?), we obtain that 
\begin{align}\label{f 0 H1}
\|\na_{y,z} \nz_x\|_2\leq C  ( n_{in},\cc_{in}).
\end{align}
This concludes the proof of the lemma.
\end{proof}
\fi

\ifx
\subsection{Previous Version}
We would like to prove that the following functional decay with fast rate $\nu^{1/3}$ {\textcolor{blue} In the following, $L^4=L_1^2,\, L^2=L_2^2$.} 
\begin{align}
F=L^4\|  n_\nq\|_2^2+L^2\| \pa_x c_\nq\|_2^2+{\color{red}(1+\frac{t^2}{A^{2/3}})}^{-1}\| \pa_y^t c_\nq\|_2^2+\| \pa_z c_\nq\|_2^2+\ep^{-2}\| \na \cc _{0;\nq}\|_2^2e^{-\frac{1}{2}\delta t/A^{1/3}}.
\end{align}
Here we define $\varphi(t)=\frac{1}{1+(\frac{t}{A^{1/3}})^2}$. The functional can be applied till time $t=A^{1/3+\zeta}, \quad \zeta<\frac{2}{15}$.
Here $\|\na d_\nq\|_2\leq C\|\na\cc_{0;\nq}\|_2(1+t)e^{-\delta t/A^{1/3}}$. The parameter $L$ will be chosen to compensate for various constants appear during the proof. In particular, it depends on the $\delta$ in the enhanced dissipation estimate of the passive scalar solutions. 

We further note that the vector field
\begin{align}
\pa_y^t=\pa_y+\int_0^tu_y(y,s)ds\pa_x, \ \
\pa_y^{t_\star,t_\star+t}=\pa_y+\int_{t_\star}^{t_\star+t}u_y(y,s)ds\pa_x
\end{align}
commutes with the transport part of the equation. We might use this to derive the ED for the $\pa_y c_\nq$ part of the functional, i.e., 
\begin{align}
[\pa_t +u(y,t_\star+t)\pa_x,\pa_y+\int_{t_\star}^{t_\star+t}u_y(y,s)ds\pa_x]=u_y(y,t_\star+t)\pa_x-u_y(y,t_\star+t)\pa_x=0.
\end{align} 
Now we estimate the term $A^{2/3}\|\pa_x c_\nq\|_2^2$. We compare it with the solution to the passive scalar equation $ S_{t_\star}^{t_\star}\pa_x c_\nq,\ S_{t_\star}^{t_\star}\pa_x c_\nq=\pa_x c_\nq(t_\star)$: 
\begin{align}
\frac{d}{dt}\frac{1}{2}A^{2/3}L^2\|\pa_x c_\nq\|_2^2\leq &-\frac{1}{4A}L^2\|\na \pa_x c_\nq\|_2^2+\frac{1}{A^{1/3}L^2}L^4\|n_\nq\|_2^2;\label{Reg_pa_x_c_nq}\\
\frac{d}{dt}\frac{1}{2}A^{2/3}L^2\|\pa_x c_\nq-S_{t_\star}^{t_\star+\tau}\pa_x c_\nq\|_2^2\leq&-\frac{1}{4A^{1/3}}L^2\|\na (\pa_x c_\nq-S_{t_\star}^{t_\star+\tau}\pa_x c_\nq)\|_2^2+\frac{1}{A^{1/3}L^2}L^4\|n_\nq\|_2^2. 
\end{align}
Hence when we integrate from $t_\star $ to $t_\star+\delta^{-1}A^{1/3}$. The deviation is less than 
\begin{align}
A^{2/3}L^2\|\pa_x c_\nq-S_{t_\star}^{t_\star+\tau}\pa_x c_\nq\|_2^2(t_\star+t)\leq \frac{1}{180}\left(L^4\sup_{t\in[0, \delta^{-1}A^{1/3}]}\|n_{\nq}(t_\star+t)\|_2^2\right),\quad t\leq\delta^{-1}A^{1/3}.
\end{align}
Now we can tune the parameter in front of $\|n_\nq\|_2^2$ to get the decay of the functional. Now we estimate the time evolution of the norm $\|\pa_y^t c_\nq\|_2^2$. We consider the following quantity
\begin{align}
\|\pa_y^{\tau} c_\nq\|_2^2=\|(\pa_y +\int_{0}^{\tau}u_y(y, s)ds\pa_x)c_\nq\|_2^2.
\end{align}
The equation for $\pa_y^\tau$ reads as follows:
\begin{align}
\pa_\tau \pa_y^\tau c_\nq+u(\tau, y)\pa_x \pa_y^\tau c_\nq=&\frac{1}{A}\de \pa_y^\tau c_\nq +\frac{1}{A}(-2 B^{(2)}\pa_x\pa_y^\tau+\pa_y(B^{(1)})^2\pa_{xx}-B^{(3)}\pa_x)c_\nq+\frac{1}{A}\pa_y^\tau n_\nq\\
=&\frac{1}{A}\de \pa_y^\tau c_\nq+\frac{1}{A}(-2 \pa_y (B^{(2)}\pa_x c _\nq)+ B^{(3)}\pa_x c_\nq)+\frac{1}{A}\pa_y^\tau n_\nq,\\
\quad B^{(k)}:=&\int_0^t u^{(k)}(s,y)ds=\int_0^t\underbrace{\pa_{y...y}}_{\text{ n derivatives}}u(s,y)ds.
\end{align}
Now let $\eta_\nq(t_\star)=\pa_y c_{\nq}(t_\star)$. Now we have that 
\begin{align}
\frac{d}{dt}&\frac{1}{2}\varphi(t_\star+t)\|\pa_y^{t_\star+t} c_\nq\|_2^2\\
=&\frac{\varphi'(t_\star+t)}{2}\|\pa_y^{t_\star+t}c_\nq\|_2^2-\frac{1}{2A}\varphi(t_\star+t)\|\na (\pa_y^{t_\star+t} c_\nq)\|_2^2\\
&+\frac{1}{A}C(u)\varphi(t_\star+t)((t_\star+t)^4\|\pa_x c_\nq\|_2^2+(t_\star+t)^2\|\pa_x c_\nq\|_2^2)+\frac{1}{A}(C\varphi(t_\star+t)\|n_\nq\|_2^2(t+t_\star)^2)\\
=&\frac{\varphi'(t_\star+t)}{2}\|\pa_y^{t_\star+t}c_\nq\|_2^2-\frac{1}{2A}\varphi(t_\star+t)\|\na (\pa_y^{t_\star+t} c_\nq)\|_2^2\\
&+\frac{1}{L^2A^{1/3}}C(u)\frac{(t_\star+t)^4+(t_\star+t)^2}{{A^{4/3}}+{(t_\star+t)^4}}(\|\pa_x c_\nq\|_2^2A^{2/3}L^2)+\frac{C}{A^{1/3}L^4}\frac{(t+t_\star)^2}{A^{2/3}+\frac{(t_\star+t)^4}{A^{2/3}}}(L^4\|n_\nq\|_2^2)\\
\leq&\frac{1}{A^{1/3}L^2}C(L^4\|n_\nq\|_2^2+L^2A^{2/3}\|\pa_x c_\nq\|_2^2);\\ 
\frac{d}{dt}&\frac{1}{2}\varphi(t_\star+t)\|\pa_y^{t_\star+t} c_\nq-\eta_\nq\|_2^2\\
=&\frac{\varphi'(t_\star+t)}{2}\|\pa_y^{t_\star+t}c_\nq-\eta_\nq\|_2^2-\frac{1}{2A}\varphi(t_\star+t)\|\na (\pa_y^{t_\star+t} c_\nq-\eta_\nq)\|_2^2\\
&+\frac{1}{A}C(u)\varphi(t_\star+t)\big((t_\star+t)^2\|\pa_x c_\nq\|_2^2\big)+\frac{1}{A}(C\varphi(t_\star+t)\|n_\nq\|_2^2(t+t_\star)^2)\\
\leq&\frac{1}{L^2 A^{1/3}}C\frac{(t_\star+t)^2}{A^{2/3}+(t_\star+t)^4/A^{2/3}}(A^{2/3}L^2\|\pa_xc_\nq\|_2^2)+\frac{1}{A^{1/3}L^4}\frac{(t+t_\star)^2}{A^{2/3} +\frac{(t_\star+t)^{4}}{A^{2/3}}}C(L^4\|n_\nq\|_2^2)\\
\leq&\frac{1}{A^{1/3}L^2}\sup_{t\in[0,\delta^{-1}A^{1/3}]}(L^4\|n_\nq(t_\star+t)\|_2^2+L^2A^{2/3}\|\pa_x c_\nq(t_\star+t)\|_2^2)).
\end{align}
\ifx We also note that 
\begin{align}
\|\pa_y c_\nq(t_\star+A^{1/3})\|_2^2\leq 4\|\pa_y^{t_\star,t_\star+A^{1/3}}c_\nq(t_\star+A^{1/3})\|_2^2+C_u A^{2/3}\|\pa_x c_\nq(t_\star+A^{1/3})\|_2^2\\\leq \frac{1}{L^4}CL^4\|n_\nq\|_2^2+\frac{1}{12}\|\pa_y c_\nq(t_\star)\|_2^2+\frac{1}{12}A^{2/3}L^2\|\pa_x c_\nq(t_\star)\|_2^2.....
\end{align}\fi
The last term in the functional should be similar? Maybe?

To estimate the $\|n_\nq\|_2^2$ term, we use similar ideas. Application of the hypothesis and the estimate $\|\na d_{\neq}\|_2\leq C(1+t)\|\na \cc_{0;\nq}\|_2e^{-\delta t/A^{1/3}}$ yields that

Next we derive the regularity estimates. The estimates are similar to the approximation estimates carried out before. First, we estimate the $L^2$ norm of the $\|n_\nq\|_2^2$,
\begin{align}
\frac{d}{dt}\frac{1}{2}(L^4\|n_\nq\|_2^2)=&-\frac{1}{A}L^4\|\na n_\nq\|_2^2+\frac{1}{A}L^4\int n_\nq\left(-\na \cdot(\na \cc_\nq n_{\neq})-\na \cdot(\na \cc_\nq \lan n\ran_x))-\na \cdot(\na \lan \cc\ran_x n_{\neq})\right)dV\\
=&-\frac{1}{A}L^4\|\na n_\nq\|_2^2+\frac{1}{A}L^4\int\na n_{\nq}\cdot(\na \cc_\nq n_\nq+\na \cc_\nq \lan n\ran_x+\na \lan \cc\ran _x n_\nq)dV. 
\end{align}
We further recall that the $\cc_{\nq}=c_\nq+d_\nq$ and the bootstrap hypothesis on the $x$-average. Now we can estimate the time evolution as follows
\begin{align}
\frac{d}{dt}&\frac{1}{2}L^4\|n_\nq(t_\star+t)\|_2^2\\
=&-\frac{1}{2A}L^4\|\na n_\nq\|_2^2+\frac{C}{A}L^4(\|\na c_\nq\|_2^2+\|\na d_\nq\|_2^2)\|n\|_\infty^2+\frac{C}{A}L^4\|\na \lan \cc\ran_x\|_\infty^2\|n_{\neq}\|_2^2\\
\leq &\frac{CL^4}{A}(\|\pa_x c_\nq\|_2^2+\|\pa_z c_\nq\|_2^2)C_{n;\infty}^2+\frac{CL^4}{A}((1+\frac{(t_\star+t)^4}{A^{4/3}})\varphi(t_\star+t)\|\pa_y^t c_\nq\|_2^2+\frac{C(u)(t_\star+t)^2}{L^2A^{2/3}}(A^{2/3}\|\pa_xc_\nq\|_2^2L^2))\\
&+\frac{CL^4}{A}(\ep^{-2}\|\na \cc_{0;\nq}\|_2^2e^{-\delta t_\star/A^{1/3}})\ep^2C_{n;\infty}^2(1+t_\star+t)^2e^{-\delta (t+t_\star)/A^{1/3}}+\frac{CL^4}{A}C_{\na \lan c\ran ; \infty}^2\|n_\nq\|_2^2.
\end{align} 
For $t_\star+t\leq A^{1/3+\ep},\ep<\frac{2}{15}$, we can take the $A$ compared to $L$ large such that for $\forall t\leq C\delta^{-1}A^{1/3},$
\begin{align}
L^4\|n_\nq\|_2^2(t_\star+t)+L^2A^{2/3}\|\pa_x c_\nq\|_2^2(t_\star+t)+\varphi(t_\star+t)\|\pa_y^{t_\star+t} c_\nq\|_2^2(t_\star+t)+\|\pa_z c_\nq\|_2^2(t_\star+t)+\ep^{-2}\| \na \cc _{0;\nq}\|_2^2e^{-\frac{1}{2}\delta (t_\star+t)/A^{1/3}}\\
\leq 2(L^4\|n_\nq(t_\star)\|^2_2+L^2A^{2/3}\|\pa_x c_\nq\|_2^2(t_\star)+\varphi(t_\star)\|\pa_y^{t_\star} c_\nq\|_2^2(t_\star)+\|\pa_z c_\nq\|_2^2(t_\star)+\ep^{-2}\| \na \cc _{0;\nq}\|_2^2e^{-\frac{1}{2}\delta t_\star/A^{1/3}}). 
\end{align} 
Similar estimate ? yields that 
\begin{align}
\frac{d}{dt}&\frac{1}{2}L^4\|n_\nq-\eta_\nq\|_2^2(t_\star+t)\\
\leq&-\frac{1}{2A}\|\na (n_\nq-\eta_\nq)\|_2^2+\frac{C}{A}(\|\na c_{\neq}\|_2^2+\|\na d_{\neq}\|_2^2)\|n\|_\infty^2+\frac{C}{A}\|\na \lan c\ran_x\|_\infty^2\|n_{\neq}\|_2^2\\
\leq &-\frac{1}{2A}\|\na (n_\nq-\eta_\nq)\|_2^2+\frac{C}{A}(\|\na c_{\neq}\|_2^2+C(1+t+t_\star)^2\|\na \cc_{0;\neq}\|_2^2e^{-2\delta (t_\star+t)/A^{1/3}})\|n\|_\infty^2+\frac{C}{A}\|\na \lan \cc\ran\|_\infty^2\|n_{\neq}\|_2^2\\
\leq&\frac{CL^4}{A}(\|\pa_x c_\nq\|_2^2+\|\pa_z c_\nq\|_2^2)C_{n;\infty}^2+\frac{CL^4}{A}((1+\frac{(t_\star+t)^4}{A^{4/3}})\varphi(t_\star+t)\|\pa_y^t c_\nq\|_2^2+\frac{C(u)(t_\star+t)^2}{L^2A^{2/3}}(A^{2/3}\|\pa_xc_\nq\|_2^2L^2))\\
&+\frac{CL^4}{A}(\ep^{-2}\|\na \cc_{0;\nq}\|_2^2e^{-\delta t_\star/A^{1/3}})\ep^2C_{n;\infty}^2(1+t_\star+t)^2e^{-\delta (t+t_\star)/A^{1/3}}+\frac{CL^4}{A}C_{\na \lan c\ran ; \infty}^2\|n_\nq\|_2^2\\
\leq&\frac{C}{A^{2/3-5\ep}}F[n_\nq, c_\nq](t_\star).
\end{align} 
Now we should have decay for all??\ifx
By the regularity estimate, we should have that 
\begin{align}
\|\na c_\nq(t_\star+t)\|_2^2+\|n_\nq (t_\star+t)\|_2^2\leq 2F_{i,j}(t_\star).
\end{align}
Recall the hypothesis, we have that for $t\leq A^{1/3}$, we can pick the $A$ and $\ep$ small to get
\begin{align}
\|n_{\neq}-\eta_{\nq}\|_2^2(t_\star+t)\leq& \frac{Ct}{A}C_{n;\infty}^2F_{i,j}(t_\star)+\epsilon^{-2}\|\na \cc_{0;\nq}\|_2^2e^{-\delta t_\star/A^{1/3}}\int_0^{t}\frac{(1+s+t_\star)^2}{A}C e^{-\delta (t_\star+s)/A^{1/3}}\ep^2 C_{n;\infty}^2ds\\
&+\frac{Ct}{A	}C_{\na c_0;\infty}^2F_{i,j}(t_\star)\\
\leq&\frac{1}{18}F_{i,j}(t_\star).
\end{align}

Next we estimate the time evolution of the $A^{2/3}\|\pa_x c_\nq\|_2^2$,
\begin{align}
\frac{d}{dt}\frac{1}{2}A^{2/3}\|\pa_x c_\nq\|_2^2=&-\frac{1}{A}\|\na \pa_x c_\nq\|_2^2+\frac{1}{A^{1/3}}\int \pa_x c_\nq\pa_x n_\nq\\
\leq&-\frac{1}{2A^{1/3}}\|\na \pa_x c_\nq\|_2^2+\frac{1}{A^{1/3}}\|n_\nq\|_2^2.
\end{align}
Now we should be okay?

For the $\|\pa_y c_\nq\|_2^2$ term, we apply similar idea as before. Instead of considering the gradient itself, we consider the vector field norm
$\|(\pa_y +\int_{t_\star}^{t_\star+t}\pa_yu(y,s)ds\pa_x)c_\nq(t_\star+t)\|_2^2.$
The time evolution of the quantity is bounded as follows
\begin{align}
\frac{d}{dt}\frac{1}{2}\|\pa_y^{t_\star,t_\star+t}c_\nq\|_2^2=&-\frac{1}{A}\|\na \pa_y^{t_\star,t_\star+t}c_\nq\|_2^2+\frac{1}{A}\int \pa_y^{t_\star,t_\star+t} c_\nq \left(\pa_y+\int_{t_\star}^{t_\star+t}\pa_y u(y,s)ds\pa_x \right)n_\nq dV\\ 
\leq&-\frac{1}{2A}\|\na \pa_y^{t_\star,t_\star+t}c_\nq(t_\star+t)\|_2^2+\frac{C}{A}\left(1+\|\int_{t_\star}^{t_\star+t}\pa_yu(s,y)ds\|_\infty^2\right)\|n_\nq\|_2^2\\
\leq&-\frac{1}{2A}\|\na \pa_y^{t_\star,t_\star+t}c_\nq(t_\star+t)\|_2^2+\frac{C}{A}(1+t^2)\|\pa_yu\|_{L^\infty}^2\|n_\nq\|_2^2.
\end{align}
Now we have that 
\begin{align}
\frac{d}{dt}&(L^4\|n_\nq(t_\star+t)\|_2^2+L^2A^{2/3}\|\pa_x c_\nq(t_\star+t)\|_2^2+\|\pa_y^{t_\star,t_\star+t}c_\nq(t_\star+t)\|_2^2+\|\pa_z c_\nq(t_\star+t)\|_2^2)\\
\leq &\frac{CL^4}{A}\bigg(\frac{(1+\|\pa_y u\|_\infty^2t^2)}{L^2A^{2/3}}(L_2^2A^{2/3}\|\pa_x c_\nq\|_2^2)+\|\pa_z c_\nq\|_2^2+\|\pa_{y}^{t_\star,t_\star+t} c_\nq\|_2^2\bigg)C_{n;\infty}^2\\
&+\frac{CL^4}{A}(\ep^{-2}\|\na \cc_{0;\nq}\|_2^2e^{-\delta (t_\star+t)/A^{1/3}})\ep^2C_{n;\infty}^2(1+t_\star+t)^2e^{-\delta (t+t_\star)/A^{1/3}}\\
&+\frac{C}{A}C_{\na \lan c\ran ; \infty}^2(L^4\|n_\nq\|_2^2)\\
&+\frac{4L^2}{L^4A^{1/3}}(L^4 \|n_\nq\|_2^2)+\frac{C}{AL^4}(1+t^2)\|\pa_y u\|_\infty^2(L^4\|n_\nq\|_2^2)\\
\leq & \left(\frac{CL^4(1+\|\pa_y u\|_\infty^2 t^2)}{L^2A^{5/3}}C_{n;\infty}^2+\frac{C}{A}C^2_{\na \lan c\ran;\infty}+\frac{4L^2}{L^4 A^{1/3}}+\frac{C}{AL_1^2}(1+t^2)\|\pa_yu\|_\infty^2\right)\times\\
&\times(L^4\|n_\nq(t_\star+t)\|_2^2+L^2A^{2/3}\|\pa_x c_\nq(t_\star+t)\|_2^2+\|\pa_y^{t_\star,t_\star+t}c_\nq(t_\star+t)\|_2^2+\|\pa_z c_\nq(t_\star+t)\|_2^2)\\
&+\frac{CL^4}{A}(\ep^{-2}\|\na \cc_{0;\nq}\|_2^2e^{-\delta (t_\star+t)/A^{1/3}})\ep^2C_{n;\infty}^2(1+t_\star+t)^2e^{-\delta (t+t_\star)/A^{1/3}}.
\end{align}
We focus on the interval $[t_\star,t_\star+2\delta^{-1}A^{1/3}]$ and observe that $t\leq 2\delta^{-1}A^{1/3}$. As a result, we choose $L\geq 24\delta^{-2}(1+\|\pa_yu\|_{L_{t,y}^\infty})$ and $A$ large compared to the bootstrap bounds $C_{\na \lan c\ran_x;\infty},\ C_{n;\infty}$ and $\delta^{-1}, L$ to obtain that\begin{align}\exp\bigg\{\int _0^t \frac{CL^4(1+\|\pa_y u\|_\infty^2 s^2)}{L^2A^{5/3}}C_{n;\infty}^2+\frac{C}{A}C^2_{\na \lan c\ran;\infty}+\frac{4L^2}{L^4 A^{1/3}}+\frac{C}{AL^4}(1+s^2)\|\pa_yu\|_\infty^2ds \bigg\}\leq 4.
\end{align} 
 Now we integrate on the time interval and end up with the following
\begin{align}
L^4&\|n_\nq(t_\star+t)\|_2^2+L^2A^{2/3}\|\pa_x c_\nq(t_\star+t)\|_2^2+\|\pa_y^{t_\star,t_\star+t}c_\nq(t_\star+t)\|_2^2+\|\pa_z c_\nq(t_\star+t)\|_2^2\\
\leq &4\left(L^4\|n_\nq(t_\star)\|_2^2+L^2A^{2/3}\|\pa_x c_\nq(t_\star)\|_2^2+\|\pa_y c_\nq(t_\star)\|_2^2+\|\pa_z c_\nq(t_\star)\|_2^2\right)\\
&+4\int_0^{2\delta^{-1}A^{1/3}} \frac{CL_1^2}{A}(\ep^{-2}\|\na \cc_{0;\nq}\|_2^2e^{-\delta t_\star/A^{1/3}})\ep^2C_{n;\infty}^2(1+t_\star+s)^2e^{-\delta (s+t_\star)/A^{1/3}}ds\\
\leq &4F(t_\star)(1+\frac{CL^4}{\delta^3}\ep^2 C_{n;\infty}^2),\forall t\in[0, 2\delta^{-1}A^{1/3}].
\end{align}If the $\ep$ is chosen small enough, we have that the sum is less than $12F(t_\star)$. Note that $\|\pa_y c_{\neq}(t_\star+t)\|_2^2\leq 4L^2 A^{2/3}\|\pa_x c_\nq(t_\star+t)\|_2^2+ 4\|\pa_y^{t_\star,t_\star+t}c_{\neq}(t_\star+t)\|_2^2$,  and we  have the regularity estimate. \fi
 \fi
 \section{{Alternating Construction}} \label{Sec:Alt}
In this section, we prove Theorem \ref{thm:alt_1} and Theorem \ref{thm:alt_2}. For the sake of notation simplicity, the implicit  constants $C$ in this section can depend on $\delta_{\mathcal Z}^{-1}, \, \delta_0^{-1}, \, C_0, \, M,\, \|u_A\|_{L^\infty_t W^{\mathbb{M}+6, \infty}}.$

To prove Theorem \ref{thm:alt_1}, we focus on Phase \# 1 \eqref{ph_1}, i.e., $t\in[0, 3A^{1/3+\te}]$. As explained in Section 2.3, the main achievement in Phase \# 1 is to guarantee that at time $T_{1\mathfrak a}=3A^{1/3+\te}$, the chemical gradient is small, i.e., $\|\cc(T_{1\mathfrak a})\|_{H^{\mathbb{M}+1}}\ll 1$. Once the smallness is reached, application of  Proposition \ref{Pro:main} yields the enhanced dissipation of the solution in Phase \# 2 ($t\geq 3A^{1/3+\te}$). We decompose the chemical density in the following way 
\begin{align}
\mathfrak {C}=\lan \lan \cc\ran\ran^{y,z}+(\lan  \cc\ran^y)_\nq^z+ \cc_\nq^y,\quad \overline{\mathfrak C}\equiv0.
\end{align}
To visualize this decomposition, we can apply a Fourier transform and see that 
\begin{align}
\wh{\lan \lan \cc\ran\ran^{y,z}}=&\mathbbm{1}_{R_1}\wh{\cc},\qquad R_1:=\{\al\neq 0, \beta =0, \gamma=0\};\\
\wh{(\lan \cc\ran^{y})_\nq^z}=&\mathbbm{1}_{R_2}\wh{\cc},\qquad R_2:=\{ \beta= 0, \gamma \neq0\};\\
\wh{\cc_\nq^{y}}=&\mathbbm{1}_{R_3}\wh{\cc},\qquad R_3:=\{ \beta\neq 0\}.
\end{align}
The explicit positions of the three Fourier domains can be found in Figure \ref{Fig:R123}. 
\begin{figure}[h]
\centering
\includegraphics[scale=0.8]{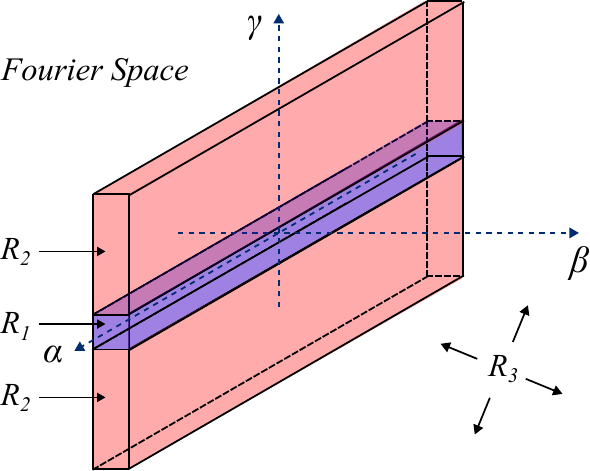}
\caption{Region $R_1, \, R_2,\, R_3$} \label{Fig:R123}
\end{figure} 
Now the plan is to show significant decay of the chemical density information in the three Fourier  domains $R_1,\ R_2, \ R_3.$ We will show that heuristically, the information stored in Fourier domains $R_1$, $R_2$ and $R_3$ will undergo significant decay in phase $\mathfrak a,$ $\mathfrak b$ and $\mathfrak c$, respectively. In Phase $\mathfrak a$ ($t\in[0, A^{1/3+\te}]$), the shear flow is in the $x$-direction, and it will efficiently damp all the chemical information in Fourier domain $\{\al\nq0\}\supset R_1$. As a result, we obtain the following lemma, which captures the main characteristics of the system at the end of Phase $\mathfrak a$, i.e., $t=T_{1\mathfrak b}$.
\ifx 
Let us start by decomposing phase $\mathfrak a$ into two sub-intervals:
\begin{align}
[T_{0\mathfrak a}, T_{0\mathfrak b}]=[0, A^{1/3+\zeta/2})\cup [A^{1/3+\zeta/2}, A^{1/3+\zeta}]=:\mathcal{I}_1\cup \mathcal{I}_2.
\end{align}
We recall that in the $\mathfrak{a}$-phase, the shear is in the $x$-direction and we can directly apply Proposition \ref{Pro:main} and Proposition \ref{pro:prp_FM}. Application of Proposition \ref{pro:prp_FM} yields that 
\begin{align}
\cc_\nq^x=&c_\nq^x+d_\nq^x;\\
\sum_{|i,j,k|\leq M+3}\|\pa_x^i\Gamma_{y;A^{1/3+\zeta/2}}^j\pa_z^k c_{\nq}^x(A^{1/3+\te/2})\|_{L^2}^2\leq &\frac{1}{A^{2/3}}C(\mathbb{F}_{G,Q}^{0;M+3}[0]+\|\lan n_{\text{in}}\ran\|_{H^{M+3}}^2+\|\cc_{\text{in}}\|_{H^{M+4}}^2);\\
\sum_{|i,j,k|\leq M+3}\|\pa_x^i\Gamma_{y;A^{1/3+\zeta/2}}^j\pa_z^k d_\nq^x(A^{1/3+\te/2})\|_{L^2}^2\leq& C\|\cc_{\text{in}}\|_{H^{M+3}}^2 e^{-\frac{\delta_{\mathcal Z}}{2}A^{\te/2}}.
\end{align}
Combining the estimate above yields that if $A$ is chosen large enough, then
\begin{align}
\sum_{|i,j,k|\leq M+3}\|\pa_x^i\Gamma_{y;A^{1/3+\te/2}}^j\pa_z^k \cc_\nq^x\|_{L^2}^2\leq \frac{1}{A^{1/2}}.
\end{align}
\fi
\begin{lem}\label{lem:Cnfg_T0b}
Consider the solution to the equation \eqref{ppPKS}, initiated from data $n_{\mathrm{in}}\in H^{\mathbb{M}+3}, \, \cc_{\mathrm{in}}\in H^{\mathbb{M}+4}$.  Recall the definitions of $f_\nq^x,\, \lan f\ran^x$ \eqref{f_iota}.  If the parameter $A$ is chosen large enough, then the following estimates hold
\begin{align} \label{conf_T0b}
\| n_\nq^x(T_{0\mathfrak b})\|_{H^{\mathbb{M}+2}}^2+\|\cc_\nq^x(T_{0\mathfrak b})\|_{H^{\mathbb{M}+3}}^2\leq& {A}^{-1},\quad \ 
 \| \lan n\ran^x(T_{0\mathfrak b})\|_{H^{\mathbb{M}+2}}^2+\| \lan \cc\ran^x (T_{0\mathfrak b})\|_{H^{ \mathbb{M}+3}}^2\leq \BB(\|n_{\mathrm{in}}\|_{H^{\mathbb{M}+3}},\|\cc_{\mathrm{in}}\|_{H^{\mathbb{M}+4}}).
\end{align}
As a consequence, the following estimate holds at $T_{0\mathfrak b}$,
\begin{align}\label{in_cnstr_bc}
\|(\lan \lan n\ran\ran^{y,z} -\overline{n})(T_{0\mathfrak b})\|_{H_{y,z}^{\mathbb{M}+2}}^2+\|\lan \lan \cc\ran\ran^{y,z}(T_{0\mathfrak b})\|_{H_{y,z}^{\mathbb{M}+3}}^2\leq {A}^{-1}. 
\end{align}
Here we recall that $\dss\lan \lan f\ran\ran^{y,z}(x)=\frac{1}{|\Torus|^2}\iint f(x,y,z)dydz.$
\end{lem} 

 \begin{proof}[Proof of Lemma \ref{lem:Cnfg_T0b}]
 Combining Proposition \ref{pro:prp_FM}, Proposition \ref{Pro:2}, we have that at time $T_{0\mathfrak{b}}$, \eqref{conf_T0b} holds. Moreover, we have that 
 \begin{align}
 \|(\lan \lan n\ran\ran^{y,z} -\overline{n})(T_{0\mathfrak b})\|_{H_{y,z}^{\mathbb{M}+2}}^2+\|\lan \lan \cc\ran\ran^{y,z}(T_{0\mathfrak b})\|_{H_{y,z}^{\mathbb{M}+3}}^2\leq \| n_\nq^x(T_{0\mathfrak b})\|_{H^{\mathbb{M}+2}}^2+\|\cc_\nq^x(T_{0\mathfrak b})\|_{H^{\mathbb{M}+3}}^2\leq& {A}^{-1}.
\end{align}   
This is estimate \eqref{in_cnstr_bc}.
 \end{proof}
 
\ifx
Through measuring the distance between the nonlinear solutions to the solutions to the passive scalar equations, we have that in phase $\mathfrak{a}$, the $x$-remainder $n_\nq^x,\, \cc_\nq^x$ undergoes enhanced dissipation, and the $x$-average $\lan n\ran^x,\, \lan \cc\ran^x$ do not grow too fast. As a result, we expect the following estimates at time $t=A^{1/3+\te}$
\begin{align}
\sum_{|i,j,k|=0}^{M}\|\pa_{x}^i\Gamma_{y;t}^j\pa_z^{k}n_\nq^x(A^{1/3+\te})\|_{L^2}+\sum_{|i,j,k|=0}^{M+1}\|\pa_{x}^i\Gamma_{y;t}^j\pa_z^{k}\cc_\nq^x (A^{1/3+\te})\|_{L^2}\leq \frac{C}{A^{1/12}};\label{Phase_1_nq}\\
\sum_{|i,j,k|=0}^{M}\|\pa_{ijk}\lan n\ran^x(A^{1/3+\te})\|_{L^2}+\sum_{|i,j,k|=0}^{M+1}\|\pa_{ijk}\lan \cc\ran^x(A^{1/3+\te})\|_{L^2}\leq C+C\ep.\label{Phase_1_average}
\end{align} 
We note that the $x$-average part of the solution does not decay in general. If nothing is changed at this moment, after time scale $\mathcal{O}(A)$, the nonlinear effect of the system will drive the growth of the chemical gradient $\na \lan \cc\ran^x$, which might create blow-up. Hence, we introduce the second phase of the control.
By rotating the shear direction by $\pi/2$, we obtain the time-dependent shear in the $z$-direction.
\fi
In phase $\mathfrak{b}$, we obtain similar estimates for $n_\nq^z,\, \lan n\ran^z,\, \cc_\nq^z,\, \lan \cc\ran^z$ as in phase $\mathfrak{a}$. Moreover, we can propagate the smallness estimate on $\lan \lan\cc \ran\ran^{y,z}=\mathbb P_{R_1}\cc$ \eqref{in_cnstr_bc}. These estimates  are summarized in the following lemma:
{
\begin{lem}\label{lem:Cnfg_T0c}
Consider the solution to the equation \eqref{ppPKS}. If the parameter $A$ is chosen large enough, then the following estimates hold
\begin{align} \label{conf_T0c_1}
\| n_\nq^z(T_{0\mathfrak c})\|_{H^{\mathbb{M}+1}}^2+\|\cc_\nq^z(T_{0\mathfrak c})\|_{H^{\mathbb{M}+2}}^2\leq& A^{-1};\quad
 \| \lan n\ran^z(T_{0\mathfrak c})\|_{H^{\mathbb{M}+1}}^2+\| \lan \cc\ran^z (T_{0\mathfrak c})\|_{H^{ \mathbb{M}+2}}^2\leq \BB(\|n_{\mathrm{in}}\|_{H^{\mathbb{M}+3}},\|\cc_{\mathrm{in}}\|_{H^{\mathbb{M}+4}}).
\end{align}
Moreover, the following estimates hold at $T_{0\mathfrak c}$,
\begin{align} \label{Ach_ab}
\| \lan \lan \cc\ran\ran^{y,z}(T_{0\mathfrak c})\|_{H^{\mathbb{M}+2}}^2\leq \| \lan \cc\ran^y(T_{0\mathfrak c})\|_{H^{\mathbb{M}+2}}^2\leq &\frac{\BB}{A^{1/3}}.
%
\end{align}
\end{lem} 
 \begin{proof}[Proof of Lemma \ref{lem:Cnfg_T0c}]
 Combining Proposition \ref{pro:prp_FM}, Proposition \ref{Pro:2}, we have that at time $T_{0\mathfrak{c}}$, \eqref{conf_T0c_1}  holds. Next we combine Proposition  \ref{pro:prp_FM}, Proposition \ref{Pro:2}, and Theorem \ref{thm:avg_yz}. As a result of an ODE argument \myc{(?) The explicit argument is as follows: Without loss of generality, we set $T_{0\mathfrak{b}}=0$ and $M=\mathbb{H}+1$, and invoke Theorem \ref{thm:avg_yz} to obtain that
\begin{align}
\frac{d}{dt}&\left(\sum_{i=0}^{ M}\|\pa_x^i  \lan \lan n\ran\ran^{y,z} (t)\|_{L_{y}^2}^2\right)\\
\leq &-\frac{1}{A}\sum_{i=0}^{M}\|\pa_x^{i+1}\lan\lan n\ran\ran^{y,z} \|_{L_y^2}^2+\frac{C}{A}{\|\lan n\ran^{z}(t)\|_{H_{x,y}^M}^2\left(\sup_{s\in[0, A^{1/3+\te}]}\|\lan  n \ran^{z}(s)\|_{H_{x,y}^{M}}^2 +\|\pa_x  \lan \cc_{\mathrm{in}}\ran^{z}\|_{H_{x,y}^M}^2\right)}\\
&+\frac{C}{A}\lf(\sum_{|i,j,k|=0}^{M+1}\|\Gamma_{x;t}^{ijk} \cc_\nq\|_{L_{x,y,z}^2}^2+   t^2  \sum_{ |i,j,k|=0}^{ M }\|\Gamma_{x;t}^{ij(k+1)} \cc_\nq\|_{L_{x,y,z}^2}^2\rg) \left(\sum_{|i,j,k|=0}^{M}\|\Gamma_{x;t}^{ijk}  n_\nq\|_{L^{2}_{x,y,z}}^2\right).
\end{align} Now we invoke the Proposition \ref{pro:prp_FM} and the linear enhanced dissipation Theorem \ref{thm:ED_gldrg} to obtain that 
\begin{align}
\frac{d}{dt} &\left(\sum_{i=0}^{ M}\|\pa_x^i  \lan \lan n\ran\ran^{y,z} (t)\|_{L_{y}^2}^2\right) \\
\leq &-\frac{1}{A}\sum_{i=0}^{M}\|\pa_x^{i+1}\lan\lan n\ran\ran^{y,z} \|_{L_y^2}^2+\frac{C}{A}\BB_{\lan n\ran^z, H^M}^2\left(\BB_{\lan n\ran^z, H^M}^2 +\|\pa_x  \lan \cc_{\mathrm{in}}\ran^{z}\|_{H_{x,y}^M}^2\right)\\
&+\frac{C(G)}{A\Phi^{4M+4}}\lf(\sum_{|i,j,k|=0}^{M+1}G^{2k}\Phi^{2i+2}\|\Gamma_{x;t}^{ijk} c_\nq\|_{L_{x,y,z}^2}^2+\frac{t^2}{A^{2/3}}\sum_{ |i,j,k|=0}^{ M } A^{2/3}G^{2k}\Phi^{2i+2} \|\Gamma_{x;t}^{ij(k+1)} c_\nq\|_{L_{x,y,z}^2}^2\rg) \mathbb{H}_{G}^M\\ 
&+\frac{C(G)}{A^{1/3}\Phi^{4M+6}}\lf(\sum_{|i,j,k|=0}^{M+1}G^{2k}\Phi^{2i}\|\Gamma_{x;t}^{ijk} S_0^t \cc_\nq(0)\|_{L_{x,y,z}^2}^2+ \sum_{ |i,j,k|=0}^{ M }  G^{2k}\Phi^{2i} \|\Gamma_{x;t}^{ij(k+1)} S_0^t\cc_\nq(0)\|_{L_{x,y,z}^2}^2\rg) \mathbb{H}_{G}^M\\
\lesssim&-\frac{1}{A}\sum_{i=0}^{M}\|\pa_x^{i+1}\lan\lan n\ran\ran^{y,z} \|_{L_y^2}^2+\frac{C}{A}\BB_{\lan n\ran^z, H^M}^2\left(\BB_{\lan n\ran^z, H^M}^2 +\|\pa_x  \lan \cc_{\mathrm{in}}\ran^{z}\|_{H_{x,y}^M}^2\right) +\frac{C(G)}{A\Phi^{4M+4}}(\mathbb{H}_{G}^M)^2\\ 
&+\frac{C(G)}{A^{1/3}\Phi^{4M+6}}\exp\lf\{-\frac{\delta_{\mathcal Z} t}{A^{1/3}} \rg\} \mathbb{H}_{G}^M\sum_{|i,j,k|=0}^{M+1}G^{2k} \|\Gamma_{x;0}^{ijk} \cc_\nq(0)\|_{L_{x,y,z}^2}^2.
\end{align}
Now recalling the bound \eqref{est_pro}, we integrate the above expression in time and take $A$ large to gain the result.  
 }, we observe that 
 if $A$ is chosen large enough,
\begin{align*}
\|\lan \lan n\ran\ran^{y,z}(T_{0\mathfrak c})\|_{H^{\mathbb M+1}}\leq C(\|n_{\mathrm{in}}\|_{H^{\mathbb M+3}},\|\cc_{\mathrm{in}}\|_{H^{\mathbb M+4}}).
\end{align*} Now Lemma \ref{Lem:gc0Lpest} yields the  estimate 
\begin{align} \label{conf_T0c}
\| \lan \lan \cc\ran\ran^{y,z}(T_{0\mathfrak c})\|_{H^{\mathbb M+2}}^2\leq &C\| \lan \lan \cc\ran\ran^{y,z}(T_{0\mathfrak b})\|_{H^{\mathbb M+2}}^2 +{A^{-1/3}}\leq {C}{A^{-1/3}}.
\end{align}
Since $$\|\lan \cc\ran^y\|_{H^{\mathbb M+2}}\leq \|\lan\lan \cc\ran\ran^{y,z}\|_{H^{\mathbb M+2}}+\|(\lan \cc\ran^y)_\nq^z\|_{H^{\mathbb M+2}}\leq \|\lan\lan \cc\ran\ran^{y,z}\|_{H^{\mathbb M+2}}+\|\cc_\nq^z\|_{H^{\mathbb M+2}}, $$
 the estimates \eqref{conf_T0c_1} and \eqref{conf_T0c} imply that \eqref{Ach_ab} holds.
\ifx
Combining it with the estimates,  \eqref{conf_T0c_2}, we have that
\begin{align}
\| (\lan n\ran^y-\overline n)(T_{0\mathfrak c})\|_{H^{M+1}}\leq C(\|n_{\mathrm{in}}\|_{H^{M+3}},\|\cc_{\mathrm{in}}\|_{H^{M+4}}).
\end{align}
 and.\fi
 \end{proof}
The main achievement in the phases $\mathfrak a, \mathfrak b$ is that we obtain smallness in the $y$-average component at the beginning of phase $\mathfrak c$, i.e., $\|\lan \cc\ran^y(T_{0\mathfrak c})\|_{H^{\mathbb{M}+2}}=\|\mathbb{P}_{R_1\cup R_2}\cc(T_{0\mathfrak c})\|_{H^{\mathbb{M}+2}}\ll 1$.
}
Hence we have nice estimate at time instance $T_{1\mathfrak a}$:
\begin{lem}\label{lem:Cnfg_T1a}
Consider the solution to the equation \eqref{ppPKS}. If the parameter $A$ is chosen large enough, then the following estimates hold at time instance $T_{1\mathfrak a}$
\begin{align} \label{conf_T1a}
\|(n-\overline {n})(T_{1\mathfrak a})\|_{H^\mathbb{M}}^2\leq&\BB(\|n_{\mathrm{in}}\|_{H^{\mathbb{M}+3}},\|\cc_{\mathrm{in}}\|_{H^{\mathbb{M}+4}});\quad \
\|\cc(T_{1\mathfrak a})\|_{H^{\mathbb{M}+1}}^2\leq \frac{\BB(\|n_{\mathrm{in}}\|_{H^{\mathbb{M}+3}},\|\cc_{\mathrm{in}}\|_{H^{\mathbb{M}+4}})}{A^{1/3}}.
\end{align}
\end{lem}

 \begin{proof}[Proof of Lemma \ref{lem:Cnfg_T1a}]
 Combining Proposition \ref{pro:prp_FM}, Proposition \ref{Pro:2}, Theorem \ref{thm:avg_yz} and Lemma \ref{Lem:gc0Lpest}
 , we have that at time $T_{1\mathfrak{a}}$, \eqref{conf_T1a} holds.  \end{proof}
 
At this point ($t=T_{1\mathfrak a}$), the norm of the chemical density is small, and we can derive Theorem \ref{thm:alt_2}.\ifx apply similar estimate to that
\begin{lem}\label{lem:I_geq_1}
Consider the solution to the equation \eqref{PKS_rsc}. If the parameter $A$ is chosen large enough, then the following estimates hold 
\begin{align}
\|(n-\overline {n})(T_{1\mathfrak a}+t)\|_{H^\mathbb{M}}^2+\|\cc(T_{1\mathfrak a}+t)\|_{H^{\mathbb{M}+1}}^2\leq& C(\|n_{\mathrm{in}}\|_{H^{\mathbb{M}+3}}^2+\|\cc_{\mathrm{in}}\|_{H^{\mathbb{M}+4}}^2)e^{-\delta t/A^{1/3+\te}},\quad\forall t\geq 0.
\end{align}Here $C,\, \delta$ are universal constants. 
As a result, the solution is globally regular.
\end{lem}
This is \fi
 

 
 \begin{proof}[Proof of Theorem \ref{thm:alt_2}]
Application of  Proposition \ref{Pro:main} and the argument \eqref{ED_argument} yields the result. 
 \end{proof}
\ifx
\textbf{Previous argument:} 
We have the following favorable constraint:
\begin{align}
\|\lan n\ran^y(2A^{1/3+\te})\|_{H^{m}}\leq \|\lan\lan n\ran^y\ran^z(2A^{1/3+\te})\|_{H^m}+\|n_\nq^z(2A^{1/3+\te})\|_{H^m}\leq\frac{C}{A^{1/12}}+C(\|n_{\mathrm{in}}\|_{H^{M+3}},\|\cc_{\mathrm{in}}\|_{H^{M+4}});\\
\|\lan \cc\ran^y(2A^{1/3+\te})\|_{H^{m+1}}\leq\|\lan \lan \cc\ran^y\ran^z(2A^{1/3+\te})\|_{H^{m+1}}+\| \cc_{\nq}^z(2A^{1/3+\te}) \|_{H^{m+1}}\leq\frac{C}{A^{1/12}}.
\end{align}
Now we can use the same argument as in Phase $\mathfrak{a}$ to estimate the solutions. At the end of phase $\mathfrak{a},\mathfrak{b},\mathfrak{c}$, we have the following
\begin{align}
\|n-\overline{n}\|_{H^m}\leq\frac{C}{A^{1/12}}+C( \|\cc_{\mathrm{in}}\|_{H^{M+4}});\\
\|\cc\|_{H^{m+1}}\leq\frac{C}{A^{1/12}}.
\end{align}\fi
\appendix
\section{General Lemmas}

\begin{lem}\label{lem:F_0andF}
Let $F$ be in $C^m(\Torus^3)$, and define $
\lan F\ran(y,z):=\frac{1}{|\mathbb{T}|}\int_{\mathbb{T}}F(x,y,z)dx,\ 
\lan\lan F\ran\ran(z):=\frac{1}{|\mathbb{T}|^2}\int_{\mathbb{T}^2}F(x,y,z)dxdy.$ The following estimates hold:
\begin{align}\label{F_0 and F}
\|\lan\lan F\ran\ran\|_{L_z^p(\Torus)}\leq&  \|\lan F\ran\|_{L^p_{y,z}(\Torus^2)}\leq\|F\|_{L^p_{x,y,z}(\mathbb{T}^3)},\quad 1\leq p\leq\infty,\\
\sum_{k=0}^m\|\pa_z^k \lan\lan F\ran\ran\|_{L^2_z(\Torus)}\leq&
\sum_{k=0}^m\|\pa_z^k\lan F\ran\|_{L^2_{y,z}(\Torus^2)}\leq \sum_{k=0}^m\|\pa_z^k F\|_{L_{x,y,z}^2(\mathbb{T}^3)}.
\end{align}
\end{lem}

\begin{proof}

Applying the H\"{o}lder's inequality yields that for $p\in(1,\infty)$,
\begin{align*}
\|\lan\lan F\ran\ran\|_{L^p_z(\Torus)}=&\left(\int_{\Torus} \bigg|\frac{1}{|2\pi| }\int_{\mathbb{T}}\lan F\ran dy\bigg|^p dz\right)^{1/p}\leq\left(\int_\Torus \left(\left(\int_{\mathbb{T}}|\lan F\ran|^pdy\right)^{1/p}\left(\int_{\mathbb{T}}|2\pi|^{-p'} dy\right)^{1/p'}\right)^p dz\right)^{1/p}\\
\leq&\left(\int_{\Torus^2} |\lan F\ran|^pdydz\right)^{1/p}=\|\lan F\ran\|_{L_{y,z}^p(\mathbb{T}^2)}= \left(\int_{\Torus^2} \bigg|\frac{1}{|2\pi| }\int_{\mathbb{T}} F dx\bigg|^p dydz\right)^{1/p}\\\leq&\left(\int_{\Torus^2} \left(\left(\int_{\mathbb{T}}| F|^pdx\right)^{1/p}\left(\int_{\mathbb{T}}|2\pi|^{-p'} dx\right)^{1/p'}\right)^p dydz\right)^{1/p}\leq\|F\|_{L_{x,y,z}^p}.
\end{align*}
Here $\frac{1}{p}+\frac{1}{p'}=1$. The $p=1$ case is a simple variant of the argument above. The estimate in the $p=\infty$ case is a  direct consequence of definition. We observe that $\pa_z^k\lan \lan F\ran\ran=\lan \pa_z^k \lan F\ran\ran =\lan\lan \pa_z^k F\ran\ran$. Hence, the estimate above yields that
\begin{align}
\sum_{k=0}^m\|\pa_z^k\lan\lan F\ran\ran\|_{L^p_z(\Torus)}=\sum_{k=0}^m\|\lan\pa_z^k\lan  F\ran\ran\|_{L_z^p(\Torus)}\leq \sum_{k=0}^m\| \pa_z^k\lan F\ran\|_{L^p_{y,z}(\Torus^2)}=\sum_{k=0}^m\|\lan \pa_z^k F\ran\|_{L^p_{y,z}(\Torus^2)}\leq \sum_{k=0}^m\| \pa_z^k F\|_{L^p_{x,y,z}(\Torus^3)}.
\end{align}
\ifx
Applying the Fourier transform and the Plancherel equality yields
\be\ba
\|\lan \pa_y\ran^s F_0\|_{L^2(\rr)}^2=C\int_\rr \lan\eta\ran^{2s}|\widehat{F_0}|^2(\eta)d\eta
\leq \sum_k \int_\rr \lan k,\eta\ran^{2s}|\widehat{F}|^2(k,\eta)d\eta=C\|\lan \pa_z,\pa_y\ran^s F\|_{L^2(\mathbb{T}\times\rr)}^2.
\ea\ee
\fi
This finishes the proof of the lemma.
\end{proof}
\begin{lem}\label{Lem:gc0Lpest} 
Consider solution to the heat equation on $\Torus^d,\, d=1,2$,
\begin{align}
\pa_t \mathfrak{c}=\frac{1}{A}\de\mathfrak{c}+\frac{1}{A}\lf(\mathfrak{n}-\overline{\mathfrak n}\rg),\quad \mathfrak{c}(t=0, X)= \mathfrak{c}_{\mathrm{in}}(X),\quad \int_{\Torus^d} \mathfrak{c}_{\mathrm{in}}dX=0.\label{Heat_eq}
\end{align}
\myc{(Double check whether the formulation is correct for $\lan \cc\ran, \lan\lan\cc\ran\ran$!) Answer: Here the key is to check that $\overline{n}=\frac{1}{|\Torus|}\int \lan\lan n \ran\ran^{x,z} dy,\overline{n}=\frac{1}{|\Torus|^2}\int \lan n \ran^{x} dydz.$ They are true.} Then the following estimate holds \myc{for$\quad \frac{1}{q}<\frac{1}{d}+\frac{1}{p},\,1\leq q\leq p\leq \infty$},
\begin{align}\label{na_c_0_est}
\|\na^{m+1}  \mathfrak c(t)\|_{L^2(\Torus^d)}\leq  C \lf(\frac{t}{A}\rg)^{\frac{1}{2}}\sup_{0\leq\tau\leq t}\|\na^m(\mathfrak n(\tau)-\overline{\mathfrak n})\|_{L^2(\Torus^d)}+\|\na^{m+1}\mathfrak c_{\mathrm{in}}\|_{L^2(\Torus^d)}.
\end{align}\myc{
\begin{align}
\|\na^{m+1}  \mathfrak c(t)\|_{L^p(\Torus^d)}\leq&\frac{C}{A}\int_{0}^t\lf (\frac{t-\tau}{A}\rg)^{-\frac{1}{2}-\frac{d}{2}(\frac{1}{q}-\frac{1}{p})} e^{-\lambda \frac{t-\tau}{3A}}\|\na^m(\mathfrak n(\tau)-\overline{\mathfrak n})\|_{L^q(\Torus^d)}d\tau+\|\na^{m+1}\mathfrak c_{\mathrm{in}}\|_{L^p(\Torus^d)}\\
\leq & C\min\lf\{1,\lf(\frac{t}{A}\rg)^{\frac{1}{2}-\frac{d}{2}(\frac{1}{q}-\frac{1}{p})}\rg\}\sup_{0\leq\tau\leq t}\|\na^m(\mathfrak n(\tau)-\overline{\mathfrak n})\|_{L^q(\Torus^d)}+\|\na^{m+1}\mathfrak c_{\mathrm{in}}\|_{L^p(\Torus^d)}.
\end{align}
Here $\lambda\myc{\,=1?}>0$ is the first nonzero eigenvalue of $-\de_{\Torus^d}$. }

\myc{Both the 1 D and 2 D cases are needed. Moreover, we need the $H^M$-version.  }
\end{lem}
\begin{proof}
By applying the Fourier transform, and the Plancherel equality, we have that
\begin{align}\label{Winklersgp2}
\|\na e^{\frac{t}{A}\de_{\Torus^d }}f\|_2\leq C \left(\frac{t}{A}\right)^{-\frac{1}{2} } \|f\|_2,\quad \forall t>0.  
\end{align} 
Now we have that
\begin{align}
  \|\na^{m+1} \mathfrak c(t)\|_{2}
\leq &
  \lf\|\exp\lf\{\frac{t}{A}\de_{\Torus^d}\rg\}\na ^{m+1} \mathfrak c(0)\rg\|_2+\frac{1}{A }
\lf\| \int_0^t \na \exp\lf\{\frac{t-s}{A}\de_{\Torus^d} \rg \}\na^m   (\mathfrak n(s)-\overline{\mathfrak n}) ds\rg\|_2\\
\leq& \|\na^{m+1}\mathfrak c_{\text{in}}\|_2+\frac{1}{A}
\int_0^{t }\lf\|\na\exp\lf\{\frac{t-s}{A}\de_{\Torus^d} \rg \}\rg\|_{L^2\rightarrow L^2}\sup_{s\in[0,t)}\|\na^{m} (\mathfrak n(s)-\overline {\mathfrak n})\|_{2} ds\\
\leq & \|\na^{m+1} \mathfrak c_{\text{in}}\|_2+C \lf(\frac{t}{A}\rg)^{\frac{1}{2}}\sup_{0\leq s\leq t}\|\na^m(\mathfrak n(s)-\overline{\mathfrak n})\|_{2}.
\end{align}
\end{proof}
\myc{
\begin{proof}The proof is similar to the proof of heat semigroup bounds in \cite{Winkler10}. 
Since the spacial derivatives commute with the differential equation \eqref{Heat_eq}, it is enough to derive the estimate \eqref{na_c_0_est_1} for $m=0$. Consider the heat equation 
\begin{align}
\pa_t h=\frac{1}{A}\de h, \quad h(t=0,X)=h_{\text{in}}(X),\quad  \overline{h_{\text{in}}}=0,\quad X\in\Torus^d.
\end{align}
Our first goal is to derive the following estimates

\noindent
a) If $d=1,2,\quad 1\leq q\leq p\leq \infty$, then 
\begin{align}\label{Winklersgp1}
\|e^{\frac{t}{A}\de_{\Torus^d }}h_{\text{in}}\|_p\leq C \left(\frac{t}{A}\right)^{-\frac{d}{2}(\frac{1}{q}-\frac{1}{p})} e^{- \frac{t}{2A}}\|h_{\text{in}}\|_q,\quad \forall t>0,
\end{align} 
holds for all $h_{\text{in}}\in L^q$ with $\overline{h_{\text{in}}}=0$;

\noindent
b) If $d=1,2,\quad 1\leq q\leq p \leq \infty$, then 
\begin{align}\label{Winklersgp2}
\|\na e^{\frac{t}{A}\de_{\Torus^d }}h_{\text{in}}\|_p\leq C \left(\frac{t}{A}\right)^{-\frac{1}{2}-\frac{d}{2}(\frac{1}{q}-\frac{1}{p})} e^{-\frac{t}{3A}}\|h_{\text{in}}\|_q,\quad \forall t>0.
\end{align}  

Application of a standard energy estimate yields that 
\begin{align}
\|e^{\frac{t}{A}\de}h_{\text{in}}\|_{L^2(\Torus^d)}\leq e^{-\frac{ t}{A} }\|h_{\text{in}}\|_{L^2(\Torus^d)},\quad \forall t\geq 0.\label{heat_L2L2}
\end{align} 

On the other hand, we recall the explicit formula for the Green's function of the heat semigroup
\begin{align}e^{\frac{t}{A}\de}h_\text{in}(X)=&\sum_{m_1\in \mathbb{Z}}\frac{1}{\sqrt{4\pi t/A} }\int_{X'\in \Torus}e^{-\frac{|X'+2m_1 \pi-X|^2}{4 t/A}}h_{\text{in}}(X')dX',\quad d=1;\\
e^{\frac{t}{A}\de}h_{\text{in}}(X)=&\sum_{(m_1, m_2)\in \mathbb{Z}^2}\frac{1}{4\pi t/A}\int_{X'\in \Torus^2}e^{-\frac{|X'+(2m_1 \pi, 2m_2 \pi)-X|^2}{4 t/A}}h_{\text{in}}(X')dX',\quad d=2.\label{heat_ker}
\end{align}Direct estimate yields that for $1\leq q\leq p\leq \infty$, 
\begin{align}\label{heat_Lp_Lq}
\|e^{\frac{t}{A}\de} h_{\text{in}}\|_{L_X^p(\Torus^d)}\leq& C \left(\frac{t}{A}\right)^{-\frac{d}{2}(\frac{1}{q}-\frac{1}{p})}\| h_{\text{in}}\|_{L_X^q(\Torus^d)},\quad \forall t\in[0, 2A];\\
\label{heatgradLpLq}\|\na e^{\frac{t}{A}\de} h_{\text{in}}\|_{L_X^p(\Torus^d)}\leq & C\left(\frac{t}{A}\right)^{-\frac{1}{2}}\|h_{\text{in}}\|_{L_X^p(\Torus^d)},\quad \forall t\in[0, 2A] .
\end{align}\myc{\textcolor{red}{(Check the index $??$!)}  {\bf(Check the second inequality!)} Method 1: We consider the function $|\na e^{\frac{t}{A}\de}h_{\text{in}}(X)|$ and it is bounded by the following expression. To justify the expression of the $\na e^{t\de/A}$ convergence, we can think of periodizing the $h_{\text{in}}\in L^\infty$ function to the whole $\rr^2$, then since the heat kernel decays fast near infinity, we can take the derivative on the $X$ variable, then we rewrite the expression as a sum over $(m_1,m_2)\in \mathbb{Z}^2$.)
\begin{align}
|\na e^{\frac{t}{A}\de}h_{\text{in}}(X)|\leq &\sum_{(m_1, m_2)\in \mathbb{Z}^2}\frac{1}{4\pi t/A}\int_{X'\in \Torus^2}\bigg|\na_{X} e^{-\frac{|(X'-X)+(2m_1 \pi, 2m_2 \pi)|^2}{4 t/A}}\bigg||h_{\text{in}}(X')|dX'\\
\underbrace{=}_{Fubini-Tonelli}&\frac{1}{4\pi t/A}\int_{X'\in \Torus^2}\underbrace{\sum_{(m_1, m_2)\in \mathbb{Z}^2}\bigg|\frac{(X'-X)+(2m_1 \pi, 2m_2 \pi)}{2t/A} e^{-\frac{|(X'-X)+(2m_1 \pi, 2m_2 \pi)|^2}{4 t/A}}\bigg|}_{=:F(t,A,X-X')}|h_{\text{in}}(X')|dX';\\
\|\na e^{\frac{t}{A}\de}h_{\text{in}}(\cdot)\|_{L^p(\Torus^d)}\underbrace{\leq}_{Young}&\frac{C}{\sqrt{t/A}}\lf\|\frac{1}{\pi (2 t/A)^{3/2}}F(t,A,\cdot)\rg\|_{L^1}\|h_{\text{in}}\|_{L^p(\Torus^d)}\leq \frac{C}{\sqrt{t/A}} \|h_{\text{in}}\|_{L^p}.
\end{align}
Method 2: 
In this paper, we just use the $L^2$ version of the theorem. Hence we can directly implement a Fourier proof:
\begin{align}
\|\na e^{\frac{t}{A}\de} h_{\mathrm{in}}\|_{L^2(\Torus^d)}^2\leq& C\sum_{k\in \mathbb{Z}^d\backslash 0^d}|k|^2 e^{-\frac{t}{A}|k|^2}|\wh{h_{\mathrm{in}}}(k)|^2\leq \frac{CA }{t}\sum_{k\in \mathbb{Z}^d\backslash 0^d}\frac{t|k|^2}{A} e^{-\frac{t}{A}|k|^2}|\wh{h_{\mathrm{in}}}(k)|^2\\
\leq&  \frac{CA }{t}\sum_{k\in \mathbb{Z}^d\backslash 0^d}|\wh{h_{\mathrm{in}}}(k)|^2\leq C\lf(\frac{t}{A}\rg)^{-1}\|h_{\mathrm{in}}\|_{L^2(\Torus^d)}^2.
\end{align} Now $\|\na e^{t\de/A}\|_{L^2\rightarrow L^2}\leq C(t/A)^{-1/2}$. Then we can say that from the expression of $\na e^{t\de/A}$, the $L^\infty-L^\infty$ semigroup estimate has bound $C(\frac{t}{A})^{-1/2}$ (No convolution structure is needed). Now an operator interpolation estimate yields the semigroup estimate for $p\in [2,\infty].$ Then we can use the duality to get the $p\in[1,2)$: Let $p^{-1}+(p')^{-1}=1, \, p'>2$ and $f,g\in C^\infty(\Torus^d), \, \overline{f}=\overline{g}=0$, then 
\begin{align}\|\na e^{t\de/A} f\|_{L^p}=\sup_{\|g\|_{p'}=1}
\int_{\Torus^d} \na e^{t\de/A}f  g dV=\sup_{\|g\|_{p'}=1}
\int_{\Torus^d}f   \na e^{t\de/A}g dV\leq \|f\|_{p}\sup_{\|g\|_{p'}=1}\| \na e^{t\de/A}g\|_{p'}\leq C\|f\|_{p}(t/A)^{-1/2}.
\end{align} 
Hence we have obtained the estimate.}
For $p\in[1,2),\, q\in[1, p]$, application of the H\"older inequality and then $\eqref{heat_Lp_Lq}_{p=2}$ yields that
\begin{align}
\|h(t)\|_{L^p(\Torus^d)}\leq C\|h(t)\|_{L^2(\Torus^d)}\leq C e^{-\frac{t-A}{A}}\|h(A)\|_{L^2}\leq Ce^{-(\frac{t}{A}-1)}\|h_{\text{in}}\|_{L^q},\quad\forall t\geq 2A. 
\end{align}
If $p\in [2,\infty],\,  q\in[1, p]$, we have that by the semigroup property of $e^{\frac{t}{A}\de}$ and the estimates $\eqref{heat_Lp_Lq}_{q=2}$, \eqref{heat_L2L2}, 
\begin{align}
\|h(t)\|_{L^p(\Torus^d)}=&\|e^{\frac{A}{A}\de}h(t-A)\|_{L^p(\Torus^d)}\leq C\|h(t-A)\|_{L^2(\Torus^d)}=C\|e^{\frac{t-2A}{A}\de}h(A)\|_{L^2(\Torus^d)}\\
\leq &Ce^{-\frac{t-2A}{A}}\|h(A)\|_{L^2(\Torus^d)},\quad \forall t\geq 2A.
\end{align}
Now we apply the H\"older inequality ($q\leq 2$) or the estimate $\eqref{heat_Lp_Lq}_{p=2}$ ($q> 2$) to obtain
\begin{align}
\|h(t)\|_{L^p(\Torus^d)}\leq Ce^{- \frac{t}{A} }\|h_{\text{in}}\|_{L^q(\Torus^d)},\quad \forall t\geq 2A.
\end{align} Combining this estimate with \eqref{heat_Lp_Lq}, we have that for $1\leq q\leq p\leq \infty$, 
\begin{align}
\|e^{\frac{t}{A}\de}h_{\text{in}}\|_p\leq C\left(\left(\frac{t}{A}\right)^{-\frac{d}{2}(\frac{1}{q}-\frac{1}{p})}\mathbbm{1}_{t\leq 2A} +e^{- \frac{t}{A}}\right)\|h_{\text{in}}\|_q\leq C\left(\frac{t}{A}\right)^{-\frac{d}{2}(\frac{1}{q}-\frac{1}{p})}e^{- \frac{t}{2A}}\|h_{\text{in}}\|_q,\quad \forall t>0.
\end{align}
Hence the proof of \eqref{Winklersgp1} is completed. 

We proceed to prove \eqref{Winklersgp2}. 
For $t\leq 2A$, we combine the semigroup property and the estimates \eqref{Winklersgp1}, 
\eqref{heatgradLpLq}  to obtain the following gradient bound  
\begin{align}
\|\na e^{\frac{t}{A}\de}h_{\text{in}}\|_{L^p(\Torus^d)}=&\|\na h(t)\|_{L^p(\Torus^d)}=\lf\|\na e^{\frac{t}{2A}\de}\lf( h\lf(\frac{t}{2}\rg)\rg)\rg\|_{L^p(\Torus^d)}\leq C\left(\frac{t}{2A}\right)^{-1/2}\lf\|h\lf(\frac{t}{2	}\rg)\rg\|_{L^p(\Torus^d)}\\
\leq& C\left(\frac{t}{2A}\right)^{-\frac{1}{2}}\left(\frac{t}{2A}\right)^{-\frac{d}{2}(\frac{1}{q}-\frac{1}{p})}\|h_{\text{in}}\|_{L^q(\Torus^d)}, \quad 1\leq q\leq p\leq\infty. \label{Winklersgp21}
\end{align}
For $t\geq 2A$, we estimate the gradient with estimates  \eqref{Winklersgp1},  \eqref{heatgradLpLq} as follows:
\begin{align}
\|\na e^{\frac{t}{A}\de}h_{\text{in}}\|_{L^p(\Torus^d)}=&\|\na h(t)\|_{L^p(\Torus^d)}=\|\na e^{\frac{A}{A}\de} (h(t-A))\|_{L^p(\Torus^d)}\leq C\|h({t}-{A})\|_{L^p(\Torus^d)}\\
\leq&C\left(\frac{t}{A}\right)^{-\frac{d}{2}(\frac{1}{q}-\frac{1}{p})}e^{-\frac{1}{2} (\frac{t}{A}-1)}\|h_{\text{in}}\|_{L^q(\Torus^d)}. \label{Winklersgp22}
\end{align}
Combinbing \eqref{Winklersgp21} and \eqref{Winklersgp22} yields that
\begin{align*}
\|\na e^{\frac{t}{A}\de}h_{\text{in}}\|_{L^p(\Torus^d)}\leq& C\lf(\left(\frac{t}{A}\right)^{-\frac{1}{2}}\mathbbm{1}_{t\leq 2A}+e^{-\frac{1}{2} \frac{t}{A}}\mathbbm{1}_{t>2A}\rg)\left(\frac{t}{A}\right)^{-\frac{d}{2}(\frac{1}{q}-\frac{1}{p})}\|h_{\text{in}}\|_{L^q(\Torus^d)}\\
\leq& C\left(\frac{t}{A}\right)^{-\frac{1}{2}-\frac{d}{2}(\frac{1}{q}-\frac{1}{p})}e^{-\frac{t}{3A}}\|h_{\text{in}}\|_{L^q(\Torus^d)}.
\end{align*}
This completes the proof of \eqref{Winklersgp2}.

Next we observe that from \eqref{Heat_eq}, 
\begin{align}
\pa_t \pa_X^m\mathfrak{c}=\frac{1}{A}\de\pa_X^m\mathfrak{c}+\frac{1}{A}\pa_X^m\lf(\mathfrak{n}-\overline{\mathfrak n}\rg).
\end{align}
The Duhamel formulation yields that
\begin{align*}
\|\pa_X^{m+1}\mathfrak{c}(t)\|_{L^p}\leq& \|e^{\frac{t}{A}\de}\pa_X^{m+1}\mathfrak{c}_{\text{in}}\|_{L^p}+\frac{1}{A}\lf\|\int_0^t e^{\frac{t-s}{A}\de}\lf(\pa_X^{m+1}(\mathfrak{n}(\tau)-\overline{\mathfrak n})\rg)d\tau \rg\|_{L^p}\\
=&\|e^{\frac{t}{A}\de}\pa_X^{m+1}\mathfrak{c}_{\text{in}}\|_{L^p}+\frac{1}{A}\lf\|\int_0^t \pa_X e^{\frac{t-s}{A}\de}\lf(\pa_X^{m}(\mathfrak{n}(\tau)-\overline{\mathfrak n})\rg)d\tau \rg\|_{L^p}.
\end{align*}
Now we invoke the estimates \eqref{Winklersgp1} and \eqref{Winklersgp2} to obtain that
\begin{align}
\|\pa_X^{m+1}\mathfrak{c}(t)\|_{L^p}\leq C\|\pa_X^{m+1}\mathfrak{c}_{\text{in}}\|_{L^p}+\frac{C}{A}\int_0^t\lf (\frac{t-\tau}{A}\rg)^{-\frac{1}{2}-\frac{d}{2}(\frac{1}{q}-\frac{1}{p})} e^{- \frac{t-\tau}{3A}}\|\pa_X^m(\mathfrak n(\tau)-\overline{\mathfrak n})\|_{L^q(\Torus^d)}d\tau
\end{align}
If $ \frac{1}{q}<\frac{1}{d}+\frac{1}{p}$, then,
\begin{align}
\|\pa_X^{m+1}\mathfrak{c}(t)\|_{L^p}\leq& C\|\pa_X^{m+1}\mathfrak{c}_{\text{in}}\|_{L^p}+C\sup_{0\leq\tau\leq t}\|\pa_X^m(\mathfrak n(\tau)-\overline{\mathfrak n})\|_{L^q(\Torus^d)}\int_0^t\lf (\frac{t-\tau}{A}\rg)^{-\frac{1}{2}-\frac{d}{2}(\frac{1}{q}-\frac{1}{p})} e^{- \frac{t-\tau}{3A}}\frac{1}{A}d\tau\\
\leq &C\|\pa_X^{m+1}\mathfrak{c}_{\text{in}}\|_{L^p}+C\sup_{0\leq\tau\leq t}\|\pa_X^m(\mathfrak n(\tau)-\overline{\mathfrak n})\|_{L^q(\Torus^d)}\int_{\max\{0,t-A\}}^t\lf (\frac{t-\tau}{A}\rg)^{-\frac{1}{2}-\frac{d}{2}(\frac{1}{q}-\frac{1}{p})} \frac{1}{A}d\tau\\
&+C\sup_{0\leq\tau\leq t}\|\pa_X^m(\mathfrak n(\tau)-\overline{\mathfrak n})\|_{L^q(\Torus^d)}\int_0^{\max\{0,t-A\}} e^{- \frac{t-\tau}{3A}}\frac{1}{A}d\tau\\
\leq&C\|\pa_X^{m+1}\mathfrak{c}_{\text{in}}\|_{L^p}+C\sup_{0\leq\tau\leq t}\|\pa_X^m(\mathfrak n(\tau)-\overline{\mathfrak n})\|_{L^q(\Torus^d)}.\label{na_c_0_est_1}
\end{align}Another way to estimate the $L^p$-norm is that, 
\begin{align}
\|\pa_X^{m+1}\mathfrak{c}(t)\|_{L^p}\leq& C\|\pa_X^{m+1}\mathfrak{c}_{\text{in}}\|_{L^p}+C\sup_{0\leq\tau\leq t}\|\pa_X^m(\mathfrak n(\tau)-\overline{\mathfrak n})\|_{L^q(\Torus^d)}\int_0^t\lf (\frac{t-\tau}{A}\rg)^{-\frac{1}{2}-\frac{d}{2}(\frac{1}{q}-\frac{1}{p})} \frac{1}{A}d\tau\\
\leq& C\|\pa_X^{m+1}\mathfrak{c}_{\text{in}}\|_{L^p}+C\sup_{0\leq\tau\leq t}\|\pa_X^m(\mathfrak n(\tau)-\overline{\mathfrak n})\|_{L^q(\Torus^d)} \lf(\frac{t}{A}\rg)^{\frac{1}{2}-\frac{d}{2}(\frac{1}{q}-\frac{1}{p})}, \quad\frac{1}{q}<\frac{1}{d}+\frac{1}{p} .\label{na_c_0_est_2}
\end{align}
Combining \eqref{na_c_0_est_1} and \eqref{na_c_0_est_2} yields \eqref{na_c_0_est}.
\ifx
Previous argument:
\textcolor{red}{The theorem requires reproving! The idea is that we can periodically extend the $c_0$ to $\rr^2$ and then use the heat semigroup on $\rr^2$ to get the $L^p$ to $L^q$ estimate. Then we can use the spectral gap of the $L^2\rightarrow L^2$-estimate of the $e^{t\de_{\Torus^2}/A}$.}

The estimates of the $L^q$-norm of the chemical gradient $\|\na \lan c\ran^\iota\|_q$ \eqref{na_c_0_est} are natural consequences of the following two estimates
\noindent
a) If $1\leq q\leq p\leq \infty$, then 
\begin{align}\label{Winkler_semigroup_1}
\|e^{\frac{t}{A}\de}h_{\text{in}}\|_p\leq C_1\left(1+\left(\frac{t}{A}\right)^{-\frac{2}{2}(\frac{1}{q}-\frac{1}{p})}\right)e^{-\lambda_1 \frac{t}{A}}\|h_{\text{in}}\|_q,\quad \forall t>0,
\end{align} 
holds for all $w\in L^q$ with $\int w=0$;

\noindent
b) If $1\leq q\leq p \leq \infty$, then 
\begin{align}\label{Winkler_semigroup_2}
\|\na e^{\frac{t}{A}\de}h_{\text{in}}\|_p\leq C_2\left(1+\left(\frac{t}{A}\right)^{-\frac{1}{2}-\frac{n}{2}(\frac{1}{q}-\frac{1}{p})}\right)e^{-\lambda_1 t/A}\|h_{\text{in}}\|_q,\quad \forall t>0.
\end{align} 
Here $\lambda_1>0$ is the first nonzero eigenvalue of the negative Laplacian $-\de$. 

Combining the Duhamel formulation and the estimates \eqref{Winkler_semigroup_1}, \eqref{Winkler_semigroup_2}, one can integrate in time to derive \eqref{na_c_0_est}. 
The proof of the two inequalities can be found in  Lemma 1.3 in \cite{Winkler10}. In the paper, the estimates are proved that for bounded domain with Neumann boundary condition. The Green's function estimate is crucial. However, on the torus, we can explicitly write down the Green' function of the heat kernel and check all the conditions in the proof of Winkler. A standard bookkeeping yields  \eqref{Winkler_semigroup_1}, \eqref{Winkler_semigroup_2}.

{\color{blue}Consider the heat equation 
\begin{align}
\pa_t h=\frac{1}{A}\de h, \quad h(t=0,X)=h_{\text{in}}(X),\, \overline{h_{\text{in}}}=0,\,\, X\in\Torus^2.
\end{align} Through standard energy estimate, we have that 
\begin{align}
\|h(t)\|_{L^2(\Torus^2)}\leq e^{-\frac{\lambda_1t}{A} }\|h_{\text{in}}\|_{L^2(\Torus^2)},\quad \forall t\geq 0.
\end{align}
Here $\lambda_1>0$ is the first nonzero eigenvalue of the negative Laplacian $-\de$. 

On the other hand, the Green's function of the heat equation on the torus can be explicitly spelled out
\begin{align}\label{heat_kernel}
h(t,X)=\sum_{(m, n)\in \mathbb{Z}^2}\frac{1}{4\pi t/A}\int_{Y\in \Torus^2}e^{-\frac{|X'+(m L, n L)-X|^2}{4 t/A}}h_{\text{in}}(X')dX'.
\end{align}
Direct estimate of the Green's function yields that 
\begin{align}\label{heat_LpLq_decay}
\|h(t)\|_{L_X^p(\Torus^2)}\leq C \left(\frac{t}{A}\right)^{-\frac{2}{2}(\frac{1}{q}-\frac{1}{p})}\|h_{\text{in}}\|_{L_X^q(\Torus^2)},\quad ( \forall t\in[0, 2A]??).
\end{align}\textcolor{red}{(Check!)}

For $p\in[1,2)$, application of the H\"older inequality yields that
\begin{align}
\|h(t)\|_{L^p(\Torus^2)}\leq |\Torus|^{\frac{4}{2-p}}\|h(t)\|_{L^2(\Torus^2)}\leq C e^{-\lambda_1\frac{(t-1)}{A}}\|h(A)\|_{L^2}\leq Ce^{-\lambda_1\frac{(t-A)}{A}}\|h_{\text{in}}\|_{L^q},\quad\forall t\geq 2A. 
\end{align}
If $p\in [2,\infty]$, we have
\begin{align}
\|h(t)\|_{L^p(\Torus^2)}\leq C\|h(t-A)\|_{L^2(\Torus^2)}\leq Ce^{-\lambda_1(\frac{t-2A}{A})}\|h(A)\|_{L^2(\Torus^2)}\leq Ce^{-\lambda_1\left(\frac{t}{A}-2\right)}\|h_{\text{in}}\|_{L^q(\Torus^2)},\quad \forall t\geq 2A.
\end{align}
For $t\leq 2A$, the estimate  \eqref{Winkler_semigroup_1} is direct from \eqref{heat_LpLq_decay}. Hence the proof of \eqref{Winkler_semigroup_1} is completed. 

In order to prove \eqref{Winkler_semigroup_2}, we first note that {\bf(Check!)}
\begin{align}
\|\na e^{\frac{t}{A}\de} h_{\text{in}}\|_{L^p(\Torus^2)}\leq C\left(\frac{t}{A}\right)^{-1/2}\|h_{\text{in}}\|_{L^p(\Torus^2)},\quad \forall t\leq A.
\end{align} 
For $t\leq 2A$, we can estimate the gradient with \eqref{Winkler_semigroup_1} as follows
\begin{align}
\|\na e^{\frac{t}{A}\de}h_{\text{in}}\|_{L^p(\Torus^2)}=&\|(\na e^{\frac{t}{2A}\de})e^{\frac{t}{2A}\de} h_{\text{in}}\|_{L^p(\Torus^2)}\leq C\left(\frac{t}{2A}\right)^{-1/2}\|e^{\frac{t}{2	A}\de}h_{\text{in}}\|_{L^p(\Torus^2)}\\
\leq& C\left(\frac{t}{2A}\right)^{-1/2}C_1\left(1+\left(\frac{t}{2A}\right)^{-\frac{2}{2}(\frac{1}{q}-\frac{1}{p})}\right)\|h_{\text{in}}\|_{L^q(\Torus^2)}.
\end{align}
{\bf I am not sure whether $\na e^{\frac{t}{A}\de}=(\na e^{\frac{t}{2A}\de})e^{\frac{t}{2A}\de}$ is justified. It feels like there is a $`2'$ factor?}
For $t\geq 2A$, we estimate it as follows:
\begin{align}
\|\na e^{\frac{t}{A}\de}h_{\text{in}}\|_{L^p(\Torus^2)}=&\|(\na e^{\de})e^{\frac{t-A}{A}\de} h_{\text{in}}\|_{L^p(\Torus^2)}\leq C\|h({t}-{A})\|_{L^p(\Torus^2)}\\
\leq&C\left(1+\left(\frac{t}{A}\right)^{-\frac{2}{2}(\frac{1}{q}-\frac{1}{p})}\right)e^{-\lambda_1 (\frac{t}{A}-1)}\|h_{\text{in}}\|_{L^q(\Torus^2)}.
\end{align}
This completes the proof.
 }\fi
\end{proof}}

\section{Technical Lemmas in Linear Theory}\label{sec:tchlem_l}
In this subsection, we collect some technical lemmas/corollaries applied in Section \ref{Sec:Lnr}.

\begin{lem}\label{lem:cm_AB}
Let $\mathcal{A}$ and $\mathcal{B}$ be two linear operators. Define the operation 
\begin{align}
\mathrm{ad}_\mathcal{A}^n (\mathcal{B}):=\underbrace{[\mathcal{A},[\mathcal{A},[\mathcal{A},...[\mathcal{A}}_{n},\mathcal{B}]]...],\quad  n\in\{1,2, 3, ...\}.
\end{align} 
Then
\begin{align}
[\mathcal{A}^n, \mathcal{B}]
=\sum_{\ell=0}^{n-1}\binom{n}{\ell} \mathrm{ad}_\mathcal{A}^{n-\ell} (\mathcal{B})\mathcal{A}^\ell. \label{[An,B]}
\end{align}
\end{lem}
\begin{proof}\footnote{This proof is kindly suggested to me by Yiyue Zhang.}
To prove the lemma, we need two extra definitions, i.e., the left multiplication $\mathcal{L}_\mathcal{A}$ and $\mathcal{R}_\mathcal{A}$
\begin{align}
\mathcal{L}_\mathcal{A} \mathcal{B}=\mathcal{AB},\quad \mathcal{R}_{\mathcal{A}}\mathcal{B}=\mathcal {BA}.
\end{align} 
Standard computation yields the following three relations: 
\begin{align}
\mathrm{ad}_\mathcal{A}\mathcal{B}=(\mathcal{L}_{\mathcal{A}}-\mathcal{R}_{\mathcal{A}})(\mathcal B),\quad 
\mathcal{L}_\mathcal{A}\mathcal{R}_\mathcal{A} \mathcal{B}=\mathcal{ABA}=\mathcal{R}_\mathcal{A}\mathcal{L}_\mathcal{A} \mathcal{B},\quad
\mathrm{ad}_\mathcal{A} \mathcal{R}_\mathcal{A}(\mathcal B)=\mathcal{R}_\mathcal{A}\mathrm{ad}_\mathcal{A}(\mathcal B).
\end{align}
Now we have that \begin{align}
[\mathcal{A}^n, \mathcal{B}]=&\mathcal{A}^n \mathcal{B}-\mathcal{BA}^n=(\mathrm{ad}_\mathcal{A}+\mathcal{R}_\mathcal{A})^n (\mathcal{B})-\mathcal{R}_\mathcal{A}^n (\mathcal{B})=\sum_{\ell=0}^{n}\binom{n}{\ell} \mathrm{ad}_\mathcal{A}^{n-\ell}\mathcal{R}_\mathcal{A}^{\ell } (\mathcal{B})-\mathcal{R}_\mathcal{A}^n (\mathcal{B})\\
=&\sum_{\ell=0}^{n-1}\binom{n}{ \ell}\mathcal{R}_\mathcal{A}^{\ell }\mathrm{ad}_\mathcal{A}^{n-\ell} (\mathcal{B})=\sum_{\ell=0}^{n-1}\binom{n}{\ell} \mathrm{ad}_\mathcal{A}^{n-\ell} (\mathcal{B})\mathcal{A}^{\ell }  .
\end{align}
\end{proof} 

The follow corollary of Lemma \ref{lem:cm_AB} is useful when we compute the commutators involving $\Gamma_y$ and $\de$.
\begin{lem}\label{cor:cm}Consider the vector field $\Gamma_y=\Gamma_{y;t}$ \eqref{Gamma}, and recall the notation \eqref{B_m_mt}: $B^{(m)}(t,y):=\int_0^t \pa_y^{m}u(s,y)ds$. 
The following commutator relations hold
\begin{subequations}\label{cm_est}
\begin{align}
[\Gamma_y,\pa_y]=&-\int_0^t \pa_{yy} u(s,y)ds\pa_x=-B^{(2)}\pa_x,\label{cm_G_y}\\
[\Gamma_y,\pa_{yy}]=&-2 B^{(2)}\pa_x\Gamma_y+\pa_y(B^{(1)})^2\pa_{xx}-B^{(3)}\pa_x,\label{cm_G_yy}\\
 [\Gamma_y^j,\pa_y]=&-\sum_{\ell=0}^{j-1}\binom{j}{\ell}B^{(j-\ell+1)}\pa_x\Gamma_y^{\ell},\label{cm_yt_j_y}\\
 [\Gamma_y^j,\pa_{yy}]=&\sum_{\ell=0}^{j-1}\binom{j}{\ell}\lf(-2B^{(j-\ell+1)}\pa_x\Gamma_y+\pa_y^{(j-\ell)}(B^{(1)})^2 \pa_{xx}-B^{(j-\ell+2)}\pa_x\rg)\Gamma_y^\ell.\label{cm_yt_j_yy}
\end{align}\end{subequations}
\end{lem}
\begin{proof}
The first two relations \eqref{cm_G_y}, \eqref{cm_G_yy} are direct consequences of computation. Next we use the general relation \eqref{[An,B]} to derive \eqref{cm_yt_j_y}, \eqref{cm_yt_j_yy}. First, we observe that 
\begin{align}\label{adG_Bm_pax}
\mathrm{ad}_{\Gamma_y}(-{B}^{(m)}\pa_x)=[\pa_y+B^{(1)}\pa_x,-{B}^{(m)}\pa_x]=-B^{(m+1)}\pa_x.
\end{align} 
As a result, we have the general formula
\begin{align}
\mathrm{ad}_{\Gamma_y}^{m}(\pa_y)=-B^{(m+1)}\pa_x,\quad \forall m\in\mathbb{Z}_+.
\end{align}
Hence by \eqref{[An,B]},
\begin{align}
[\Gamma_y^j, \pa_y ]=\sum_{\ell=0}^{j-1}\binom {j}{\ell}\mathrm{ad}_{\Gamma_y}^{j-\ell}(\pa_y)\Gamma_y^\ell=-\sum_{\ell=0}^{j-1}\binom{j}{\ell}B^{(j-\ell+1)}\pa_x\Gamma_y^\ell.
\end{align}
This concludes the proof of \eqref{cm_yt_j_y}. 

Finally, we derive \eqref{cm_yt_j_yy}. By \eqref{[An,B]}, it is enough to compute $\mathrm{ad}_{\Gamma_y}^m(\pa_{yy})$. We apply the induction to prove that
\begin{align}\label{ad_G_m_yy}
\mathrm{ad}_{\Gamma_y}^m(\pa_{yy})=-2B^{{(m+1)}}\pa_x\Gamma_y +\pa_y^{m}(B^{(1)})^2\pa_{xx}-B^{(m+2)}\pa_x,\quad \forall m\in\mathbb{Z}_+.
\end{align} 
This relation, when combined with \eqref{[An,B]}, implies \eqref{cm_yt_j_yy}. The relation \eqref{ad_G_m_yy} holds when $m=1$ (\eqref{cm_G_yy}). Assume that the relation holds on the $m-1(\geq 1)$ level, then
\begin{align}
\mathrm{ad}_{\Gamma_y}^{m}(\pa_{yy})=&\mathrm{ad}_{\Gamma_y}(\mathrm{ad}_{\Gamma_y}^{m-1}(\pa_{yy})) 
= \mathrm{ad}_{\Gamma_y}\lf(-B^{(m+1)}\pa_x-2B^{{(m)}}\pa_x\Gamma_y +\pa_y^{m-1}(B^{(1)})^2\pa_{xx}\rg).
\end{align}  
Recalling \eqref{adG_Bm_pax}, we have 
\begin{align*}
\mathrm{ad}_{\Gamma_y}^{m}(\pa_{yy})=&-B^{(m+2)}\pa_x-\Gamma_y(2B^{(m)})\pa_x\Gamma_y+\Gamma_y(\pa_y^{m-1}(B^{(1 )})^2)\pa_{xx}\\
=&-B^{(m+2)}\pa_x-2B^{(m+1)}\pa_x\Gamma_y+\pa_y^m(B^{(1)})^2\pa_{xx}.
\end{align*}
This concludes the proof of \eqref{ad_G_m_yy}. Hence the proof of the lemma is complete.
\end{proof}

We also use the following lemma.\myc{(I didn't found a reason to include the $(\cdot)_\nq$ restriction. So I remove them. Please double check!)}
\begin{lem}\label{lem:gliding}
Consider the gliding vector fields $\Gamma_y=\pa_y+\int_0^t\pa_y u(s,y)ds\pa_x$, and functions $f,g\in H^m(\Torus^3)$. Then the following two estimate hold \myc{(Check!)}
\begin{align}
\|f\|_{L_{x,y,z}^\infty}\leq& C \sum_{|i,j,k|\leq 2}\|\pa_x^i\Gamma_y^j\pa_z^k f\|_{L_{x,y,z}^2}.\label{Linfty_estgld}
\end{align}
Moreover, we have that for $m\geq 3$,
\begin{align}\label{Prd_est_gld}\qquad\lf(
\sum_{|i,j,k|=0}^{M}\|\pa_x^i \Gamma_y^j\pa_z^k (f g)\|_{L_{x,y,z}^2}^{2}\rg)^{1/2}\leq &C\left(\sum_{|i,j,k|=0}^{M}\|\pa_x^i \Gamma_y^j\pa_z^k f\|_{L_{x,y,z}^2}^2\right)^{1/2} \left(\sum_{|i,j,k|=0}^{M}\|\pa_x^i \Gamma_y^j\pa_z^k g\|_{L_{x,y,z}^2}^2\right)^{1/2}.
\end{align}
\myc{Previous\begin{align}
\sum_{|i,j,k|=0}^{M}\|\pa_x^i \Gamma_y^j\pa_z^k (f g)\|_{L_{x,y,z}^2}\leq &C\left(\sum_{|i,j,k|=0}^{M}\|\pa_x^i \Gamma_y^j\pa_z^k f\|_{L_{x,y,z}^2}\right) \left(\sum_{|i,j,k|=0}^{M}\|\pa_x^i \Gamma_y^j\pa_z^k g\|_{L_{x,y,z}^2}\right).
\end{align}
But the two norms are equivalent, i.e.,
\begin{align}
\sum_{i=1}^N |f_i|\approx (\sum_{i=1}^N |f_i|^2)^{1/2}
.\end{align} So there should be no problem in the main text. }
\end{lem}
\begin{proof}\myc{There is an alternative simpler proof: By the Fourier transform in $x,z$-variable, the Minkowski's integral inequality and the Gagliardo-Nirenberg inequality, 
\begin{align}
\|f&\|_{L_{x,y,z}^\infty}= C\lf\|\sum_{\al,\gamma\in \mathbb{Z}}\wh f(\al,y,\gamma)e^{i\al x}e^{i\gamma z}\rg\|_{L_{x,y,z}^\infty}\leq C \sum_{\al,\gamma\in \mathbb{Z}}\lf\|\wh f(\al,y,\gamma)e^{i\al x}e^{i\gamma z}\rg\|_{L_{x,y,z}^\infty}=C \sum_{\al,\gamma\in \mathbb{Z}}\lf\|\wh f(\al,\cdot,\gamma)\rg\|_{L_{y}^\infty}\\
\leq &C \sum_{\al,\gamma\in \mathbb{Z}}\lf\|\wh f(\al,\cdot, \gamma)\rg\|_{L_y^2}^{1/2}\lf\|\pa_y\lf(\wh f(\al,\cdot, \gamma)\exp\lf\{i\al\int_0^t u(s,\cdot) ds\rg\}\rg)\rg\|_{L_y^2}^{1/2}+C \sum_{\al,\gamma\in \mathbb{Z}}\lf\|\wh f(\al,\cdot, \gamma)\rg\|_{L_y^2}.
\end{align}
Now we focus on the case where $\al,\gamma\in \mathbb{Z}\backslash \{0\}$, and the other cases are similar. By invoking the H\"older inequality, we have that
\begin{align}
&\sum_{\al,\gamma\in \mathbb{Z}\backslash\{0\}}\lf\|\wh f(\al,\cdot,\gamma)\rg\|_{L_{y}^\infty}\\
&\leq C\lf(\sum_{\al,\gamma\in \mathbb{Z}\backslash\{0\}}\lf\|(|\al|+|\gamma|)^2\wh f\rg \|_{L^2}^2\rg)^{1/4}\lf(\sum_{\al,\gamma\in \mathbb{Z}\backslash\{0\}}\lf\|(|\al|+|\gamma|)\lf(\pa_y+i\al \int_0^t\pa_y uds \rg)\wh f\rg \|_{L^2}^2\rg)^{1/4}\\
&\qquad+C\lf(\sum_{\al,\gamma\in \mathbb{Z}\backslash\{0\}}\lf\|(|\al|+|\gamma|)^2\wh f\rg \|_{L^2}^2\rg)^{1/2}\\ 
&\leq C\sum_{|i,j,k|\leq 2}\lf\|\pa_x^i\Gamma_k^j\pa_z^k f\rg\|_{L^2}. 
\end{align}}
 Consider the following change of variables:
\begin{align}
X=x-\int_0^t u(s,y )ds,\quad Y=y ,\quad Z=z .
\end{align}
The Jacobian of this change of coordinate is $1$. We also check that it is one-one and onto. Define
\begin{align}
F(X,Y,Z)=f(x,y,z),\quad G(X,Y, Z)=g(x,y,z). 
\end{align} Now by classical Sobolev embedding, we have 
\begin{align}
\|f\|_{L^\infty_{x,y,z}}=\|F\|_{L^{\infty}_{X,Y,Z}}\leq C\sum_{|i,j,k|\leq 2}\|\pa_X^i\pa_Y^j\pa_Z^kF\|_{L^2_{X,Y,Z}}.
\end{align}\myc{(Check! $\|F\|_\infty\leq C\|\pa_X\pa_Y\pa_ZF\|_1$ not quite work for $F(X,Y,Z)=F(Y,Z)$. Therefore we need lower order term. $\|F\|_\infty\leq C\sum_{i,j,k\in\{0,1\}}\| \pa_X^i\pa_Y^j\pa_Z^k F\|_{L^1_{X,Y,Z}}?$)}
It is direct to check that 
\begin{align}
\pa_X=\pa_x, \, \pa_Y=\Gamma_y, \,\pa_Z =\pa_z. 
\end{align}
Hence,
\begin{align}
\|f\|_{L^\infty_{x,y,z}}\leq  C\lf(\sum_{|i,j,k|\leq 2}\|\pa_X^i\pa_Y^j\pa_Z^kF\|_{L^2_{X,Y,Z}}^2\rg)^{1/2} = C\lf(\sum_{|i,j,k|\leq 2}\|\pa_x^i\Gamma_y^j\pa_z^k f \|_{L^2_{x,y,z}} ^2\rg)^{1/2}.
\end{align}
This is \eqref{Linfty_estgld}.
Next we recall the product estimate in the $(X,Y,Z)$-coordinate
\begin{align}
\|FG\|_{H^m_{X,Y,Z}}\leq C
\|F\|_{H^m_{X,Y,Z}}\|G\|_{L_{X, Y, Z}^\infty}+C
\|G\|_{H^m_{X,Y,Z}}\|F\|_{L_{X, Y, Z}^\infty}.
\end{align}
As a result,
\begin{align*}
\bigg(\sum_{|i,j,k|=0}^{M}\|\pa_x^i  &\Gamma_y^j\pa_z^k (f g)\|_{L_{x,y,z}^2}^2\bigg)^{1/2}
\leq C\|FG\|_{H^m_{X,Y,Z}} 
\leq C\|F\|_{H^m_{X,Y,Z}}\|G\|_{L_{X, Y, Z}^\infty}+C
\|G\|_{H^m_{X,Y,Z}}\|F\|_{L_{X, Y, Z}^\infty}\\
\leq &C\bigg(\sum_{|i,j,k|=0}^{M}\|\pa_x^i \Gamma_y^j\pa_z^k f \|_{L_{x,y,z}^2}^2\bigg)^{1/2}\|g\|_{L^\infty_{x,y,z}}+C\bigg(\sum_{|i,j,k|=0}^{M}\|\pa_x^i \Gamma_y^j\pa_z^k g\|_{L_{x,y,z}^2}^2\bigg)^{1/2}\|f\|_{L^\infty_{x,y,z}}.&
\end{align*}
Combining this with \eqref{Linfty_estgld} yields \eqref{Prd_est_gld}.
\end{proof}

To conclude this section, we present a lemma.
\begin{lem} Recall the vector field $\Gamma_y=\Gamma_{y;t}$ \eqref{Gamma}. The gliding regularity norms of the gradient of functions $f\in H^{M+1}$ are bounded as follows, 
\begin{align}\label{Gld_Reg_grd}
\sum_{|i,j,k|=0}^{M}\lf\|\Gamma_{y;t}^{ijk}\na f\rg\|_{L^p}
\leq& {\sum_{|i,j,k|=0}^{ M+1}\lf\|\Gamma_{y;t}^{ijk} f\rg\|_{L^p}+C t  \sum_{ |i,j,k|=0}^ { M} \lf \|  \Gamma_{y;t}^{(i+1)jk} f\rg\|_{L^p},\quad\ \Gamma_{y;t}^{ijk}=\pa_x^i\Gamma_{y;t}^j\pa_z^k.}  
\end{align}
{Here the constant $C$ depends only on $M,\, \|\pa_y u\|_{L_t^\infty   W^{M,\infty}}$ and $p\in[1,\infty]$.}
\end{lem}
\begin{proof}Recalling the definition of commutators, we have that
\begin{align*}
\sum_{|i,j,k|=0}^{M}\|\pa_x^i\Gamma_{y;t}^j\pa_z^k\na f_\nq\|_{L^p}\leq &\sum_{|i,j,k|=0}^{M}\|\na \pa_x^i\Gamma_{y;t}^j\pa_z^k f_\nq\|_{L^p}+\sum_{|i,j,k|=0}^{M}\mathbbm{1}_{j\geq 1}\|[\Gamma_{y;t}^{j},\pa_y]\pa_{x}^{i}\pa_z^{k} f_\nq\|_{L^p}.
\end{align*}
Direct applications of \eqref{cm_yt_j_y} and the relation $\pa_y=\Gamma_{y;t}-\int_0^t \pa_yu(s,y) ds \pa_x$ yield that
\begin{align}
\sum_{|i,j,k|=0}^ M\|\pa_x^i\Gamma_{y;t}^j\pa_z^k\na f_\nq\|_{L^p}\leq &\sum_{|i,j,k|=0}^ { M+1}\|\pa_x^i\Gamma_{y;t}^j\pa_z^k f_\nq\|_{L^p}+C\sum_{ |i,j,k|=0 }^  M\int_0^t\| \pa_y u(s,\cdot)\|_{ L^{\infty}_y}ds  \|\Gamma_{y;t}^{(i+1)jk} f_\nq\|_{L^p}\\
&+C\sum_{|i,j,k|=0}^{ M}\mathbbm{1}_{j\geq 1}\sum_{j'=0}^{j-1}\binom{j}{j'}\left(\int_0^t\|\pa_y^{j-j'+1}u(s,\cdot)\|_{L_y^\infty}ds\right) \|\pa_{x}^{i+1}\Gamma_{y;t}^{j'}\pa_z^k f_\nq\|_{L^p}.
\end{align}
Now direct estimation yields the result.
\end{proof}

\noindent
\thanks{\textbf{Acknowledgment.} SH was supported in part by NSF grants DMS 2006660, DMS 2304392, DMS 2006372. The author would like to thank Tarek Elgindi for providing the key construction for the time-dependent shear flows and many crucial suggestions,  discussions and guidance. The author would also like to thank Yiyue Zhang for his suggestion of a clean proof of Lemma  \ref{lem:cm_AB} and Zhongtian Hu for finding out typos in the first version of the manuscript. }

\small
\bibliographystyle{abbrv}
\bibliography{nonlocal_eqns,JacobBib,SimingBib}

\end{document}